\documentclass[12pt]{article}
\usepackage{natbib}

\usepackage{hyperref}

\usepackage{a4wide}
\usepackage{amscd}
\usepackage{enumerate}

\usepackage{graphicx}

\usepackage{amsmath,multirow}
\usepackage{array}
\usepackage{mathptmx}
\usepackage[utf8]{inputenc}
\usepackage{mathtools}
\usepackage{amsfonts}
\usepackage{ed}
\usepackage[mathscr]{eucal}

\usepackage{amsthm}

\theoremstyle{plain}
\newtheorem{theorem}{Theorem}[section]
\newtheorem{lemma}{Lemma}[section]
\newtheorem{remark}{Remark}

\newtheorem{definition}{Definition}[section]

\numberwithin{equation}{section}

\begin{document}
%
%










{
  \title{\bf Run and Frequency quotas for $q$-binary trials}
 \author{Jungtaek Oh\thanks{Corresponding Author:
 (e-mail: jungtaekoh0191@gmail.com)}\hspace{.2cm}\\
}
  \maketitle
} 

\begin{abstract}
We study the distributions
of waiting times in variations of the $q$-sooner and later waiting
time problem. One variation imposes length and frequency quotas on
the runs of successes and failures. Another case considers binary
trials for which the probability of ones is geometrically varying. First, we consider sooner case (or later case) impose the run quotas both runs of successes and failures. Theorem \ref{thm:3.1} gives an probability function of the $q$-sooner waiting time distribution of order $(k_1,k_2)$ when a quota is imposed on runs of successes and failures. Theorem \ref{thm:3.2} gives an probability function of the $q$-later waiting time distribution of order $(k_1,k_2)$ when a quota is imposed on runs of successes and failures. Next, we consider sooner case (or later case) impose a frequency quota on successes and a run quota on failures. Theorem \ref{thm:4.1} gives an probability function of the $q$-sooner waiting time distribution impose a frequency quota on successes and a run quota on failures. Theorem \ref{thm:4.2} gives an probability function of the $q$-later waiting time distribution impose a frequency quota on successes and a run quota on failures. Next, we consider sooner case (or later case) impose a run quota on successes and a frequency quota on failures. Theorem \ref{thm:4.3} gives an probability function of the $q$-sooner waiting time distribution impose a run quota on successes and a frequency quota on failures. Theorem \ref{thm:4.4} gives an probability function of the $q$-later waiting time distribution impose a run quota on successes and a frequency quota on failures. Next, we consider sooner case (or later case) impose a frequency quotas on successes and failures. Theorem \ref{thm:5.1} gives an probability function of the $q$-sooner waiting time distribution impose a frequency quotas on successes and failures. Theorem \ref{thm:5.3} gives an probability function of the $q$-later waiting time distribution impose a frequency quotas on both successes and failures.

We also study the distributions of Longest run under the same variations. The main theorems are sooner and later waiting time problem and the joint distribution of the length of the longest success and longest failure runs when a run and frequency quotas imposed on runs of successes and failure. In the present work, we consider a sequence of independent binary zero and one trials with not necessarily identical distributed with probability of ones varying according to a geometric rule. Exact formulae for the distributions obtained by means of enumerative combinatorics.






\end{abstract}
\noindent%
{\it Keywords:}  Sooner/later waiting time problems, Run quota, Frequency quota, Longest run, Runs, Binary trials, $q$-Distributions
\vfill
\tableofcontents
\section{Introduction}
\citet{charalambides2010a} studied discrete $q$-distributions on Bernoulli trials with a geometrically varying success probability. Let us consider a sequence $X_{1}$,...,$X_{n}$ of zero(failure)-one(success) Bernoulli trials, such that the trials of the subsequence after the $(i-1)$st zero until the $i$th zero are independent with equal failure probability. For $i>0$ the $i$'s geometric sequences of trials is the subsequence after the $(i-1)$'st zero. The probability $q_i$ of zero is fixed for all trials during the $i$'th geometric sequence and is given by
\noindent
\begin{equation}\label{failureprobofq}
\begin{split}
q_{i}=1-\theta q^{i-1},\quad 0\leq\theta\leq1,\quad 0\leq q<1.
\end{split}
\end{equation}
\noindent
By the standard definitions it follows from \eqref{failureprobofq} that the success probability $1-q_i$ is geometrically decreasing. Let $S_{j}^{(0)}=\Sigma_{m=1}^{j}(1-X_{m})$ denote the number of zeros in the first $j$ trials. Because the probability of zeros at the $i$th geometric sequence of trials is in fact the conditional probability $q_{j,i}$ of occurrence of a zero at any trial $j$ given the occurrence of exactly $i-1$ zeros in the previous trials. We can write as follows.
\noindent
\begin{equation}\label{failureprobofq2}
\begin{split}
q_{j,i}=p\Big(X_{j}=0\ \bigm|\ S_{j-1}^{(0)}=i-1\Big)=1-\theta q^{i-1},\quad \text{for}\ i\leq j.
\end{split}
\end{equation}
\noindent
To make the preceding more clear, we consider the example of the sequence 111011110111100110 for $n=18$. Each subsequence in brackets has own success and failure probabilities according to the geometric rule.
\begin{figure}[t]
\noindent
  \begin{center}
\includegraphics[scale=0.65]{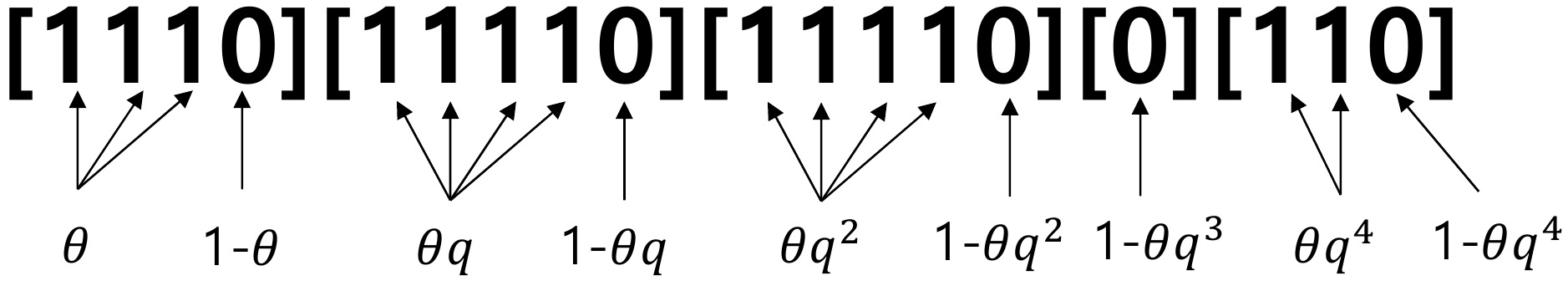}
\end{center}
\end{figure}

This stochastic model \eqref{failureprobofq} or \eqref{failureprobofq2} has interesting applications, studied as a reliability growth model by \cite{Dubman&Sherman1969}, and applies to a $q$-boson theory in physics by \cite{Jing&Fan1994} and \cite{Jing1994}. This stochastic model \eqref{failureprobofq} also applies to start-up demonstration tests, as a sequential-intervention model which is proposed by \cite{BBV1995}.\\

The stochastic model \eqref{failureprobofq} is the $q$-analogue of the classical binomial distribution with geometrically varying probabilities of zeros. This provides a stochastic model for independent identically distributed (IID) trials with failure probability function $\pi$ given by
\noindent
\begin{equation}\label{bernoullifailureprob}
\pi_{j}=P\left(X_{j}=0\right)=1-\theta,\ j>0,
\end{equation}
\noindent
where $0\leq \theta<1$.

As $q$ tends toward $1$, the stochastic model \eqref{failureprobofq} reduces  to IID(Bernoulli) model \eqref{bernoullifailureprob}, since $q_{i}\rightarrow\pi_{i}$, $i=1,2,\ldots$ or $q_{j,i}\rightarrow1-\theta$, $i=1,2,\ldots,j$, $j=1,2,\ldots.$\\

The discrete $q$-distributions based on the stochastic model of the sequence of independent Bernoulli trials have been investigated by numerous researchers, for a lucid review and comprehensive list of publications on this area the interested reader may consult the monographs by \cite{charalambides2010a,Charalambides2010b,Charalambides2016}.\\

From a Mathematical and Statistical point of view, \cite{Charalambides2016} mentioned in the preface of his book \textit{"It should be noticed that a stochastic model of a sequence of independent Bernoulli trials, in which the probability of success at a trial is assumed to vary with the number of trials and/or the number of successes, is advantageous in the sense that it permits incorporating the experience gained from previous trials and/or successes. If the probability of success at a trial is a very general function of the number of trials and/or the number successes, very little can be inferred from it about the distributions of the various random variables that may be defined on this model. The assumption that the probability of success (or failure) at a trial varies geometrically, with rate (proportion) $q$, leads to the introduction of discrete $q$-distributions"}.\\

Following the work of \citet{philippou1982waiting}, \citet{philippou1983generalized} and \citet{aki1984discrete}, the distribution theory of runs and patterns has been intensively developed in the last few decades, because of its theoretical interest and applications in a wide variety of research areas such as hypothesis testing, system reliability, quality control, physics, psychology, radar astronomy, molecular biology, computer science, insurance, and finance. During the past few decades up to recently, the meaningful progress on runs and pattern statistics has been wonderfully surveyed in \citet{balakrishnan2003runs} as well as in \citet{fu2003distribution} and references therein. Furthermore, there are some more recent contributions on the topic such as \citet{arapis2018}, \citet{eryilmaz2018}, \citet{kong2019}, \citet{makri2019}, and \citet{aki2019}.\\

Waiting time distributions related to runs and patterns received much attention recently in applied probability. In many situations, it is necessary to study the waiting time problems associated with two or more runs or patterns; for example, an experiment stops or a system fails whenever one of several predefined runs or patterns occurs. Sometimes, the order among these runs or patterns is important and must be taken into consideration; for example, the DNA sequence of a virus contains certain patterns that occur in order. We introduce another important class of waiting time distributions called sooner waiting time distributions. We will consider both success and failure runs of fixed lengths, instead of simply focus on success runs of fixed length. An interesting type of distribution is the waiting time until the first occurrence of a success run of fixed length or a failure run of fixed length in a sequence of Bernoulli trials was introduced by \cite{feller1968introduction} for the first time, instead of just concentrate on success runs of fixed length. The related problems, named as \textit{sooner/later succession} or \textit{run quota problems}. Let $X_{1},X_{2},\ldots$ be a sequence of binary trials, each resulting in either a Success (S) or Failure (F) and fix two positive integers $k_{1}\geq 1$, $k_{2}\geq 1$. Let $T_{k_{1},k_{2}}^{(S)}$ denote the waiting time until a sequence of $k_{1}$ consecutive successes or $k_{2}$ consecutive failures are observed for the first time.\\

The distribution of $T_{k_{1},k_{2}}^{(S)}$ named as \textit{sooner succession quota waiting time distribution}, or \textit{sooner run quota waiting time distribution}, or simply \textit{sooner waiting time distribution} by \cite{ebneshahrashoob1990sooner}. \citet{ling1990geometric} named as \textit{geometric distribution of order} $(k_{1},k_{2})$.\\

The sooner waiting time random variable $T_{k_{1},k_{2}}^{(S)}$ is closely related with the random variable $T_{k_{1},k_{2}}^{(L)}$, which denote stopping the trials when both a success run of length $k_1$ and a failure run of length $k_2$ have been observed. This distribution named as \textit{later succession quota waiting time distribution}, or \textit{later run quota waiting time distribution}, or simply \textit{later waiting time distribution}.
\citet{feller1968introduction} applied renewal theory to derive the p.g.f. of sooner waiting time distribution by using recurrent events. \citet{ebneshahrashoob1990sooner} first introduced the generalized probability generating function (gpgf), which provides a mechanism to take into account certain partitions into the probability space of the random variable of interest. They investigated probability generating functions of the sooner and later waiting time problems. \citet{ling1992generalization} and \cite{sobel1992quota} studied the sooner and later problems when a frequency quota is imposed on runs of successes and failures. \citet{aki1993discrete} derived the p.g.f. of the sooner and later waiting time problems in the first order Markov dependent trials. \citet{balasubramanian1993sooner} studied sooner and later waiting time problems in a markov correlated Bernoulli trials. \citet{aki1996sooner} studied sooner and later waiting time problems in a higher order Markov dependent trials. \citet{aki1999sooner} studied sooner and later waiting time problems in a Markov dependent bivariate trials. \citet{koutras1997sooner} studied the problem in the trinary trials.
 \citet{han1999joint} studied sooner and later waiting problems for runs in both identical and nonidentical independent sequences of multi-state trials. \citet{han2003sooner} investigated sooner and later waiting time problems for patterns in multi-state Markov dependent trials. \citet{kim2013waiting} derived probabilty distribution of sooner waiting time for both independent and homogeneous two-state Markovian Bernoulli trials, using a generalized Fibonacci sequence of order $k$. \citet{eryilmaz2016generalized} studied the generalized sooner waiting time problems in a sequence of trinary trials with random rewards. \citet{kim2019sooner} investigated the sooner waiting time problems in a sequence of independent multi-state trials with random rewards, using the analysis of a Markov chain with a special transition structure. They are generalized the sooner waiting time problems of \citet{eryilmaz2016generalized} for a sequence of trinary trials with random rewards in the case of multi-state trials.\\

 The random variable is the $L_{n}$ length of the longest success run in $n$ trials. Longest runs which can be considered as the extremes of run lengths. Denote by $R_{n}^{(j)}$ is the number of runs of type $j$ $(j=0,1)$ and let $x_{i}(y_{i})$ denote the length of the $i$th run of successes (failures) in $X_{1}$, $\ldots$,$X_{n}$, respectively. Let $L_{n}^{(j)}$ is denote denote the longest run of type $j$($j =0, 1$). More precisely, a fixed number $n$,
\begin{equation*}
\begin{split}
L_{n}^{(j)}&=\text{max}\Big\{b\ \bigm|\ X_{a+1}=\cdots=X_{a+b}=j\ \wedge\ 0\leq a\leq n-b\Big\}\\
&=\underset{1\leq j \leq R_{n}^{(j)}}{\text{max}}x_{j}.
\end{split}
\end{equation*}
To make more clear we mention by way of example that a sequence of ten trials $0011110110$. We have $R_{10}^{(0)}=3$, $R_{10}^{(1)}=2$, $L_{n}^{(0)}=2$ and $L_{n}^{(1)}=4$.\\
\citet{burr1961longest} studied the exact expression for the distribution of the longest run in i.i.d. Bernoulli trials. \citet{fu1994distribution} obtain the distribution of the longest run, using a Markov chain imbedding technique. \citet{fu1996distribution} developed a method of finite Markov-chain imbedding technique to obtain the exact distribution of $L_n$ in sequences of multi-state trials. \citet{wendy1996runs} studied the joint distribution of the length of the longest success run and the number of successes, as well as the conditional distribution of the length of the longest success run given the number of successes in sequences of Bernoulli trials. \citet{vaggelatou2003length} investigated the asymptotic results for the longest run statistics in a multi state markov chain. \citet{makri2007shortest} derived the exact probability function of the the longest run statistics for both linear and circular sequence by means of combinatorial analysis. Many authors have contributed to the longest run statistics in the cases of Bernoulli trials, Markov dependent trials, and exchangeable binary trials. See. \citet{philippou1985longest,philippou1986successes,philippou1990longest},  \citet{antzoulakos1997probability,antzoulakos1998longest}, \citet{charalambides1994success}, \citet{eryilmaz2005distribution,eryilmaz2005longest,eryilmaz2006some}, \citet{muselli1996simple}.\\

As mentioned in \cite{johnson2005univariate}, "\textit{The formulas for the probabilities for these distributions are not at first sight very illuminating; also in some cases there is more than one correct published expression for the probabilities}". \citet{fu1994distribution} mentioned that some of the formulae are extremely complicated and hard to compute. To the best of our knowledge, this is the first work that derive the exact probability functions of sooner and later waiting time problem and longest run statistics in the binary trials with probability of ones varying according to a geometric rule.\\

In this paper we study of the $q$-sooner/later waiting time distribution of order $(k_1,k_2)$ and longest run statistics when a quota is imposed on runs of successes and failures, with probability of ones varying according to a geometric rule. Run (or succession) quota is based on the length of the current run in that trial and frequency quota is based on the total frequency accumulated in that trial (i.e. number of the totality of frequency).\\

The paper is organized as follows. In section 2 we introduce basic definitions and necessary notations that will be useful throughout  this article. In Section 3 we shall study of the $q$-sooner and later waiting time distribution of order $(k_1,k_2)$ when a quota is imposed on runs of successes and failures. We derive exact probability function of $q$-sooner and later waiting time by means of combinatorial analysis. In section 4 we shall study of the $q$-sooner and later waiting times to be discussed will arise by setting quotas on both runs and frequencies of successes and failures. First, we consider sooner cases (or later cases) impose a frequency quota on successes and a run quota on failures. Next, we consider sooner cases (or later cases) impose a run quota on successes and a frequency quota on failures. In section 5 we study the joint distribution of the lengths of the longest success and longest failure runs.

\section{Terminology and notation}
We first recall some definitions, notation and known results in which will be used in this paper. Throughout the paper, we suppose that $0<q<1$. First, we introduce the following notation.
\begin{itemize}
\item $L_n^{(1)}:$ the length of the longest run of successes in $X_{1}, X_{2}, \ldots, X_{n}$;
\item $L_n^{(0)}:$ the length of the longest run of failures in $X_{1}, X_{2}, \ldots, X_{n}$;
\item $S_{n}:$ the total number of successes in $X_{1}, X_{2}, \ldots, X_{n}$;
\item $F_{n}:$ the total number of failures in $X_{1}, X_{2}, \ldots, X_{n}$.
\end{itemize}
\noindent
Next, let us introduce some basic $q$-sequences and functions and their properties, which are useful in the sequel. The $q$-shifted factorials are defined as
\begin{equation}\label{}
(a;q)_{0}=1,\quad (a;q)_{n} =\prod_{k=0}^{n-1}(1-aq^{k}),\quad (a;q)_{\infty} =\prod_{k=0}^{\infty}(1-aq^{k}).
\end{equation}
\noindent
Let $m$, $n$ and $i$ be positive integer and $z$ and $q$ be real numbers, with $q\neq 1$. The number $[z]_{q} = (1-q^{z})/(1-q)$ is called $q$-number and in particular $[z]_{q}$ is called $q$-integer. The $m$ th order factorial of the $q$-number $[z]_{q}$, which is defined by
\begin{equation}\label{}
\begin{split}
[z]_{m,q}&=\prod_{i=1}^{m}[z-i+1]_{q}= [z]_{q} [z-1]_{q}\cdots[z-m+1]_{q}\\
&=\frac{(1-q^{z})(1-q^{z-1})\cdots(1-q^{z-m+1})}{(1-q)^{m}},\ z=1,2,\ldots,\ m=0,1,\ldots,z.
\end{split}
\end{equation}
\noindent
is called $q$-factorial of $z$ of order $m$. In particular, $[m]_{q}! = [1]_{q}[2]_{q}...[m]_{q}$ is called $q$-factorial of $m$. The $q$-binomial coefficient (or Gaussian polynomial) is defined by
\noindent
\begin{equation}\label{}
\begin{split}
\begin{bmatrix}
n\\
m
\end{bmatrix}_{q}
&=\frac{[n]_{m,q}}{[m]_{q}!}=\frac{[n]_{q}!}{[m]_{q}![n-m]_{q}!}=\frac{(1-q^{n})(1-q^{n-1})\cdots(1-q^{n-m+1})}{(1-q^{m})(1-q^{m-1})\cdots(1-q)}\\
&=\frac{(q;q)_{n}}{(q;q)_{m}(q;q)_{n-m}},\ m=1,2,\ldots,
\end{split}
\end{equation}
\noindent
The $q$-binomial ($q$-Newton's binomial) formula is expressed as
\noindent
\begin{equation}\label{}
\prod_{i=1}^{n}(1+zq^{i-1})=\sum_{k=0}^{n}q^{k(k-1)/2}\begin{bmatrix}
n\\
k
\end{bmatrix}_{q}
z^{k},\ -\infty<z<\infty,\ n=1,2,\ldots.
\end{equation}
\noindent
For $q\rightarrow 1$ the $q$-analogs tend to their classical counterparts, that is
\noindent
\begin{equation*}\label{eq: 1.1}
\begin{split}
\lim\limits_{q\rightarrow 1}
\begin{bmatrix}
n\\
r
\end{bmatrix}_{q}
={n \choose r}
\end{split}
\end{equation*}
\noindent
Let us consider again a sequence of independent geometric sequences of trials with probability of failure at the $i$th geometric sequence of trials given by \eqref{failureprobofq} or \eqref{failureprobofq2}. We are interesting now is focused on the study of the number of successes in a given number of trials in this stochastic model.
\begin{definition}
Let $Z_{n}$ be the number of successes in a sequence of $n$ independent Bernoulli trials, with probability of success at the $i$th geometric sequence of trials given by \eqref{failureprobofq} or \eqref{failureprobofq2}. The distribution of the random variable $Z_{n}$ is called $q$-binomial distribution, with parameters $n$, $\theta$, and $q$.
\end{definition}
\noindent
Let we introduce a $q$-analogue of the binomial distribution with the probability function of the number $Z_{n}$ of successes in $n$ trials $X_{1}, \ldots, X_{n}$ is given by
\begin{equation}\label{q-binomialdist}
\begin{split}
P_{q,\theta}\{Z_{n}=r\}=
\begin{bmatrix}
n\\
r
\end{bmatrix}_{q}
\theta^{r}\prod_{i=1}^{n-r}(1-\theta q^{i-1}),
\end{split}
\end{equation}\\
for $r = 0, 1, . . . , n$, $0 < q < 1$. The distribution  is called a $q$-binomial distribution. For $q \rightarrow 1$, because
\begin{equation*}\label{eq: 1.1}
\begin{split}
\lim\limits_{q\rightarrow 1}
\begin{bmatrix}
n\\
r
\end{bmatrix}_{q}
={n \choose r}
\end{split}
\end{equation*}
so that the q-binomial distribution converges to the usual binomial distribution as $q \rightarrow 1$, as follows
\begin{equation}
P_{\theta}\left(Z_{n}=r\right)={n\choose r}\theta^{r}(1-\theta)^{n-r},\ r=0,1,\ldots,n,
\end{equation}
with parameters $n$ and $\theta$.
The $q$-binomial distribution studied by \citet{charalambides2010a,Charalambides2016}, which is connected with $q$-Berstein polynomial. \citet{Jing1994} introduced probability function \eqref{q-binomialdist} as a $q$-deformed binomial distribution, also derived recurrence relation of its probability distribution.
\noindent
In the sequel, $P_{q ,\theta}(.)$ and $P_{\theta}(.)$ denote probabilities related with the stochastic model \eqref{failureprobofq} and \eqref{bernoullifailureprob}, respectively.
\newpage

\section{Succession (run) quota for $q$-binary trials}

{\rm
In this section, we shall study of the $q$-sooner and later waiting time distribution of order $(k_1,k_2)$ when a quota is imposed on runs of successes and failures. Let $E_{1}$ be the event that a $k_{1}$ consecutive successes (ones) occurs, and $E_{0}$ be the event that a $k_{2}$ consecutive failures (zeros) occurs.

Let $W$ be a nonnegative integer valued random variable defined on a probability space $(\Omega,F,P).$ Suppose that we are given a partition of $\Omega$, i.e. $\Omega$ can be written as a disjoint union of finite or countably many subsets $\Omega_{i}$ of $\Omega$. Then the PMF of event $W=n$ can be expressed as
\noindent
\begin{equation*}\label{eq: 1.1}
\begin{split}
P_{q,\theta}(W=n)=\sum P_{q,\theta}\left(\{\omega\in \Omega : W(\omega)=n\}\right),\ 0<\theta<1,\ 0<q<1.
\end{split}
\end{equation*}

\noindent

First, we solve the sooner waiting time problem by means of combinatorial analysis. Let $W_{S}$ be a random variable denoting that the waiting time until either $k_{1}$ consecutive successes or $k_{2}$ consecutive failures are occurred, whichever event observe sooner. To be precise, we let $\Omega_{1}=\{\omega \in \Omega | \omega_{n} \in E_1 \setminus E_2 \}$ and $\Omega_{2}=\{\omega_{n} \in \Omega| \omega \in E_2 \setminus E_1 \},$ where $\omega_n$ for every $\omega$ as the string of the first $n$ bits of $\omega$. The support (range set) of $W_S$ is
\noindent
\begin{equation*}
\begin{split}
\mathfrak{R}(W_{S})=\{\text{min}(k_1,k_2),\ldots\}.
\end{split}
\end{equation*}
\noindent

Next, we solve the later waiting time problem by means of combinatorial analysis. Let $W_{L}$ be a random variable denoting that the waiting time until both of  $k_{1}$ consecutive successes and $k_{2}$ consecutive failures are occurred, whichever event observe later. To be precise, we let $\Omega_{3}=\{\omega_n \in \Omega| \Omega_{2}\cap E_1\}$ and $\Omega_{4}=\{\omega_n \in \Omega| \Omega_{1}\cap E_2\}.$ The support (range set) of $W_L$ is
\noindent
\begin{equation*}
\begin{split}
\mathfrak{R}(W_{L})=\{k_1+k_2,\ldots\}.
\end{split}
\end{equation*}
\noindent

As an illustration we mention that, suppose that $E_1$ means four consecutive ones: $1111$, and $E_2$ means three consecutive zeros: $000$. And we consider a bitstring $S = 11010001001111000\ldots$ . Then $S$ belongs to both events $E_1$ and $E_2$. But the beginning string $S' = 1101000$ belongs only to $E_2$ and not to $E_1$, so it belongs to $E_2\setminus E_1$, and the definition of $\Omega_2$ can be expressed by saying this precisely. Similarly the beginning string $S" = 11010001001111$ belongs to the intersection of $\Omega_2$ and $E_1$, which means that $E_1$ comes later than $E_2$ in $S$.
}

\subsection{$q$-sooner waiting time distribution of order $(k_1,k_2)$}

The problem of waiting time that will be discussed in this section is one of the 'sooner cases' and it emerges when a run quota is imposed on runs of successes and failures. More specifically, Binary (zero and one) trials with probability of ones varying according to a geometric rule, are performed sequentially until $k_1$ consecutive successes or $k_2$ consecutive failures are observed, whichever event occurs first. Let $W_S^{(1)}$  be a random variable denoting that the waiting time until  $k_{1}$ consecutive successes are observed for the first time and no failure run of length $k_2$ or more has appeared before and $W_S^{(0)}$  be a random variable denoting that the waiting time until  $k_{2}$ consecutive failures are observed for the first time and no success run of length $k_1$ or more has appeared before.

We will also simply write $f^{(1)}_{q,S}(n;\theta) = P_{q,\theta}(W_S^{(1)}=n)$ and $f^{(0)}_{q,S}(n;\theta)=P_{q,\theta}(W_{S}^{(0)}=n)$. Therefore we can write the probability of event $W_{S}=n$ as follows
\begin{equation*}\label{eq: 1.1}
\begin{split}
P_{q,\theta}(W_{S}=n)=f^{(1)}_{q,S}(n;\theta)+f^{(0)}_{q,S}(n;\theta),\ n\geq \text{min}(k_{1},\ k_{2}).
\end{split}
\end{equation*}
\noindent
We now make some useful Definition and Lemma for the proofs of Theorem in the sequel.
\noindent
\begin{definition}
For $0<q\leq1$, we define

\begin{equation*}\label{eq: 1.1}
\begin{split}
A_{q}^{k_1,k_2}(m,r,s)={\sum_{x_{1},\ldots,x_{s-1}}}\
\sum_{y_{1},\ldots,y_{s}} q^{y_{1}x_{1}+(y_{1}+y_{2})x_{2}+\cdots+(y_{1}+\cdots+y_{s-1})x_{s-1}},\
\end{split}
\end{equation*}
\noindent
where the summation is over all integers $x_1,\ldots,x_{s-1},$ and $y_1,\ldots,y_s$ satisfying
\noindent

\begin{equation*}\label{eq:1}
\begin{split}
0<x_{j}<k_{1}\ \text{for}\ j=1,\ldots,s-1,
\end{split}
\end{equation*}
\noindent
\begin{equation*}\label{eq:1}
\begin{split}
x_{1}+\cdots+x_{s-1}=m,\ \text{and}
\end{split}
\end{equation*}
\noindent
\begin{equation*}\label{eq:2}
\begin{split}
0<y_{j}<k_{2}\ \text{for}\ j=1,\ldots,s,
\end{split}
\end{equation*}
\noindent
\begin{equation*}\label{eq:2}
\begin{split}
y_{1}+\cdots+y_{s}=r.
\end{split}
\end{equation*}

\end{definition}

\noindent
The following gives a recurrence relation useful for the computation of $A_{q}^{k_1,k_2}(m,r,s)$.\\
\noindent

\begin{lemma}
\label{lemma:3.1}
For $0<q\leq1$, $A_{q}^{k_1,k_2}(m,r,s)$ obeys the following recurrence relation.
\noindent
\begin{equation*}\label{eq: 1.1}
\begin{split}
A&_{q}^{k_1,k_2}(m,r,s)\\
&=\left\{
  \begin{array}{ll}
    \sum_{a=1}^{k_{1}-1}\sum_{b=1}^{k_{2}-1}q^{a(r-b)}A_{q}^{k_1,k_2}(m-a,r-b,s-1) & \text{for}\ s>1,\ s-1\leq m\leq(s-1)(k_{1}-1)\\
    &\text{and}\ s\leq r\leq s(k_{2}-1), \\
    1 & \text{for}\ s=1,\ m=0,\ \text{and}\ 1\leq r\leq k_{2}-1, \text{or}\\
    0 & \text{otherwise.}\\
  \end{array}
\right.
\end{split}
\end{equation*}
\end{lemma}

\begin{proof}
For $s > 1$, $ s-1\leq m\leq(s-1)(k_{1}-1)$ and $s\leq r\leq s(k_{2}-1)$, we observe that since $x_{s-1}$ can assume the values $1,\ldots,k_{1}-1$, then $A_{q}^{k_1,k_2}(m,r,s)$ can be written as
\noindent

\begin{equation*}
\begin{split}
A&_{q}^{k_1,k_2}(m,r,s)\\
=&\sum_{x_{s-1}=1}^{k_1-1}{\sum_{\substack{x_{1}+\cdots+x_{s-2}=m-x_{s-1}\\ x_{1},\ldots,x_{s-2} \in \{1,\ldots,k_{1}-1\}}}}\quad
{\sum_{\substack{y_{1}+\cdots+y_{s}=r\\ y_{1},\ldots,y_{s} \in \{1,\ldots,k_{2}-1\}}}}q^{x_{s-1}(r-y_{s})}\ q^{y_{1}x_{1}+(y_{1}+y_{2})x_{2}+\cdots+(y_{1}+\cdots+y_{s-2})x_{s-2}}.
\end{split}
\end{equation*}
Similarly, we observe that since $y_s$ can assume the values $1,\ldots,k_{2}-1$, then $A_{q}^{k_1,k_2}(m,r,s)$ can be rewritten as

\begin{equation*}
\begin{split}
A&_{q}^{k_1,k_2}(m,r,s)\\
=&\sum_{y_{s}=1}^{k_2-1}\sum_{x_{s-1}=1}^{k_1-1}q^{x_{s-1}(r-y_{s})}{\sum_{\substack{x_{1}+\cdots+x_{s-2}=m-x_{s-1}\\ x_{1},\ldots,x_{s-2} \in \{1,\ldots,k_{1}-1\}}}}\quad
{\sum_{\substack{y_{1}+\cdots+y_{s-1}=r-y_s\\ y_{1},\ldots,y_{s-1} \in \{1,\ldots,k_{2}-1\}}}} q^{y_{1}x_{1}+(y_{1}+y_{2})x_{2}+\cdots+(y_{1}+\cdots+y_{s-2})x_{s-2}}\\
=&\sum_{a=1}^{k_{1}-1}\sum_{b=1}^{k_{2}-1}q^{a(r-b)}A_{q}^{k_1,k_2}(m-a,r-b,s-1).
\end{split}
\end{equation*}
The other cases are obvious and thus the proof is completed.
\end{proof}

\begin{remark}
{\rm
We observe that $A_1^{k_1,k_2}(m,r,s)$ is the number of integer solutions $(x_{1},\ \ldots,\ x_{s-1})$ and $(y_{1},\ \ldots,\ y_{s})$ of
\noindent
\begin{equation*}\label{eq:1}
\begin{split}
0<x_{j}<k_{1}\ \text{for}\ j=1,\ldots,s-1,
\end{split}
\end{equation*}
\noindent
\begin{equation*}\label{eq:1}
\begin{split}
x_{1}+\cdots+x_{s-1}=m,\ \text{and}
\end{split}
\end{equation*}
\noindent
\begin{equation*}\label{eq:2}
\begin{split}
0<y_{j}<k_{2}\ \text{for}\ j=1,\ldots,s,
\end{split}
\end{equation*}
\noindent
\begin{equation*}\label{eq:2}
\begin{split}
y_{1}+\cdots+y_{s}=r
\end{split}
\end{equation*}
\noindent
which is
\noindent
\begin{equation*}\label{eq: 1.1}
\begin{split}
A_{1}^{k_1,k_2}(m,r,s)=S(s-1,\ k_{1},\ m)S(s,\ k_{2},\ r),
\end{split}
\end{equation*}
\noindent
where $S(a,\ b,\ c)$ denotes the total number of integer solutions $x_{1}+x_{2}+\cdots+x_{a}=c$ such that $0<x_{i}<b$ for $i=1,2,\ldots,a$. Alternatively, it is the number of ways of distributing $c$ identical balls into $a$ different cells with no containing more or equal than $b$ balls. The number can be expressed as
\noindent
\begin{equation*}\label{eq: 1.1}
\begin{split}
S(a,\ b,\ c)=\sum_{j=0}^{min(a,\left[\frac{c-a}{b-1}\right])}(-1)^{j}{a \choose j}{c-j(b-1)-1 \choose a-1}.
\end{split}
\end{equation*}
\noindent
See, e.g. \citet{charalambides2002enumerative}.
}
\end{remark}

\begin{definition}

For $0<q\leq1$, we define
\noindent
\begin{equation*}\label{eq: 1.1}
\begin{split}
B_{q}^{k_1,k_2}(m,r,s)={\sum_{x_{1},\ldots,x_{s}}}\
\sum_{y_{1},\ldots,y_{s}} q^{y_{1}x_{2}+(y_{1}+y_{2})x_{3}+\cdots+(y_{1}+\cdots+y_{s-1})x_{s}},\
\end{split}
\end{equation*}
\noindent
where the summation is over all integers $x_1,\ldots,x_{s},$ and $y_1,\ldots,y_s$ satisfying
\noindent
\begin{equation*}\label{eq:1}
\begin{split}
0<x_{j}<k_{1}\ \text{for}\ j=1,\ldots,s,
\end{split}
\end{equation*}
\noindent
\begin{equation*}\label{eq:1}
\begin{split}
x_{1}+\cdots+x_{s}=m,\ \text{and}
\end{split}
\end{equation*}
\noindent
\begin{equation*}\label{eq:2}
\begin{split}
0<y_{j}<k_{2}\ \text{for}\ j=1,\ldots,s,
\end{split}
\end{equation*}
\noindent
\begin{equation*}\label{eq:2}
\begin{split}
y_{1}+\cdots+y_{s}=r.
\end{split}
\end{equation*}

\end{definition}

\noindent
The following gives a recurrence relation useful for the computation of $B_{q}^{k_1,k_2}(m,r,s)$.
\noindent

\begin{lemma}
\label{lemma:3.2}
For $0<q\leq1$, $B_{q}^{k_1,k_2}(m,r,s)$ obeys the following recurrence relation.
\noindent
\begin{equation*}\label{eq: 1.1}
\begin{split}
B&_{q}^{k_1,k_2}(m,r,s)\\
&=\left\{
  \begin{array}{ll}
    \sum_{a=1}^{k_{1}-1}\sum_{b=1}^{k_{2}-1}q^{a(r-b)}B_{q}^{k_1,k_2}(m-a,r-b,s-1) & \text{for}\ s>1,\ s\leq m\leq s(k_{1}-1)\\
    &\text{and}\ s\leq r\leq s(k_{2}-1), \\
    1 & \text{for}\ s=1,\ 1\leq m\leq k_1-1\\
     &\text{and}\ 1\leq r\leq k_{2}-1, \text{or}\\
    0 & \text{otherwise.}\\
  \end{array}
\right.
\end{split}
\end{equation*}

\end{lemma}

\begin{proof}

For $s > 1$, $s\leq m\leq s(k_{1}-1)$ and $s\leq r\leq s(k_{2}-1)$, we observe that since $x_{s}$ can assume the values $1,\ldots,k_{1}-1$, then $B_{q}^{k_1,k_2}(m,r,s)$ can be written as
\noindent
\begin{equation*}
\begin{split}
B&_{q}^{k_1,k_2}(m,r,s)\\
=&\sum_{x_{s}=1}^{k_1-1}{\sum_{\substack{x_{1}+\cdots+x_{s-1}=m-x_{s}\\ x_{1},\ldots,x_{s-1} \in \{1,\ldots,k_{1}-1\}}}}\quad
{\sum_{\substack{y_{1}+\cdots+y_{s}=r\\ y_{1},\ldots,y_{s} \in \{1,\ldots,k_{2}-1\}}}}q^{x_{s}(r-y_{s})}\ q^{y_{1}x_{2}+(y_{1}+y_{2})x_{3}+\cdots+(y_{1}+\cdots+y_{s-2})x_{s-1}}.
\end{split}
\end{equation*}
\noindent
Similarly, we observe that since $y_s$ can assume the values $1,\ldots,k_{2}-1$, then $B_{q}^{k_1,k_2}(m,r,s)$ can be rewritten as
\noindent
\begin{equation*}
\begin{split}
B&_{q}^{k_1,k_2}(m,r,s)\\
=&\sum_{y_{s}=1}^{k_2-1}\sum_{x_{s}=1}^{k_1-1}q^{x_{s}(r-y_{s})}{\sum_{\substack{x_{1}+\cdots+x_{s-1}=m-x_{s}\\ x_{1},\ldots,x_{s-1} \in \{1,\ldots,k_{1}-1\}}}}\quad
{\sum_{\substack{y_{1}+\cdots+y_{s-1}=r-y_s\\ y_{1},\ldots,y_{s-1} \in \{1,\ldots,k_{2}-1\}}}}\ q^{y_{1}x_{2}+(y_{1}+y_{2})x_{3}+\cdots+(y_{1}+\cdots+y_{s-2})x_{s-1}}\\
=&\sum_{a=1}^{k_{1}-1}\sum_{b=1}^{k_{2}-1}q^{a(r-b)}B_{q}^{k_1,k_2}(m-a,r-b,s-1).
\end{split}
\end{equation*}
The other cases are obvious and thus the proof is completed.
\end{proof}

\begin{remark}
{\rm
We observe that $B_1^{k_1,k_2}(m,r,s)$ is the number of integer solutions $(x_{1},\ldots,x_{s})$ and $(y_{1},\ldots,y_{s})$ of
\noindent
\begin{equation*}\label{eq:1}
\begin{split}
0<x_{j}<k_{1}\ \text{for}\ j=1,\ldots,s,
\end{split}
\end{equation*}
\noindent
\begin{equation*}\label{eq:1}
\begin{split}
x_{1}+\cdots+x_{s}=m,\ \text{and}
\end{split}
\end{equation*}
\noindent
\begin{equation*}\label{eq:2}
\begin{split}
0<y_{j}<k_{2}\ \text{for}\ j=1,\ldots,s,
\end{split}
\end{equation*}
\noindent
\begin{equation*}\label{eq:2}
\begin{split}
y_{1}+\cdots+y_{s}=r
\end{split}
\end{equation*}
\noindent
which is
\noindent
\begin{equation*}\label{eq: 1.1}
\begin{split}
B_{1}^{k_1,k_2}(m,r,s)=S(s,\ k_{1},\ m)S(s,\ k_{2},\ r),
\end{split}
\end{equation*}
\noindent
where $S(a,\ b,\ c)$ denotes the total number of integer solutions $x_{1}+x_{2}+\cdots+x_{a}=c$ such that $0<x_{i}<b$ for $i=1,2,\ldots,a$. The number can be expressed as
\noindent
\begin{equation*}\label{eq: 1.1}
\begin{split}
S(a,\ b,\ c)=\sum_{j=0}^{a}(-1)^{j}{a \choose j}{c-j(b-1)-1 \choose a-1}.
\end{split}
\end{equation*}
\noindent
See, e.g. \citet{charalambides2002enumerative}.}
\end{remark}

\begin{definition}
For $0<q\leq1$, we define
\noindent
\begin{equation*}\label{eq: 1.1}
\begin{split}
C_{q}^{k_1,k_2}(m,r,s)={\sum_{x_{1},\ldots,x_{s}}}\
\sum_{y_{1},\ldots,y_{s-1}} q^{y_{1}x_{2}+(y_{1}+y_{2})x_{3}+\cdots+(y_{1}+\cdots+y_{s-2})x_{s-1}+(y_{1}+\cdots+y_{s-1})x_{s}},\
\end{split}
\end{equation*}
\noindent
where the summation is over all integers $x_1,\ldots,x_{s},$ and $y_1,\ldots,y_{s-1}$ satisfying
\noindent
\begin{equation*}\label{eq:1}
\begin{split}
0<x_{j}<k_{1}\ \text{for}\ j=1,\ldots,s,
\end{split}
\end{equation*}
\noindent
\begin{equation*}\label{eq:1}
\begin{split}
x_{1}+\cdots+x_{s}=m,\ \text{and}
\end{split}
\end{equation*}
\noindent
\begin{equation*}\label{eq:2}
\begin{split}
0<y_{j}<k_{2}\ \text{for}\ j=1,\ldots,s-1,
\end{split}
\end{equation*}
\noindent
\begin{equation*}\label{eq:2}
\begin{split}
y_{1}+\cdots+y_{s-1}=r.
\end{split}
\end{equation*}
\end{definition}
\noindent
The following gives a recurrence relation useful for the computation of $C_{q}^{k_1,k_2}(m,r,s)$.
\noindent
\begin{lemma}
\label{lemma:3.3}
For $0<q\leq1$, $C_{q}^{k_1,k_2}(m,r,s)$ obeys the following recurrence relation.
\noindent
\begin{equation*}\label{eq: 1.1}
\begin{split}
C&_{q}^{k_1,k_2}(m,r,s)\\
&=\left\{
  \begin{array}{ll}
    \sum_{a=1}^{k_{1}-1}\sum_{b=1}^{k_{2}-1}q^{ra}C_{q}(m-a,r-b,s-1) & \text{for}\ s>1,\ s\leq m\leq s(k_{1}-1)\\
    &\text{and}\ s-1\leq r\leq (s-1)(k_{2}-1), \\
    1 & \text{for}\ s=1,\ 1\leq m\leq k_1-1\ \text{and}\ r=0, \text{or}\\
    0 & \text{otherwise.}\\
  \end{array}
\right.
\end{split}
\end{equation*}

\end{lemma}
\begin{proof}
For $s > 1$, $s\leq m\leq s(k_{1}-1)$ and $s-1\leq r\leq (s-1)(k_{2}-1),$ we observe that since $x_{s}$ can assume the values $1,\ldots,k_{1}-1$, then $C_{q}^{k_1,k_2}(m,r,s)$ can be written as
\noindent
\begin{equation*}
\begin{split}
C&_{q}^{k_1,k_2}(m,r,s)\\
=&\sum_{x_{s}=1}^{k_1-1}{\sum_{\substack{x_{1}+\cdots+x_{s-1}=m-x_{s}\\ x_{1},\ldots,x_{s-1} \in \{1,\ldots,k_{1}-1\}}}}\quad
{\sum_{\substack{y_{1}+\cdots+y_{s}=r\\ y_{1},\ldots,y_{s} \in \{1,\ldots,k_{2}-1\}}}}q^{rx_{s}}\ q^{y_{1}x_{2}+(y_{1}+y_{2})x_{3}+\cdots+(y_{1}+\cdots+y_{s-2})x_{s-1}}.
\end{split}
\end{equation*}
\noindent
Similarly, we observe that since $y_s$ can assume the values $1,\ldots,k_{2}-1$, then $C_{q}^{k_1,k_2}(m,r,s)$ can be rewritten as

\noindent
\begin{equation*}
\begin{split}
C&_{q}^{k_1,k_2}(m,r,s)\\
=&\sum_{y_{s}=1}^{k_2-1}\sum_{x_{s}=1}^{k_1-1}q^{rx_{s}}{\sum_{\substack{x_{1}+\cdots+x_{s-1}=m-x_{s}\\ x_{1},\ldots,x_{s-1} \in \{1,\ldots,k_{1}-1\}}}}\quad
{\sum_{\substack{y_{1}+\cdots+y_{s-1}=r-y_s\\ y_{1},\ldots,y_{s-1} \in \{1,\ldots,k_{2}-1\}}}}\ q^{y_{1}x_{2}+(y_{1}+y_{2})x_{3}+\cdots+(y_{1}+\cdots+y_{s-2})x_{s-1}}\\
=&\sum_{a=1}^{k_{1}-1}\sum_{b=1}^{k_{2}-1}q^{ra}C_{q}^{k_1,k_2}(m-a,r-b,s-1).
\end{split}
\end{equation*}
The other cases are obvious and thus the proof is completed.
\end{proof}

\begin{remark}
{\rm
We observe that $C_1^{k_1,k_2}(m,r,s)$ is the number of integer solutions $(x_{1},\ldots,x_{s})$ and $(y_{1},\ldots,y_{s-1})$ of
\noindent
\begin{equation*}\label{eq:1}
\begin{split}
0<x_{j}<k_{1}\ \text{for}\ j=1,\ldots,s,
\end{split}
\end{equation*}
\noindent
\begin{equation*}\label{eq:1}
\begin{split}
x_{1}+\cdots+x_{s}=m,\ \text{and}
\end{split}
\end{equation*}
\noindent
\begin{equation*}\label{eq:2}
\begin{split}
0<y_{j}<k_{2}\ \text{for}\ j=1,\ldots,s-1,
\end{split}
\end{equation*}
\noindent
\begin{equation*}\label{eq:2}
\begin{split}
y_{1}+\cdots+y_{s-1}=r
\end{split}
\end{equation*}
\noindent
which is
\noindent
\begin{equation*}\label{eq: 1.1}
\begin{split}
C_{1}^{k_1,k_2}(m,r,s)=S(s,\ k_{1},\ m)S(s-1,\ k_{2},\ r),
\end{split}
\end{equation*}
\noindent
where $S(a,\ b,\ c)$ denotes the total number of integer solutions $x_{1}+x_{2}+\cdots+x_{a}=c$ such that $0<x_{i}<b$ for $i=1,\ 2,\ldots,a$. The number can be expressed as
\noindent
\begin{equation*}\label{eq: 1.1}
\begin{split}
S(a,\ b,\ c)=\sum_{j=0}^{a}(-1)^{j}{a \choose j}{c-j(b-1)-1 \choose a-1}.
\end{split}
\end{equation*}
\noindent
See, e.g. \citet{charalambides2002enumerative}..
}
\end{remark}

\begin{definition}
For $0<q\leq1$, we define
\noindent
\begin{equation*}\label{eq: 1.1}
\begin{split}
D_{q}^{k_1,k_2}(m,r,s)={\sum_{x_{1},\ldots,x_{s}}}\
\sum_{y_{1},\ldots,y_{s}} q^{y_{1}x_{1}+(y_{1}+y_{2})x_{2}+\cdots+(y_{1}+\cdots+y_{s-1})x_{s-1}+(y_{1}+\cdots+y_{s})x_{s}},\
\end{split}
\end{equation*}
\noindent
where the summation is over all integers $x_1,\ldots,x_{s},$ and $y_1,\ldots,y_{s}$ satisfying
\noindent
\begin{equation*}\label{eq:1}
\begin{split}
0<x_{j}<k_{1}\ \text{for}\ j=1,\ldots,s,
\end{split}
\end{equation*}
\noindent
\begin{equation*}\label{eq:1}
\begin{split}
x_{1}+\cdots+x_{s}=m,\ \text{and}
\end{split}
\end{equation*}
\noindent
\begin{equation*}\label{eq:2}
\begin{split}
0<y_{j}<k_{2}\ \text{for}\ j=1,\ldots,s,
\end{split}
\end{equation*}
\noindent
\begin{equation*}\label{eq:2}
\begin{split}
y_{1}+\cdots+y_{s}=r.
\end{split}
\end{equation*}
\end{definition}
\noindent
The following gives a recurrence relation useful for the computation of $D_{q}^{k_1,k_2}(m,r,s)$.
\noindent

\begin{lemma}
\label{lemma:3.4}
For $0<q\leq1$, $D_{q}^{k_1,k_2}(m,r,s)$ obeys the following recurrence relation.
\noindent
\begin{equation*}\label{eq: 1.1}
\begin{split}
D&_{q}^{k_1,k_2}(m,r,s)\\
&=\left\{
  \begin{array}{ll}
    \sum_{a=1}^{k_{1}-1}\sum_{b=1}^{k_{2}-1}q^{ra}D_{q}^{k_1,k_2}(m-a,r-b,s-1) & \text{for}\ s>1,\ s\leq m\leq s(k_{1}-1)\\
    &\text{and}\ s\leq r\leq s(k_{2}-1), \\
    1 & \text{for}\ s=1,\ 1\leq m\leq k_1-1\\
    &\text{and}\ 1\leq r\leq k_2-1, \text{or}\\
    0 & \text{otherwise.}\\
  \end{array}
\right.
\end{split}
\end{equation*}
\end{lemma}
\begin{proof}
For $s > 1$, $s\leq m\leq s(k_{1}-1)$ and $s\leq r\leq s(k_{2}-1)$, we observe that since $x_{s}$ can assume the values $1,\ldots,k_{1}-1$, then $D_{q}^{k_1,k_2}(m,r,s)$ can be written as
\noindent
\begin{equation*}
\begin{split}
D&_{q}^{k_1,k_2}(m,r,s)\\
=&\sum_{x_{s}=1}^{k_1-1}{\sum_{\substack{x_{1}+\cdots+x_{s-1}=m-x_{s}\\ x_{1},\ldots,x_{s-1} \in \{1,\ldots,k_{1}-1\}}}}\quad
{\sum_{\substack{y_{1}+\cdots+y_{s}=r\\ y_{1},\ldots,y_{s} \in \{1,\ldots,k_{2}-1\}}}}q^{rx_{s}}\ q^{y_{1}x_{1}+(y_{1}+y_{2})x_{2}+\cdots+(y_{1}+\cdots+y_{s-1})x_{s-1}}.
\end{split}
\end{equation*}
\noindent
Similarly, we observe that since $y_s$ can assume the values $1,\ldots,k_{2}-1$, then $D_{q}^{k_1,k_2}(m,r,s)$ can be rewritten as
\noindent
\begin{equation*}
\begin{split}
D&_{q}^{k_1,k_2}(m,r,s)\\
=&\sum_{y_{s}=1}^{k_2-1}\sum_{x_{s}=1}^{k_1-1}q^{rx_{s}}{\sum_{\substack{x_{1}+\cdots+x_{s-1}=m-x_{s}\\ x_{1},\ldots,x_{s-1} \in \{1,\ldots,k_{1}-1\}}}}\quad
{\sum_{\substack{y_{1}+\cdots+y_{s-1}=r-y_s\\ y_{1},\ldots,y_{s-1} \in \{1,\ldots,k_{2}-1\}}}}\ q^{y_{1}x_{1}+(y_{1}+y_{2})x_{2}+\cdots+(y_{1}+\cdots+y_{s-1})x_{s-1}}\\
=&\sum_{a=1}^{k_{1}-1}\sum_{b=1}^{k_{2}-1}q^{ra}D_{q}^{k_1,k_2}(m-a,r-b,s-1).
\end{split}
\end{equation*}
The other cases are obvious and thus the proof is completed.
\end{proof}
\begin{remark}
{\rm
We observe that $D_1^{k_1,k_2}(m,r,s)$ is the number of integer solutions $(x_{1},\ldots,x_{s})$ and $(y_{1},\ldots,y_{s})$ of
\noindent
\begin{equation*}\label{eq:1}
\begin{split}
0<x_{j}<k_{1}\ \text{for}\ j=1,\ldots,s,
\end{split}
\end{equation*}
\noindent
\begin{equation*}\label{eq:1}
\begin{split}
x_{1}+\cdots+x_{s}=m,\ \text{and}
\end{split}
\end{equation*}
\noindent
\begin{equation*}\label{eq:2}
\begin{split}
0<y_{j}<k_{2}\ \text{for}\ j=1,\ldots,s,
\end{split}
\end{equation*}
\noindent
\begin{equation*}\label{eq:2}
\begin{split}
y_{1}+\cdots+y_{s}=r
\end{split}
\end{equation*}
\noindent
which is
\noindent
\begin{equation*}\label{eq: 1.1}
\begin{split}
D_{1}^{k_1,k_2}(m,r,s)=S(s,\ k_{1},\ m)S(s,\ k_{2},\ r),
\end{split}
\end{equation*}
\noindent
where $S(a,\ b,\ c)$ denotes the total number of integer solutions $x_{1}+x_{2}+\cdots+x_{a}=c$ such that $0<x_{i}<b$ for $i=1,2,\ldots,a$. The number can be expressed as
\noindent
\begin{equation*}\label{eq: 1.1}
\begin{split}
S(a,\ b,\ c)=\sum_{j=0}^{a}(-1)^{j}{a \choose j}{c-j(b-1)-1 \choose a-1}.
\end{split}
\end{equation*}
\noindent
See, e.g. \citet{charalambides2002enumerative}.
}
\end{remark}

\noindent
The probability function of the $q$-sooner waiting time distribution of order $(k_1,k_2)$ when a quota is imposed on runs of successes and failures is obtained by the following theorem. It is evident that
\noindent
\begin{equation*}
P_{q,\theta}(W_{S}=n)=0\ \text{for}\ 0\leq n<\text{min}(k_1,k_2)
\end{equation*}
\noindent
and so we shall focus on determining the probability mass function for $n\geq\min(k_{1},\ k_{2})$.
\noindent

\begin{theorem}
\label{thm:3.1}
The PMF $f_{q,S}(n;\theta)=P_{q,\theta}(W_{S}=n)$ satisfies
\noindent
\begin{equation*}\label{eq:bn}
\begin{split}
f_{q,S}(n;\theta)=f^{(1)}_{q,S}(n;\theta)+f^{(0)}_{q,S}(n;\theta),\ \text{for}\ n\geq\min(k_{1},\ k_{2}),
\end{split}
\end{equation*}
\noindent
where $f^{(1)}_{q,S}(k_1;\theta)=\theta^{k_{1}}$, $f^{(0)}_{q,S}(k_2;\theta)={\prod_{j=1}^{k_2}}(1-\theta q^{j-1})$,
\noindent
\begin{equation*}\label{eq:bn1}
\begin{split}
f^{(1)}_{q,S}(n;\theta)=\sum_{i=1}^{n-k_{1}}\theta^{n-i} q^{ik_{1}}{\prod_{j=1}^{i}}(1-\theta q^{j-1})\sum_{s=1}^{i}\bigg[&A_{q}^{k_1,k_2}(n-k_{1}-i,\ i,\ s)\\
&+B_{q}^{k_1,k_2}(n-k_{1}-i,\ i,\ s)\bigg],\ n>k_{1}\\
\end{split}
\end{equation*}
\noindent
and
\noindent
\begin{equation*}\label{eq:bn1}
\begin{split}
f^{(0)}_{q,S}(n;\theta)=\sum_{i=1}^{n-k_{2}}\theta^{n-k_{2}-i}{\prod_{j=1}^{i+k_2}}(1-\theta q^{j-1})\sum_{s=1}^{n-k_2-i}\bigg[&C_{q}^{k_1,k_2}(n-k_{2}-i,\ i,\ s)\\
&+D_{q}^{k_1,k_2}(n-k_{2}-i,\ i,\ s)\bigg],\ n>k_{2}.\\
\end{split}
\end{equation*}
\end{theorem}
\begin{proof}
We start with the study of $P_{q,S}^{(1)}(n)$. It is easy to see that $P_{q,S}^{(1)}(k_1)=\big(\theta q^{0}\big)^{k_{1}}=\theta^{k_{1}}$. From now on we assume $n > k_1,$ and we can rewrite $P_{q,S}^{(1)}(n)$ as follows
\noindent
\begin{equation*}\label{eq:bn}
\begin{split}
P_{q,S}^{(1)}(n)&=P_{q,\theta}\left(L_{n-k_{1}}^{(1)}< k_{1}\ \wedge\ L_{n-k_{1}}^{(0)}< k_{2}\ \wedge\ X_{n-k_{1}}=0\ \wedge\ X_{n-k_{1}+1}=\cdots =X_{n}=1\right).\\
\end{split}
\end{equation*}
\noindent
We partition the event $W_{S}^{(1)}=n$ into disjoint events given by $F_{n-k_1}=i,$ for $i=1,\ldots,n-k_1.$ Adding the probabilities we have
\noindent
\begin{equation*}\label{eq:bn}
\begin{split}
P_{q,S}^{(1)}(n)=\sum_{i=1}^{n-k_{1}}P_{q,\theta}\Big(L_{n-k_{1}}^{(1)}&< k_{1}\ \wedge\ L_{n-k_{1}}^{(0)}< k_{2}\ \wedge\ F_{n-k_{1}}=i\ \wedge\ X_{n-k_{1}}=0\ \wedge\\
&X_{n-k_{1}+1}=\cdots =X_{n}=1\Big).\\
\end{split}
\end{equation*}
\noindent
If the number of $0$'s in the first $n-k_1$ trials is equal to $i,$ that is, $F_{n-k_1}=i,$ then in each of the $(n-k_1+1)$ to $n$-th trials the probability of success is
\noindent
\begin{equation*}\label{eq: kk}
\begin{split}
p_{n-k_{1}+1}=\cdots=p_{n}=\theta q^{i}.
\end{split}
\end{equation*}
\noindent
We will write $E_{n,i}^{(0)}$ for the event $\left\{L_n^{(1)}< k_1\ \wedge\ L_n^{(0)}< k_2\ \wedge\ X_n=0\ \wedge\ F_n=i\right\}.$
\noindent
We can now rewrite as follows.
\noindent
\begin{equation*}\label{eq:bn1}
\begin{split}
P_{q,S}^{(1)}(n)=\sum_{i=1}^{n-k_{1}}&P_{q,\theta}\Big(L_{n-k_{1}}^{(1)}< k_{1}\ \wedge\ L_{n-k_{1}}^{(0)}< k_{2}\ \wedge\ F_{n-k_{1}}=i\ \wedge\ X_{n-k_{1}}=0\Big)\\
&\times P_{q,\theta}\Big(X_{n-k_{1}+1}=\cdots =X_{n}=1\mid F_{n-k_{1}}=i\Big)\\
=\sum_{i=1}^{n-k_{1}}&P_{q,\theta}\Big(E_{n-k_{1},i}^{(0)}\Big)\Big(\theta q^{i}\Big)^{k_1}.\\
\end{split}
\end{equation*}
\noindent
We are going to focus on the event $E_{n-k_{1},i}^{(0)}$. For $i=1,\ldots,n-k_1$, a typical element of the event $\{L_{n-k_1}^{(1)}< k_1\ \wedge\ L_{n-k_1}^{(0)}< k_2\ \wedge\ F_{n-k_1}=i\}$ is an ordered sequence which consists of $n-k_{1}-i$ successes and $i$ failures such that the length of the longest success run is less than $k_1$ and the length of the longest failure run is less than $k_2$. The number of these sequences can be derived as follows.
 First we will distribute the $i$ failures. Let $s$ $(1\leq s \leq i)$ be the number of runs of failures in the typical element of the event $E_{n-k_{1},i}^{(0)}$. Next, we will distribute the $n-k_{1}-i$ successes. We divide into two cases: starting with a failure run or starting with a success run. Thus, we distinguish between two types of sequences in the event  $$\left\{L_{n-k_1}^{(1)}< k_1 \wedge\ L_{n-k_1}^{(0)}< k_2 \wedge\ F_{n-k_1}=i\right\},$$ respectively named $(s-1,\ s)$-type and $(s,\ s)$-type, which are defined as follows.
\noindent
\begin{equation*}\label{eq:bn}
\begin{split}
(s-1,\ s)\text{-type}\ :&\quad \overbrace{0\ldots 0}^{y_{1}}\mid\overbrace{1\ldots 1}^{x_{1}}\mid\overbrace{0\ldots 0}^{y_{2}}\mid\overbrace{1\ldots 1}^{x_{2}}\mid \ldots \mid \overbrace{0\ldots 0}^{y_{s-1}}\mid\overbrace{1\ldots 1}^{x_{s-1}}\mid\overbrace{0\ldots 0}^{y_{s}},\\
\end{split}
\end{equation*}
\noindent
with $i$ $0$'s and $n-k_{1}-i$ $1$'s, where $x_{j}$ $(j=1,\ldots,s-1)$ represents a length of run of $1$'s and $y_{j}$ $(j=1,\ldots,s)$ represents the length of a run of $0$'s. And all integers $x_{1},\ldots,x_{s-1}$, and $y_{1},\ldots ,y_{s}$ satisfy the conditions
\noindent
\begin{equation*}\label{eq:bn}
\begin{split}
0 < x_j < k_1 \mbox{\ for\ } j=1,...,s-1, \mbox{\ and\ } x_1+\cdots +x_{s-1} = n-k_1-i,
\end{split}
\end{equation*}
\noindent
\begin{equation*}\label{eq:bn}
\begin{split}
0 < y_j < k_2 \mbox{\ for\ } j=1,...,s, \mbox{\ and\ } y_1+\cdots +y_s=i.
\end{split}
\end{equation*}

\begin{equation*}\label{eq:bn}
\begin{split}
(s,\ s)\text{-type}\  :&\quad \overbrace{1\ldots 1}^{x_{1}}\mid\overbrace{0\ldots 0}^{y_{1}}\mid\overbrace{1\ldots 1}^{x_{2}}\mid\overbrace{0\ldots 0}^{y_{2}}\mid\overbrace{1\ldots 1}^{x_{3}}\mid \ldots \mid \overbrace{0\ldots 0}^{y_{s-1}}\mid\overbrace{1\ldots 1}^{x_{s}}\mid\overbrace{0\ldots 0}^{y_{s}},
\end{split}
\end{equation*}
\noindent
with $i$ $0$'s and $n-k_{1}-i$ $1$'s, where $x_{j}$ $(j=1,\ldots,s)$ represents a length of run of $1$'s and $y_{j}$ $(j=1,\ldots,s)$ represents the length of a run of $0$'s. Here all of $x_1,\ldots, x_{s},$ and $y_1,\ldots, y_s$ are integers, and they satisfy
\noindent
\begin{equation*}\label{eq:bn}
\begin{split}
0 < x_j < k_1 \mbox{\ for\ } j=1,...,s, \mbox{\ and\ } x_1+\cdots +x_s = n-k_1-i,
\end{split}
\end{equation*}
\noindent
\begin{equation*}\label{eq:bn}
\begin{split}
0 < y_j < k_2 \mbox{\ for\ } j=1,...,s, \mbox{\ and\ } y_1+\cdots +y_s =i.
\end{split}
\end{equation*}
\noindent
Then the probability of the event $E_{n-k_1,i}^{(0)}$ is given by
\noindent
\begin{equation*}\label{eq: 1.1}
\begin{split}
P_{q,\theta}\Big(E_{n-k_{1},i}^{(0)}&\Big)=\\
\sum_{s=1}^{i}\Bigg[\bigg\{&\sum_{\substack{x_{1}+\cdots+x_{s-1}=n-k_{1}-i\\ x_{1},\ldots,x_{s-1} \in \{1,\ldots,k_{1}-1\}}}\
\sum_{\substack{y_{1}+\cdots+y_{s}=i\\ y_{1},\ldots,y_{s} \in \{1,\ldots,k_{2}-1\}}} \big(1-\theta q^{0}\big)\cdots \big(1-\theta q^{y_{1}-1}\big)\times\\
&\Big(\theta q^{y_{1}}\Big)^{x_{1}}\big(1-\theta q^{y_{1}}\big)\cdots \big(1-\theta q^{y_{1}+y_{2}-1}\big)\times\\
&\Big(\theta q^{y_{1}+y_{2}}\Big)^{x_{2}}\big(1-\theta q^{y_{1}+y_2}\big)\cdots \big(1-\theta q^{y_{1}+y_{2}+y_3-1}\big)\times \\
&\quad \quad \quad \quad\quad \quad\quad \quad \quad \quad \quad \vdots\\
&\Big(\theta q^{y_{1}+\cdots+y_{s-1}}\Big)^{x_{s-1}}
\big(1-\theta q^{y_{1}+\cdots+y_{s-1}}\big)\cdots \big(1-\theta q^{y_{1}+\cdots+y_{s}-1}\big)\bigg\}+\\
\end{split}
\end{equation*}
\noindent
\begin{equation*}\label{eq: 1.1}
\begin{split}
\quad \quad \quad \quad \quad \quad &\bigg\{\sum_{\substack{x_{1}+\cdots+x_{s}=n-k_{1}-i\\ x_{1},\ldots,x_{s} \in \{1,\ldots,k_{1}-1\}}}\
\sum_{\substack{y_{1}+\cdots+y_{s}=i\\ y_{1},\ldots,y_{s} \in \{1,\ldots,k_{2}-1\}}}\big(\theta q^{0}\big)^{x_{1}}\big(1-\theta q^{0}\big)\cdots (1-\theta q^{y_{1}-1})\times\\
&\Big(\theta q^{y_{1}}\Big)^{x_{2}}\big(1-\theta q^{y_{1}+y_2}\big)\cdots \big(1-\theta q^{y_{1}+y_{2}-1}\big)\times\\
&\Big(\theta q^{y_{1}+y_{2}}\Big)^{x_{3}}\big(1-\theta q^{y_{1}}\big)\cdots \big(1-\theta q^{y_{1}+y_{2}+y_3-1}\big)\times \\
&\quad \quad \quad \quad\quad \quad\quad \quad \quad \quad \quad \vdots\\
&\Big(\theta q^{y_{1}+\cdots+y_{s-1}}\Big)^{x_{s}}
\big(1-\theta q^{y_{1}+\cdots+y_{s-1}}\big)\cdots \big(1-\theta q^{y_{1}+\cdots+y_{s}-1}\big)\bigg\}\Bigg].\\
\end{split}
\end{equation*}
\noindent
Using simple exponentiation algebra arguments to simplify,
\noindent
\begin{equation*}\label{eq: 1.1}
\begin{split}
&P_{q,\theta}\Big(E_{n-k_{1},i}^{(0)}\Big)=\\
&\theta^{n-k_{1}-i}{\prod_{j=1}^{i}}\ \left   (1-\theta q^{j-1}\right)\times\\
&\sum_{s=1}^{i}\Bigg[\sum_{\substack{x_{1}+\cdots+x_{s-1}=n-k_{1}-i\\ x_{1},\ldots,x_{s-1} \in \{1,\ldots,k_{1}-1\}}}\
\sum_{\substack{y_{1}+\cdots+y_{s}=i\\ y_{1},\ldots,y_{s} \in \{1,\ldots,k_{2}-1\}}}q^{y_{1}x_{1}+(y_{1}+y_{2})x_{2}+\cdots+(y_{1}+\cdots+y_{s-1})x_{s-1}}+\\
&\quad \quad \ \sum_{\substack{x_{1}+\cdots+x_{s}=n-k_{1}-i\\ x_{1},\ldots,x_{s} \in \{1,\ldots,k_{1}-1\}}}\
\sum_{\substack{y_{1}+\cdots+y_{s}=i\\ y_{1},\ldots,y_{s} \in \{1,\ldots,k_{2}-1\}}}q^{y_{1}x_{2}+(y_{1}+y_{2})x_{3}+\cdots+(y_{1}+\cdots+y_{s-1})x_{s}}\Bigg].\\
\end{split}
\end{equation*}
\noindent
Using Lemma \ref{lemma:3.1} and Lemma \ref{lemma:3.2}, we can rewrite as follows.
\noindent
\begin{equation*}\label{eq: 1.1}
\begin{split}
P_{q,\theta}\Big(E_{n-k_{1},i}^{(0)}\Big)=\theta^{n-k_{1}-i}{\prod_{j=1}^{i}}\left(1-\theta q^{j-1}\right)\sum_{s=1}^{i}\Bigg[A&_{q}^{k_1,k_2}(n-k_{1}-i,\ i,\ s)+B_{q}^{k_1,k_2}(n-k_{1}-i,\ i,\ s)\Bigg],\\
\end{split}
\end{equation*}\\
\noindent
where
\noindent
\begin{equation*}\label{eq: 1.1}
\begin{split}
A_{q}^{k_1,k_2}(n&-k_{1}-i,\ i,\ s)\\
&=\sum_{\substack{x_{1}+\cdots+x_{s-1}=n-k_{1}-i\\ x_{1},\ldots,x_{s-1} \in \{1,\ldots,k_{1}-1\}}}\
\sum_{\substack{y_{1}+\cdots+y_{s}=i\\ y_{1},\ldots,y_{s} \in \{1,\ldots,k_{2}-1\}}} q^{y_{1}x_{1}+(y_{1}+y_{2})x_{2}+\cdots+(y_{1}+\cdots+y_{s-1})x_{s-1}},
\end{split}
\end{equation*}
\noindent
and
\noindent
\begin{equation*}\label{eq: 1.1}
\begin{split}
B_{q}^{k_1,k_2}(n&-k_{1}-i,\ i,\ s)\\
&=\sum_{\substack{x_{1}+\cdots+x_{s}=n-k_{1}-i\\ x_{1},\ldots,x_{s} \in \{1,\ldots,k_{1}-1\}}}\
\sum_{\substack{y_{1}+\cdots+y_{s}=i\\ y_{1},\ldots,y_{s} \in \{1,\ldots,k_{2}-1\}}} q^{y_{1}x_{2}+(y_{1}+y_{2})x_{3}+\cdots+(y_{1}+\cdots+y_{s-1})x_{s}}.
\end{split}
\end{equation*}
\noindent
Therefore we can compute the probability of the event $W_{S}^{(1)}=n$ as follows.
\noindent
\begin{equation*}\label{eq:bn1}
\begin{split}
P_{q,S}^{(1)}(n)=&\sum_{i=1}^{n-k_{1}}P\Big(E_{n-k_{1},i}^{(0)}\Big)\Big(\theta q^{i}\Big)^{k_{1}}\\
=&\sum_{i=1}^{n-k_{1}}\theta^{n-k_{1}-i}{\prod_{j=1}^{i}}\left(1-\theta q^{j-1}\right)\sum_{s=1}^{i}\bigg[A_{q}^{k_1,k_2}(n-k_{1}-i,\ i,\ s)\\
&\quad\quad\quad\quad\quad\quad\quad\quad\quad\quad\quad\quad\quad\quad +B_{q}^{k_1,k_2}(n-k_{1}-i,\ i,\ s)\bigg]\Big(\theta q^{i}\Big)^{k_{1}}.\\
\end{split}
\end{equation*}
\noindent
Finally, applying typical factorization algebra arguments, $P_{S}^{(1)}(n)$ can be rewritten as
\noindent
\begin{equation*}\label{eq:bn1}
\begin{split}
P_{q,S}^{(1)}(n)=\sum_{i=1}^{n-k_{1}}\theta^{n-i} q^{ik_{1}}{\prod_{j=1}^{i}}(1-\theta q^{j-1})\sum_{s=1}^{i}\bigg[A_{q}^{k_1,k_2}(n-k_{1}-i,\ i,\ s)+B_{q}^{k_1,k_2}(n-k_{1}-i,\ i,\ s)\bigg].\\
\end{split}
\end{equation*}
\noindent
Next, we are now going to study $P_{q,S}^{(0)}(n)$. It is easy to see that $P_{q,S}^{(0)}(k_2)={\prod_{j=1}^{k_2}}(1-\theta q^{j-1})$. From now on we assume $n > k_2,$ and we can write $P_{q,S}^{(0)}(n)$ as
\noindent
\begin{equation*}\label{eq:bn}
\begin{split}
P_{q,S}^{(0)}(n)&=P\left(L_{n-k_{2}}^{(1)}< k_{1}\ \wedge\ L_{n-k_{2}}^{(0)}< k_{2}\ \wedge\ X_{n-k_{2}}=1\ \wedge\ X_{n-k_{2}+1}=\cdots =X_{n}=0\right).\\
\end{split}
\end{equation*}
\noindent
We partition the event $W_{S}^{(0)}=n$ into disjoint events given by $F_{n-k_2}=i,$ for $i=1,\ldots,n-k_2.$ Adding the probabilities we have
\noindent
\begin{equation*}\label{eq:bn}
\begin{split}
P_{q,S}^{(0)}(n)=\sum_{i=1}^{n-k_{2}}P_{q,\theta}\Big(L_{n-k_{2}}^{(1)}&< k_{1}\ \wedge\ L_{n-k_{2}}^{(0)}< k_{2}\ \wedge\ F_{n-k_{2}}=i\ \wedge\ X_{n-k_{2}}=1\ \wedge\\
&X_{n-k_{2}+1}=\cdots =X_{n}=0\Big).\\
\end{split}
\end{equation*}
\noindent
If the number of $0$'s in the first $n-k_2$ trials is equal to $i,$ that is, $F_{n-k_2}=i,$ then the probability of failures in all of the $(n-k_2+1)$-th to $n$-th trials is
\noindent
\begin{equation*}\label{eq: kk}
\begin{split}
P_{q,\theta}(X_{n-k_2+1}=\cdots = X_n = 0 \,\mid \, F_{n-k_2}=i)={\prod_{j=i+1}^{i+k_2}}(1-\theta q^{j-1}).
\end{split}
\end{equation*}
\noindent
We write $E_{n,i}^{(1)}$ for the event $\left\{L_n^{(1)}< k_1\ \wedge\ L_n^{(0)}< k_2\ \wedge\ X_n=1\ \wedge\ F_n=i\right\}.$ We can now rewrite as follows
\noindent
\begin{equation*}\label{eq:bn1}
\begin{split}
P_{q,S}^{(0)}(n)=\sum_{i=1}^{n-k_{2}}&P_{q,\theta}\Big(L_{n-k_{2}}^{(1)}< k_{1}\ \wedge\ L_{n-k_{2}}^{(0)}< k_{2}\ \wedge\ F_{n-k_{2}}=i\ \wedge\ X_{n-k_{2}}=1\Big)\\
&\times P_{q,\theta}\Big(X_{n-k_{2}+1}=\cdots =X_{n}=0\mid F_{n-k_{2}}=i\Big)\\
=\sum_{i=1}^{n-k_{2}}&P_{q,\theta}\Big(E_{n-k_{2},i}^{(1)}\Big){\prod_{j=i+1}^{i+k_2}}(1-\theta q^{j-1}).\\
\end{split}
\end{equation*}

\noindent
We are going to focus on the event $E_{n-k_{2},i}^{(1)}$. For $i=1,\ldots,n-k_2$, a typical element of the event $\{L_{n-k_1}^{(1)}< k_1 \wedge\ L_{n-k_1}^{(0)}< k_2 \wedge\ F_{n-k_1}=i\}$ is an ordered sequence which consists of $n-k_{2}-i$ successes and $i$ failures such that the length of the longest success run is less than $k_1$ and the length of the longest failure run is less than $k_2$. The number of these sequences can be derived as follows.
 First we will distribute the $i$ failures. Let $s$ $(1\leq s \leq i)$ be the number of runs of failures in the typical element of the event  $E_{n-k_{2},i}^{(1)}$. Next, we will distribute the $n-k_{2}-i$ successes. We divide into two cases: starting with a success run or starting with a failure run. Thus, we distinguish between two types of sequences in the event  $$\left\{L_{n-k_1}^{(1)} <k_1 \wedge\ L_{n-k_1}^{(0)}< k_2 \wedge\ F_{n-k_1}=i\right\},$$ respectively named $(s-1,\ s)$-type and $(s,\ s)$-type, which are defined as follows.
\noindent

\begin{equation*}\label{eq:bn}
\begin{split}
(s,\ s-1)\text{-type}\ :&\quad \overbrace{1\ldots 1}^{x_{1}}\mid\overbrace{0\ldots 0}^{y_{1}}\mid\overbrace{1\ldots 1}^{x_{2}}\mid\overbrace{0\ldots 0}^{y_{2}}\mid \ldots \mid\overbrace{0\ldots 0}^{y_{s-1}}\mid\overbrace{1\ldots 1}^{x_{s}},\\
\end{split}
\end{equation*}
\noindent
with $i$ $0$'s and $n-k_{2}-i$ $1$'s, where $x_{j}$ $(j=1,\ldots,s)$ represents the length of a run of $1$'s and $y_{j}$ $(j=1,\ldots,s-1)$ represents the length of a run of $0$'s. And all integers $x_{1},\ldots,x_{s}$, and $y_{1},\ldots,y_{s-1}$ satisfy the conditions
\noindent
\begin{equation*}\label{eq:bn}
\begin{split}
0 < x_j < k_1 \mbox{\ for\ } j=1,...,s, \mbox{\ and\ } x_1+\cdots +x_{s} = n-k_2-i,
\end{split}
\end{equation*}
\noindent
\begin{equation*}\label{eq:bn}
\begin{split}
0 < y_j < k_2 \mbox{\ for\ } j=1,...,s-1, \mbox{\ and\ } y_1+\cdots +y_{s-1}=i.
\end{split}
\end{equation*}

\begin{equation*}\label{eq:bn}
\begin{split}
(s,\ s)\text{-type}\  :&\quad \overbrace{0\ldots 0}^{y_{1}}\mid\overbrace{1\ldots 1}^{x_{1}}\mid\overbrace{0\ldots 0}^{y_{2}}\mid\overbrace{1\ldots 1}^{x_{2}}\mid\overbrace{0\ldots 0}^{y_{3}}\mid \ldots \mid\overbrace{0\ldots 0}^{y_{s}}\mid\overbrace{1\ldots 1}^{x_{s}},
\end{split}
\end{equation*}
\noindent
with $i$ $0$'s and $n-k_{2}-i$ $1$'s, where $x_{j}$ $(j=1,\ldots,s)$ represents the length of a run of $1$'s and $y_{j}$ $(j=1,\ldots,s)$ represents the length of a run of $0$'s. Here all of $x_1,\ldots, x_{s},$ and $y_1,\ldots, y_s$ are integers, and they satisfy
\noindent
\begin{equation*}\label{eq:bn}
\begin{split}
0 < x_j < k_1 \mbox{\ for\ } j=1,...,s, \mbox{\ and\ } x_1+\cdots +x_s = n-k_2-i,
\end{split}
\end{equation*}
\noindent
\begin{equation*}\label{eq:bn}
\begin{split}
0 < y_j < k_2 \mbox{\ for\ } j=1,...,s, \mbox{\ and\ } y_1+\cdots +y_s =i.
\end{split}
\end{equation*}
\noindent
Then the probability of the event $E_{n-k_2,i}^{(1)}$ is given by
\noindent
\begin{equation*}\label{eq: 1.1}
\begin{split}
P_{q,\theta}\Big(E_{n-k_{2},i}^{(1)}&\Big)=\\
\sum_{s=1}^{n-k_2-i}\Bigg[\bigg\{&\sum_{\substack{x_{1}+\cdots+x_{s}=n-k_{2}-i\\ x_{1},\ldots,x_{s} \in \{1,\ldots,k_{1}-1\}}}\
\sum_{\substack{y_{1}+\cdots+y_{s-1}=i\\ y_{1},\ldots,y_{s-1} \in \{1,\ldots,k_{2}-1\}}} \big(\theta q^{0}\big)^{x_{1}}\big(1-\theta q^{0}\big)\cdots (1-\theta q^{y_{1}-1})\times\\
&\Big(\theta q^{y_{1}}\Big)^{x_{2}}\big(1-\theta q^{y_{1}}\big)\cdots \big(1-\theta q^{y_{1}+y_{2}-1}\big)\times\\
&\Big(\theta q^{y_{1}+y_{2}}\Big)^{x_{3}}\big(1-\theta q^{y_{1}+y_2}\big)\cdots \big(1-\theta q^{y_{1}+y_{2}+y_3-1}\big)\times \\
&\quad \quad \quad \quad\quad \quad\quad \quad \quad \quad \quad\vdots\\
&\Big(\theta q^{y_{1}+\cdots+y_{s-1}}\Big)^{x_{s}}\bigg\}+\\
\end{split}
\end{equation*}

\noindent
\begin{equation*}\label{eq: 1.1}
\begin{split}
\quad \quad \quad \quad&\bigg\{\sum_{\substack{x_{1}+\cdots+x_{s}=n-k_{2}-i\\ x_{1},\ldots,x_{s} \in \{1,\ldots,k_{1}-1\}}}\
\sum_{\substack{y_{1}+\cdots+y_{s}=i\\ y_{1},\ldots,y_{s} \in \{1,\ldots,k_{2}-1\}}}\big(1-\theta q^{0}\big)\cdots \big(1-\theta q^{y_{1}-1}\big)\Big(\theta q^{y_{1}}\Big)^{x_{1}}\times\\
&\big(1-\theta q^{y_{1}}\big)\cdots \big(1-\theta q^{y_{1}+y_{2}-1}\big)\Big(\theta q^{y_{1}+y_{2}}\Big)^{x_{2}}\times\\
&\big(1-\theta q^{y_{1}+y_2}\big)\cdots \big(1-\theta q^{y_{1}+y_{2}+y_3-1}\big)\Big(\theta q^{y_{1}+y_{2}+y_3}\Big)^{x_{3}}\times \\
&\quad \quad \quad \quad\quad \quad\quad \quad \quad \quad \quad\vdots\\
&\big(1-\theta q^{y_{1}+\cdots+y_{s-1}}\big)\cdots \big(1-\theta q^{y_{1}+\cdots+y_{s}-1}\big)\Big(\theta q^{y_{1}+\cdots+y_{s}}\Big)^{x_{s}}\bigg\}\Bigg].\\
\end{split}
\end{equation*}

\noindent
Using simple exponentiation algebra arguments to simplify,
\noindent
\begin{equation*}\label{eq: 1.1}
\begin{split}
&P_{q,\theta}\Big(E_{n-k_{2},i}^{(1)}\Big)=\\
&\theta^{n-k_{2}-i}{\prod_{j=1}^{i}}\ (1-\theta q^{j-1})\times\\
&\sum_{s=1}^{n-k_2-i}\Bigg[\sum_{\substack{x_{1}+\cdots+x_{s}=n-k_{2}-i\\ x_{1},\ldots,x_{s} \in \{1,\ldots,k_{1}-1\}}}\
\sum_{\substack{y_{1}+\cdots+y_{s-1}=i\\ y_{1},\ldots,y_{s-1} \in \{1,\ldots,k_{2}-1\}}}q^{y_{1}x_{2}+(y_{1}+y_{2})x_{3}+\cdots+(y_{1}+\cdots+y_{s-1})x_{s}}+\\
&\quad \quad \quad \ \sum_{\substack{x_{1}+\cdots+x_{s}=n-k_{2}-i\\ x_{1},\ldots,x_{s} \in \{1,\ldots,k_{1}-1\}}}\
\sum_{\substack{y_{1}+\cdots+y_{s}=i\\ y_{1},\ldots,y_{s} \in \{1,\ldots,k_{2}-1\}}}q^{y_{1}x_{1}+(y_{1}+y_{2})x_{2}+\cdots+(y_{1}+\cdots+y_{s})x_{s}}\Bigg].\\
\end{split}
\end{equation*}
\noindent
Using Lemma \ref{lemma:3.3} and Lemma \ref{lemma:3.4}, we can rewrite as follows.

\noindent
\begin{equation*}\label{eq: 1.1}
\begin{split}
P_{q,\theta}\Big(E_{n-k_{2},i}^{(1)}\Big)=\theta^{n-k_{2}-i}{\prod_{j=1}^{i}}(1-\theta q^{j-1})\sum_{s=1}^{n-k_2-i}\Bigg[C&_{q}^{k_1,k_2}(n-k_{2}-i,\ i,\ s)+D_{q}^{k_1,k_2}(n-k_{2}-i,\ i,\ s)\Bigg],\\
\end{split}
\end{equation*}\\
\noindent
where
\noindent
\begin{equation*}\label{eq: 1.1}
\begin{split}
C_{q}^{k_1,k_2}(n&-k_{2}-i,\ i,\ s)\\
&=\sum_{\substack{x_{1}+\cdots+x_{s}=n-k_{2}-i\\ x_{1},\ldots,x_{s} \in \{1,\ldots,k_{1}-1\}}}\
\sum_{\substack{y_{1}+\cdots+y_{s-1}=i\\ y_{1},\ldots,y_{s-1} \in \{1,\ldots,k_{2}-1\}}}q^{y_{1}x_{2}+(y_{1}+y_{2})x_{3}+\cdots+(y_{1}+\cdots+y_{s-1})x_{s}},
\end{split}
\end{equation*}
\noindent
and\noindent
\begin{equation*}\label{eq: 1.1}
\begin{split}
D_{q}^{k_1,k_2}(n&-k_{2}-i,\ i,\ s)\\
&=\sum_{\substack{x_{1}+\cdots+x_{s}=n-k_{2}-i\\ x_{1},\ldots,x_{s} \in \{1,\ldots,k_{1}-1\}}}\
\sum_{\substack{y_{1}+\cdots+y_{s}=i\\ y_{1},\ldots,y_{s} \in \{1,\ldots,k_{2}-1\}}}q^{y_{1}x_{1}+(y_{1}+y_{2})x_{2}+\cdots+(y_{1}+\cdots+y_{s})x_{s}}.
\end{split}
\end{equation*}
\noindent
Therefore we can compute the probability of the event $W_{S}^{(0)}=n$ as follows
\noindent
\begin{equation*}\label{eq:bn1}
\begin{split}
P_{q,S}^{(0)}(n)=&\sum_{i=1}^{n-k_{2}}P_{q,\theta}\Big(E_{n-k_{2},i}^{(1)}\Big){\prod_{j=i+1}^{i+k_2}}(1-\theta q^{j-1})\\
=&\sum_{i=1}^{n-k_{2}}\theta^{n-k_{2}-i}{\prod_{j=1}^{i}}\left(1-\theta q^{j-1}\right)\sum_{s=1}^{n-k_2-i}\bigg[C_{q}^{k_1,k_2}(n-k_{2}-i,\ i,\ s)+D_{q}^{k_1,k_2}(n-k_{2}-i,\ i,\ s)\bigg]\\
&\times{\prod_{j=i+1}^{i+k_2}}\left(1-\theta q^{j-1}\right).\\
\end{split}
\end{equation*}
\noindent
Finally, applying typical factorization algebra arguments, $P_{q,S}^{(0)}(n)$ can be rewritten as follows
\noindent
\begin{equation*}\label{eq:bn1}
\begin{split}
P_{q,S}^{(0)}(n)=\sum_{i=1}^{n-k_{2}}\theta^{n-k_{2}-i}{\prod_{j=1}^{i+k_2}}(1-\theta q^{j-1})\sum_{s=1}^{n-k_2-i}\bigg[C_{q}^{k_1,k_2}(n-k_{2}-i,\ i,\ s)+D_{q}^{k_1,k_2}(n-k_{2}-i,\ i,\ s)\bigg].\\
\end{split}
\end{equation*}
Thus proof is completed.
\end{proof}
\noindent
It is worth mentioning here that the PMF $f_{q,S}(n;\theta)$ approaches the probability function of the sooner waiting time distribution of order $(k_1,k_2)$
in the limit as $q$ tends to 1 when a run(succession) run quota is imposed on runs of successes and failures of IID model. The details are presented in the following remark.
\noindent
\begin{remark}
{\rm For $q=1$, the PMF $f_{q,S}(n;\theta)$ reduces to the PMF $f_{S}(n;\theta)=P_{\theta}(W_{S}=n)$ for $n\geq min(k_1,k_2)$ is given by

\noindent
\begin{equation*}\label{eq:bn}
\begin{split}
P_{\theta}(W_{S}=n)=P_{S}^{(1)}(n)+P_{S}^{(0)}(n),
\end{split}
\end{equation*}
\noindent
where $P_S^{(1)}(k_1)=\theta^{k_{1}}$, $P_S^{(0)}(k_2)=(1-\theta)^{k_2}$,
\noindent
\begin{equation*}\label{eq:bn1}
\begin{split}
P_{S}^{(1)}(n)=\sum_{i=1}^{n-k_{1}}\theta^{n-i} (1-\theta )^{i}\sum_{s=1}^{i}S(s,k_2,i)\bigg[&S(s-1,k_1,n-k_1-i)\\
&+S(s-1,k_1,n-k_1-i)\bigg],\ n>k_1\\
\end{split}
\end{equation*}
\noindent
and
\noindent
\begin{equation*}\label{eq:bn1}
\begin{split}
P_{S}^{(0)}(n)=\sum_{i=1}^{n-k_{2}}\theta^{n-k_{2}-i}(1-\theta )^{i+k_2}\sum_{s=1}^{n-k_2-i}S(s,k_1,n-k_2-i)\bigg[S(s-1,k_2,i)+S(s,k_2,i)\bigg],\ n>k_2.\\
\end{split}
\end{equation*}

where

\noindent
\begin{equation*}\label{eq: 1.1}
\begin{split}
S(a,\ b,\ c)=\sum_{j=0}^{min(a,\left[\frac{c-a}{b-1}\right])}(-1)^{j}{a \choose j}{c-j(b-1)-1 \choose a-1}.
\end{split}
\end{equation*}
\noindent
See, e.g. \citet{charalambides2002enumerative}.
}
\end{remark}

\subsection{$q$-later waiting time distribution of order $(k_1,k_2)$}

The problem of waiting time that will be discussed in this section is one of the 'later cases' and it emerges when a quota is imposed on runs of successes and failures. More specifically, Binary (zero and one) trials with probability of ones varying according to a geometric rule, are performed sequentially until $k_1$ consecutive successes and $k_2$ consecutive failures are observed, whichever event gets observed later. Let $W_L^{(1)}$  be a random variable denoting that the waiting time until  $k_{1}$ consecutive successes are observed for the first time and failure run of length $k_2$ or more has appeared before and $W_L^{(0)}$  be a random variable denoting that the waiting time until  $k_{2}$ consecutive failures are observed for the first time and success run of length $k_1$ or more has appeared before.

We will also simply write $f^{(1)}_{q,L}(n;\theta) = P_{q,\theta}(W_L^{(1)}=n)$ and $f^{(0)}_{q,L}(n;\theta)=P_{q,\theta}(W_{L}^{(0)}=n)$. Therefore we can write the probability of event $W_{L}=n$ as follows
\begin{equation*}\label{eq: 1.1}
\begin{split}
P_{q,\theta}(W_{L}=n)=f^{(1)}_{q,L}(n;\theta)+f^{(0)}_{q,L}(n;\theta),\ n\geq k_{1}+ k_{2}.
\end{split}
\end{equation*}

\noindent
We now make some useful Definition and Lemma for the proofs of Theorem in the sequel.
\begin{definition}
For $s,r,k \in \mathbb{N}$ we define
\noindent
\begin{equation*}\label{eq: 1.1}
\begin{split}
S_{r,s}^{k} = \left\{ (i_1,\ldots ,i_s) \in \mathbb{Z}_+^s \  \mid \ \sum_{j=1}^s i_j = r\ \wedge \ \max(i_1,\ldots,i_s) > k \right\}.
\end{split}
\end{equation*}
\end{definition}

\begin{definition}
For $0<q\leq1$, we define

\begin{equation*}\label{eq: 1.1}
\begin{split}
\overline{E}_{q,0,0}^{k_1,\infty}(u,v,s)={\sum_{x_{1},\ldots,x_{s-1}}}\
\sum_{y_{1},\ldots,y_{s}} q^{y_{1}x_{1}+(y_{1}+y_{2})x_{2}+\cdots+(y_{1}+\cdots+y_{s-1})x_{s-1}},\
\end{split}
\end{equation*}
\noindent
where the summation is over all integers $x_1,\ldots,x_{s-1},$ and $y_1,\ldots,y_{s}$ satisfying

\begin{equation*}\label{eq:1}
\begin{split}
0<x_{j}<k_1\ \text{for}\ j=1,\ldots,s-1,
\end{split}
\end{equation*}

\noindent
\begin{equation*}\label{eq:1}
\begin{split}
(x_1,\ldots,x_{s-1})\in S_{u,s-1}^{0},\ \text{and}
\end{split}
\end{equation*}

\noindent
\begin{equation*}\label{eq:2}
\begin{split}
(y_{1},\ldots,y_{s}) \in S_{v,s}^{0}.
\end{split}
\end{equation*}
\end{definition}
\noindent
The following gives a recurrence relation useful for the computation of $\overline{E}_{q,0,0}^{k_1,\infty}(u,v,s)$.\\
\noindent
\begin{lemma}
\label{lemma:3.5}
$\overline{E}_{q,0,0}^{k_1,\infty}(u,v,s)$ obeys the following recurrence relation.
\begin{equation*}\label{eq: 1.1}
\begin{split}
\overline{E}&_{q,0,0}^{k_1,\infty}(u,v,s)\\
&=\left\{
  \begin{array}{ll}
    \sum_{b=1}^{v-(s-1)}\sum_{a=1}^{k_1-1}q^{a(v-b)} \overline{E}_{q,0,0}^{k_1,\infty}(u-a,v-b,s-1), & \text{for}\ s>1,\ s-1\leq u\leq(s-1)(k_{1}-1)\\
    &\text{and}\ s\leq v \\
    1, & \text{for}\ s=1,\ u=0,\ \text{and}\ 1\leq v\\
    0, & \text{otherwise.}\\
  \end{array}
\right.
\end{split}
\end{equation*}
\end{lemma}
\begin{proof}
For $s > 1$, $s-1\leq u\leq(s-1)(k_{1}-1)$ and $s\leq v$, we observe that $x_{s-1}$ may assume any value $1,\ldots,k_{1}-1$, then $\overline{E}_{q,0,0}^{k_1,\infty}(u,v,s)$ can be written as

\begin{equation*}
\begin{split}
\overline{E}&_{q,0,0}^{k_1,\infty}(u,v,s)\\
=&\sum_{x_{s-1}=1}^{k_1-1}{\sum_{\substack{0<x_1,\ldots,x_{s-2}<k_1\\(x_1,\ldots,x_{s-2})\in S_{u-x_{s-1},s-2}^{0}}}}{\hspace{0.1cm}\sum_{\substack{(y_{1},\ldots,y_{s}) \in S_{v,s}^{0}}}}q^{x_{s-1}(v-y_{s})}\ q^{y_{1}x_{1}+(y_{1}+y_{2})x_{2}+\cdots+(y_{1}+\cdots+y_{s-2})x_{s-2}}
\end{split}
\end{equation*}
Similarly, we observe that since $y_s$ can assume the values $1,\ldots, v-(s-1)$, then $\overline{E}_{q,0,0}^{k_1,\infty}(u,v,s)$ can be rewritten as

\begin{equation*}
\begin{split}
\overline{E}&_{q,0,0}^{k_1,\infty}(u,v,s)\\
=&\sum_{y_{s}=1}^{v-(s-1)}\sum_{x_{s-1}=1}^{k_1-1}q^{x_{s-1}(v-y_{s})}{\hspace{-1cm}\sum_{\substack{0<x_1,\ldots,x_{s-2}<k_1\\(x_1,\ldots,x_{s-2})\in S_{u-x_{s-1},s-2}^{0}}}}
{\hspace{0.1cm}\sum_{\substack{(y_{1},\ldots,y_{s-1}) \in S_{v-y_{s},s-1}^{0}}}} q^{y_{1}x_{1}+(y_{1}+y_{2})x_{2}+\cdots+(y_{1}+\cdots+y_{s-2})x_{s-2}}\\
=&\sum_{b=1}^{v-(s-1)}\sum_{a=1}^{k_1-1}q^{a(v-b)} \overline{E}_{q,0,0}^{k_1,\infty}(u-a,v-b,s-1).
\end{split}
\end{equation*}
The other cases are obvious and thus the proof is completed.
\end{proof}
\begin{remark}
{\rm
We observe that for $\overline{E}_{1,0,0}^{k_1,\infty}(m,r,s)$ is the number of integer solutions $(x_{1},\ldots,x_{s-1})$ and $(y_{1},\ldots,y_{s})$ of
\noindent
\begin{equation*}\label{eq:1}
\begin{split}
0<x_{j}<k_1\ \text{for}\ j=1,\ldots,s-1,
\end{split}
\end{equation*}

\noindent
\begin{equation*}\label{eq:1}
\begin{split}
(x_1,\ldots,x_{s-1})\in S_{u,s-1}^{0},\ \text{and}
\end{split}
\end{equation*}

\noindent
\begin{equation*}\label{eq:2}
\begin{split}
(y_{1},\ldots,y_{s}) \in S_{r,s}^{0}
\end{split}
\end{equation*}
\noindent
which is
\noindent
\begin{equation*}\label{eq: 1.1}
\begin{split}
\overline{E}_{1,0,0}^{k_1,\infty}(m,r,s)=S(s-1,\ k_{1},\ u)M(s,\ v ),
\end{split}
\end{equation*}
\noindent
where $S(a,\ b,\ c)$ denotes the total number of integer solution $x_{1}+x_{2}+\cdots+x_{a}=c$ such that $0<x_{i}<b$ for $i=1,2,\ldots,a$. The number is given by
\noindent
\begin{equation*}\label{eq: 1.1}
\begin{split}
S(a,\ b,\ c)=\sum_{j=0}^{min\left(a,\ \left[\frac{c-a}{b-1}\right]\right) }(-1)^{j}{a \choose j}{c-j(b-1)-1 \choose a-1}.
\end{split}
\end{equation*}
\noindent
where $M(a,\ b)$ denotes the total number of integer solution $y_{1}+x_{2}+\cdots+y_{a}=b$ such that $(y_{1},\ldots,y_{a}) \in S_{b,a}^{0}$. The number is given by
\noindent
\begin{equation*}\label{eq: 1.1}
\begin{split}
M(a,\ b)={b-1 \choose a-1}.
\end{split}
\end{equation*}
\noindent
See, e.g. \citet{charalambides2002enumerative}.
}
\end{remark}

\vskip 5mm

\begin{definition}
For $0<q\leq1$, we define

\begin{equation*}\label{eq: 1.1}
\begin{split}
E_{q,0,k_2}^{k_1,\infty}(m,r,s)={\sum_{x_{1},\ldots,x_{s-1}}}\
\sum_{y_{1},\ldots,y_{s}} q^{y_{1}x_{1}+(y_{1}+y_{2})x_{2}+\cdots+(y_{1}+\cdots+y_{s-1})x_{s-1}},\
\end{split}
\end{equation*}
\noindent
where the summation is over all integers $x_1,\ldots,x_{s-1},$ and $y_1,\ldots,y_s$ satisfying
\noindent
\begin{equation*}\label{eq:1}
\begin{split}
0<x_{j}<k_1\ \text{for}\ j=1,\ldots,s-1,
\end{split}
\end{equation*}
\noindent
\begin{equation*}\label{eq:1}
\begin{split}
(x_1,\ldots,x_{s-1})\in S_{m,s-1}^{0},\ \text{and}
\end{split}
\end{equation*}
\noindent
\begin{equation*}\label{eq:2}
\begin{split}
(y_1,\ldots,y_s)\in S_{r,s}^{k_2-1}.
\end{split}
\end{equation*}
\end{definition}
\noindent
The following gives a recurrence relation useful for the computation of $E_{q,0,k_2}^{k_1,\infty}(m,r,s)$.\\
\noindent
\begin{lemma}
\label{lemma:3.6}
$E_{q,0,k_2}^{k_1,\infty}(m,r,s)$ obeys the following recurrence relation.
\begin{equation*}\label{eq: 1.1}
\begin{split}
E&_{q,0,k_2}^{k_1,\infty}(m,r,s)\\
&=\left\{
  \begin{array}{ll}
    \sum_{b=1}^{k_2-1}\sum_{a=1}^{k_1-1}q^{a(r-b)} E_{q,0,k_2}^{k_1,\infty}(m-a,r-b,s-1)\\
    +\sum_{b=k_2}^{r-(s-1)}\sum_{a=1}^{k_1-1} q^{a(r-b)} \overline{E}_{q,0,0}^{k_1,\infty}(m-a,r-b,s-1), & \text{for}\ s>1,\ s-1\leq m\leq(s-1)(k_{1}-1)\\
    &\text{and}\ s-1+k_2\leq r \\
    1, & \text{for}\ s=1,\ m=0,\ \text{and}\ k_2\leq r\\
    0, & \text{otherwise.}\\
  \end{array}
\right.
\end{split}
\end{equation*}
\end{lemma}
\begin{proof}
For $s > 1$, $s-1\leq m\leq(s-1)(k_{1}-1)$ and $s-1+k_2\leq r$, we observe that $x_{s-1}$ may assume any value $1,\ldots,k_{1}-1$, then $E_{q,0,k_2}^{k_1,\infty}(m,r,s)$ can be written as

\begin{equation*}
\begin{split}
E&_{q,0,k_2}^{k_1,\infty}(m,r,s)\\
=&\sum_{x_{s-1}=1}^{k_1-1}{\sum_{\substack{0<x_1,\ldots,x_{s-2}<k_1\\(x_1,\ldots,x_{s-2})\in S_{m-x_{s-1},s-2}^{0}}}}\
{\hspace{0.1cm}\sum_{\substack{(y_{1},\ldots,y_{s}) \in S_{r,s}^{k_{2}-1}}}}q^{x_{s-1}(r-y_{s})}\ q^{y_{1}x_{1}+(y_{1}+y_{2})x_{2}+\cdots+(y_{1}+\cdots+y_{s-2})x_{s-2}}
\end{split}
\end{equation*}
\noindent
Similarly, we observe that since $y_s$ can assume the values $1,\ldots,k_{2}-1$, then $E_{q,0,k_2}^{k_1,\infty}(m,r,s)$ can be rewritten as
\noindent
\begin{equation*}
\begin{split}
E&_{q,0,k_2}^{k_1,\infty}(m,r,s)\\
=&\sum_{y_{s}=1}^{k_2-1}\sum_{x_{s-1}=1}^{k_1-1}q^{x_{s-1}(r-y_{s})}{\hspace{-0.1cm}\sum_{\substack{0<x_1,\ldots,x_{s-2}<k_1\\(x_1,\ldots,x_{s-2})\in S_{m-x_{s-1},s-2}^{0}}}}\
{\sum_{\substack{(y_{1},\ldots,y_{s-1}) \in S_{r-y_s,s-1}^{k_2-1}}}}\ q^{y_{1}x_{1}+(y_{1}+y_{2})x_{2}+\cdots+(y_{1}+\cdots+y_{s-2})x_{s-2}}\\
&+\sum_{y_{s}=k_2}^{r-(s-1)}\sum_{x_{s-1}=1}^{k_1-1}q^{x_{s-1}(r-y_{s})}{\hspace{-1cm}\sum_{\substack{0<x_1,\ldots,x_{s-2}<k_1\\(x_1,\ldots,x_{s-2})\in S_{m-x_{s-1},s-2}^{0}}}}\
{\sum_{\substack{(y_{1},\ldots,y_{s-1}) \in S_{r-y_s,s-1}^{0}}}}\ q^{y_{1}x_{1}+(y_{1}+y_{2})x_{2}+\cdots+(y_{1}+\cdots+y_{s-2})x_{s-2}}\\
=&\sum_{b=1}^{k_2-1}\sum_{a=1}^{k_1-1}q^{a(r-b)} E_{q,0,k_2}^{k_1,\infty}(m-a,r-b,s-1)+\sum_{b=k_2}^{r-(s-1)}\sum_{a=1}^{k_1-1} q^{a(r-b)} \overline{E}_{q,0,0}^{k_1,\infty}(m-a,r-b,s-1).
\end{split}
\end{equation*}
The other cases are obvious and thus the proof is completed.
\end{proof}

\vskip 5mm
\begin{remark}
{\rm
We observe that for $E_{1,0,k_2}^{k_1,\infty}(m,r,s)$ is the number of integer solutions $(x_{1},\ldots, x_{s-1})$ and $(y_{1},\ldots,y_{s})$ of

\noindent
\begin{equation*}\label{eq:1}
\begin{split}
0<x_{j}<k_1\ \text{for}\ j=1,\ldots,s-1,
\end{split}
\end{equation*}
\noindent
\begin{equation*}\label{eq:1}
\begin{split}
(x_1,\ldots,x_{s-1})\in S_{m,s-1}^{0},\ \text{and}
\end{split}
\end{equation*}
\noindent
\begin{equation*}\label{eq:2}
\begin{split}
(y_{1},\ldots,y_{s}) \in S_{r,s}^{k_2-1}
\end{split}
\end{equation*}
\noindent
which is
\noindent
\begin{equation*}\label{eq: 1.1}
\begin{split}
E_{1,0,k_2}^{k_1,\infty}(m,r,s)=S(s-1,\ k_{1},\ m)R(s,\ k_{2},\ r),
\end{split}
\end{equation*}
\noindent
where $S(a,\ b,\ c)$ denotes the total number of integer solution $x_{1}+x_{2}+\cdots+x_{a}=c$ such that $0<x_{i}<b$ for $i=1,2,\ldots,a$. The number is given by
\noindent
\begin{equation*}\label{eq: 1.1}
\begin{split}
S(a,\ b,\ c)=\sum_{j=0}^{min\left(a,\ \left[\frac{c-a}{b-1}\right]\right) }(-1)^{j}{a \choose j}{c-j(b-1)-1 \choose a-1}.
\end{split}
\end{equation*}
\noindent
where $R(a,\ b,\ c)$ denotes the total number of integer solution $y_{1}+x_{2}+\cdots+y_{a}=c$ such that $(y_{1},\ldots,y_{a}) \in S_{c,a}^{b-1}$. Alternatively, it is the number of ways in which $c$ identical objects can be arranged to form $a$ groups with at least one group containing more than or equal to $b$ objects. The number is given by
\noindent
\begin{equation*}\label{eq: 1.1}
\begin{split}
R(a,\ b,\ c)=\sum_{j=1}^{min\left(a,\ \left[\frac{c-a}{b-1}\right]\right) }(-1)^{j+1}{a \choose j}{c-j(b-1)-1 \choose a-1}.
\end{split}
\end{equation*}
\noindent
See, e.g. \citet{charalambides2002enumerative}.
}
\end{remark}

\vskip 5mm
\begin{definition}
For $0<q\leq1$, we define

\begin{equation*}\label{eq: 1.1}
\begin{split}
\overline{F}_{q,0,0}^{k_1,\infty}(u,v,s)={\sum_{x_{1},\ldots,x_{s}}}\
\sum_{y_{1},\ldots,y_{s}} q^{y_{1}x_{2}+(y_{1}+y_{2})x_{3}+\cdots+(y_{1}+\cdots+y_{s-1})x_{s}},\
\end{split}
\end{equation*}
\noindent
where the summation is over all integers $x_1,\ldots,x_{s},$ and $y_1,\ldots,y_s$ satisfying
\noindent
\begin{equation*}\label{eq:1}
\begin{split}
0<x_{j}<k_{1}\ \text{for}\ j=1,\ldots,s,
\end{split}
\end{equation*}
\noindent
\begin{equation*}\label{eq:1}
\begin{split}
(x_1,\ldots,x_{s})\in S_{m,s}^{0},\ \text{and}\end{split}
\end{equation*}
\noindent
\begin{equation*}\label{eq:2}
\begin{split}
(y_{1},\ldots,y_{s})\in S_{r,s}^{0}.
\end{split}
\end{equation*}
\end{definition}
\noindent
The following gives a recurrence relation useful for the computation of $\overline{F}_{q,0,0}^{k_1,\infty}(u,v,s)$.\\
\noindent
\begin{lemma}
\label{lemma:3.7}
$\overline{F}_{q,0,0}^{k_1,\infty}(u,v,s)$ obeys the following recurrence relation.
\begin{equation*}\label{eq: 1.1}
\begin{split}
\overline{F}&_{q,0,0}^{k_1,\infty}(u,v,s)\\
&=\left\{
  \begin{array}{ll}
    \sum_{b=1}^{v-(s-1)}\sum_{a=1}^{k_1-1}q^{a(v-b)} \overline{F}_{q,0,0}^{k_1,\infty}(u-a,v-b,s-1), & \text{for}\ s>1,\ s\leq u\leq s(k_{1}-1)\\
    &\text{and}\ s\leq v \\
    1, & \text{for}\ s=1,\ 1\leq u\leq k_{1}-1,\ \text{and}\ 1\leq v\\
    0, & \text{otherwise.}\\
  \end{array}
\right.
\end{split}
\end{equation*}
\end{lemma}
\begin{proof}
For $s > 1$, $s\leq u\leq s(k_{1}-1)$ and $s\leq v$, we observe that $x_{s}$ may assume any value $1,\ldots,k_{1}-1$, then $\overline{F}_{q,0,0}^{k_1,\infty}(u,v,s)$ can be written as

\begin{equation*}
\begin{split}
\overline{F}&_{q,0,0}^{k_1,\infty}(u,v,s)\\
=&\sum_{x_{s}=1}^{k_1-1}{\sum_{\substack{0<x_1,\ldots,x_{s-1}<k_1\\(x_1,\ldots,x_{s-1})\in S_{u-x_{s},s-1}^{0}}}}\
{\sum_{\substack{(y_{1},\ldots,y_{s}) \in S_{v,s}^{0}}}}q^{x_{s}(v-y_{s})}\ q^{y_{1}x_{2}+(y_{1}+y_{2})x_{3}+\cdots+(y_{1}+\cdots+y_{s-2})x_{s-1}}
\end{split}
\end{equation*}
Similarly, we observe that since $y_s$ can assume the values $1,\ldots,v-(s-1)$, then $\overline{F}_{q,0,0}^{k_1,\infty}(u,v,s)$ can be rewritten as

\begin{equation*}
\begin{split}
\overline{F}&_{q,0,0}^{k_1,\infty}(u,v,s)\\
=&\sum_{y_{s}=1}^{v-(s-1)}\sum_{x_{s}=1}^{k_1-1}q^{x_{s}(v-y_{s})}{\sum_{\substack{0<x_1,\ldots,x_{s-2}<k_1\\(x_1,\ldots,x_{s-2})\in S_{u-x_{s-1},s-2}^{0}}}}\
{\sum_{\substack{(y_{1},\ldots,y_{s-1}) \in S_{v-y_{s},s-1}^{0}}}}q^{y_{1}x_{2}+(y_{1}+y_{2})x_{3}+\cdots+(y_{1}+\cdots+y_{s-2})x_{s-1}}\\
=&\sum_{b=1}^{v-(s-1)}\sum_{a=1}^{k_1-1}q^{a(v-b)} \overline{F}_{q,0,0}^{k_1,\infty}(u-a,v-b,s-1).
\end{split}
\end{equation*}
The other cases are obvious and thus the proof is completed.
\end{proof}

{\hfill $\Box$}
\vskip 5mm
\begin{remark}
{\rm
We observe that $\overline{F}_{1,0,0}^{k_1,\infty}(m,r,s)$ is the number of integer solutions $(x_{1},\ldots,x_{s})$ and $(y_{1},\ldots,y_{s})$ of
\noindent
\begin{equation*}\label{eq:1}
\begin{split}
0<x_{j}<k_1\ \text{for}\ j=1,\ldots,s,
\end{split}
\end{equation*}

\noindent
\begin{equation*}\label{eq:1}
\begin{split}
(x_1,\ldots,x_{s})\in S_{u,s}^{0},\ \text{and}
\end{split}
\end{equation*}

\noindent
\begin{equation*}\label{eq:2}
\begin{split}
(y_{1},\ldots,y_{s}) \in S_{r,s}^{0}
\end{split}
\end{equation*}
\noindent
which is
\noindent
\begin{equation*}\label{eq: 1.1}
\begin{split}
\overline{F}_{1,0,0}^{k_1,\infty}(m,r,s)=S(s,\ k_{1},\ u)M(s,\ v ),
\end{split}
\end{equation*}
\noindent
where $S(a,\ b,\ c)$ denotes the total number of integer solution $x_{1}+x_{2}+\cdots+x_{a}=c$ such that $0<x_{i}<b$ for $i=1,2,\ldots,a$. The number is given by
\noindent
\begin{equation*}\label{eq: 1.1}
\begin{split}
S(a,\ b,\ c)=\sum_{j=0}^{min\left(a,\ \left[\frac{c-a}{b-1}\right]\right) }(-1)^{j}{a \choose j}{c-j(b-1)-1 \choose a-1}.
\end{split}
\end{equation*}
\noindent
where $M(a,\ b)$ denotes the total number of integer solution $y_{1}+x_{2}+\cdots+y_{a}=b$ such that $(y_{1},\ldots,y_{a}) \in S_{b,a}^{0}$. The number is given by
\noindent
\begin{equation*}\label{eq: 1.1}
\begin{split}
M(a,\ b)={b-1 \choose a-1}.
\end{split}
\end{equation*}
\noindent
See, e.g. \citet{charalambides2002enumerative}.
}
\end{remark}

\vskip 5mm
\begin{definition}
For $0<q\leq1$, we define
\begin{equation*}\label{eq: 1.1}
\begin{split}
F_{q,0,k_2}^{k_1,\infty}(m,r,s)={\sum_{x_{1},\ldots,x_{s}}}\
\sum_{y_{1},\ldots,y_{s}}q^{y_{1}x_{2}+(y_{1}+y_{2})x_{3}+\cdots+(y_{1}+\cdots+y_{s-1})x_{s}},\
\end{split}
\end{equation*}
\noindent
where the summation is over all integers $x_1,\ldots,x_{s},$ and $y_1,\ldots,y_s$ satisfying
\noindent
\begin{equation*}\label{eq:1}
\begin{split}
0<x_{j}<k_{1},\ \text{for} j=1,\ \ldots,\ s,
\end{split}
\end{equation*}
\noindent
\begin{equation*}\label{eq:1}
\begin{split}
(x_1,\ldots,x_{s})\in S_{m,s}^{0},\ \text{and}\end{split}
\end{equation*}
\noindent
\begin{equation*}\label{eq:2}
\begin{split}
(y_{1},\ldots,y_{s})\in S_{r,s}^{k_2-1}.
\end{split}
\end{equation*}
\end{definition}

\noindent
The following gives a recurrence relation useful for the computation of $F_{q,0,k_2}^{k_1,\infty}(m,r,s)$.\\
\noindent
\begin{lemma}
\label{lemma:3.8}
$F_{q,0,k_2}^{k_1,\infty}(m,r,s)$ obeys the following recurrence relation.
\begin{equation*}\label{eq: 1.1}
\begin{split}
F&_{q,0,k_2}^{k_1,\infty}(m,r,s)\\
&=\left\{
  \begin{array}{ll}
    \sum_{b=1}^{k_{2}-1}\ \sum_{a=1}^{k_{1}-1}q^{a(r-b)}F_{q,0,k_2}^{k_1,\infty}(m-a,r-b,s-1)\\
    +\sum_{b=k_{2}}^{r-(s-1)}\ \sum_{a=1}^{k_{1}-1}q^{a(r-b)}\overline{F}_{q,0,0}^{k_1,\infty}(m-a,r-b,s-1), & \text{for}\ s>1,\ s\leq m\leq s(k_{1}-1)\\
    &\text{and}\ (s-1)+k_{2}\leq r \\
    1, & \text{for}\ s=1,\ 0<m<k_1,\ \text{and}\ k_{2}\leq r\\
    0, & \text{otherwise.}\\
  \end{array}
\right.
\end{split}
\end{equation*}
\end{lemma}
\begin{proof}
For $s > 1$, $s\leq m\leq s(k_{1}-1)$ and $(s-1)+k_{2}\leq r$, we observe that $x_{s}$ may assume any value $1,\ldots,k_{1}-1$, then $F_{q,0,k_2}^{k_1,\infty}(m,r,s)$ can be written as

\begin{equation*}
\begin{split}
F&_{q,0,k_2}^{k_1,\infty}(m,r,s)\\
=&\sum_{x_{s}=1}^{k_1-1}{\sum_{\substack{x_{1}+\cdots+x_{s-1}=m-x_{s}\\ x_{1},\ldots,x_{s-1} \in \{1,\ldots,k_{1}-1\}}}}\quad
{\sum_{\substack{(y_{1},\ldots,y_{s}) \in S_{r,s}^{k_2-1}}}}q^{x_{s}(r-y_{s})}\ q^{y_{1}x_{2}+(y_{1}+y_{2})x_{3}+\cdots+(y_{1}+\cdots+y_{s-2})x_{s-1}}
\end{split}
\end{equation*}
Similarly, we observe that since $y_s$ can assume the values $1,\ldots,r-(s-1)$, then $F_{q,0,k_2}^{k_1,\infty}(m,r,s)$ can be rewritten as

\begin{equation*}
\begin{split}
F&_{q,0,k_2}^{k_1,\infty}(m,r,s)\\
=&\sum_{y_{s}=1}^{k_{2}-1}\sum_{x_{s}=1}^{k_1-1}q^{x_{s}(r-y_{s})}{\hspace{-0.3cm}\sum_{\substack{x_{1}+\cdots+x_{s-1}=m-x_{s}\\ x_{1},\ldots,x_{s-1} \in \{1,\ldots,k_{1}-1\}}}}\quad
{\sum_{\substack{(y_{1},\ldots,y_{s-1}) \in S_{r-y_{s},s-1}^{k_2-1}}}} q^{y_{1}x_{2}+(y_{1}+y_{2})x_{3}+\cdots+(y_{1}+\cdots+y_{s-2})x_{s-1}}\\
&+\sum_{y_{s}=k_{2}}^{r-(s-1)}\sum_{x_{s}=1}^{k_1-1}q^{x_{s}(r-y_{s})}{\hspace{-0.3cm}\sum_{\substack{x_{1}+\cdots+x_{s-1}=m-x_{s}\\ x_{1},\ \ldots,\ x_{s-1} \in \{1,\ \ldots,\ k_{1}-1\}}}}\quad
{\sum_{\substack{(y_{1},\ldots,y_{s-1}) \in S_{r-y_{s},s-1}^{0}}}}\hspace{-0.5cm} q^{y_{1}x_{2}+(y_{1}+y_{2})x_{3}+\cdots+(y_{1}+\cdots+y_{s-2})x_{s-1}}\\
=&\sum_{b=1}^{k_{2}-1}\ \sum_{a=1}^{k_{1}-1}q^{a(r-b)}F_{q,0,k_2}^{k_1,\infty}(m-a,r-b,s-1)+\sum_{b=k_{2}}^{r-(s-1)}\ \sum_{a=1}^{k_{1}-1}q^{a(r-b)}\overline{F}_{q,0,0}^{k_1,\infty}(m-a,r-b,s-1).
\end{split}
\end{equation*}
The other cases are obvious and thus the proof is completed.
\end{proof}

%

{\hfill $\Box$}

\vskip 5mm

\begin{remark}
{\rm
We observe that $F_{1,0,k_2}^{k_1,\infty}(m,r,s)$ is the number of integer solutions $(x_{1},\ldots,x_{s})$ and $(y_{1},\ldots,y_{s})$ of
\noindent
\noindent
\begin{equation*}\label{eq:1}
\begin{split}
0<x_{j}<k_{1}\ \text{for}\ j=1,\ldots,s,
\end{split}
\end{equation*}
\noindent
\begin{equation*}\label{eq:1}
\begin{split}
(x_1,\ldots,x_{s})\in S_{m,s}^{0},\ \text{and}\end{split}
\end{equation*}
\noindent
\begin{equation*}\label{eq:2}
\begin{split}
(y_{1},\ldots,y_{s})\in S_{r,s}^{k_2-1}.
\end{split}
\end{equation*}
\noindent
which is
\noindent
\begin{equation*}\label{eq: 1.1}
\begin{split}
F_{1,0,k_2}^{k_1,\infty}(m,r,s)=S(s,\ k_{1},\ m)R(s,\ k_{2},\ r),
\end{split}
\end{equation*}
\noindent
where $S(a,\ b,\ c)$ denotes the total number of integer solution $x_{1}+x_{2}+\cdots+x_{a}=c$ such that $0<x_{i}<b$ for $i=1,2.\ldots,a$. The number is given by
\noindent
\begin{equation*}\label{eq: 1.1}
\begin{split}
S(a,\ b,\ c)=\sum_{j=0}^{min\left(a,\ \left[\frac{c-a}{b-1}\right]\right)}(-1)^{j}{a \choose j}{c-j(b-1)-1 \choose a-1}.
\end{split}
\end{equation*}
\noindent
where $R(a,\ b,\ c)$ denotes the total number of integer solution $y_{1}+x_{2}+\cdots+y_{a}=c$ such that $y_{1},\ldots,y_{a} \in S_{c,a}^{b-1}$. The number is given by
\noindent
\begin{equation*}\label{eq: 1.1}
\begin{split}
R(a,\ b,\ c)=\sum_{j=1}^{min\left(a,\ \left[\frac{c-a}{b-1}\right]\right)}(-1)^{j+1}{a \choose j}{c-j(b-1)-1 \choose a-1}.
\end{split}
\end{equation*}
\noindent
See, e.g. \citet{charalambides2002enumerative}.
}
\end{remark}

\vskip 5mm

\begin{definition}
For $0<q\leq1$, we define

\begin{equation*}\label{eq: 1.1}
\begin{split}
\overline{G}_{q,0,0}^{\infty,k_2}(u,v,s)={\sum_{x_{1},\ldots,x_{s}}}\
\sum_{y_{1},\ldots,y_{s-1}} q^{y_{1}x_{2}+(y_{1}+y_{2})x_{3}+\cdots+(y_{1}+\cdots+y_{s-1})x_{s}},\
\end{split}
\end{equation*}
\noindent
where the summation is over all integers $x_1,\ldots,x_{s},$ and $y_1,\ldots,y_{s-1}$ satisfying

\begin{equation*}\label{eq:1}
\begin{split}
(x_1,\ldots,x_{s})\in S_{u,s}^{0},\ \text{and}
\end{split}
\end{equation*}
\noindent
\begin{equation*}\label{eq:1}
\begin{split}
0<y_{j}<k_2\ \text{for}\ j=1,\ldots,s-1,
\end{split}
\end{equation*}
\noindent
\begin{equation*}\label{eq:2}
\begin{split}
(y_{1},\ldots,y_{s-1}) \in S_{v,s-1}^{0}.
\end{split}
\end{equation*}
\end{definition}
\noindent
The following gives a recurrence relation useful for the computation of $\overline{G}_{q,0,0}^{\infty,k_2}(u,v,s)$.
\noindent
\begin{lemma}
\label{lemma:3.9}
$\overline{G}_{q,0,0}^{\infty,k_2}(u,v,s)$ obeys the following recurrence relation.
\begin{equation*}\label{eq: 1.1}
\begin{split}
\overline{G}&_{q,0,0}^{\infty,k_2}(u,v,s)\\
&=\left\{
  \begin{array}{ll}
    \sum_{a=1}^{u-(s-1)}\sum_{b=1}^{k_2-1}q^{va} \overline{G}_{q,0,0}^{\infty,k_2}(u-a,v-b,s-1), & \text{for}\ s>1,\ s\leq u\\
    &\text{and}\ s-1\leq v\leq(s-1)(k_{2}-1) \\
    1, & \text{for}\ s=1,\ 1\leq u,\ \text{and}\  v=0\\
    0, & \text{otherwise.}\\
  \end{array}
\right.
\end{split}
\end{equation*}
\end{lemma}
\begin{proof}
For $s > 1$, $s\leq u$ and $s-1\leq v\leq(s-1)(k_{2}-1)$, we observe that $x_{s}$ may assume any value $1,\ldots,u-(s-1)$, then $\overline{G}_{q,0,0}^{\infty,k_2}(u,v,s)$ can be written as

\noindent
\begin{equation*}
\begin{split}
\overline{G}&_{q,0,0}^{\infty,k_2}(u,v,s)\\
=&\sum_{x_{s}=1}^{u-(s-1)}{\sum_{\substack{(x_1,\ldots,x_{s-1})\in S_{u-x_{s},s-1}^{0}}}}\
{\hspace{0.1cm}\sum_{\substack{y_{1},\ldots,y_{s-1} \in \{1,\ldots,k_{2}-1\}\\(y_{1},\ldots,y_{s-1}) \in S_{v,s-1}^{0}}}}q^{vx_{s}}q^{y_{1}x_{2}+(y_{1}+y_{2})x_{3}+\cdots+(y_{1}+\cdots+y_{s-2})x_{s-1}}.
\end{split}
\end{equation*}
\noindent
Similarly, we observe that since $y_s$ can assume the values $1,\ldots,k_{2}-1$, then $\overline{G}_{q,0,0}^{\infty,k_2}(u,v,s)$ can be rewritten as
\noindent

\begin{equation*}
\begin{split}
\overline{G}&_{q,0,0}^{\infty,k_2}(u,v,s)\\
=&\sum_{y_{s-1}=1}^{k_2-1}\sum_{x_{s}=1}^{u-(s-1)}{\sum_{\substack{(x_1,\ldots,x_{s-1})\in S_{u-x_{s},s-1}^{0}}}}\
{\hspace{0.1cm}\sum_{\substack{y_{1},\ldots,y_{s-2} \in \{1,\ldots,k_{2}-1\}\\(y_{1},\ldots,y_{s-2}) \in S_{v-y_{s-1},s-2}^{0}}}}q^{vx_{s}}\ q^{y_{1}x_{2}+(y_{1}+y_{2})x_{3}+\cdots+(y_{1}+\cdots+y_{s-2})x_{s-1}}\\
=&\sum_{b=1}^{k_2-1}\sum_{a=1}^{u-(s-1)}q^{va}\overline{G}_{q,0,0}^{\infty,k_2}(u-a,v-b,s-1)
\end{split}
\end{equation*}
The other cases are obvious and thus the proof is completed.
\end{proof}

\vskip 5mm
\begin{remark}
{\rm
We observe that $\overline{G}_{1,0,0}^{\infty,k_2}(u,v,s)$ is the number of integer solutions $(x_{1},\ldots,x_{s})$ and $(y_{1},\ldots,y_{s-1})$ of

\noindent
\begin{equation*}\label{eq:1}
\begin{split}
(x_1,\ldots,x_{s})\in S_{u,s}^{0},\ \text{and}
\end{split}
\end{equation*}

\noindent
\begin{equation*}\label{eq:1}
\begin{split}
0<y_{j}<k_2\ \text{for}\ j=1,\ldots,s-1,
\end{split}
\end{equation*}

\noindent
\begin{equation*}\label{eq:2}
\begin{split}
(y_{1},\ldots,y_{s-1}) \in S_{v,s-1}^{0}
\end{split}
\end{equation*}
\noindent
which is\noindent
\begin{equation*}\label{eq: 1.1}
\begin{split}
\overline{G}_{1,0,0}^{\infty,k_2}(u,v,s)=M(s,\ u )S(s-1,\ k_{2},\ v),
\end{split}
\end{equation*}

\noindent
where $M(a,\ b)$ denotes the total number of integer solution $x_{1}+x_{2}+\cdots+x_{a}=b$ such that $(x_{1},\ldots,x_{a}) \in S_{b,a}^{0}$. The number is given by
\noindent
\begin{equation*}\label{eq: 1.1}
\begin{split}
M(a,\ b)={b-1 \choose a-1}.
\end{split}
\end{equation*}

\noindent
where $S(a,\ b,\ c)$ denotes the total number of integer solution $y_{1}+y_{2}+\cdots+y_{a}=c$ such that $0<y_{i}<b$ for $i=1,2,\ldots,a$. The number is given by
\noindent
\begin{equation*}\label{eq: 1.1}
\begin{split}
S(a,\ b,\ c)=\sum_{j=0}^{min\left(a,\ \left[\frac{c-a}{b-1}\right]\right) }(-1)^{j}{a \choose j}{c-j(b-1)-1 \choose a-1}.
\end{split}
\end{equation*}
\noindent
See, e.g. \citet{charalambides2002enumerative}.
}
\end{remark}

\vskip 5mm
\begin{definition}
We define

\begin{equation*}\label{eq: 1.1}
\begin{split}
G_{q,k_1,0}^{\infty,k_2}(m,r,s)={\sum_{x_{1},\ldots,x_{s}}}\
\sum_{y_{1},\ldots,y_{s-1}} q^{y_{1}x_{2}+(y_{1}+y_{2})x_{3}+\cdots+(y_{1}+\cdots+y_{s-1})x_{s}},\
\end{split}
\end{equation*}
\noindent
where the summation is over all integers $x_1,\ldots,x_{s},$ and $y_1,\ldots,y_{s-1}$ satisfying

\begin{equation*}\label{eq:1}
\begin{split}
(x_1,\ldots,x_{s})\in S_{m,s}^{k_1-1},\ \text{and}
\end{split}
\end{equation*}
\noindent
\begin{equation*}\label{eq:1}
\begin{split}
0<y_{j}<k_2\ \text{for}\ j=1,\ldots,s-1,
\end{split}
\end{equation*}
\noindent
\begin{equation*}\label{eq:2}
\begin{split}
(y_{1},\ldots,y_{s-1}) \in S_{r,s-1}^{0}.
\end{split}
\end{equation*}
\end{definition}
\noindent
The following gives a recurrence relation useful for the computation of $G_{q,k_1,0}^{\infty,k_2}(m,r,s)$.\\
\noindent
\begin{lemma}
\label{lemma:3.10}
$G_{q,k_1,0}^{\infty,k_2}(m,r,s)$ obeys the following recurrence relation.
\begin{equation*}\label{eq: 1.1}
\begin{split}
G&_{q,k_1,0}^{\infty,k_2}(m,r,s)\\
&=\left\{
  \begin{array}{ll}
    \sum_{b=1}^{k_2-1}\sum_{a=1}^{k_{1}-1}q^{ra}G_{q,k_1,0}^{\infty,k_2}(m-a,r-b,s-1)\\
    +\sum_{b=1}^{k_2-1}\sum_{a=k_{1}}^{m-(s-1)}q^{ra}\overline{G}_{q,0,0}^{\infty,k_2}(m-a,r-b,s-1), & \text{for}\ s>1,\ (s-1)+k_{1}\leq m\\
    &\text{and}\ s-1\leq r\leq(s-1)(k_{2}-1) \\
    1, & \text{for}\ s=1,\ k_1\leq m,\ \text{and}\ r=0\\
    0, & \text{otherwise.}\\
  \end{array}
\right.
\end{split}
\end{equation*}
\end{lemma}
\begin{proof}
For $s > 1$, $(s-1)+k_{1}\leq m$ and $s-1\leq r\leq(s-1)(k_{2}-1)$, we observe that $x_{s}$ may assume any value $1,\ldots,m-(s-1)$, then $G_{q,k_1,0}^{\infty,k_2}(m,r,s)$ can be written as

\noindent
\begin{equation*}
\begin{split}
G&_{q,k_1,0}^{\infty,k_2}(m,r,s)\\
=&\sum_{x_{s}=1}^{k_{1}-1}{\sum_{\substack{(x_1,\ldots,x_{s-1})\in S_{m-x_{s},s-1}^{k_{1}-1}}}}
{\hspace{0.3cm}\sum_{\substack{y_{1},\ldots,y_{s-1} \in \{1,\ldots,k_{2}-1\}\\(y_{1},\ldots,y_{s-1}) \in S_{r,s-1}^{0}}}}q^{rx_{s}}\ \ q^{y_{1}x_{2}+(y_{1}+y_{2})x_{3}+\cdots+(y_{1}+\cdots+y_{s-2})x_{s-1}}\\
&+\sum_{x_{s}=k_{1}}^{m-(s-1)}{\sum_{\substack{(x_1,\ldots,x_{s-1})\in S_{m-x_{s},s-1}^{0}}}}{\hspace{0.3cm}\sum_{\substack{y_{1},\ldots,y_{s-1} \in \{1,\ldots,k_{2}-1\}\\(y_{1},\ldots,y_{s-1}) \in S_{r,s-1}^{0}}}}q^{rx_{s}}\ \ q^{y_{1}x_{2}+(y_{1}+y_{2})x_{3}+\cdots+(y_{1}+\cdots+y_{s-2})x_{s-1}}.
\end{split}
\end{equation*}
\noindent
Similarly, we observe that since $y_{s-1}$ can assume the values $1,\ldots,k_{2}-1$, then $G_{q,k_1,0}^{\infty,k_2}(m,r,s)$ can be rewritten as
\noindent

\begin{equation*}
\begin{split}
G&_{q,k_1,0}^{\infty,k_2}(m,r,s)\\
=&\sum_{y_{s-1}=1}^{k_2-1}\sum_{x_{s}=1}^{k_{1}-1}{\hspace{0.3cm}\sum_{\substack{(x_1,\ldots,x_{s-1})\in S_{m-x_{s},s-1}^{k_{1}-1}}}}\
{\hspace{0.3cm}\sum_{\substack{y_{1},\ldots,y_{s-2} \in \{1,\ldots,k_{2}-1\}\\(y_{1},\ldots,y_{s-2}) \in S_{r-y_{s-1},s-2}^{0}}}}q^{rx_{s}}\ \ q^{y_{1}x_{2}+(y_{1}+y_{2})x_{3}+\cdots+(y_{1}+\cdots+y_{s-2})x_{s-1}}\\
&+\sum_{y_{s-1}=1}^{k_2-1}\sum_{x_{s}=k_{1}}^{m-(s-1)}{\sum_{\substack{(x_1,\ldots,x_{s-1})\in S_{m-x_{s},s-1}^{0}}}}{\hspace{0.3cm}\sum_{\substack{y_{1},\ldots,y_{s-2} \in \{1,\ldots,k_{2}-1\}\\(y_{1},\ldots,y_{s-2}) \in S_{r-y_{s-1},s-2}^{0}}}}q^{rx_{s}}\ \ q^{y_{1}x_{2}+(y_{1}+y_{2})x_{3}+\cdots+(y_{1}+\cdots+y_{s-2})x_{s-1}}\\
=&\sum_{b=1}^{k_2-1}\sum_{a=1}^{k_{1}-1}q^{ra}G_{q,k_1,0}^{\infty,k_2}(u-a,v-b,s-1)+\sum_{b=1}^{k_2-1}\ \sum_{a=k_{1}}^{m-(s-1)}q^{ra} \overline{G}_{q,0,0}^{\infty,k_2}(u-a,v-b,s-1).
\end{split}
\end{equation*}
The other cases are obvious and thus the proof is completed.
\end{proof}
\vskip 5mm
\begin{remark}
{\rm
We observe that $G_{1,k_1,0}^{\infty,k_2}(m,r,s)$ is the number of integer solutions $(x_{1},\ldots,x_{s})$ and $(y_{1},\ldots,y_{s-1})$ of

\noindent
\begin{equation*}\label{eq:1}
\begin{split}
(x_1,\ldots,x_{s})\in S_{m,s}^{k_{1}-1},\ \text{and}
\end{split}
\end{equation*}

\noindent
\begin{equation*}\label{eq:1}
\begin{split}
0<y_{j}<k_2\ \text{for}\ j=1,\ldots,s-1,
\end{split}
\end{equation*}

\noindent
\begin{equation*}\label{eq:2}
\begin{split}
(y_{1},\ldots,y_{s-1}) \in S_{r,s-1}^{0}
\end{split}
\end{equation*}
\noindent
which is
\noindent
\begin{equation*}\label{eq: 1.1}
\begin{split}
G_{1,k_1,0}^{\infty,k_2}(m,r,s)=R(s,\ k_{1},\ m)S(s-1,\ k_{2},\ r),
\end{split}
\end{equation*}
\noindent
where $R(a,\ b,\ c)$ denotes the total number of integer solution $y_{1}+x_{2}+\cdots+y_{a}=c$ such that $(y_{1},\ldots,y_{a}) \in S_{c,a}^{b-1}$. The number is given by
\noindent
\begin{equation*}\label{eq: 1.1}
\begin{split}
R(a,\ b,\ c)=\sum_{j=1}^{min\left(a,\ \left[\frac{c-a}{b-1}\right]\right) }(-1)^{j+1}{a \choose j}{c-j(b-1)-1 \choose a-1}.
\end{split}
\end{equation*}
where $S(a,\ b,\ c)$ denotes the total number of integer solution $x_{1}+x_{2}+\cdots+x_{a}=c$ such that $0<x_{i}<b$ for $i=1,2,\ldots,a$. The number is given by
\noindent
\begin{equation*}\label{eq: 1.1}
\begin{split}
S(a,\ b,\ c)=\sum_{j=0}^{min\left(a,\ \left[\frac{c-a}{b-1}\right]\right) }(-1)^{j}{a \choose j}{c-j(b-1)-1 \choose a-1}.
\end{split}
\end{equation*}
See, e.g. \citet{charalambides2002enumerative}.
}
\end{remark}
\noindent
\begin{definition}
For $0<q\leq1$, we define

\begin{equation*}\label{eq: 1.1}
\begin{split}
\overline{H}_{q,0,0}^{\infty,k_2}(u,v,s)={\sum_{x_{1},\ldots,x_{s}}}\
\sum_{y_{1},\ldots,y_{s}} q^{y_{1}x_{1}+(y_{1}+y_{2})x_{2}+\cdots+(y_{1}+\cdots+y_{s})x_{s}},\
\end{split}
\end{equation*}
\noindent
where the summation is over all integers $x_1,\ldots,x_{s},$ and $y_1,\ldots,y_{s}$ satisfying

\begin{equation*}\label{eq:1}
\begin{split}
(x_1,\ldots,x_{s})\in S_{u,s}^{0},\ \text{and}
\end{split}
\end{equation*}
\noindent
\begin{equation*}\label{eq:1}
\begin{split}
0<y_{j}<k_2\ \text{for}\ j=1,\ldots,s,
\end{split}
\end{equation*}
\noindent
\begin{equation*}\label{eq:2}
\begin{split}
(y_{1},\ldots,y_{s}) \in S_{v,s}^{0}.
\end{split}
\end{equation*}
\end{definition}
\noindent
The following gives a recurrence relation useful for the computation of $\overline{H}_{q,0,0}^{\infty,k_2}(u,v,s)$.
\noindent
\begin{lemma}
\label{lemma:3.11}
$\overline{H}_{q,0,0}^{\infty,k_2}(u,v,s)$ obeys the following recurrence relation.
\begin{equation*}\label{eq: 1.1}
\begin{split}
\overline{H}&_{q,0,0}^{\infty,k_2}(u,v,s)\\
&=\left\{
  \begin{array}{ll}
    \sum_{a=1}^{u-(s-1)}\sum_{b=1}^{k_2-1}q^{va} \overline{H}_{q,0,0}^{\infty,k_2}(u-a,v-b,s-1), & \text{for}\ s>1,\ s\leq u\\
    &\text{and}\ s\leq v\leq s(k_{2}-1) \\
    1, & \text{for}\ s=1,\ 1\leq u,\ \text{and}\  1\leq v\leq k_2-1\\
    0, & \text{otherwise.}\\
  \end{array}
\right.
\end{split}
\end{equation*}
\end{lemma}
\begin{proof}
For $s > 1$, $s\leq u$ and $s\leq v\leq s(k_{2}-1)$, we observe that $x_{s}$ may assume any value $1,\ldots,u-(s-1)$, then $\overline{H}_{q,0,0}^{\infty,k_2}(u,v,s)$ can be written as

\noindent
\begin{equation*}
\begin{split}
\overline{H}&_{q,0,0}^{\infty,k_2}(u,v,s)\\
=&\sum_{x_{s}=1}^{u-(s-1)}{\sum_{\substack{(x_1,\ldots,x_{s-1})\in S_{u-x_{s},s-1}^{0}}}}{\hspace{0.3cm}\sum_{\substack{y_{1},\ldots,y_{s} \in \{1,\ldots,k_{2}-1\}\\(y_{1},\ldots,y_{s}) \in S_{v,s}^{0}}}}q^{vx_{s}}\ \ q^{y_{1}x_{1}+(y_{1}+y_{2})x_{2}+\cdots+(y_{1}+\cdots+y_{s-1})x_{s-1}}.
\end{split}
\end{equation*}
Similarly, we observe that since $y_s$ can assume the values $1,\ldots,k_{2}-1$, then $\overline{H}_{q,0,0}^{\infty,k_2}(u,v,s)$ can be rewritten as
\begin{equation*}
\begin{split}
\overline{H}&_{q,0,0}^{\infty,k_2}(u,v,s)\\
=&\sum_{y_{s}=1}^{k_2-1}\sum_{x_{s}=1}^{u-(s-1)}{\sum_{\substack{(x_1,\ldots,x_{s-1})\in S_{u-x_{s},s-1}^{0}}}}\
{\hspace{0.3cm}\sum_{\substack{y_{1},\ldots,y_{s-1} \in \{1,\ldots,k_{2}-1\}\\(y_{1},\ldots,y_{s-1}) \in S_{v-y_{s},s-1}^{0}}}}q^{vx_{s}}\ q^{y_{1}x_{1}+(y_{1}+y_{2})x_{2}+\cdots+(y_{1}+\cdots+y_{s-1})x_{s-1}}\\
=&\sum_{b=1}^{k_2-1}\sum_{a=1}^{u-(s-1)}q^{va}\overline{H}_{q,0,0}^{\infty,k_2}(u-a,v-b,s-1)
\end{split}
\end{equation*}
The other cases are obvious and thus the proof is completed.
\end{proof}

\vskip 5mm

\begin{remark}
{\rm
We observe that $\overline{H}_{1,0,0}^{\infty,k_2}(u,v,s)$ is the number of integer solutions $(x_{1},\ldots,x_{s})$ and $(y_{1},\ldots,y_{s})$ of

\noindent
\begin{equation*}\label{eq:1}
\begin{split}
(x_1,\ldots,x_{s})\in S_{u,s}^{0},\ \text{and}
\end{split}
\end{equation*}

\noindent
\begin{equation*}\label{eq:1}
\begin{split}
0<y_{j}<k_2\ \text{for}\ j=1,\ldots,s,
\end{split}
\end{equation*}

\noindent
\begin{equation*}\label{eq:2}
\begin{split}
(y_{1},\ldots,y_{s}) \in S_{v,s}^{0}
\end{split}
\end{equation*}
\noindent
which is
\noindent
\begin{equation*}\label{eq: 1.1}
\begin{split}
\overline{H}_{1,0,0}^{\infty,k_2}(u,v,s)=M(s,\ u )S(s,\ k_{2},\ v),
\end{split}
\end{equation*}

\noindent
where $M(a,\ b)$ denotes the total number of integer solution $x_{1}+x_{2}+\cdots+x_{a}=b$ such that $(x_{1},\ldots,x_{a}) \in S_{b,a}^{0}$. The number is given by
\noindent
\begin{equation*}\label{eq: 1.1}
\begin{split}
M(a,\ b)={b-1 \choose a-1}.
\end{split}
\end{equation*}

\noindent
where $S(a,\ b,\ c)$ denotes the total number of integer solution $y_{1}+y_{2}+\cdots+y_{a}=c$ such that $0<y_{i}<b$ for $i=1,2,\ldots,a$. The number is given by
\noindent
\begin{equation*}\label{eq: 1.1}
\begin{split}
S(a,\ b,\ c)=\sum_{j=0}^{min\left(a,\ \left[\frac{c-a}{b-1}\right]\right) }(-1)^{j}{a \choose j}{c-j(b-1)-1 \choose a-1}.
\end{split}
\end{equation*}
\noindent
See, e.g. \citet{charalambides2002enumerative}.
}
\end{remark}

\vskip 5mm

\begin{definition}
For $0<q\leq1$, we define

\begin{equation*}\label{eq: 1.1}
\begin{split}
H_{q,k_1,0}^{\infty,k_2}(m,r,s)={\sum_{x_{1},\ldots,x_{s}}}\
\sum_{y_{1},\ldots,y_{s}} q^{y_{1}x_{1}+(y_{1}+y_{2})x_{2}+\cdots+(y_{1}+\cdots+y_{s})x_{s}},\
\end{split}
\end{equation*}
\noindent
where the summation is over all integers $x_1,\ldots,x_{s},$ and $y_1,\ldots,y_{s}$ satisfying

\begin{equation*}\label{eq:1}
\begin{split}
(x_1,\ldots,x_{s})\in S_{m,s}^{k_1-1},\ \text{and}
\end{split}
\end{equation*}
\noindent
\begin{equation*}\label{eq:1}
\begin{split}
0<y_{j}<k_2\ \text{for}\ j=1,\ldots,s,
\end{split}
\end{equation*}
\noindent
\begin{equation*}\label{eq:2}
\begin{split}
(y_{1},\ldots,y_{s}) \in S_{r,s}^{0}.
\end{split}
\end{equation*}
\end{definition}
\noindent
The following gives a recurrence relation useful for the computation of $H_{q,k_1,0}^{\infty,k_2}(m,r,s)$.
\noindent
\begin{lemma}
\label{lemma:3.12}
$H_{q,k_1,0}^{\infty,k_2}(m,r,s)$ obeys the following recurrence relation.
\begin{equation*}\label{eq: 1.1}
\begin{split}
H&_{q,k_1,0}^{\infty,k_2}(m,r,s)\\
&=\left\{
  \begin{array}{ll}
    \sum_{b=1}^{k_2-1}\sum_{a=1}^{k_{1}-1}q^{ra}H_{q,k_1,0}^{\infty,k_2}(m-a,r-b,s-1)\\
    +\sum_{b=1}^{k_2-1}\sum_{a=k_{1}}^{m-(s-1)}q^{ra}\overline{H}_{q,0,0}^{\infty,k_2}(m-a,r-b,s-1), & \text{for}\ s>1,\ (s-1)+k_{1}\leq m\\
    &\text{and}\ s\leq r\leq s(k_{2}-1) \\
    1, & \text{for}\ s=1,\ k_1\leq m,\ \text{and}\ 1\leq r \leq k_2-1\\
    0, & \text{otherwise.}\\
  \end{array}
\right.
\end{split}
\end{equation*}
\end{lemma}
\begin{proof}
For $s > 1$, $(s-1)+k_{1}\leq m$ and $s\leq r\leq s(k_{2}-1)$, we observe that $x_{s}$ may assume any value $1,\ldots,m-(s-1)$, then $H_{q,k_1,0}^{\infty,k_2}(m,r,s)$ can be written as

\noindent
\begin{equation*}
\begin{split}
H&_{q,k_1,0}^{\infty,k_2}(m,r,s)\\
=&\sum_{x_{s}=1}^{k_{1}-1}{\hspace{0.3cm}\sum_{\substack{(x_1,\ldots,x_{s-1})\in S_{m-x_{s},s-1}^{k_{1}-1}}}}\
{\hspace{0.5cm}\sum_{\substack{y_{1},\ldots,y_{s} \in \{1,\ldots,k_{2}-1\}\\(y_{1},\ldots,y_{s}) \in S_{r,s}^{0}}}}q^{rx_{s}}q^{y_{1}x_{1}+(y_{1}+y_{2})x_{2}+\cdots+(y_{1}+\cdots+y_{s-1})x_{s-1}}\\
&+\sum_{x_{s}=k_{1}}^{m-(s-1)}{\sum_{\substack{(x_1,\ldots,x_{s-1})\in S_{m-x_{s},s-1}^{0}}}}{\hspace{0.5cm}\sum_{\substack{y_{1},\ldots, y_{s} \in \{1,\ldots,k_{2}-1\}\\(y_{1},\ldots,y_{s}) \in S_{r,s}^{0}}}}q^{rx_{s}} q^{y_{1}x_{1}+(y_{1}+y_{2})x_{2}+\cdots+(y_{1}+\cdots+y_{s-1})x_{s-1}}.
\end{split}
\end{equation*}
\noindent
Similarly, we observe that since $y_s$ can assume the values $1,\ldots,k_{2}-1$, then $H_{q,k_1,0}^{\infty,k_2}(m,r,s)$ can be rewritten as
\noindent
\begin{equation*}
\begin{split}
H&_{q,k_1,0}^{\infty,k_2}(m,r,s)\\
=&\sum_{y_{s}=1}^{k_2-1}\sum_{x_{s}=1}^{k_{1}-1}{\hspace{0.3cm}\sum_{\substack{(x_1,\ldots,x_{s-1})\in S_{m-x_{s},s-1}^{k_{1}-1}}}}{\hspace{0.5cm}\sum_{\substack{y_{1},\ldots,y_{s-1} \in \{1,\ldots,k_{2}-1\}\\(y_{1},\ldots,y_{s-1}) \in S_{r-y_{s},s-1}^{0}}}}q^{rx_{s}}q^{y_{1}x_{1}+(y_{1}+y_{2})x_{2}+\cdots+(y_{1}+\cdots+y_{s-1})x_{s-1}}\\
&+\sum_{y_{s}=1}^{k_2-1}\sum_{x_{s}=k_{1}}^{m-(s-1)}{\sum_{\substack{(x_1,\ldots,x_{s-1})\in S_{m-x_{s},s-1}^{0}}}}{\hspace{0.5cm}\sum_{\substack{y_{1},\ldots,y_{s-1} \in \{1,\ldots,k_{2}-1\}\\(y_{1},\ldots,y_{s-1}) \in S_{r-y_{s},s-1}^{0}}}}q^{rx_{s}}q^{y_{1}x_{1}+(y_{1}+y_{2})x_{2}+\cdots+(y_{1}+\cdots+y_{s-1})x_{s-1}}\\
=&\sum_{b=1}^{k_2-1}\sum_{a=1}^{k_{1}-1}q^{ra}H_{q,k_1,0}^{\infty,k_2}(m-a,r-b,s-1)+\sum_{b=1}^{k_2-1}\sum_{a=k_{1}}^{m-(s-1)}q^{ra}\overline{H}_{q,0,0}^{\infty,k_2}(m-a,r-b,s-1).
\end{split}
\end{equation*}
The other cases are obvious and thus the proof is completed.
\end{proof}
\vskip 5mm
\begin{remark}
{\rm
We observe that $H_{1,k_1,0}^{\infty,k_2}(m,r,s)$ is the number of integer solutions $(x_{1},\ldots,x_{s})$ and $(y_{1},\ldots,y_{s})$ of

\noindent
\begin{equation*}\label{eq:1}
\begin{split}
(x_1,\ldots,x_{s})\in S_{m,s}^{k_{1}-1},\ \text{and}
\end{split}
\end{equation*}

\noindent
\begin{equation*}\label{eq:1}
\begin{split}
0<y_{j}<k_2\ \text{for}\ j=1,\ldots,s,
\end{split}
\end{equation*}

\noindent
\begin{equation*}\label{eq:2}
\begin{split}
(y_{1},\ldots,y_{s}) \in S_{r,s}^{0}
\end{split}
\end{equation*}
\noindent
which is
\noindent
\begin{equation*}\label{eq: 1.1}
\begin{split}
H_{1,k_1,0}^{\infty,k_2}(m,r,s)=R(s,\ k_{1},\ m)S(s,\ k_{2},\ r),
\end{split}
\end{equation*}
\noindent
where $R(a,\ b,\ c)$ denotes the total number of integer solution $y_{1}+x_{2}+\cdots+y_{a}=c$ such that $(y_{1},\ldots,y_{a}) \in S_{c,a}^{b-1}$. The number is given by
\noindent
\begin{equation*}\label{eq: 1.1}
\begin{split}
R(a,\ b,\ c)=\sum_{j=1}^{min\left(a,\ \left[\frac{c-a}{b-1}\right]\right) }(-1)^{j+1}{a \choose j}{c-j(b-1)-1 \choose a-1}.
\end{split}
\end{equation*}
where $S(a,\ b,\ c)$ denotes the total number of integer solution $x_{1}+x_{2}+\cdots+x_{a}=c$ such that $0<x_{i}<b$ for $i=1,2,\ldots,a$. The number is given by
\noindent
\begin{equation*}\label{eq: 1.1}
\begin{split}
S(a,\ b,\ c)=\sum_{j=0}^{min\left(a,\ \left[\frac{c-a}{b-1}\right]\right) }(-1)^{j}{a \choose j}{c-j(b-1)-1 \choose a-1}.
\end{split}
\end{equation*}
See, e.g. \citet{charalambides2002enumerative}.}

\end{remark}
\noindent

\vskip 5mm

\noindent
The probability function of the $q$-later waiting time distribution of order $(k_1,k_2)$ when a quota is imposed on runs of successes and failures is obtained by the following theorem. It is evident that
\noindent
\begin{equation*}
P_{q,\theta}(W_{L}=n)=0\ \text{for}\ n< k_1+k_2
\end{equation*}
\noindent
and so we shall focus on determining the probability mass function for $n\geq k_{1}+k_{2}$.
\noindent
\begin{theorem}
\label{thm:3.2}
The PMF $f_{q,L}(n;\theta)=P_{q,\theta}(W_{L}=n)$ satisfies
\noindent
\begin{equation*}\label{eq: 1.1}
\begin{split}
P_{q,\theta}(W_{L}=n)=P_{q,L}^{(1)}(n)+P_{q,L}^{(0)}(n)\ \text{for}\ n\geq k_1+k_2,
\end{split}
\end{equation*}
where
\begin{equation*}\label{eq:bn1}
\begin{split}
P_{q,L}^{(1)}(n)=\sum_{i=k_2}^{n-k_{1}}\theta^{n-i} q^{ik_{1}}{\prod_{j=1}^{i}}(1-\theta q^{j-1})\sum_{s=1}^{i-k_2+1}\bigg[E_{q,0,k_2-1}^{k_1,\infty}&(n-k_{1}-i,\ i,\ s)\\
&+F_{q,0,k_2-1}^{k_1,\infty}(n-k_{1}-i,\ i,\ s)\bigg],\\
\end{split}
\end{equation*}
and

\begin{equation*}\label{eq:bn1}
\begin{split}
P_{q,L}^{(0)}(n)=\sum_{i=k_1}^{n-k_{2}}\theta^{i}{\prod_{j=1}^{n-i}}(1-\theta q^{j-1})\sum_{s=1}^{i-k_1+1}\Bigg[G_{q,k_1,0}^{\infty,k_2}(i,\ n-k_{2}-i,\ s)+H_{q,k_1,0}^{\infty,k_2}(i,\ n-k_{2}-i,\ s)\Bigg].\\
\end{split}
\end{equation*}
\end{theorem}
\begin{proof}
We start with the study of $P_{q,L}^{(1)}(n)$. It is easy to see that $P_{q,L}^{(1)}(k_1+k_2)={\prod_{j=1}^{k_2}}(1-\theta q^{j-1})(\theta q^{k_2})^{k_{1}}.$ From now on we assume $n > k_1+k_2,$ and we can rewrite $P_{q,L}^{(1)}(n)$ as follows
\begin{equation*}\label{eq:bn}
\begin{split}
P_{q,L}^{(1)}(n)&=P_{q,\theta}\left(L_{n-k_{1}}^{(1)}< k_{1}\ \wedge\ L_{n-k_{1}}^{(0)}\geq k_{2}\ \wedge\ X_{n-k_{1}}=0\ \wedge\ X_{n-k_{1}+1}=\cdots =X_{n}=1\right).\\
\end{split}
\end{equation*}
\noindent
We partition the event $W_{q,L}^{(1)}=n$ into disjoint events given by $F_{n-k_1}=i,$ for $i=k_2,\ldots,n-k_1.$ Adding the probabilities we have
\noindent
\begin{equation*}\label{eq:bn}
\begin{split}
P_{q,L}^{(1)}(n)=\sum_{i=k_2}^{n-k_{1}}P_{q,\theta}\Big(L_{n-k_{1}}^{(1)}< k_{1}\ \wedge\ L_{n-k_{1}}^{(0)}\geq k_{2}\ \wedge\ F_{n-k_{1}}=i&\ \wedge\ X_{n-k_{1}}=0\ \wedge\\
&X_{n-k_{1}+1}=\cdots =X_{n}=1\Big).\\
\end{split}
\end{equation*}
\noindent
If the number of $0$'s in the first $n-k_1$ trials is equal to $i,$ that is, $F_{n-k_1}=i,$ then in the $(n-k_{1}+1)$-th to $n$-th trials with probability of success
\noindent
\begin{equation*}\label{eq: kk}
\begin{split}
p_{n-k_{1}+1}=\cdots=p_{n}=\theta q^{i}.
\end{split}
\end{equation*}
\noindent
We will write $E_{n,i}^{(0)}$ for the event $\{L_n^{(1)}< k_1\ \wedge\ L_n^{(0)}\geq k_2\ \wedge\ X_n=0\ \wedge\ F_n=i\}.$ We can now rewrite as follows
\noindent
\begin{equation*}\label{eq:bn1}
\begin{split}
P_{q,L}^{(1)}(n)=\sum_{i=k_2}^{n-k_{1}}&P_{q,\theta}\Big(L_{n-k_{1}}^{(1)}< k_{1}\ \wedge\ L_{n-k_{1}}^{(0)}\geq k_{2}\ \wedge\ F_{n-k_{1}}=i\ \wedge\ X_{n-k_{1}}=0\Big)\\
&\times P_{q,\theta}\Big(X_{n-k_{1}+1}=\cdots =X_{n}=1\mid F_{n-k_{1}}=i\Big)\\
=\sum_{i=k_2}^{n-k_{1}}&P_{q,\theta}\Big(E_{n-k_{1},\ i}^{(0)}\Big)\Big(\theta q^{i}\Big)^{k_1}.\\
\end{split}
\end{equation*}
\noindent
We are going to focus on the event $E_{n-k_{1},i}^{(0)}$. For $i=1,\ldots,n-k_1$, a typical element of the event $\left\{L_n^{(1)}< k_1\ \wedge\ L_n^{(0)}\geq k_2\ \wedge\ X_n=0\ \wedge\ F_n=i\right\}$ is an ordered sequence which consists of $n-k_{1}-i$ successes and $i$ failures such that the length of the longest success run is less than $k_1$ and the length of the longest failure run is greater than or equal to $k_2$. The number of these sequences can be derived as follows. First we will distribute the $i$ failures. Let $s$ $(1\leq s \leq i-k_2+1)$ be the number of runs of failures in the event $E_{n-k_{1},i}^{(0)}$. Next, we will distribute the $n-k_{1}-i$ successes. We divide into two cases: starting with a failure run or starting with a success run. Thus, we distinguish between two types of sequences in the event $\left\{L_n^{(1)}< k_1\ \wedge\ L_n^{(0)}\geq k_2\ \wedge\ X_n=0\ \wedge\ F_n=i\right\},$ respectively named $(s-1,\ s)$-type and $(s,\ s)$-type, which are defined as follows
\noindent
\begin{equation*}\label{eq:bn}
\begin{split}
(s-1,\ s)\text{-type}\ :&\quad \overbrace{0\ldots 0}^{y_{1}}\mid\overbrace{1\ldots 1}^{x_{1}}\mid\overbrace{0\ldots 0}^{y_{2}}\mid\overbrace{1\ldots 1}^{x_{2}}\mid \ldots \mid \overbrace{0\ldots 0}^{y_{s-1}}\mid\overbrace{1\ldots 1}^{x_{s-1}}\mid\overbrace{0\ldots 0}^{y_{s}},\\
\end{split}
\end{equation*}
\noindent
with $i$ $0$'s and $n-k_{1}-i$ $1$'s, where $x_{j}$ $(j=1,\ldots,s-1)$ represents a length of run of $1$'s and $y_{j}$ $(j=1,\ldots,s)$ represents the length of a run of $0$'s. And all integers $x_{1},\ldots,x_{s-1}$, and $y_{1},\ldots,y_{s}$ satisfy the conditions
\noindent
\begin{equation*}\label{eq:bn}
\begin{split}
0 < x_j < k_1 \mbox{\ for\ } j=1,...,s-1, \mbox{\ and\ } x_1+\cdots +x_{s-1} = n-k_1-i,
\end{split}
\end{equation*}
\noindent
\begin{equation*}\label{eq:bn}
\begin{split}
(y_{1},\ldots,y_{s})\in S_{i,s}^{k_2-1},
\end{split}
\end{equation*}
\noindent
\begin{equation*}\label{eq:bn}
\begin{split}
(s,\ s)\text{-type}\  :&\quad \overbrace{1\ldots 1}^{x_{1}}\mid\overbrace{0\ldots 0}^{y_{1}}\mid\overbrace{1\ldots 1}^{x_{2}}\mid\overbrace{0\ldots 0}^{y_{2}}\mid\overbrace{1\ldots 1}^{x_{3}}\mid \ldots \mid \overbrace{0\ldots 0}^{y_{s-1}}\mid\overbrace{1\ldots 1}^{x_{s}}\mid\overbrace{0\ldots 0}^{y_{s}},
\end{split}
\end{equation*}
\noindent
with $i$ $0$'s and $n-k_{1}-i$ $1$'s, where $x_{j}$ $(j=1,\ldots,s)$ represents a length of run of $1$'s and $y_{j}$ $(j=1,\ldots,s)$ represents the length of a run of $0$'s. Here all of $x_1,\ldots, x_{s-1},$ and $y_1,\ldots, y_s$ are integers, and they satisfy
\noindent
\begin{equation*}\label{eq:bn}
\begin{split}
0 < x_j < k_1 \mbox{\ for\ } j=1,...,s, \mbox{\ and\ } x_1+\cdots +x_{s} = n-k_1-i,
\end{split}
\end{equation*}
\noindent
\begin{equation*}\label{eq:bn}
\begin{split}
(y_{1},\ldots,y_{s})\in S_{i,s}^{k_2-1},
\end{split}
\end{equation*}
\noindent
Then the probability of the event $E_{n-k_1,i}^{(0)}$ is given by

\noindent
\begin{equation*}\label{eq: 1.1}
\begin{split}
P_{q,\theta}\Big(E_{n-k_{1},\ i}^{(0)}&\Big)=\\
\sum_{s=1}^{i-k_2+1}\Bigg[\bigg\{&\sum_{\substack{x_{1}+\cdots+x_{s-1}=n-k_{1}-i\\ x_{1},\ldots,x_{s-1} \in \{1,\ldots,k_{1}-1\}}}{\hspace{0.3cm}\sum_{\substack{(y_{1},\ldots,y_{s})\in S_{i,s}^{k_2-1}}}}\big(1-\theta q^{0}\big)\cdots \big(1-\theta q^{y_{1}-1}\big)\times\\
&\Big(\theta q^{y_{1}}\Big)^{x_{1}}\big(1-\theta q^{y_{1}}\big)\cdots \big(1-\theta q^{y_{1}+y_{2}-1}\big)\times\\
&\Big(\theta q^{y_{1}+y_{2}}\Big)^{x_{2}}\big(1-\theta q^{y_{1}}\big)\cdots \big(1-\theta q^{y_{1}+y_{2}+y_3-1}\big)\times \\
&\quad \quad \quad \quad\quad \quad\quad \quad \quad \quad \quad  \vdots\\
&\Big(\theta q^{y_{1}+\cdots+y_{s-1}}\Big)^{x_{s-1}}
\big(1-\theta q^{y_{1}+\cdots+y_{s-1}}\big)\cdots \big(1-\theta q^{y_{1}+\cdots+y_{s}-1}\big)\bigg\}+\\
\end{split}
\end{equation*}

\noindent
\begin{equation*}\label{eq: 1.1}
\begin{split}
\quad \quad \quad \quad \quad \quad&\bigg\{\sum_{\substack{x_{1}+\cdots+x_{s}=n-k_{1}-i\\ x_{1},\ldots,x_{s} \in \{1,\ldots,k_{1}-1\}}}{\hspace{0.3cm}\sum_{\substack{(y_{1},\ldots,y_{s})\in S_{i,s}^{k_2-1}}}}\big(\theta q^{0}\big)^{x_{1}}\big(1-\theta q^{0}\big)\cdots (1-\theta q^{y_{1}-1})\times\\
&\Big(\theta q^{y_{1}}\Big)^{x_{2}}\big(1-\theta q^{y_{1}}\big)\cdots \big(1-\theta q^{y_{1}+y_{2}-1}\big)\times\\
&\Big(\theta q^{y_{1}+y_{2}}\Big)^{x_{3}}\big(1-\theta q^{y_{1}}\big)\cdots \big(1-\theta q^{y_{1}+y_{2}+y_3-1}\big)\times \\
&\quad \quad \quad \quad\quad \quad\quad \quad \quad \quad \quad  \vdots\\
&\Big(\theta q^{y_{1}+\cdots+y_{s-1}}\Big)^{x_{s}}
\big(1-\theta q^{y_{1}+\cdots+y_{s-1}}\big)\cdots \big(1-\theta q^{y_{1}+\cdots+y_{s}-1}\big)\bigg\}\Bigg],\\
\end{split}
\end{equation*}
\noindent

and then using simple exponentiation algebra arguments to simplify,
\noindent
\begin{equation*}\label{eq: 1.1}
\begin{split}
&P_{q,\theta}\Big(E_{n-k_{1},\ i}^{(0)}\Big)=\\
&\theta^{n-k_{1}-i}{\prod_{j=1}^{i}}\ (1-\theta q^{j-1})\times\\
&\sum_{s=1}^{i-k_2+1}\Bigg[\sum_{\substack{x_{1}+\cdots+x_{s-1}=n-k_{1}-i\\ x_{1},\ldots,x_{s-1} \in \{1,\ldots,k_{1}-1\}}}\
\sum_{\substack{(y_{1},\ldots,y_{s})\in S_{i,s}^{k_2-1}}}q^{y_{1}x_{1}+(y_{1}+y_{2})x_{2}+\cdots+(y_{1}+\cdots+y_{s-1})x_{s-1}}+\\
&\quad\quad\quad\quad\sum_{\substack{x_{1}+\cdots+x_{s}=n-k_{1}-i\\ x_{1},\ldots,x_{s} \in \{1,\ldots,k_{1}-1\}}}\
\sum_{\substack{(y_{1},\ldots,y_{s})\in S_{i,s}^{k_2-1}}}q^{y_{1}x_{2}+(y_{1}+y_{2})x_{3}+\cdots+(y_{1}+\cdots+y_{s-1})x_{s}}\Bigg]\\
\end{split}
\end{equation*}
\noindent
Using Lemma \ref{lemma:3.6} and Lemma \ref{lemma:3.8}, we can rewrite as follows
\noindent
\begin{equation*}\label{eq: 1.1}
\begin{split}
P_{q,\theta}\Big(E_{n-k_{1},\ i}^{(0)}\Big)=\theta^{n-k_{1}-i}{\prod_{j=1}^{i}}(1-\theta q^{j-1})\sum_{s=1}^{i-k_2+1}\Bigg[E&_{q,0,k_2}^{k_1,\infty}(n-k_{1}-i,\ i,\ s)\\
&+F_{q,0,k_2}^{k_1,\infty}(n-k_{1}-i,\ i,\ s)\Bigg],\\
\end{split}
\end{equation*}\\
\noindent
where
\noindent
\begin{equation*}\label{eq: 1.1}
\begin{split}
E_{q,0,k_2}^{k_1,\infty}(n&-k_{1}-i,\ i,\ s)\\
&=\sum_{\substack{x_{1}+\cdots+x_{s-1}=n-k_{1}-i\\ x_{1},\ldots,x_{s-1} \in \{1,\ldots,k_{1}-1\}}}\
\sum_{\substack{(y_{1},\ldots,y_{s})\in S_{i,s}^{k_2-1}}}q^{y_{1}x_{1}+(y_{1}+y_{2})x_{2}+\cdots+(y_{1}+\cdots+y_{s-1})x_{s-1}},
\end{split}
\end{equation*}
\noindent
and
\noindent
\begin{equation*}\label{eq: 1.1}
\begin{split}
F_{q,0,k_2}^{k_1,\infty}(n&-k_{1}-i,\ i,\ s)\\
&=\sum_{\substack{x_{1}+\cdots+x_{s}=n-k_{1}-i\\ x_{1},\ldots,x_{s} \in \{1,\ldots,k_{1}-1\}}}\
\sum_{\substack{(y_{1},\ldots,y_{s})\in S_{i,s}^{k_2-1}}}q^{y_{1}x_{2}+(y_{1}+y_{2})x_{3}+\cdots+(y_{1}+\cdots+y_{s-1})x_{s}}.
\end{split}
\end{equation*}
\noindent
Therefore we can compute the probability of the event $W_{S}^{(1)}=n$ as follows
\noindent
\begin{equation*}\label{eq:bn1}
\begin{split}
P_{q,L}^{(1)}(n)=&\sum_{i=k_2}^{n-k_{1}}P_{q,\theta}\Big(E_{n-k_{1},\ i}^{(0)}\Big)\Big(\theta q^{i}\Big)^{k_{1}}\\
=&\sum_{i=k_2}^{n-k_{1}}\theta^{n-k_{1}-i}{\prod_{j=1}^{i}}(1-\theta q^{j-1})\sum_{s=1}^{i-k_2+1}\bigg[E_{q,0,k_2}^{k_1,\infty}(n-k_{1}-i,\ i,\ s)\\
&\quad\quad\quad\quad\quad\quad\quad\quad\quad\quad\quad\quad\quad\quad+F_{q,0,k_2}^{k_1,\infty}(n-k_{1}-i,\ i,\ s)\bigg]\Big(\theta q^{i}\Big)^{k_{1}}.\\
\end{split}
\end{equation*}
\noindent
Finally, applying typical factorization algebra arguments, $P_{S}^{(1)}(n)$ can be rewritten as follows
\noindent
\begin{equation*}\label{eq:bn1}
\begin{split}
P_{q,L}^{(1)}(n)=\sum_{i=k_2}^{n-k_{1}}\theta^{n-i} q^{ik_{1}}{\prod_{j=1}^{i}}(1-\theta q^{j-1})\sum_{s=1}^{i-k_2+1}\bigg[E_{q,0,k_2}^{k_1,\infty}(n-k_{1}-i,\ i,\ s)+F_{q,0,k_2}^{k_1,\infty}(n-k_{1}-i,\ i,\ s)\bigg].\\
\end{split}
\end{equation*}

\noindent
We are now going to study of $P_{q,L}^{(0)}(n)$. It is easy to see that $P_{q,L}^{(0)}(k_1+k_2)=\theta^{k_{1}}{\prod_{j=1}^{k_2}}(1-\theta q^{j-1})$. From now on we assume $n > k_1+k_2,$ and we can write
\noindent
$P_{q,L}^{(0)}(n)$ as follows.

\noindent
\begin{equation*}\label{eq:bn}
\begin{split}
P_{q,L}^{(0)}(n)&=P\left(L_{n-k_{2}}^{(1)}\geq k_{1}\ \wedge\ L_{n-k_{2}}^{(0)}< k_{2}\ \wedge\ X_{n-k_{2}}=1\ \wedge\ X_{n-k_{2}+1}=\cdots =X_{n}=0\right).\\
\end{split}
\end{equation*}


\noindent
We partition the event $W_{L}^{(0)}=n$ into disjoint events given by $S_{n-k_2}=i$ and $F_{n-k_2}=n-k_2-i,$ for $i=k_1,\ \ldots,\ n-k_2.$ Adding the probabilities we have
\noindent
\begin{equation*}\label{eq:bn}
\begin{split}
P_{q,L}^{(0)}(n)=\sum_{i=k_1}^{n-k_{2}}P_{q,\theta}\Big(L_{n-k_{2}}^{(1)}\geq k_{1}\ \wedge\ L_{n-k_{2}}^{(0)}< k_{2}\ \wedge\ &S_{n-k_{2}}=i\ \wedge\ X_{n-k_{2}}=1\ \wedge\\
&X_{n-k_{2}+1}=\cdots =X_{n}=0\Big).\\
\end{split}
\end{equation*}
\noindent
If the number of $0$'s in the first $n-k_2$ trials is equal to $n-k_2-i,$ that is, $F_{n-k_2}=n-k_2-i,$ then the probability of failures in all of the $(n-k_{2}+1)$-th to $n$-th trials is
\noindent
\begin{equation*}\label{eq: kk}
\begin{split}
P_{q,\theta}(X_{n-k_2+1}=\cdots = X_n = 0 \,\mid \, F_{n-k_2}=n-k_2-i)={\prod_{j=n-k_2-i+1}^{n-i}}\left(1-\theta q^{j-1}\right).
\end{split}
\end{equation*}
\noindent
We will write $E_{n,i}^{(1)}$ for the event $\left\{L_n^{(1)}\geq k_1\ \wedge\ L_n^{(1)}< k_2\ \wedge\ X_n=1\ \wedge\ F_n=i\right\}.$ We can now rewrite as follows
\noindent
\begin{equation*}\label{eq:bn1}
\begin{split}
P_{q,L}^{(0)}(n)=\sum_{i=k_1}^{n-k_{2}}&P_{q,\theta}\Big(L_{n-k_{1}}^{(1)}\geq k_{1}\ \wedge\ L_{n-k_{1}}^{(0)}< k_{2}\ \wedge\ F_{n-k_{2}}=n-k_2-i\ \wedge\ X_{n-k_{2}}=1\Big)\\
&\times P_{q,\theta}\Big(X_{n-k_{2}+1}=\cdots =X_{n}=0\mid F_{n-k_{2}}=n-k_2-i\Big)\\
=\sum_{i=k_1}^{n-k_{2}}&P_{q,\theta}\Big(E_{n-k_{2},\ i}^{(1)}\Big){\prod_{j=n-k_2-i+1}^{n-i}}\left(1-\theta q^{j-1}\right).\\
\end{split}
\end{equation*}
\noindent
We are going to focus on the event $E_{n-k_{2},i}^{(1)}$. For $i=k_1,\ldots,n-k_2$, a typical element of the event $\left\{L_{n-k_1}^{(1)}> k_1-1\ \wedge\ L_{n-k_1}^{(0)}\leq k_2-1\ \wedge\ F_{n-k_2}=i\right\}$ is an ordered sequence which consists of $i$ successes and $n-k_{2}-i$ failures such that the length of the longest success run is greater than or equal to $k_1$ and the length of the longest failure run is less than $k_2$. The number of these sequences can be derived as follows. First we will distribute the $i$ successes. Let $s$ $(1\leq s \leq i-k_1+1)$ be the number of runs of successes in the event $E_{n-k_{2},i}^{(1)}$. Next, we will distribute the $n-k_{1}-i$ failures. We divide into two cases: starting with a success run or starting with a failure run. Thus, we distinguish between two types of sequences in the event $\left\{L_{n-k_1}^{(1)}\geq k_1\ \wedge\ L_{n-k_1}^{(0)}< k_2\ \wedge\ F_{n-k_2}=i\right\},$ respectively named $(s-1,\ s)$-type and $(s,\ s)$-type, which are defined as follows.
\noindent
\begin{equation*}\label{eq:bn}
\begin{split}
(s,\ s-1)\text{-type}\ :&\quad \overbrace{1\ldots 1}^{x_{1}}\mid\overbrace{0\ldots 0}^{y_{1}}\mid\overbrace{1\ldots 1}^{x_{2}}\mid\overbrace{0\ldots 0}^{y_{2}}\mid \ldots \mid\overbrace{0\ldots 0}^{y_{s-1}}\mid\overbrace{1\ldots 1}^{x_{s}},\\
\end{split}
\end{equation*}
\noindent
with $n-k_{2}-i$ $0$'s and $i$ $1$'s, where $x_{j}$ $(j=1,\ldots,s)$ represents a length of run of $1$'s and $y_{j}$ $(j=1,\ldots,s-1)$ represents the length of a run of $0$'s. And all integers $x_{1},\ldots,x_{s}$, and $y_{1},\ldots ,y_{s-1}$ satisfy the conditions
\noindent
\begin{equation*}\label{eq:bn}
\begin{split}
(x_{1},\ldots,x_{s})\in S_{i,s}^{k_1-1}, \text{and}
\end{split}
\end{equation*}
\noindent
\begin{equation*}\label{eq:bn}
\begin{split}
0 < y_j < k_2 \mbox{\ for\ } j=1,...,s-1, \mbox{\ and\ } y_1+\cdots +y_{s-1} = n-k_2-i,
\end{split}
\end{equation*}
\noindent
\begin{equation*}\label{eq:bn}
\begin{split}
(s,\ s)\text{-type}\  :&\quad \overbrace{0\ldots 0}^{y_{1}}\mid\overbrace{1\ldots 1}^{x_{1}}\mid\overbrace{0\ldots 0}^{y_{2}}\mid\overbrace{1\ldots 1}^{x_{2}}\mid\overbrace{0\ldots 0}^{y_{3}}\mid \ldots \mid\overbrace{0\ldots 0}^{y_{s}}\mid\overbrace{1\ldots 1}^{x_{s}},
\end{split}
\end{equation*}
\noindent
with $n-k_{2}-i$ $0$'s and $i$ $1$'s, where $x_{j}$ $(j=1,\ldots,s)$ represents a length of run of $1$'s and $y_{j}$ $(j=1,\ldots,s)$ represents the length of a run of $0$'s. And all integers $x_{1},\ldots,x_{s}$, and $y_{1},\ldots ,y_{s}$ satisfy the conditions
\noindent
\begin{equation*}\label{eq:bn}
\begin{split}
(x_{1},\ldots,x_{s})\in S_{i,s}^{k_1-1}, \text{and}
\end{split}
\end{equation*}
\noindent
\begin{equation*}\label{eq:bn}
\begin{split}
0 < y_j < k_2 \mbox{\ for\ } j=1,...,s, \mbox{\ and\ } y_1+\cdots +y_{s} = n-k_2-i,
\end{split}
\end{equation*}
\noindent
Then the probability of the event $E_{n-k_2,i}^{(1)}$ is given by
\noindent
\begin{equation*}\label{eq: 1.1}
\begin{split}
P_{q,\theta}\Big(E_{n-k_{2},\ i}^{(1)}&\Big)=\\
\sum_{s=1}^{i-k_1+1}\Bigg[\bigg\{&\sum_{\substack{(x_{1},\ldots,x_{s})\in S_{i,s}^{k_1-1}}}\
\sum_{\substack{y_1+\cdots +y_{s-1} = n-k_2-i\\ y_{1},\ldots,y_{s-1} \in \{1,\ldots,k_{2}-1\}}} \big(\theta q^{0}\big)^{x_{1}}\big(1-\theta q^{0}\big)\cdots (1-\theta q^{y_{1}-1})\times\\
&\Big(\theta q^{y_{1}}\Big)^{x_{2}}\big(1-\theta q^{y_{1}}\big)\cdots \big(1-\theta q^{y_{1}+y_{2}-1}\big)\times\\
&\Big(\theta q^{y_{1}+y_{2}}\Big)^{x_{3}}\big(1-\theta q^{y_{1}+y_2}\big)\cdots \big(1-\theta q^{y_{1}+y_{2}+y_3-1}\big)\times \\
&\quad \quad \quad \vdots\\
&\Big(\theta q^{y_{1}+\cdots+y_{s-1}}\Big)^{x_{s}}\bigg\}+\\
\end{split}
\end{equation*}

\noindent
\begin{equation*}\label{eq: 1.1}
\begin{split}
\quad \quad \quad \quad \quad \quad \bigg\{&\sum_{\substack{(x_{1},\ldots,x_{s})\in S_{i,s}^{k_1-1}}}\
\sum_{\substack{y_{1}+\cdots+y_{s}=n-k_2-i\\ y_{1},\ldots,y_{s} \in \{1,\ldots,k_{2}-1\}}}\big(1-\theta q^{0}\big)\cdots \big(1-\theta q^{y_{1}-1}\big)\Big(\theta q^{y_{1}}\Big)^{x_{1}}\times\\
&\big(1-\theta q^{y_{1}}\big)\cdots \big(1-\theta q^{y_{1}+y_{2}-1}\big)\Big(\theta q^{y_{1}+y_{2}}\Big)^{x_{2}}\times\\
&\big(1-\theta q^{y_{1}+y_2}\big)\cdots \big(1-\theta q^{y_{1}+y_{2}+y_3-1}\big)\Big(\theta q^{y_{1}+y_{2}+y_3}\Big)^{x_{3}}\times \\
&\quad \quad \quad \quad\quad \quad\quad \quad \quad \quad \quad \vdots\\
&\big(1-\theta q^{y_{1}+\cdots+y_{s-1}}\big)\cdots \big(1-\theta q^{y_{1}+\cdots+y_{s}-1}\big)\Big(\theta q^{y_{1}+\cdots+y_{s}}\Big)^{x_{s}}\bigg\}\Bigg].\\
\end{split}
\end{equation*}
\noindent
and then using simple exponentiation algebra arguments to simplify,
\noindent
\begin{equation*}\label{eq: 1.1}
\begin{split}
&P_{q,\theta}\Big(E_{n-k_{2},\ i}^{(1)}\Big)=\\
&\theta^{i}{\prod_{j=1}^{n-k_{2}-i}}\ (1-\theta q^{j-1})\times\\
&\sum_{s=1}^{i-k_1+1}\Bigg[\sum_{\substack{(x_{1},\ldots,x_{s})\in S_{i,s}^{k_1-1}}}\
\sum_{\substack{y_1+\cdots +y_{s-1} = n-k_2-i\\ y_{1},\ldots,y_{s-1} \in \{1,\ldots,k_{2}-1\}}}q^{y_{1}x_{2}+(y_{1}+y_{2})x_{3}+\cdots+(y_{1}+\cdots+y_{s-1})x_{s}}+\\
&\quad\quad\quad\sum_{\substack{(x_{1},\ldots,x_{s})\in S_{i,s}^{k_1-1}}}\
\sum_{\substack{y_1+\cdots +y_{s} = n-k_2-i\\ y_{1},\ldots,y_{s} \in \{1,\ldots,k_{2}-1\}}}q^{y_{1}x_{1}+(y_{1}+y_{2})x_{2}+\cdots+(y_{1}+\cdots+y_{s})x_{s}}\Bigg]\\
\end{split}
\end{equation*}
\noindent
Using Lemma \ref{lemma:3.10} and Lemma \ref{lemma:3.12}, we can rewrite as follows
\noindent
\begin{equation*}\label{eq: 1.1}
\begin{split}
P_{q,\theta}\Big(E_{n-k_{2},\ i}^{(1)}\Big)=\theta^{i}{\prod_{j=1}^{n-k_{2}-i}}\ (1-\theta q^{j-1})\sum_{s=1}^{i-k_1+1}\Bigg[G_{q,k_1,0}^{\infty,k_2}(i,\ n-k_{2}-i,\ s)+H_{q,k_1,0}^{\infty,k_2}(i,\ n-k_{2}-i,\ s)\Bigg],\\
\end{split}
\end{equation*}\\
\noindent
where
\noindent
\begin{equation*}\label{eq: 1.1}
\begin{split}
G_{q,k_1,0}^{\infty,k_2}(&i,\ n-k_{2}-i,\ s)\\
&=\sum_{\substack{(x_{1},\ldots,x_{s})\in S_{i,s}^{k_1-1}}}\
\sum_{\substack{y_1+\cdots +y_{s-1} = n-k_2-i\\ y_{1},\ldots,y_{s-1} \in \{1,\ldots,k_{2}-1\}}}q^{y_{1}x_{2}+(y_{1}+y_{2})x_{3}+\cdots+(y_{1}+\cdots+y_{s-1})x_{s}},
\end{split}
\end{equation*}
\noindent
and
\noindent
\begin{equation*}\label{eq: 1.1}
\begin{split}
H_{q,k_1,0}^{\infty,k_2}(&i,\ n-k_{2}-i,\ s)\\
&=\sum_{\substack{(x_{1},\ldots,x_{s})\in S_{i,s}^{k_1-1}}}\
\sum_{\substack{y_1+\cdots +y_{s} = n-k_2-i\\ y_{1},\ldots,y_{s} \in \{1,\ldots,k_{2}-1\}}}q^{y_{1}x_{1}+(y_{1}+y_{2})x_{2}+\cdots+(y_{1}+\cdots+y_{s})x_{s}}.
\end{split}
\end{equation*}
\noindent
Therefore we can compute the probability of the event $W_{L}^{(0)}=n$ as follows
\noindent
\begin{equation*}\label{eq:bn1}
\begin{split}
P_{q,L}^{(0)}(n)=&\sum_{i=k_1}^{n-k_{2}}P\Big(E_{n-k_{2},\ i}^{(1)}\Big){\prod_{j=n-k_2-i+1}^{n-i}}(1-\theta q^{j-1})\\
=&\sum_{i=k_1}^{n-k_{2}}\theta^{i}{\prod_{j=1}^{n-k_{2}-i}}\ (1-\theta q^{j-1})\sum_{s=1}^{i-k_1+1}\Bigg[G_{q,k_1,0}^{\infty,k_2}(i,\ n-k_{2}-i,\ s)+H_{q,k_1,0}^{\infty,k_2}(i,\ n-k_{2}-i,\ s)\Bigg]\\
&{\prod_{j=n-k_2-i+1}^{n-i}}(1-\theta q^{j-1}).\\
\end{split}
\end{equation*}
\noindent
Finally, applying typical factorization algebra arguments, $P_{L}^{(0)}(n)$ can be rewritten as follows
\noindent
\begin{equation*}\label{eq:bn1}
\begin{split}
P_{q,L}^{(0)}(n)=\sum_{i=k_1}^{n-k_{2}}\theta^{i}{\prod_{j=1}^{n-i}}(1-\theta q^{j-1})\sum_{s=1}^{i-k_1+1}\Bigg[G_{q,k_1,0}^{\infty,k_2}(i,\ n-k_{2}-i,\ s)+H_{q,k_1,0}^{\infty,k_2}(i,\ n-k_{2}-i,\ s)\Bigg].\\
\end{split}
\end{equation*}
Thus proof is completed.
\end{proof}
\noindent
It is worth mentioning here that the PMF $f_{q,L}(n;\theta)$ approaches the probability function of the later waiting time distribution of order $(k_1,k_2)$
in the limit as $q$ tends to 1 when a quota is imposed on runs of successes and failures of IID model. The details are presented in the following remark.
\noindent
\begin{remark}
{\rm
For $q=1$, the PMF $f_{q,L}(n;\theta)$ reduces to the PMF $f_{L}(n;\theta)=P_{\theta}(W_{L}=n)$ for $n\geq k_1+k_2$ is given by

\begin{equation*}\label{eq: 1.1}
\begin{split}
P(W_{L}=n)=P_{L}^{(1)}(n)+P_{L}^{(0)}(n),\ n\geq k_1+k_2,
\end{split}
\end{equation*}
where
\begin{equation*}\label{eq:bn1}
\begin{split}
P_{L}^{(1)}(n)=\sum_{i=k_2}^{n-k_{1}}\theta^{n-i} (1-\theta)^{i}\sum_{s=1}^{i-k_2+1}R(s,\ k_2,\ i)\Big[S(s-1,\ k_{1},\ n-k_{1}-i)+S(s,\ k_{1},\ n-k_{1}-i)\Big],\\
\end{split}
\end{equation*}
and

\begin{equation*}\label{eq:bn1}
\begin{split}
P_{L}^{(0)}(n)=\sum_{i=k_1}^{n-k_{2}}\theta^{i}(1-\theta)^{n-i}\sum_{s=1}^{i-k_1+1}R(s,\ k_1,\ i)\Big[S(s-1,\ k_2, n-k_{2}-i)+S(s,\ k_2,\ n-k_{2}-i)\Big],\\
\end{split}
\end{equation*}

where

\noindent
\begin{equation*}\label{eq: 1.1}
\begin{split}
R(a,\ b,\ c)=\sum_{j=1}^{min\left(a,\ \left[\frac{c-a}{b-1}\right]\right)}(-1)^{j+1}{a \choose j}{c-j(b-1)-1 \choose a-1}.
\end{split}
\end{equation*}
\noindent
and
\noindent
\begin{equation*}\label{eq: 1.1}
\begin{split}
S(a,\ b,\ c)=\sum_{j=0}^{min(a,\left[\frac{c-a}{b-1}\right])}(-1)^{j}{a \choose j}{c-j(b-1)-1 \choose a-1}.
\end{split}
\end{equation*}
\noindent
See, e.g. \citet{charalambides2002enumerative}.
}
\end{remark}

\section{Run and frequency quotas for $q$-binary trials}
In the present section we shall study of the $q$-sooner and later waiting times to be discussed will arise by setting quotas on both runs and frequencies of successes and failures. First, we consider sooner cases (or later cases) impose a frequency quota on successes and a run quota on failures. Next, we consider sooner cases (or later cases) impose a run quota on successes and a frequency quota on failures.

\subsection{A frequency quota on successes and a run quota on failures are imposed}
we shall discuss of the $q$-sooner and later waiting times impose a frequency quota on successes and a run quota on failures. Each cases will be described in the following subsection.

\vskip 5mm

\subsubsection{Sooner waiting time}
The problem of waiting time that will be discussed in this section is one of the 'sooner cases' and it emerges when a frequency quota on successes and a run quota on failures are imposed. More specifically, Binary (zero and one) trials with probability of ones varying according to a geometric rule, are performed sequentially until $k_1$ successes in total or $k_2$ consecutive failures are observed, whichever event occurs first. Let $\widehat{W}_{S}$ be a random variable denoting that the waiting time until either $k_{1}$ successes in total or $k_{2}$ consecutive failures are occurred, whichever event observe sooner.
\noindent
The probability function of the $q$-sooner waiting time distribution impose a frequency quota on successes and a run quota on failures is obtained by the following theorem. It is evident that
\noindent
\begin{equation*}
P_{q,\theta}(\widehat{W}_{S}=n)=0\ \text{for}\ 0\leq n<\text{min}(k_1,k_2)
\end{equation*}
\noindent
and so we shall focus on determining the probability mass function for $n\geq\min(k_{1},\ k_{2})$.
\noindent

\begin{theorem}
\label{thm:4.1}
The PMF $P_{q,\theta}(\widehat{W}_{S}=n)$ satisfies
\noindent
\begin{equation*}\label{eq:bn}
\begin{split}
P_{q,\theta}(\widehat{W}_{S}=n)=\widehat{P}_{q,S}^{(1)}(n)+\widehat{P}_{q,S}^{(0)}(n),\ \text{for}\ n\geq\min(k_{1},\ k_{2}),
\end{split}
\end{equation*}
\noindent
where $\widehat{P}_{q,S}^{(1)}(k_1)=\theta^{k_{1}}$, $\widehat{P}_{q,S}^{(0)}(k_2)={\prod_{j=1}^{k_2}}(1-\theta q^{j-1})$,
\noindent
\begin{equation*}\label{eq:bn1}
\begin{split}
\widehat{P}_{q,S}^{(1)}(n)=\theta^{k_{1}}\ {\prod_{j=1}^{n-k_1}}\ (1-\theta q^{j-1})\ \sum_{s=1}^{k_1}\bigg[\overline{H}_{q,0,0}^{\infty,k_2}&(k_{1},\ n-k_1,\ s)\\
&+\overline{G}_{q,0,0}^{\infty,k_2}(k_{1},\ n-k_1,\ s)\bigg],\ n>k_{1}\\
\end{split}
\end{equation*}
\noindent
and
\noindent
\begin{equation*}\label{eq:bn1}
\begin{split}
\widehat{P}_{q,S}^{(0)}(n)=\sum_{i=1}^{\min(n-k_2,k_1-1)}\theta^{i}\ {\prod_{j=1}^{n-i}}(1-\theta q^{j-1})\sum_{s=1}^{i}\bigg[\overline{G}_{q,0,0}^{\infty,k_2}&(i,\ n-k_{2}-i,\ s)\\
&+\overline{H}_{q,0,0}^{\infty,k_2}(i,\ n-k_{2}-i,\ s)\bigg],\ n>k_{2}.\\
\end{split}
\end{equation*}
\end{theorem}
\begin{proof}
We start with the study of $\widehat{P}_{q,S}^{(1)}(n)$. It is easy to see that $\widehat{P}_{q,S}^{(1)}(k_1)=\big(\theta q^{0}\big)^{k_{1}}=\theta^{k_{1}}$. From now on we assume $n > k_1,$ and we can write $\widehat{P}_{q,S}^{(1)}(n)$ as follows.
\noindent
\begin{equation*}\label{eq:bn}
\begin{split}
\widehat{P}_{q,S}^{(1)}(n)&=P_{q,\theta}\left(L_{n}^{(0)}< k_{2}\ \wedge\ X_{n}=1\ \wedge\ F_{n}=n-k_1\right).\\
\end{split}
\end{equation*}

\noindent
We are going to focus on the event $\left\{L_{n}^{(0)}< k_{2}\ \wedge\ X_{n}=1\ \wedge\ F_{n}=n-k_1\right\}$. A typical element of the event $\left\{L_{n}^{(0)}< k_{2}\ \wedge\ X_{n}=1\ \wedge\ F_{n}=n-k_1\right\}$ is an ordered sequence which consists of $k_{1}$ successes and $n-k_1$ failures such that the length of the longest failure run is less than $k_2$. The number of these sequences can be derived as follows. First we will distribute the $k_1$ successes. Let $s$ $(1\leq s \leq k_1)$ be the number of runs of successes in the event $\left\{L_{n}^{(0)}< k_{2}\ \wedge\ X_{n}=1\ \wedge\ F_{n}=n-k_1\right\}$. We divide into two cases: starting with a failure run or starting with a success run. Thus, we distinguish between two types of sequences in the event  $$\left\{L_{n}^{(0)}< k_{2}\ \wedge\ X_{n}=1\ \wedge\ F_{n}=n-k_1\right\},$$ respectively named $(s,\ s)$-type and $(s,\ s-1)$-type, which are defined as follows.
\noindent
\begin{equation*}\label{eq:bn}
\begin{split}
(s,\ s)\text{-type}\ :&\quad \overbrace{0\ldots 0}^{y_{1}}\mid\overbrace{1\ldots 1}^{x_{1}}\mid\overbrace{0\ldots 0}^{y_{2}}\mid\overbrace{1\ldots 1}^{x_{2}}\mid \ldots \mid \overbrace{0\ldots 0}^{y_{s-1}}\mid\overbrace{1\ldots 1}^{x_{s-1}}\mid\overbrace{0\ldots 0}^{y_{s}}\mid\overbrace{1\ldots 1}^{x_{s}},\\
\end{split}
\end{equation*}
\noindent
with $n-k_1$ $0$'s and $k_{1}$ $1$'s, where $x_{j}$ $(j=1,\ldots,s)$ represents a length of run of $1$'s and $y_{j}$ $(j=1,\ldots,s)$ represents the length of a run of $0$'s. And all integers $x_{1},\ldots,x_{s}$, and $y_{1},\ldots ,y_{s}$ satisfy the conditions
\noindent
\begin{equation*}\label{eq:bn}
\begin{split}
0 < x_j  \mbox{\ for\ } j=1,...,s, \mbox{\ and\ } x_1+\cdots +x_{s} = k_1,
\end{split}
\end{equation*}
\noindent
\begin{equation*}\label{eq:bn}
\begin{split}
0 < y_j < k_2 \mbox{\ for\ } j=1,...,s, \mbox{\ and\ } y_1+\cdots +y_s=n-k_1.
\end{split}
\end{equation*}

\begin{equation*}\label{eq:bn}
\begin{split}
(s,\ s-1)\text{-type}\  :&\quad \overbrace{1\ldots 1}^{x_{1}}\mid\overbrace{0\ldots 0}^{y_{1}}\mid\overbrace{1\ldots 1}^{x_{2}}\mid\overbrace{0\ldots 0}^{y_{2}}\mid \ldots\mid\overbrace{1\ldots 1}^{x_{s-1}} \mid \overbrace{0\ldots 0}^{y_{s-1}}\mid\overbrace{1\ldots 1}^{x_{s}},
\end{split}
\end{equation*}
\noindent
with $n-k_1$ $0$'s and $k_{1}$ $1$'s, where $x_{j}$ $(j=1,\ldots,s)$ represents a length of run of $1$'s and $y_{j}$ $(j=1,\ldots,s-1)$ represents the length of a run of $0$'s. Here all of $x_1,\ldots, x_{s},$ and $y_1,\ldots, y_{s-1}$ are integers, and they satisfy
\noindent
\begin{equation*}\label{eq:bn}
\begin{split}
0 < x_j  \mbox{\ for\ } j=1,...,s, \mbox{\ and\ } x_1+\cdots +x_s = k_1,
\end{split}
\end{equation*}
\noindent
\begin{equation*}\label{eq:bn}
\begin{split}
0 < y_j < k_2 \mbox{\ for\ } j=1,...,s-1, \mbox{\ and\ } y_1+\cdots +y_{s-1} =n-k_1.
\end{split}
\end{equation*}
\noindent
Then the probability of the event $\left\{L_{n}^{(0)}< k_{2}\ \wedge\ X_{n}=1\ \wedge\ F_{n}=n-k_1\right\}$ is given by
\noindent
\begin{equation*}\label{eq: 1.1}
\begin{split}
P_{q,\theta}\Big(L_{n}^{(0)}&< k_{2}\ \wedge\ X_{n}=1\ \wedge\ F_{n}=n-k_1\Big)=\\
\sum_{s=1}^{k_1}\Bigg[\bigg\{&\sum_{\substack{x_{1},\ldots,x_{s} \in S_{k_{1},s}^{0}}}\
\sum_{\substack{y_{1}+\cdots+y_{s}=n-k_1\\ y_{1},\ldots,y_{s} \in \{1,\ldots,k_{2}-1\}}} \big(1-\theta q^{0}\big)\cdots \big(1-\theta q^{y_{1}-1}\big)\Big(\theta q^{y_{1}}\Big)^{x_{1}}\times\\
&\big(1-\theta q^{y_{1}}\big)\cdots \big(1-\theta q^{y_{1}+y_{2}-1}\big)\Big(\theta q^{y_{1}+y_{2}}\Big)^{x_{2}}\times\\
&\quad \quad \quad \quad\quad \quad\quad \quad \quad \quad \quad \vdots\\
&\big(1-\theta q^{y_{1}+\cdots+y_{s-1}}\big)\cdots \big(1-\theta q^{y_{1}+\cdots+y_{s}-1}\big)\Big(\theta q^{y_{1}+\cdots+y_{s}}\Big)^{x_{s}}\bigg\}+\\
\end{split}
\end{equation*}

\noindent
\begin{equation*}\label{eq: 1.1}
\begin{split}
\quad \quad \quad  &\bigg\{\sum_{\substack{x_{1},\ldots,x_{s} \in S_{k_{1},s}^{0}}}\
\sum_{\substack{y_{1}+\cdots+y_{s-1}=n-k_1\\ y_{1},\ldots,y_{s-1} \in \{1,\ldots,k_{2}-1\}}}\big(\theta q^{0}\big)^{x_{1}}\big(1-\theta q^{0}\big)\cdots (1-\theta q^{y_{1}-1})\times\\
&\Big(\theta q^{y_{1}}\Big)^{x_{2}}\big(1-\theta q^{y_{1}}\big)\cdots \big(1-\theta q^{y_{1}+y_{2}-1}\big)\times\\
&\quad \quad \quad \quad\quad \quad\quad \quad \quad \quad \quad \vdots\\
&\Big(\theta q^{y_{1}+\cdots+y_{s-2}}\Big)^{x_{s-1}}
\big(1-\theta q^{y_{1}+\cdots+y_{s-2}}\big)\cdots \big(1-\theta q^{y_{1}+\cdots+y_{s-1}-1}\big)\times\\
&\Big(\theta q^{y_{1}+\cdots+y_{s-1}}\Big)^{x_{s}}\bigg\}\Bigg].\\
\end{split}
\end{equation*}
\noindent
Using simple exponentiation algebra arguments to simplify,
\noindent
\begin{equation*}\label{eq: 1.1}
\begin{split}
&P_{q,\theta}\Big(L_{n}^{(0)}< k_{2}\ \wedge\ X_{n}=1\ \wedge\ F_{n}=n-k_1\Big)=\\
&\theta^{k_{1}}{\prod_{j=1}^{n-k_1}}\ (1-\theta q^{j-1})\times\\
&\sum_{s=1}^{k_1}\Bigg[\sum_{\substack{(x_{1},\ldots,x_{s}) \in S_{k_{1},s}^{0}}}
{\hspace{0.3cm}\sum_{\substack{y_{1}+\cdots+y_{s}=n-k_1\\ y_{1},\ldots,y_{s} \in \{1,\ldots,k_{2}-1\}}}}q^{y_{1}x_{1}+(y_{1}+y_{2})x_{2}+\cdots+(y_{1}+\cdots+y_{s})x_{s}}+\\
&\quad\quad\sum_{\substack{(x_{1},\ldots,x_{s}) \in S_{k_{1},s}^{0}}}\
{\hspace{0.3cm}\sum_{\substack{y_{1}+\cdots+y_{s-1}=n-k_1\\ y_{1},\ldots,y_{s-1} \in \{1,\ldots,k_{2}-1\}}}}q^{y_{1}x_{2}+(y_{1}+y_{2})x_{3}+\cdots+(y_{1}+\cdots+y_{s-1})x_{s}}\Bigg].\\
\end{split}
\end{equation*}
\noindent
Using Lemma \ref{lemma:3.9} and Lemma \ref{lemma:3.11}, we can rewrite as follows.
\noindent
\begin{equation*}\label{eq: 1.1}
\begin{split}
P_{q,\theta}\Big(&L_{n}^{(0)}< k_{2} \wedge\ X_{n}=1\ \wedge\ F_{n}=n-k_1\Big)\\
&=\theta^{k_{1}}\ {\prod_{j=1}^{n-k_1}}\ (1-\theta q^{j-1})\ \sum_{s=1}^{k_1}\bigg[\overline{H}_{q,0,0}^{\infty,k_2}(k_{1},\ n-k_1,\ s)+\overline{G}_{q,0,0}^{\infty,k_2}(k_{1},\ n-k_1,\ s)\bigg],\\
\end{split}
\end{equation*}\\
\noindent
where
\noindent
\begin{equation*}\label{eq: 1.1}
\begin{split}
\overline{H}_{q,0,0}^{\infty,k_2}(k_{1},\ n-k_1,\ s)=\sum_{\substack{(x_{1},\ldots,x_{s}) \in S_{k_{1},s}^{0}}}
{\hspace{0.3cm}\sum_{\substack{y_{1}+\cdots+y_{s}=n-k_1\\ y_{1},\ldots,y_{s} \in \{1,\ldots,k_{2}-1\}}}}q^{y_{1}x_{1}+(y_{1}+y_{2})x_{2}+\cdots+(y_{1}+\cdots+y_{s})x_{s}},
\end{split}
\end{equation*}
\noindent
and
\noindent
\begin{equation*}\label{eq: 1.1}
\begin{split}
\overline{G}_{q,0,0}^{\infty,k_2}(k_{1},\ n-k_1,\ s)=\sum_{\substack{ (x_{1},\ldots,x_{s}) \in S_{k_{1},s}^{0}}}\
{\hspace{0.3cm}\sum_{\substack{y_{1}+\cdots+y_{s-1}=n-k_1\\ (y_{1},\ldots,y_{s-1}) \in \{1,\ldots,k_{2}-1\}}}}q^{y_{1}x_{2}+(y_{1}+y_{2})x_{3}+\cdots+(y_{1}+\cdots+y_{s-1})x_{s}}.
\end{split}
\end{equation*}
\noindent
Therefore we can compute the probability of the event $W_{S}^{(1)}=n$ as follows.
\noindent
\begin{equation*}\label{eq:bn1}
\begin{split}
\widehat{P}_{S}^{(1)}(n)=\theta^{k_{1}}\ {\prod_{j=1}^{n-k_1}}\ (1-\theta q^{j-1})\ \sum_{s=1}^{k_1}\bigg[\overline{H}_{q,0,0}^{\infty,k_2}(k_{1},\ n-k_1,\ s)+\overline{G}_{q,0,0}^{\infty,k_2}(k_{1},\ n-k_1,\ s)\bigg].\\
\end{split}
\end{equation*}
\noindent

We are now going to study $\widehat{P}_{q,S}^{(0)}(n)$. It is easy to see that $\widehat{P}_{q,S}^{(0)}(k_2)={\prod_{j=1}^{k_2}}(1-\theta q^{j-1})$. From now on we assume $n > k_2,$ and we can write $\widehat{P}_{q,S}^{(0)}(n)$ as

\noindent
\begin{equation*}\label{eq:bn}
\begin{split}
\widehat{P}_{q,S}^{(0)}(n)&=P_{q,\theta}\left(L_{n-k_{2}}^{(0)}< k_{2}\ \wedge\ X_{n-k_{2}}=1\ \wedge\ X_{n-k_{2}+1}=\cdots =X_{n}=0\right).\\
\end{split}
\end{equation*}

We partition the event $\widehat{W}_{S}^{(0)}=n$ into disjoint events given by $S_{n-k_2}=i,$ for $i=1, \ldots,\min(n-k_2,k_1-1).$ Adding the probabilities we have
\noindent
\begin{equation*}\label{eq:bn}
\begin{split}
\widehat{P}_{q,S}^{(0)}(n)=\sum_{i=1}^{\min(n-k_2,k_1-1)}P_{q,\theta}\Big(L_{n-k_{2}}^{(0)}< k_{2}\ \wedge\ S_{n-k_2}=i&\ \wedge\ X_{n-k_{2}}=1\ \wedge\\
&X_{n-k_{2}+1}=\cdots =X_{n}=0\Big).\\
\end{split}
\end{equation*}
\noindent

Using $F_{n-k_2}=n-k_2-S_{n-k_2},$ we can rewrite as follows.
\begin{equation*}\label{eq:bn}
\begin{split}
\widehat{P}_{q,S}^{(0)}(n)=\sum_{i=1}^{\min(n-k_2,k_1-1)}P_{q,\theta}\Big(L_{n-k_{2}}^{(0)}< k_{2}\ \wedge\ F_{n-k_2}=n-k_2-i&\ \wedge\ X_{n-k_{2}}=1\ \wedge\\
&X_{n-k_{2}+1}=\cdots =X_{n}=0\Big).\\
\end{split}
\end{equation*}
\noindent
If the number of $0$'s in the first $n-k_2$ trials is equal to $n-k_2-i,$ that is, $F_{n-k_2}=n-k_2-i,$ then the probability of failures in all of the $(n-k_2+1)$-th to $n$-th trials is
\noindent
\begin{equation*}\label{eq: kk}
\begin{split}
P(X_{n-k_2+1}=\cdots = X_n = 0 \,\mid \, F_{n-k_2}=n-k_2-i)={\prod_{j=n-k_2-i+1}^{n-i}}(1-\theta q^{j-1}).
\end{split}
\end{equation*}
\noindent
We write $E_{n,i}^{(1)}$ for the event $\left\{L_n^{(0)}< k_2\ \wedge\ X_n=1\ \wedge\ F_n=n-i\right\}.$ We can now rewrite as follows.
\noindent
\begin{equation*}\label{eq:bn1}
\begin{split}
\widehat{P}_{q,S}^{(0)}(n)=\sum_{i=1}^{\min(n-k_2,k_1-1)}P_{q,\theta}\Big(L_{n-k_{2}}^{(0)}< &k_{2}\ \wedge\ F_{n-k_2}=n-k_2-i\ \wedge\ X_{n-k_{2}}=1\Big)\\
&\times P_{q,\theta}\Big(X_{n-k_{2}+1}=\cdots =X_{n}=0\mid F_{n-k_{2}}=n-k_2-i\Big)\\
=\sum_{i=1}^{\min(n-k_2,k_1-1)}P_{q,\theta}\Big(E_{n-k_{2},\ i}^{(1)}&\Big){\prod_{j=n-k_2-i+1}^{n-i}}(1-\theta q^{j-1}).\\
\end{split}
\end{equation*}
\noindent
We are going to focus on the event $E_{n-k_{2},i}^{(1)}$. A typical element of the event $E_{n-k_{2},i}^{(1)}$ is an ordered sequence which consists of $i$ successes and $n-k_2-i$ failures such that the length of the longest failure run is less than $k_2$. The number of these sequences can be derived as follows. First we will distribute the $i$ successes. Let $s$ $(1\leq s \leq i)$ be the number of runs of successes in the event $E_{n-k_{2},\ i}^{(1)}$. We divide into two cases:starting with a success run or starting with a failure run. Thus, we distinguish between two types of sequences in the event $$\left\{L_{n-k_{2}}^{(0)}< k_{2}\ \wedge\ F_{n-k_2}=n-k_2-i\ \wedge\ X_{n-k_{2}}=1\right\},$$ respectively named $(s,\ s-1)$-type and $(s,\ s)$-type, which are defined as follows.
\noindent
\begin{equation*}\label{eq:bn}
\begin{split}
(s,\ s-1)\text{-type}\ :&\quad \overbrace{1\ldots 1}^{x_{1}}\mid\overbrace{0\ldots 0}^{y_{1}}\mid\overbrace{1\ldots 1}^{x_{2}}\mid\overbrace{0\ldots 0}^{y_{2}}\mid \ldots \mid\overbrace{0\ldots 0}^{y_{s-1}}\mid\overbrace{1\ldots 1}^{x_{s}},\\
\end{split}
\end{equation*}
\noindent
with $n-k_{2}-i$ $0$'s and $i$ $1$'s, where $x_{j}$ $(j=1,\ldots,s)$ represents the length of a run of $1$'s and $y_{j}$ $(j=1,\ldots,s-1)$ represents the length of a run of $0$'s. And all integers $x_{1},\ldots,x_{s}$, and $y_{1},\ldots ,y_{s-1}$ satisfy the conditions
\noindent
\begin{equation*}\label{eq:bn}
\begin{split}
0 < x_j \mbox{\ for\ } j=1,...,s, \mbox{\ and\ } x_1+\cdots +x_{s} = i,
\end{split}
\end{equation*}
\noindent
\begin{equation*}\label{eq:bn}
\begin{split}
0 < y_j < k_2 \mbox{\ for\ } j=1,...,s-1, \mbox{\ and\ } y_1+\cdots +y_{s-1}=n-k_2-i.
\end{split}
\end{equation*}

\begin{equation*}\label{eq:bn}
\begin{split}
(s,\ s)\text{-type}\  :&\quad \overbrace{0\ldots 0}^{y_{1}}\mid\overbrace{1\ldots 1}^{x_{1}}\mid\overbrace{0\ldots 0}^{y_{2}}\mid\overbrace{1\ldots 1}^{x_{2}}\mid\overbrace{0\ldots 0}^{y_{3}}\mid \ldots \mid\overbrace{0\ldots 0}^{y_{s}}\mid\overbrace{1\ldots 1}^{x_{s}},
\end{split}
\end{equation*}
\noindent
with $n-k_{2}-i$ $0$'s and $i$ $1$'s, where $x_{j}$ $(j=1,\ldots,s)$ represents the length of a run of $1$'s and $y_{j}$ $(j=1,\ldots,s)$ represents the length of a run of $0$'s. Here all of $x_1,\ldots, x_{s},$ and $y_1,\ldots, y_s$ are integers, and they satisfy
\noindent
\begin{equation*}\label{eq:bn}
\begin{split}
0 < x_j \mbox{\ for\ } j=1,...,s, \mbox{\ and\ } x_1+\cdots +x_s = i,
\end{split}
\end{equation*}
\noindent
\begin{equation*}\label{eq:bn}
\begin{split}
0 < y_j < k_2 \mbox{\ for\ } j=1,...,s, \mbox{\ and\ } y_1+\cdots +y_s =n-k_2-i.
\end{split}
\end{equation*}

Then the probability of the event $E_{n-k_2,i}^{(1)}$ is given by
\noindent
\begin{equation*}\label{eq: 1.1}
\begin{split}
P_{q,\theta}\Big(E_{n-k_{2},\ i}^{(1)}&\Big)=\\
\sum_{s=1}^{i}\Bigg[\bigg\{&\sum_{\substack{(x_{1},\ldots,x_{s}) \in S_{i,s}^{0}}}{\hspace{0.5cm}\sum_{\substack{y_{1}+\cdots+y_{s-1}=n-k_2-i\\ y_{1},\ldots,y_{s-1} \in \{1,\ldots,k_{2}-1\}}}} \big(\theta q^{0}\big)^{x_{1}}\big(1-\theta q^{0}\big)\cdots (1-\theta q^{y_{1}-1})\times\\
&\Big(\theta q^{y_{1}}\Big)^{x_{2}}\big(1-\theta q^{y_{1}}\big)\cdots \big(1-\theta q^{y_{1}+y_{2}-1}\big)\times\\
&\Big(\theta q^{y_{1}+y_{2}}\Big)^{x_{3}}\big(1-\theta q^{y_{1}+y_2}\big)\cdots \big(1-\theta q^{y_{1}+y_{2}+y_3-1}\big)\times \\
&\quad \quad \quad \vdots\\
&\Big(\theta q^{y_{1}+\cdots+y_{s-1}}\Big)^{x_{s}}\bigg\}+\\
\end{split}
\end{equation*}

\noindent
\begin{equation*}\label{eq: 1.1}
\begin{split}
\quad \quad \quad &\bigg\{\sum_{\substack{(x_{1},\ldots,x_{s}) \in S_{i,s}^{0}}}\
\sum_{\substack{y_{1}+\cdots+y_{s}=n-k_2-i\\ y_{1},\ldots,y_{s} \in \{1,\ldots,k_{2}-1\}}}\big(1-\theta q^{0}\big)\cdots \big(1-\theta q^{y_{1}-1}\big)\Big(\theta q^{y_{1}}\Big)^{x_{1}}\times\\
&\big(1-\theta q^{y_{1}}\big)\cdots \big(1-\theta q^{y_{1}+y_{2}-1}\big)\Big(\theta q^{y_{1}+y_{2}}\Big)^{x_{2}}\times\\
&\big(1-\theta q^{y_{1}+y_2}\big)\cdots \big(1-\theta q^{y_{1}+y_{2}+y_3-1}\big)\Big(\theta q^{y_{1}+y_{2}+y_3}\Big)^{x_{3}}\times \\
&\quad \quad \quad  \quad \quad \quad \quad \quad \quad \quad \vdots\\
&\big(1-\theta q^{y_{1}+\cdots+y_{s-1}}\big)\cdots \big(1-\theta q^{y_{1}+\cdots+y_{s}-1}\big)\Big(\theta q^{y_{1}+\cdots+y_{s}}\Big)^{x_{s}}\bigg\}\Bigg].\\
\end{split}
\end{equation*}
\noindent
Using simple exponentiation algebra arguments to simplify,
\noindent
\begin{equation*}\label{eq: 1.1}
\begin{split}
&P_{q,\theta}\Big(E_{n-k_{2},\ i}^{(1)}\Big)=\\
&\theta^{i}\ {\prod_{j=1}^{n-k_2-i}}\ (1-\theta q^{j-1})\times\\
&\sum_{s=1}^{i}\Bigg[\sum_{\substack{(x_{1},\ldots,x_{s}) \in S_{i,s}^{0}}}\
\sum_{\substack{y_{1}+\cdots+y_{s-1}=n-k_2-i\\ y_{1},\ldots,y_{s-1} \in \{1,\ldots,k_{2}-1\}}}q^{y_{1}x_{2}+(y_{1}+y_{2})x_{3}+\cdots+(y_{1}+\cdots+y_{s-1})x_{s}}+\\
&\quad \quad \sum_{\substack{(x_{1},\ldots,x_{s}) \in S_{i,s}^{0}}}\
\sum_{\substack{y_{1}+\cdots+y_{s}=n-k_2-i\\ y_{1},\ldots,y_{s} \in \{1,\ldots,k_{2}-1\}}}q^{y_{1}x_{1}+(y_{1}+y_{2})x_{2}+\cdots+(y_{1}+\cdots+y_{s})x_{s}}\Bigg].\\
\end{split}
\end{equation*}
\noindent
Using Lemma \ref{lemma:3.9} and Lemma \ref{lemma:3.11}, we can rewrite as follows.
\noindent
\begin{equation*}\label{eq: 1.1}
\begin{split}
P_{q,\theta}\Big(E_{n-k_{2},\ i}^{(1)}\Big)=\theta^{i}\ {\prod_{j=1}^{n-k_2-i}}\ \left(1-\theta q^{j-1}\right)\sum_{s=1}^{i}\bigg[\overline{G}_{q,0,0}^{\infty,k_2}&(i,\ n-k_{2}-i,\ s)\\
&+\overline{H}_{q,0,0}^{\infty,k_2}(i,\ n-k_{2}-i,\ s)\bigg],\\
\end{split}
\end{equation*}\\
\noindent
where
\noindent
\begin{equation*}\label{eq: 1.1}
\begin{split}
\overline{G}_{q,0,0}^{\infty,k_2}(i,\ n-k_{2}-i,\ s)=\sum_{\substack{(x_{1},\ldots,x_{s}) \in S_{i,s}^{0}}}\
\sum_{\substack{y_{1}+\cdots+y_{s-1}=n-k_2-i\\ y_{1},\ldots,y_{s-1} \in \{1,\ldots,k_{2}-1\}}}q^{y_{1}x_{2}+(y_{1}+y_{2})x_{3}+\cdots+(y_{1}+\cdots+y_{s-1})x_{s}},
\end{split}
\end{equation*}
\noindent
and
\noindent
\begin{equation*}\label{eq: 1.1}
\begin{split}
\overline{H}_{q,0,0}^{\infty,k_2}(i,\ n-k_{2}-i,\ s)=\sum_{\substack{(x_{1},\ldots,x_{s}) \in S_{i,s}^{0}}}\
\sum_{\substack{y_{1}+\cdots+y_{s}=n-k_2-i\\ y_{1},\ldots,y_{s} \in \{1,\ldots,k_{2}-1\}}}q^{y_{1}x_{1}+(y_{1}+y_{2})x_{2}+\cdots+(y_{1}+\cdots+y_{s})x_{s}}.
\end{split}
\end{equation*}
\noindent
Therefore we can compute the probability of the event $\widehat{W}_{S}^{(0)}=n$ as follows
\noindent
\begin{equation*}\label{eq:bn1}
\begin{split}
\widehat{P}_{q,S}^{(0)}(n)=&\sum_{i=1}^{\min(n-k_2,k_1-1)}P\Big(E_{n-k_{2},\ i}^{(1)}\Big){\prod_{j=n-k_2-i+1}^{n-i}}(1-\theta q^{j-1})\\
=&\sum_{i=1}^{\min(n-k_2,k_1-1)}\theta^{i}{\prod_{j=1}^{n-k_2-i}}\ (1-\theta q^{j-1})\sum_{s=1}^{i}\bigg[\overline{G}_{q,0,0}^{\infty,k_2}(i,\ n-k_{2}-i,\ s)\\
&+\overline{H}_{q,0,0}^{\infty,k_2}(i,\ n-k_{2}-i,\ s)\bigg]\times{\prod_{j=n-k_2-i+1}^{n-i}}(1-\theta q^{j-1}).\\
\end{split}
\end{equation*}

\noindent
Finally, applying typical factorization algebra arguments, $\widehat{P}_{S}^{(0)}(n)$ can be rewritten as follows
\noindent
\begin{equation*}\label{eq:bn1}
\begin{split}
\widehat{P}_{q,S}^{(0)}(n)=\sum_{i=1}^{\min(n-k_2,k_1-1)}\theta^{i}\ {\prod_{j=1}^{n-i}}(1-\theta q^{j-1})\sum_{s=1}^{i}\bigg[\overline{G}_{q,0,0}^{\infty,k_2}&(i,\ n-k_{2}-i,\ s)\\
&+\overline{H}_{q,0,0}^{\infty,k_2}(i,\ n-k_{2}-i,\ s)\bigg].\\
\end{split}
\end{equation*}
Thus proof is completed.
\end{proof}
\noindent
It is worth mentioning here that the PMF $\widehat{f}_{q,S}(n;\theta)$ approaches the probability function of the sooner waiting time distribution of order $(k_1,k_2)$
in the limit as $q$ tends to 1 when a frequency quota on runs of successes and a run quota on failures are imposed of IID model. The details are presented in the following remark.
\noindent
\begin{remark}
{\rm For $q=1$, the PMF $\widehat{f}_{q,S}(n;\theta)$ reduces to the PMF $\widehat{f}_{S}(n;\theta)=P_{\theta}(\widehat{W}_{S}=n)$ for $n\geq min(k_1,k_2)$ is given by

\noindent
\begin{equation*}\label{eq:bn}
\begin{split}
P(\widehat{W}_{S}=n)=\widehat{P}_{S}^{(1)}(n)+\widehat{P}_{S}^{(0)}(n),
\end{split}
\end{equation*}
\noindent
where $\widehat{P}_S^{(1)}(k_1)=\theta^{k_{1}}$ and $\widehat{P}_S^{(0)}(k_2)=(1-\theta)^{k_2}$,
\noindent
\begin{equation*}\label{eq:bn1}
\begin{split}
\widehat{P}_{S}^{(1)}(n)=\theta^{k_{1}}(1-\theta)^{n-k_1}\sum_{s=1}^{k_1}M(s,\ k_1)\bigg[S(s,\ k_{2},\ n-k_1)+S(s-1,\ k_2,\ n-k_1)\bigg],\ n>k_{1}\\
\end{split}
\end{equation*}
\noindent
and
\noindent
\begin{equation*}\label{eq:bn1}
\begin{split}
\widehat{P}_{S}^{(0)}(n)=\sum_{i=1}^{\min(n-k_2,k_1-1)}\theta^{i}\ (1-\theta)^{n-i}\sum_{s=1}^{i}M(s,\ i)\bigg[&S(s-1,\ k_2,\ n-k_{2}-i)\\
&+S(s,\ k_2,\ n-k_{2}-i)\bigg],\ n>k_{2},\\
\end{split}
\end{equation*}

where

\noindent
\begin{equation*}\label{eq: 1.1}
\begin{split}
M(a,\ b)={b-1 \choose a-1}.
\end{split}
\end{equation*}
\noindent
and
\noindent
\begin{equation*}\label{eq: 1.1}
\begin{split}
S(a,\ b,\ c)=\sum_{j=0}^{min(a,\left[\frac{c-a}{b-1}\right])}(-1)^{j}{a \choose j}{c-j(b-1)-1 \choose a-1}.
\end{split}
\end{equation*}
\noindent
See, e.g. \citet{charalambides2002enumerative}.
}
\end{remark}
\subsubsection{Later waiting time}
The problem of waiting time that will be discussed in this section is one of the 'later cases' and it emerges when a frequency quota on successes and a run quota on failures are imposed. More specifically, Binary (zero and one) trials with probability of ones varying according to a geometric rule, are performed sequentially until $k_1$ successes in total or $k_2$ consecutive failures are observed, whichever event gets observed later.
Let random variable $\widehat{W}_{L}$ denote the waiting time until both $k_{1}$ successes in total and $k_{2}$ consecutive failures have observed, whichever event gets observed later. Let random variable $\widehat{W}_{L}$ denote the waiting time until both $k_{1}$ successes in total and $k_{2}$ consecutive failures have observed, whichever event gets observed later. We now make some Lemma for the proofs of Theorem in the sequel.
\noindent
\begin{definition}
For $0<q\leq1$, we define

\begin{equation*}\label{eq: 1.1}
\begin{split}
\overline{I}_{q}(u,v,s)={\sum_{x_{1},\ldots,x_{s}}}\
\sum_{y_{1},\ldots,y_{s-1}} q^{y_{1}x_{2}+(y_{1}+y_{2})x_{3}+\cdots+(y_{1}+\cdots+y_{s-1})x_{s}},\
\end{split}
\end{equation*}
\noindent
where the summation is over all integers $x_1,\ldots,x_{s},$ and $y_1,\ldots,y_{s-1}$ satisfying
\noindent

\begin{equation*}\label{eq:1}
\begin{split}
(x_1,\ldots,x_{s})\in S_{u,s}^{0},\ \text{and}
\end{split}
\end{equation*}
\noindent
\begin{equation*}\label{eq:2}
\begin{split}
(y_{1},\ldots,y_{s-1}) \in S_{v,s-1}^{0}.
\end{split}
\end{equation*}
\end{definition}
\noindent
The following gives a recurrence relation useful for the computation of $\overline{I}_{q}(u,v,s)$.
\noindent
\begin{lemma}
\label{lemma:4.1}
$\overline{I}_{q}(u,v,s)$ obeys the following recurrence relation.
\begin{equation*}\label{eq: 1.1}
\begin{split}
\overline{I}&_{q}(u,v,s)\\
&=\left\{
  \begin{array}{ll}
    \sum_{b=1}^{v-(s-2)}\sum_{a=1}^{u-(s-1)}q^{va}\overline{I}_{q}(u-a,v-b,s-1), & \text{for}\ s>1,\ s\leq u\ \text{and}\ (s-1)\leq v \\
    1, & \text{for}\ s=1,\ 1\leq u,\ \text{and}\  v=0\\
    0, & \text{otherwise.}\\
  \end{array}
\right.
\end{split}
\end{equation*}
\end{lemma}
\begin{proof}
For $s > 1$, $s\leq u$ and $(s-1)\leq v$, we observe that $x_{s}$ may assume any value $1,\ldots,u-(s-1)$, then $\overline{I}_{q}(u,v,s)$ can be written as
\noindent
\begin{equation*}
\begin{split}
\overline{I}&_{q}(u,v,s)\\
=&\sum_{x_{s}=1}^{u-(s-1)}{\sum_{\substack{(x_1,\ldots,x_{s-1})\in S_{u-x_{s},s-1}^{0}}}}\
{\sum_{(y_{1},\ldots,y_{s-1}) \in S_{v,s-1}^{0}}}q^{vx_{s}}q^{y_{1}x_{2}+(y_{1}+y_{2})x_{3}+\cdots+(y_{1}+\cdots+y_{s-2})x_{s-1}}.
\end{split}
\end{equation*}
\noindent
Similarly, we observe that since $y_{s-1}$ can assume the values $1,\ldots,v-(s-2)$, then $\overline{I}_{q}(u,v,s)$ can be rewritten as
\noindent
\begin{equation*}
\begin{split}
\overline{I}&_{q}(u,v,s)\\
=&\sum_{y_{s-1}=1}^{v-(s-2)}\sum_{x_{s}=1}^{u-(s-1)}{\sum_{\substack{(x_1,\ldots,x_{s-1})\in S_{u-x_{s},s-1}^{0}}}}\
{\sum_{(y_{1},\ldots,y_{s-2}) \in S_{v-y_{s-1},s-2}^{0}}}q^{vx_{s}}\ q^{y_{1}x_{2}+(y_{1}+y_{2})x_{3}+\cdots+(y_{1}+\cdots+y_{s-2})x_{s-1}}\\
=&\sum_{b=1}^{v-(s-2)}\sum_{a=1}^{u-(s-1)}q^{va}\overline{I}_{q}(u-a,v-b,s-1)
\end{split}
\end{equation*}
\noindent
The other cases are obvious and thus the proof is completed.
\end{proof}
\begin{remark}
{\rm
We observe that $\overline{I}_{1}(u,v,s)$ is the number of integer solutions $(x_{1},\ldots,x_{s})$ and $(y_{1},\ldots,y_{s-1})$ of
\noindent
\begin{equation*}\label{eq:1}
\begin{split}
(x_1,\ldots,x_{s})\in S_{u,s}^{0},\ \text{and}
\end{split}
\end{equation*}
\noindent
\begin{equation*}\label{eq:2}
\begin{split}
(y_{1},\ldots,y_{s-1}) \in S_{v,s-1}^{0}
\end{split}
\end{equation*}
\noindent
given by
\noindent
\begin{equation*}\label{eq: 1.1}
\begin{split}
\overline{I}_{1}(u,v,s)=M(s,\ u )M(s-1,\ v),
\end{split}
\end{equation*}
\noindent
where $M(a,\ b)$ denotes the total number of integer solution $x_{1}+x_{2}+\cdots+x_{a}=b$ such that $(x_{1},\ldots,x_{a}) \in S_{b,a}^{0}$. The number is given by
\noindent
\begin{equation*}\label{eq: 1.1}
\begin{split}
M(a,\ b)={b-1 \choose a-1}.
\end{split}
\end{equation*}
\noindent
See, e.g. \citet{charalambides2002enumerative}.
}
\end{remark}
\begin{definition}
For $0<q\leq1$, we define

\begin{equation*}\label{eq: 1.1}
\begin{split}
I_{q,0,k_2}(m,r,s)={\sum_{x_{1},\ldots,x_{s}}}\
\sum_{y_{1},\ldots,y_{s-1}} q^{y_{1}x_{2}+(y_{1}+y_{2})x_{3}+\cdots+(y_{1}+\cdots+y_{s-1})x_{s}},\
\end{split}
\end{equation*}
\noindent
where the summation is over all integers $x_1,\ldots,x_{s},$ and $y_1,\ldots,y_{s-1}$ satisfying

\begin{equation*}\label{eq:1}
\begin{split}
(x_1,\ldots,x_{s})\in S_{m,s}^{0},\ \text{and}
\end{split}
\end{equation*}
\noindent
\begin{equation*}\label{eq:2}
\begin{split}
(y_{1},\ldots,y_{s-1}) \in S_{r,s-1}^{k_2-1}.
\end{split}
\end{equation*}
\end{definition}
\noindent
The following gives a recurrence relation useful for the computation of $I_{q,0,k_2}(m,r,s)$.
\noindent
\begin{lemma}
\label{lemma:4.2}
$I_{q,0,k_2}(m,r,s)$ obeys the following recurrence relation.
\begin{equation*}\label{eq: 1.1}
\begin{split}
I&_{q,0,k_2}(m,r,s)\\
&=\left\{
  \begin{array}{ll}
    \sum_{b=1}^{k_2-1}\sum_{a=1}^{m-(s-1)}q^{ra}I_{q,0,k_2}(m-a,r-b,s-1)\\
    +\sum_{b=k_2}^{r-(s-2)}\sum_{a=1}^{m-(s-1)}q^{ra}\overline{I}_{q}(m-a,r-b,s-1), & \text{for}\ s>1,\ s\leq m\\
    &\text{and}\ (s-2)+k_{2}\leq r\\
    0, & \text{otherwise.}\\
  \end{array}
\right.
\end{split}
\end{equation*}
\end{lemma}
\begin{proof}
For $s > 1$, $s\leq m$ and $(s-2)+k_{2}\leq r$, we observe that $x_{s}$ may assume any value $1,\ldots,m-(s-1)$, then $I_{q,0,k_2}(m,r,s)$ can be written as

\noindent
\begin{equation*}
\begin{split}
I&_{q,0,k_2}(m,r,s)\\
=&\sum_{x_{s}=1}^{m-(s-1)}{\sum_{\substack{(x_1,\ldots,x_{s-1})\in S_{m-x_{s},s-1}^{0}}}}\
{\sum_{(y_{1},\ldots,y_{s-1}) \in S_{r,s-1}^{k_2-1}}}q^{rx_{s}}q^{y_{1}x_{2}+(y_{1}+y_{2})x_{3}+\cdots+(y_{1}+\cdots+y_{s-2})x_{s-1}}.
\end{split}
\end{equation*}
\noindent
Similarly, we observe that since $y_{s-1}$ can assume the values $1,\ldots,k_{2}-1$, then $I_{q,0,k_2}(m,r,s)$ can be rewritten as
\noindent
\begin{equation*}
\begin{split}
I&_{q,0,k_2}(m,r,s)\\
=&\sum_{y_{s-1}=1}^{k_2-1}\sum_{x_{s}=1}^{m-(s-1)}{\sum_{\substack{(x_1,\ldots,x_{s-1})\in S_{m-x_{s},s-1}^{0}}}}{\hspace{0.3cm}\sum_{(y_{1},\ldots,y_{s-2}) \in S_{r-y_{s-1},s-2}^{k_2-1}}}\hspace{-0.3cm} q^{rx_{s}} q^{y_{1}x_{2}+(y_{1}+y_{2})x_{3}+\cdots+(y_{1}+\cdots+y_{s-2})x_{s-1}}\\
&+\sum_{y_{s-1}=k_2}^{r-(s-2)}\sum_{x_{s}=1}^{m-(s-1)}{\sum_{\substack{(x_1,\ldots,x_{s-1})\in S_{m-x_{s},s-1}^{0}}}}{\hspace{0.3cm}\sum_{(y_{1},\ldots,y_{s-2}) \in S_{r-y_{s-1},s-2}^{0}}}\hspace{-0.3cm}q^{rx_{s}}q^{y_{1}x_{2}+(y_{1}+y_{2})x_{3}+\cdots+(y_{1}+\cdots+y_{s-2})x_{s-1}}\\
=&\sum_{b=1}^{k_2-1}\sum_{a=1}^{m-(s-1)}q^{ra}I_{q,0,k_2}(m-a,r-b,s-1)+\sum_{b=k_2}^{r-(s-2)}\sum_{a=1}^{m-(s-1)}q^{ra}\overline{I}_{q}(m-a,r-b,s-1).
\end{split}
\end{equation*}
\noindent
The other cases are obvious and thus the proof is completed.
\end{proof}
\begin{remark}
{\rm
We observe that $I_{1,0,k_2}(m,r,s)$ is the number of integer solutions $(x_{1},\ldots,x_{s})$ and $(y_{1},\ldots,y_{s-1})$ of

\noindent
\begin{equation*}\label{eq:1}
\begin{split}
(x_1,\ldots,x_{s})\in S_{m,s}^{0},\ \text{and}
\end{split}
\end{equation*}

\noindent
\begin{equation*}\label{eq:2}
\begin{split}
(y_{1},\ldots,y_{s-1}) \in S_{r,s-1}^{k_2-1}
\end{split}
\end{equation*}
\noindent
given by
\noindent
\begin{equation*}\label{eq: 1.1}
\begin{split}
I_{1,0,k_2}(m,r,s)=M(s,\ m)R(s-1,\ k_{2},\ r),
\end{split}
\end{equation*}

\noindent
where $M(a,\ b)$ denotes the total number of integer solution $x_{1}+x_{2}+\cdots+x_{a}=b$ such that $(x_{1},\ldots,x_{a}) \in S_{b,a}^{0}$. The number is given by
\noindent
\begin{equation*}\label{eq: 1.1}
\begin{split}
M(a,\ b)={b-1 \choose a-1}.
\end{split}
\end{equation*}

\noindent
where $R(a,\ b,\ c)$ denotes the total number of integer solution $y_{1}+x_{2}+\cdots+y_{a}=c$ such that $(y_{1},\ldots,y_{a}) \in S_{r,a}^{k_2}$. The number is given by
\noindent
\begin{equation*}\label{eq: 1.1}
\begin{split}
R(a,\ b,\ c)=\sum_{j=1}^{min\left(a,\ \left[\frac{c-a}{b-1}\right]\right) }(-1)^{j+1}{a \choose j}{c-j(b-1)-1 \choose a-1}.
\end{split}
\end{equation*}
\noindent
See, e.g. \citet{charalambides2002enumerative}.
}
\end{remark}
\begin{definition}
For $0<q\leq1$, we define

\begin{equation*}\label{eq: 1.1}
\begin{split}
\overline{J}_{q}(u,v,s)={\sum_{x_{1},\ldots,x_{s}}}\
\sum_{y_{1},\ldots,y_{s}} q^{y_{1}x_{1}+(y_{1}+y_{2})x_{2}+\cdots+(y_{1}+\cdots+y_{s})x_{s}},\
\end{split}
\end{equation*}
\noindent
where the summation is over all integers $x_1,\ldots,x_{s},$ and $y_1,\ldots,y_{s}$ satisfying

\begin{equation*}\label{eq:1}
\begin{split}
(x_1,\ldots,x_{s})\in S_{u,s}^{0},\ \text{and}
\end{split}
\end{equation*}
\noindent
\begin{equation*}\label{eq:2}
\begin{split}
(y_{1},\ldots,y_{s}) \in S_{v,s}^{0}.
\end{split}
\end{equation*}
\end{definition}
\noindent
The following gives a recurrence relation useful for the computation of $\overline{J}_{q}(u,v,s)$.
\noindent
\begin{lemma}
\label{lemma:4.3}
$\overline{J}_{q}(u,v,s)$ obeys the following recurrence relation.
\begin{equation*}\label{eq: 1.1}
\begin{split}
\overline{J}&_{q}(u,v,s)\\
&=\left\{
  \begin{array}{ll}
    \sum_{b=1}^{v-(s-1)}\sum_{a=1}^{u-(s-1)}q^{va}\overline{J}_{q}(u-a,v-b,s-1), & \text{for}\ s>1,\ s\leq u\ \text{and}\ s\leq v \\
    1, & \text{for}\ s=1,\ 1\leq u,\ \text{and}\  1\leq v\\
    0, & \text{otherwise.}\\
  \end{array}
\right.
\end{split}
\end{equation*}
\end{lemma}
\begin{proof}
For $s > 1$, $s\leq u$ and $s\leq v$, we observe that $x_{s}$ may assume any value $1,\ldots,u-(s-1)$, then $\overline{J}_{q}(u,v,s)$ can be written as
\noindent
\begin{equation*}
\begin{split}
\overline{J}&_{q}(u,v,s)\\
=&\sum_{x_{s}=1}^{u-(s-1)}{\sum_{\substack{(x_1,\ldots,x_{s-1})\in S_{u-x_{s},s-1}^{0}}}}\
{\sum_{(y_{1},\ \ldots,\ y_{s}) \in S_{v,s}^{0}}}q^{vx_{s}}\ \ q^{y_{1}x_{1}+(y_{1}+y_{2})x_{2}+\cdots+(y_{1}+\cdots+y_{s})x_{s}}.
\end{split}
\end{equation*}
\noindent
Similarly, we observe that since $y_s$ can assume the values $1,\ldots,v-(s-1)$, then $\overline{J}_{q}(u,v,s)$ can be rewritten as
\noindent
\begin{equation*}
\begin{split}
\overline{J}&_{q}(u,v,s)\\
=&\sum_{y_{s}=1}^{v-(s-1)}\sum_{x_{s}=1}^{u-(s-1)}{\sum_{\substack{(x_1,\ldots,x_{s-1})\in S_{u-x_{s},s-1}^{0}}}}\
{\sum_{(y_{1},\ldots,y_{s-1}) \in S_{v-y_{s},s-1}^{0}}}q^{vx_{s}}\ q^{y_{1}x_{1}+(y_{1}+y_{2})x_{2}+\cdots+(y_{1}+\cdots+y_{s})x_{s}}\\
=&\sum_{b=1}^{v-(s-1)}\sum_{a=1}^{u-(s-1)}q^{va}\overline{J}_{q}(u-a,v-b,s-1).
\end{split}
\end{equation*}
\noindent
The other cases are obvious and thus the proof is completed.
\end{proof}
\begin{remark}
{\rm
We observe that $\overline{J}_{1}(u,v,s)$ is the number of integer solutions $(x_{1},\ldots,x_{s})$ and $(y_{1},\ldots,y_{s})$ of
\noindent
\begin{equation*}\label{eq:1}
\begin{split}
(x_1,\ldots,x_{s})\in S_{u,s}^{0},\ \text{and}
\end{split}
\end{equation*}
\noindent
\begin{equation*}\label{eq:2}
\begin{split}
(y_{1},\ldots,y_{s}) \in S_{v,s}^{0}
\end{split}
\end{equation*}
\noindent
given by
\noindent
\begin{equation*}\label{eq: 1.1}
\begin{split}
\overline{J}_{1}(u,v,s)=M(s,\ u )M(s,\ v),
\end{split}
\end{equation*}
\noindent
where $M(a,\ b)$ denotes the total number of integer solution $x_{1}+x_{2}+\cdots+x_{a}=b$ such that $(x_{1},\ldots,x_{a}) \in S_{b,s}^{0}$. The number is given by
\noindent
\begin{equation*}\label{eq: 1.1}
\begin{split}
M(a,\ b)={b-1 \choose a-1}.
\end{split}
\end{equation*}
\noindent
See, e.g. \citet{charalambides2002enumerative}.
}
\end{remark}
\begin{definition}
For $0<q\leq1$, we define

\begin{equation*}\label{eq: 1.1}
\begin{split}
J_{q,0,k_2}(m,r,s)={\sum_{x_{1},\ldots,x_{s}}}\
\sum_{y_{1},\ldots,y_{s}} q^{y_{1}x_{1}+(y_{1}+y_{2})x_{2}+\cdots+(y_{1}+\cdots+y_{s})x_{s}},\
\end{split}
\end{equation*}
\noindent
where the summation is over all integers $x_1,\ldots,x_{s},$ and $y_1,\ldots,y_{s}$ satisfying
\noindent
\begin{equation*}\label{eq:1}
\begin{split}
(x_1,\ldots,x_{s})\in S_{m,s}^{0},\ \text{and}
\end{split}
\end{equation*}
\noindent
\begin{equation*}\label{eq:2}
\begin{split}
(y_{1},\ldots,y_{s}) \in S_{r,s}^{k_2-1}.
\end{split}
\end{equation*}
\end{definition}
\noindent
The following gives a recurrence relation useful for the computation of $J_{q,0,k_2}(m,r,s)$.
\noindent
\begin{lemma}
\label{lemma:4.4}
$J_{q,0,k_2}(m,r,s)$ obeys the following recurrence relation.
\begin{equation*}\label{eq: 1.1}
\begin{split}
J&_{q}(m,r,s)\\
&=\left\{
  \begin{array}{ll}
    \sum_{b=1}^{k_2-1}\sum_{a=1}^{m-(s-1)}q^{ra}J_{q,0,k_2}(m-a,r-b,s-1)\\
    +\sum_{b=k_2}^{r-(s-1)}\sum_{a=1}^{m-(s-1)}q^{ra}\overline{J}_{q}(m-a,r-b,s-1), & \text{for}\ s>1,\ s\leq m\\
    &\text{and}\ (s-1)+k_{2}\leq r \\
    1, & \text{for}\ s=1,\ 1\leq m,\ \text{and}\ k_2\leq r \\
    0, & \text{otherwise.}\\
  \end{array}
\right.
\end{split}
\end{equation*}
\end{lemma}
\begin{proof}
For $s > 1$, $s\leq m$ and $(s-1)+k_{2}\leq r$, we observe that $x_{s}$ may assume any value $1,\ldots,m-(s-1)$, then $J_{q,0,k_2}(m,r,s)$ can be written as
\noindent
\begin{equation*}
\begin{split}
J&_{q,0,k_2}(m,r,s)\\
=&\sum_{x_{s}=1}^{m-(s-1)}{\sum_{(x_1,\ldots,x_{s-1})\in S_{m-x_{s},s-1}^{0}}}\
{\hspace{0.3cm}\sum_{(y_{1},\ldots,y_{s}) \in S_{r,s}^{k_2-1}}}q^{rx_{s}}q^{y_{1}x_{1}+(y_{1}+y_{2})x_{2}+\cdots+(y_{1}+\cdots+y_{s-1})x_{s-1}}.
\end{split}
\end{equation*}
\noindent
Similarly, we observe that since $y_s$ can assume the values $1,\ldots,r-(s-1)$, then $J_{q,0,k_2}(m,r,s)$ can be rewritten as
\noindent
\begin{equation*}
\begin{split}
J&_{q,0,k_2}(m,r,s)\\
=&\sum_{y_{s}=1}^{k_2-1}\sum_{x_{s}=1}^{m-(s-1)}{\sum_{(x_1,\ldots,x_{s-1})\in S_{m-x_{s},s-1}^{0}}}{\hspace{0.3cm}\sum_{(y_{1},\ldots, y_{s-1}) \in S_{r-y_{s},s-1}^{k_2-1}}}q^{rx_{s}}q^{y_{1}x_{1}+(y_{1}+y_{2})x_{2}+\cdots+(y_{1}+\cdots+y_{s-1})x_{s-1}}\\
&+\sum_{y_{s}=k_2}^{r-(s-1)}\sum_{x_{s}=1}^{m-(s-1)}{\sum_{(x_1,\ldots,x_{s-1})\in S_{m-x_{s},s-1}^{0}}}\
{\sum_{(y_{1},\ldots,y_{s-1}) \in S_{r-y_{s},s-1}^{k_2-1}}}q^{rx_{s}}q^{y_{1}x_{1}+(y_{1}+y_{2})x_{2}+\cdots+(y_{1}+\cdots+y_{s-1})x_{s-1}}\\
=&\sum_{b=1}^{k_2-1}\sum_{a=1}^{m-(s-1)}q^{ra}J_{q,0,k_2}(m-a,r-b,s-1)+\sum_{b=k_2}^{r-(s-1)}\sum_{a=1}^{m-(s-1)}q^{ra}\overline{J}_{q}(m-a,r-b,s-1).
\end{split}
\end{equation*}
\noindent
The other cases are obvious and thus the proof is completed.
\end{proof}
\begin{remark}
{\rm
We observe that $J_{1,0,k_2}(m,r,s)$ is the number of integer solutions $(x_{1},\ldots,x_{s})$ and $(y_{1},\ldots,y_{s})$ of
\noindent
\begin{equation*}\label{eq:1}
\begin{split}
(x_1,\ldots,x_{s})\in S_{m,s}^{0},\ \text{and}
\end{split}
\end{equation*}
\noindent
\begin{equation*}\label{eq:2}
\begin{split}
(y_{1},\ldots,y_{s}) \in S_{r,s}^{k_2-1}
\end{split}
\end{equation*}
\noindent
given by
\noindent
\begin{equation*}\label{eq: 1.1}
\begin{split}
J_{1,0,k_2}(m,r,s)=M(s,\ m)R(s,\ k_{2},\ r),
\end{split}
\end{equation*}
\noindent
where $M(a,\ b)$ denotes the total number of integer solution $x_{1}+x_{2}+\cdots+x_{a}=b$ such that $(x_{1},\ldots,x_{a}) \in S_{b,a}^{0}$. The number is given by
\noindent
\begin{equation*}\label{eq: 1.1}
\begin{split}
M(a,\ b)={b-1 \choose a-1}.
\end{split}
\end{equation*}
\noindent
where $R(a,\ b,\ c)$ denotes the total number of integer solution $y_{1}+x_{2}+\cdots+y_{a}=c$ such that $(y_{1},\ldots,y_{a}) \in S_{r,a}^{k_2}$. The number is given by
\noindent
\begin{equation*}\label{eq: 1.1}
\begin{split}
R(a,\ b,\ c)=\sum_{j=1}^{min\left(a,\ \left[\frac{c-a}{b-1}\right]\right) }(-1)^{j+1}{a \choose j}{c-j(b-1)-1 \choose a-1}.
\end{split}
\end{equation*}
}
\end{remark}
\noindent
The probability function of the $q$-later waiting time distribution impose a frequency quota on successes and a run quota on failures is obtained by the following theorem. It is evident that
\noindent
\begin{equation*}
P_{q,\theta}(\widehat{W}_{L}=n)=0\ \text{for}\ n< k_1+k_2
\end{equation*}
\noindent
and so we shall focus on determining the probability mass function for $n\geq k_{1}+k_{2}$.
\noindent
\begin{theorem}
\label{thm:4.2}
The PMF $P_{q,\theta}(\widehat{W}_{L}=n)$ satisfies
\noindent
\begin{equation*}\label{eq: 1.1}
\begin{split}
P_{q,\theta}(\widehat{W}_{L}=n)=\widehat{P}_{q,L}^{(1)}(n)+\widehat{P}_{q,L}^{(0)}(n)\ \text{for}\ n\geq k_1+k_2,
\end{split}
\end{equation*}
where

\begin{equation*}\label{eq:bn1}
\begin{split}
\widehat{P}_{q,L}^{(1)}(n)=\theta^{k_{1}}\ {\prod_{j=1}^{n-k_1}}\ (1-\theta q^{j-1})\ \sum_{s=1}^{k_1}\bigg[I_{q,0,k_2}(k_{1},\ n-k_1,\ s)+J_{q,0,k_2}(k_{1},\ n-k_1,\ s)\bigg].\\
\end{split}
\end{equation*}
\noindent

and

\begin{equation*}\label{eq:bn1}
\begin{split}
\widehat{P}_{q,L}^{(0)}(n)=\sum_{i=k_1}^{n-k_{2}}\theta^{i}{\prod_{j=1}^{n-i}}(1-\theta q^{j-1})\sum_{s=k_1}^{i}\Bigg[\overline{G}_{q,0,0}^{\infty,k_2}(i,\ n-k_{2}-i,\ s)+\overline{H}_{q,0,0}^{\infty,k_2}(i,\ n-k_{2}-i,\ s)\Bigg].\\
\end{split}
\end{equation*}

\end{theorem}
\begin{proof}
We start with the study of $\widehat{P}_{q,L}^{(1)}(n)$. From now on we assume $n \geq k_1+k_2,$ and we can write $\widehat{P}_{q,L}^{(1)}(n)$ as follows.
\noindent
\begin{equation*}\label{eq:bn}
\begin{split}
\widehat{P}_{q,L}^{(1)}(n)&=P_{q,\theta}\left(L_{n}^{(0)}\geq k_{2}\ \wedge\ X_{n}=1\ \wedge\ F_{n}=n-k_1\right).\\
\end{split}
\end{equation*}

\noindent
We are going to focus on the event $\Big\{L_{n}^{(0)}\geq k_{2}\ \wedge\ X_{n}=1\ \wedge\ F_{n}=n-k_1\Big\}$. A typical element of the event $\Big\{L_{n}^{(0)}\geq k_{2}\ \wedge\ X_{n}=1\ \wedge\ F_{n}=n-k_1\Big\}$ is an ordered sequence which consists of $k_{1}$ successes and $n-k_1$ failures such that the length of the longest failure run is greater than or equal to $k_2$. The number of these sequences can be derived as follows. First we will distribute the $k_1$ successes. Let $s$ $(1\leq s \leq k_1)$ be the number of runs of successes in the event $\Big\{L_{n}^{(0)}\geq k_{2}\ \wedge\ X_{n}=1\ \wedge\ F_{n}=n-k_1\Big\}$. We divide into two cases: starting with a success run or starting with a failure run. Thus, we distinguish between two types of sequences in the event  $$\Big\{L_{n}^{(0)}\geq k_{2}\ \wedge\ X_{n}=1\ \wedge\ F_{n}=n-k_1\Big\},$$ respectively named $(s,\ s-1)$-type and $(s,\ s)$-type, which are defined as follows.
\noindent

\begin{equation*}\label{eq:bn}
\begin{split}
(s,\ s-1)\text{-type}\  :&\quad \overbrace{1\ldots 1}^{x_{1}}\mid\overbrace{0\ldots 0}^{y_{1}}\mid\overbrace{1\ldots 1}^{x_{2}}\mid\overbrace{0\ldots 0}^{y_{2}}\mid \ldots\mid\overbrace{1\ldots 1}^{x_{s-1}} \mid \overbrace{0\ldots 0}^{y_{s-1}}\mid\overbrace{1\ldots 1}^{x_{s}},
\end{split}
\end{equation*}
\noindent
with $n-k_1$ $0$'s and $k_{1}$ $1$'s, where $x_{j}$ $(j=1,\ldots,s)$ represents a length of run of $1$'s and $y_{j}$ $(j=1,\ldots,s-1)$ represents the length of a run of $0$'s. Here all of $x_1,\ldots, x_{s},$ and $y_1,\ldots, y_{s-1}$ are integers, and they satisfy
\noindent
\begin{equation*}\label{eq:bn}
\begin{split}
0 < x_j  \mbox{\ for\ } j=1,...,s, \mbox{\ and\ } x_1+\cdots +x_s = k_1,
\end{split}
\end{equation*}
\noindent
\noindent
\begin{equation*}\label{eq:bn}
\begin{split}
(y_{1},\ldots,y_{s-1})\in S_{n-k_1,s-1}^{k_2-1},
\end{split}
\end{equation*}

\noindent

\begin{equation*}\label{eq:bn}
\begin{split}
(s,\ s)\text{-type}\ :&\quad \overbrace{0\ldots 0}^{y_{1}}\mid\overbrace{1\ldots 1}^{x_{1}}\mid\overbrace{0\ldots 0}^{y_{2}}\mid\overbrace{1\ldots 1}^{x_{2}}\mid \ldots \mid \overbrace{0\ldots 0}^{y_{s-1}}\mid\overbrace{1\ldots 1}^{x_{s-1}}\mid\overbrace{0\ldots 0}^{y_{s}}\mid\overbrace{1\ldots 1}^{x_{s}},\\
\end{split}
\end{equation*}
\noindent
with $n-k_1$ $0$'s and $k_{1}$ $1$'s, where $x_{j}$ $(j=1,\ldots,s)$ represents a length of run of $1$'s and $y_{j}$ $(j=1,\ldots,s)$ represents the length of a run of $0$'s. And all integers $x_{1},\ldots,x_{s}$, and $y_{1},\ldots ,y_{s}$ satisfy the conditions
\noindent
\begin{equation*}\label{eq:bn}
\begin{split}
0 < x_j  \mbox{\ for\ } j=1,...,s, \mbox{\ and\ } x_1+\cdots +x_{s} = k_1,
\end{split}
\end{equation*}
\noindent
\begin{equation*}\label{eq:bn}
\begin{split}
(y_{1},\ldots,y_{s})\in S_{n-k_1,s}^{k_2-1},
\end{split}
\end{equation*}

\noindent
Then the probability of the event $\left\{L_{n}^{(0)}\geq k_{2}\ \wedge\ X_{n}=1\ \wedge\ F_{n}=n-k_1\right\}$ is given by
\noindent
\begin{equation*}\label{eq: 1.1}
\begin{split}
P_{q,\theta}\Big(L_{n}^{(0)}&\geq k_{2}\ \wedge\ X_{n}=1\ \wedge\ F_{n}=n-k_1\Big)=\\
\sum_{s=1}^{k_1}\Bigg[\bigg\{&\sum_{(x_{1},\ldots,x_{s}) \in S_{k_{1},s}^{0}}\
\sum_{(y_{1},\ldots,y_{s-1})\in S_{n-k_1,s-1}^{k_2-1}}\big(\theta q^{0}\big)^{x_{1}}\big(1-\theta q^{0}\big)\cdots (1-\theta q^{y_{1}-1})\times\\
&\Big(\theta q^{y_{1}}\Big)^{x_{2}}\big(1-\theta q^{y_{1}}\big)\cdots \big(1-\theta q^{y_{1}+y_{2}-1}\big)\times\\
&\quad \quad \quad \quad\quad \quad\quad \quad \quad \quad \quad \vdots\\
&\Big(\theta q^{y_{1}+\cdots+y_{s-2}}\Big)^{x_{s-1}}
\big(1-\theta q^{y_{1}+\cdots+y_{s-2}}\big)\cdots \big(1-\theta q^{y_{1}+\cdots+y_{s-1}-1}\big)\times\\
&\Big(\theta q^{y_{1}+\cdots+y_{s-1}}\Big)^{x_{s}}\bigg\}+\\
\end{split}
\end{equation*}

\noindent
\begin{equation*}\label{eq: 1.1}
\begin{split}
\quad \bigg\{&\sum_{\substack{(x_{1},\ldots,x_{s}) \in S_{k_{1},s}^{0}}}\
\sum_{(y_{1},\ldots,y_{s})\in S_{n-k_1,s}^{k_2-1}} \big(1-\theta q^{0}\big)\cdots \big(1-\theta q^{y_{1}-1}\big)\Big(\theta q^{y_{1}}\Big)^{x_{1}}\times\\
&\big(1-\theta q^{y_{1}}\big)\cdots \big(1-\theta q^{y_{1}+y_{2}-1}\big)\Big(\theta q^{y_{1}+y_{2}}\Big)^{x_{2}}\times\\
&\quad \quad \quad \quad\quad \quad\quad \quad \quad \quad \quad \vdots\\
&\big(1-\theta q^{y_{1}+\cdots+y_{s-1}}\big)\cdots \big(1-\theta q^{y_{1}+\cdots+y_{s}-1}\big)\Big(\theta q^{y_{1}+\cdots+y_{s}}\Big)^{x_{s}}\bigg\}\Bigg].\\
\end{split}
\end{equation*}
\noindent
Using simple exponentiation algebra arguments to simplify,
\noindent
\begin{equation*}\label{eq: 1.1}
\begin{split}
&P_{q,\theta}\Big(L_{n}^{(0)}\geq k_{2}\ \wedge\ X_{n}=1\ \wedge\ F_{n}=n-k_1\Big)=\\
&\theta^{k_{1}}{\prod_{j=1}^{n-k_1}}\ (1-\theta q^{j-1})\times\\
&\sum_{s=1}^{k_1}\Bigg[\sum_{\substack{(x_{1},\ldots,x_{s}) \in S_{k_{1},s}^{0}}}\
\sum_{(y_{1},\ldots,y_{s-1})\in S_{n-k_1,s-1}^{k_2-1}}q^{y_{1}x_{2}+(y_{1}+y_{2})x_{3}+\cdots+(y_{1}+\cdots+y_{s-1})x_{s}}+\\
&\quad \quad \sum_{\substack{(x_{1},\ldots,x_{s}) \in S_{k_{1},s}^{0}}}\
\sum_{(y_{1},\ldots,y_{s})\in S_{n-k_1,s}^{k_2-1}}q^{y_{1}x_{1}+(y_{1}+y_{2})x_{2}+\cdots+(y_{1}+\cdots+y_{s})x_{s}}\Bigg].\\
\end{split}
\end{equation*}
\noindent
Using Lemma \ref{lemma:4.2} and Lemma \ref{lemma:4.4} we can rewrite as follows.
\noindent
\begin{equation*}\label{eq: 1.1}
\begin{split}
P_{q,\theta}\Big(&L_{n}^{(0)}\geq k_{2} \wedge\ X_{n}=1\ \wedge\ F_{n}=n-k_1\Big)\\
&=\theta^{k_{1}}\ {\prod_{j=1}^{n-k_1}}\ (1-\theta q^{j-1})\ \sum_{s=1}^{k_1}\bigg[I_{q,0,k_2}(k_{1},\ n-k_1,\ s)+J_{q,0,k_2}(k_{1},\ n-k_1,\ s)\bigg],\\
\end{split}
\end{equation*}\\
\noindent
where
\noindent
\begin{equation*}\label{eq: 1.1}
\begin{split}
I_{q,0,k_2}(k_{1}&,\ n-k_1,\ s)\\
&=\sum_{\substack{(x_{1},\ldots,x_{s}) \in S_{k_{1},s}^{0}}}{\hspace{0.3cm}\sum_{(y_{1},\ldots,y_{s-1})\in S_{n-k_1,s-1}^{k_2-1}}}q^{y_{1}x_{2}+(y_{1}+y_{2})x_{3}+\cdots+(y_{1}+\cdots+y_{s-1})x_{s}},
\end{split}
\end{equation*}
and
\begin{equation*}\label{eq: 1.1}
\begin{split}
J_{q,0,k_2}(k_{1}&,\ n-k_1,\ s)\\
&=\sum_{\substack{(x_{1},\ldots,x_{s}) \in S_{k_{1},s}^{0}}}{\hspace{0.3cm}\sum_{(y_{1},\ldots,y_{s})\in S_{n-k_1,s}^{k_2-1}}}q^{y_{1}x_{1}+(y_{1}+y_{2})x_{2}+\cdots+(y_{1}+\cdots+y_{s})x_{s}}.
\end{split}
\end{equation*}
\noindent
Therefore we can compute the probability of the event $\widehat{W}_{L}^{(1)}=n$ as follows.
\noindent
\begin{equation*}\label{eq:bn1}
\begin{split}
\widehat{P}_{q,L}^{(1)}(n)=\theta^{k_{1}}\ {\prod_{j=1}^{n-k_1}}\ (1-\theta q^{j-1})\ \sum_{s=1}^{k_1}\bigg[I_{q,0,k_2}(k_{1},\ n-k_1,\ s)+J_{q,0,k_2}(k_{1},\ n-k_1,\ s)\bigg].\\
\end{split}
\end{equation*}
\noindent

We are now going to study of $\widehat{P}_{q,L}^{(0)}(n)$. From now on we assume $n \geq k_1+k_2,$ and we can write $\widehat{P}_{q,L}^{(0)}(n)$ as

\noindent
\begin{equation*}\label{eq:bn}
\begin{split}
\widehat{P}_{q,L}^{(0)}(n)&=P_{q,\theta}\left(L_{n-k_{2}}^{(0)}< k_{2}\ \wedge\ X_{n-k_{2}}=1\ \wedge\ S_{n-k_{2}}\geq k_1\ \wedge\ X_{n-k_{2}+1}=\cdots =X_{n}=0\right).\\
\end{split}
\end{equation*}

\noindent
We partition the event $\widehat{W}_{L}^{(0)}=n$ into disjoint events given by $S_{n-k_2}=i$ and $F_{n-k_2}=n-k_2-i,$ for $i=k_1,\ \ldots,\ n-k_2.$ Adding the probabilities we have
\noindent
\begin{equation*}\label{eq:bn}
\begin{split}
\widehat{P}_{q,L}^{(0)}(n)=\sum_{i=k_1}^{n-k_{2}}P_{q,\theta}\Big(L_{n-k_{2}}^{(0)}< k_{2}\ \wedge\ &S_{n-k_{2}}=i\ \wedge\ X_{n-k_{2}}=1\ \wedge\\
&X_{n-k_{2}+1}=\cdots =X_{n}=0\Big).\\
\end{split}
\end{equation*}
\noindent
If the number of $0$'s in the first $n-k_2$ trials is equal to $n-k_2-i,$ that is, $F_{n-k_2}=n-k_2-i,$ then the probability of failures in all of the $(n-k_{2}+1)$-th to $n$-th trials is
\noindent
\begin{equation*}\label{eq: kk}
\begin{split}
P_{q,\theta}(X_{n-k_2+1}=\cdots = X_n = 0 \,\mid \, F_{n-k_2}=n-k_2-i)={\prod_{j=n-k_2-i+1}^{n-i}}\left(1-\theta q^{j-1}\right).
\end{split}
\end{equation*}
\noindent
We will write $E_{n,i}^{(1)}$ for the event $\left\{L_n^{(1)}< k_2\ \wedge\ X_n=1\ \wedge\ S_n=i\right\}$, and we can now rewrite as follows.
\noindent
\begin{equation*}\label{eq:bn1}
\begin{split}
\widehat{P}_{q.L}^{(0)}(n)=\sum_{i=k_1}^{n-k_{2}}&P_{q,\theta}\Big(L_{n-k_{2}}^{(0)}< k_{2}\ \wedge\ S_{n-k_{1}}=i\ \wedge\ F_{n-k_{2}}=n-k_2-i\ \wedge\ X_{n-k_{2}}=1\Big)\\
&\times P_{q,\theta}\Big(X_{n-k_{2}+1}=\cdots =X_{n}=0\mid F_{n-k_{2}}=n-k_2-i\Big)\\
=\sum_{i=k_1}^{n-k_{2}}&P_{q,\theta}\Big(E_{n-k_{2},\ i}^{(1)}\Big){\prod_{j=n-k_2-i+1}^{n-i}}\left(1-\theta q^{j-1}\right).\\
\end{split}
\end{equation*}
\noindent
We are going to focus on the event $E_{n-k_{2},\ i}^{(1)}$. A typical element of the event $E_{n-k_{2},\ i}^{(1)}$ is an ordered sequence which consists of $i$ successes and $n-k_{2}-i$ failures such that the length of the length of the longest failure run is less than $k_2$. The number of these sequences can be derived as follows. First we will distribute the $i$ successes. Let $s$ $(1\leq s \leq i)$ be the number of runs of successes in the event $E_{n-k_{2},\ i}^{(1)}$. We divide into two cases : starting with a success run or starting with a failure run. Thus, we distinguish between two types of sequences in the event $\Big\{L_{n-k_2}^{(0)}< k_2\ \wedge\ S_{n-k_{2}}=i\ \wedge\ F_{n-k_{2}}=n-k_2-i\ \wedge\ X_{n-k_{2}}=1\Big\},$ respectively named $(s,\ s-1)$-type and $(s,\ s)$-type, which are defined as follows.
\noindent
\begin{equation*}\label{eq:bn}
\begin{split}
(s,\ s-1)\text{-type}\ :&\quad \overbrace{1\ldots 1}^{x_{1}}\mid\overbrace{0\ldots 0}^{y_{1}}\mid\overbrace{1\ldots 1}^{x_{2}}\mid\overbrace{0\ldots 0}^{y_{2}}\mid \ldots \mid\overbrace{0\ldots 0}^{y_{s-1}}\mid\overbrace{1\ldots 1}^{x_{s}},\\
\end{split}
\end{equation*}
\noindent
with $n-k_{2}-i$ $0$'s and $i$ $1$'s, where $x_{j}$ $(j=1,\ldots,s)$ represents a length of run of $1$'s and $y_{j}$ $(j=1,\ldots,s-1)$ represents the length of a run of $0$'s. And all integers $x_{1},\ldots,x_{s}$, and $y_{1},\ldots ,y_{s-1}$ satisfy the conditions
\noindent
\begin{equation*}\label{eq:bn}
\begin{split}
(x_{1},\ldots,x_{s})\in S_{i,s}^{0}, \text{and}
\end{split}
\end{equation*}
\noindent
\begin{equation*}\label{eq:bn}
\begin{split}
0 < y_j < k_2 \mbox{\ for\ } j=1,...,s-1, \mbox{\ and\ } y_1+\cdots +y_{s-1} = n-k_2-i,
\end{split}
\end{equation*}
\noindent
\begin{equation*}\label{eq:bn}
\begin{split}
(s,\ s)\text{-type}\  :&\quad \overbrace{0\ldots 0}^{y_{1}}\mid\overbrace{1\ldots 1}^{x_{1}}\mid\overbrace{0\ldots 0}^{y_{2}}\mid\overbrace{1\ldots 1}^{x_{2}}\mid\overbrace{0\ldots 0}^{y_{3}}\mid \ldots \mid\overbrace{0\ldots 0}^{y_{s}}\mid\overbrace{1\ldots 1}^{x_{s}},
\end{split}
\end{equation*}
\noindent
with $n-k_{2}-i$ $0$'s and $i$ $1$'s, where $x_{j}$ $(j=1,\ldots,s)$ represents a length of run of $1$'s and $y_{j}$ $(j=1,\ldots,s)$ represents the length of a run of $0$'s. And all integers $x_{1},\ldots,x_{s}$, and $y_{1},\ldots ,y_{s}$ satisfy the conditions
\noindent
\begin{equation*}\label{eq:bn}
\begin{split}
(x_{1},\ldots,x_{s})\in S_{i,s}^{0}, \text{and}
\end{split}
\end{equation*}
\noindent
\begin{equation*}\label{eq:bn}
\begin{split}
0 < y_j < k_2 \mbox{\ for\ } j=1,...,s, \mbox{\ and\ } y_1+\cdots +y_{s} = n-k_2-i,
\end{split}
\end{equation*}
\noindent
Then the probability of the event $E_{n-k_2,i}^{(1)}$ is given by
\noindent
\begin{equation*}\label{eq: 1.1}
\begin{split}
P_{q,\theta}\Big(E_{n-k_{2},\ i}^{(1)}&\Big)=\\
\sum_{s=k_1}^{i}\Bigg[\bigg\{&\sum_{\substack{(x_{1},\ldots,x_{s})\in S_{i,s}^{0}}}{\hspace{0.3cm}\sum_{\substack{y_1+\cdots +y_{s-1} = n-k_2-i\\ y_{1},\ldots,y_{s-1} \in \{1,\ldots,k_{2}-1\}}}} \big(\theta q^{0}\big)^{x_{1}}\big(1-\theta q^{0}\big)\cdots (1-\theta q^{y_{1}-1})\times\\
&\Big(\theta q^{y_{1}}\Big)^{x_{2}}\big(1-\theta q^{y_{1}}\big)\cdots \big(1-\theta q^{y_{1}+y_{2}-1}\big)\times\\
&\Big(\theta q^{y_{1}+y_{2}}\Big)^{x_{3}}\big(1-\theta q^{y_{1}+y_2}\big)\cdots \big(1-\theta q^{y_{1}+y_{2}+y_3-1}\big)\times \\
&\quad \quad \quad \vdots\\
&\Big(\theta q^{y_{1}+\cdots+y_{s-1}}\Big)^{x_{s}}\bigg\}+\\
\end{split}
\end{equation*}

\noindent
\begin{equation*}\label{eq: 1.1}
\begin{split}
\quad \quad \quad \bigg\{&\sum_{\substack{(x_{1},\ldots,x_{s})\in S_{i,s}^{0}}}{\hspace{0.3cm}\sum_{\substack{y_{1}+\cdots+y_{s}=n-k_2-i\\ y_{1},\ldots,y_{s} \in \{1,\ldots,k_{2}-1\}}}}\big(1-\theta q^{0}\big)\cdots \big(1-\theta q^{y_{1}-1}\big)\Big(\theta q^{y_{1}}\Big)^{x_{1}}\times\\
&\big(1-\theta q^{y_{1}}\big)\cdots \big(1-\theta q^{y_{1}+y_{2}-1}\big)\Big(\theta q^{y_{1}+y_{2}}\Big)^{x_{2}}\times\\
&\big(1-\theta q^{y_{1}+y_2}\big)\cdots \big(1-\theta q^{y_{1}+y_{2}+y_3-1}\big)\Big(\theta q^{y_{1}+y_{2}+y_3}\Big)^{x_{3}}\times \\
&\quad \quad \quad \quad\quad \quad\quad \quad \quad \quad \quad \vdots\\
&\big(1-\theta q^{y_{1}+\cdots+y_{s-1}}\big)\cdots \big(1-\theta q^{y_{1}+\cdots+y_{s}-1}\big)\Big(\theta q^{y_{1}+\cdots+y_{s}}\Big)^{x_{s}}\bigg\}\Bigg].\\
\end{split}
\end{equation*}
\noindent
and then using simple exponentiation algebra arguments to simplify,
\noindent
\begin{equation*}\label{eq: 1.1}
\begin{split}
&P_{q,\theta}\Big(E_{n-k_{2},\ i}^{(1)}\Big)=\\
&\theta^{i}{\prod_{j=1}^{n-k_{2}-i}}\ (1-\theta q^{j-1})\\
&\sum_{s=1}^{i}\Bigg[\sum_{\substack{(x_{1},\ldots,x_{s})\in S_{i,s}^{0}}}
{\hspace{0.3cm}\sum_{\substack{y_1+\cdots +y_{s-1} = n-k_2-i\\ y_{1},\ldots,y_{s-1} \in \{1,\ldots,k_{2}-1\}}}}q^{y_{1}x_{2}+(y_{1}+y_{2})x_{3}+\cdots+(y_{1}+\cdots+y_{s-1})x_{s}}+\\
&\quad \quad \sum_{\substack{(x_{1},\ldots,x_{s})\in S_{i,s}^{0}}}{\hspace{0.3cm}\sum_{\substack{y_1+\cdots +y_{s} = n-k_2-i\\ y_{1},\ldots,y_{s} \in \{1,\ldots,k_{2}-1\}}}}q^{y_{1}x_{1}+(y_{1}+y_{2})x_{2}+\cdots+(y_{1}+\cdots+y_{s})x_{s}}\Bigg]\\
\end{split}
\end{equation*}
\noindent
Using Lemma \ref{lemma:3.9} and Lemma \ref{lemma:3.11}, we can rewrite as follows
\noindent
\begin{equation*}\label{eq: 1.1}
\begin{split}
P_{q,\theta}\Big(E_{n-k_{2},\ i}^{(1)}\Big)=\theta^{i}{\prod_{j=1}^{n-k_{2}-i}}\ (1-\theta q^{j-1})\sum_{s=k_1}^{i}\Bigg[\overline{G}_{q,0,0}^{\infty,k_2}(i,\ n-k_{2}-i,\ s)+\overline{H}_{q,0,0}^{\infty,k_2}(i,\ n-k_{2}-i,\ s)\Bigg],\\
\end{split}
\end{equation*}\\
\noindent
where
\noindent

\begin{equation*}\label{eq: 1.1}
\begin{split}
\overline{G}_{q,0,0}^{\infty,k_2}(&i,\ n-k_{2}-i,\ s)\\
&=\sum_{\substack{(x_{1},\ldots,x_{s})\in S_{i,s}^{0}}}\
\sum_{\substack{y_1+\cdots +y_{s-1} = n-k_2-i\\ y_{1},\ldots,y_{s-1} \in \{1,\ldots,k_{2}-1\}}}q^{y_{1}x_{2}+(y_{1}+y_{2})x_{3}+\cdots+(y_{1}+\cdots+y_{s-1})x_{s}},
\end{split}
\end{equation*}
\noindent
and
\noindent
\begin{equation*}\label{eq: 1.1}
\begin{split}
\overline{H}_{q,0,0}^{\infty,k_2}(&i,\ n-k_{2}-i,\ s)\\
&=\sum_{\substack{(x_{1},\ldots,x_{s})\in S_{i,s}^{0}}}\
\sum_{\substack{y_1+\cdots +y_{s} = n-k_2-i\\ y_{1},\ldots,y_{s} \in \{1,\ldots,k_{2}-1\}}}q^{y_{1}x_{1}+(y_{1}+y_{2})x_{2}+\cdots+(y_{1}+\cdots+y_{s})x_{s}}.
\end{split}
\end{equation*}
\noindent
Therefore we can compute the probability of the event $W_{L}^{(0)}=n$ as follows
\noindent
\begin{equation*}\label{eq:bn1}
\begin{split}
\widehat{P}_{q,L}^{(0)}(n)=&\sum_{i=k_1}^{n-k_{2}}P_{q,\theta}\Big(E_{n-k_{2},\ i}^{(1)}\Big){\prod_{j=n-k_2-i+1}^{n-i}}(1-\theta q^{j-1})\\
=&\sum_{i=k_1}^{n-k_{2}}\theta^{i}{\prod_{j=1}^{n-k_{2}-i}}\ (1-\theta q^{j-1})\sum_{s=k_1}^{i}\Bigg[\overline{G}_{q,0,0}^{\infty,k_2}(i,\ n-k_{2}-i,\ s)+\overline{H}_{q,0,0}^{\infty,k_2}(i,\ n-k_{2}-i,\ s)\Bigg]\\
&{\prod_{j=n-k_2-i+1}^{n-i}}(1-\theta q^{j-1}).\\
\end{split}
\end{equation*}
\noindent
Finally, applying typical factorization algebra arguments, $\widehat{P}_{L}^{(0)}(n)$ can be rewritten as follows
\noindent
\begin{equation*}\label{eq:bn1}
\begin{split}
\widehat{P}_{q,L}^{(0)}(n)=\sum_{i=k_1}^{n-k_{2}}\theta^{i}{\prod_{j=1}^{n-i}}(1-\theta q^{j-1})\sum_{s=k_1}^{i}\Bigg[\overline{G}_{q,0,0}^{\infty,k_2}(i,\ n-k_{2}-i,\ s)+\overline{H}_{q,0,0}^{\infty,k_2}(i,\ n-k_{2}-i,\ s)\Bigg].\\
\end{split}
\end{equation*}
\noindent
Thus proof is completed.
\end{proof}
\noindent
It is worth mentioning here that the PMF $\widehat{f}_{q,L}(n;\theta)$ approaches the probability function of the later waiting time distribution of order $(k_1,k_2)$
in the limit as $q$ tends to 1 when a frequency quota on runs of successes and a run quota on failures are imposed of IID model. The details are presented in the following remark.
\begin{remark}
{\rm
For $q=1$, the PMF $\widehat{f}_{q,L}(n;\theta)$ reduces to the PMF $\widehat{f}_{L}(n;\theta)=P_{\theta}(\widehat{W}_{L}=n)$ for $n\geq k_1+k_2$ is given by

\begin{equation*}\label{eq: 1.1}
\begin{split}
P(\widehat{W}_{L}=n)=\widehat{P}_{L}^{(1)}(n)+\widehat{P}_{L}^{(0)}(n),\ n\geq k_1+k_2,
\end{split}
\end{equation*}
where
\begin{equation*}\label{eq:bn1}
\begin{split}
\widehat{P}_{L}^{(1)}(n)=\theta^{k_{1}}\ (1-\theta)^{n-k_1}\ \sum_{s=1}^{k_1}M(s,\ k_1)\bigg[R(s-1,\ k_{2},\ n-k_1)+R(s,\ k_{2},\ n-k_1)\bigg].\\
\end{split}
\end{equation*}
\noindent

and

\begin{equation*}\label{eq:bn1}
\begin{split}
\widehat{P}_{L}^{(0)}(n)=\sum_{i=k_1}^{n-k_{2}}\theta^{i}(1-\theta)^{n-i}\sum_{s=k_1}^{i}M(s,\ i)\Bigg[S(s-1,\ k_2,\ n-k_{2}-i)+S(s,\ k_2\ n-k_{2}-i)\Bigg],\\
\end{split}
\end{equation*}

\noindent
where $M(a,\ b)$ denotes the total number of integer solution $x_{1}+x_{2}+\cdots+x_{a}=b$ such that $x_{1},\ldots,x_{a} \in S_{b,a}^{0}$. The number is given by
\noindent
\begin{equation*}\label{eq: 1.1}
\begin{split}
M(a,\ b)={b-1 \choose a-1}.
\end{split}
\end{equation*}
\noindent
and

\noindent
\begin{equation*}\label{eq: 1.1}
\begin{split}
S(a,\ b,\ c)=\sum_{j=0}^{min\left(a,\ \left[\frac{c-a}{b-1}\right]\right) }(-1)^{j}{a \choose j}{c-j(b-1)-1 \choose a-1}.
\end{split}
\end{equation*}
\noindent
See, e.g. \citet{charalambides2002enumerative}.
}
\end{remark}

\subsection{A run quota on successes and a frequency quota on failures are imposed}
We shall discuss of the $q$-sooner and later waiting times impose a run quota on successes and a frequency quota on failures. Each cases will be described in the following subsection.

\subsubsection{Sooner waiting time}
The problem of waiting time that will be discussed in this section is one of the 'sooner cases' and it emerges when a run quota on successes and a frequency quota on failures are imposed. More specifically, Binary (zero and one) trials with probability of ones varying according to a geometric rule, are performed sequentially until $k_1$ consecutive successes or $k_2$ failures in total are observed, whichever event occurs first. Let $\widetilde{W}_{S}$ be a random variable denoting that the waiting time until either $k_1$ consecutive successes in total or $k_2$ failures in total are occurred, whichever event observe sooner. The probability function of the $q$-sooner waiting time distribution impose a run quota on successes and a frequency quota on failures is obtained by the following theorem. It is evident that
\noindent
\begin{equation*}
P_{q,\theta}\left(\widetilde{W}_{S}=n\right)=0\ \text{for}\ 0\leq n<\text{min}(k_1,k_2)
\end{equation*}
\noindent
and so we shall focus on determining the probability mass function for $n\geq\min(k_{1},\ k_{2})$.
\noindent
\begin{theorem}
\label{thm:4.3}
The PMF $P_{q,\theta}\left(\widetilde{W}_{S}=n\right)$ satisfies
\noindent
\begin{equation*}\label{eq:bn}
\begin{split}
P_{q,\theta}\left(\widetilde{W}_{S}=n\right)=\widetilde{P}_{q,S}^{(1)}(n)+\widetilde{P}_{q,S}^{(0)}(n)\ \text{for}\ n\geq\min(k_{1},\ k_{2}),
\end{split}
\end{equation*}
\noindent
where $\widetilde{P}_{q,S}^{(1)}(k_1)=\theta^{k_{1}}$, $\widetilde{P}_{q,S}^{(0)}(k_2)={\prod_{j=1}^{k_2}}\left(1-\theta q^{j-1}\right)$,
\noindent
\begin{equation*}\label{eq:bn1}
\begin{split}
\widetilde{P}_{q,S}^{(1)}(n)=\sum_{i=1}^{\min(n-k_1,k_2-1)}\theta^{n-i} q^{ik_{1}}{\prod_{j=1}^{i}}(1-\theta q^{j-1})\sum_{s=1}^{i}\Bigg[\overline{E}_{q}&(n-k_{1}-i,\ i,\ s)\\
&+\overline{F}_{q}(n-k_{1}-i,\ i,\ s)\Bigg],\ n>k_{1}\\
\end{split}
\end{equation*}
\noindent
and
\noindent
\begin{equation*}\label{eq:bn1}
\begin{split}
\widetilde{P}_{q,S}^{(0)}(n)=\theta^{n-k_{2}}&{\prod_{j=1}^{k_2}}\ (1-\theta q^{j-1})\ \sum_{s=1}^{k_2}\bigg[\overline{E}_{q}(n-k_2,\ k_{2} ,\ s)+\overline{F}_{q}(n-k_2,\ k_{2} ,\ s)\bigg],\ n>k_2.\\
\end{split}
\end{equation*}
\end{theorem}
\begin{proof}
\noindent
We are now going to study $\widetilde{P}_{q,S}^{(1)}(n)$. It is easy to see that $\widetilde{P}_{q,S}^{(1)}(k_1)=\theta^{k_1}$. From now on we assume $n > k_1,$ and we can write $\widetilde{P}_{q,S}^{(1)}(n)$ as
\noindent
\begin{equation*}\label{eq:bn}
\begin{split}
\widetilde{P}_{q,S}^{(1)}(n)=P_{q,\theta}\left(L_{n-k_{1}}^{(0)}< k_{1}\ \wedge\ X_{n-k_{1}}=0\ \wedge\ X_{n-k_{1}+1}=\cdots =X_{n}=1\right).\\
\end{split}
\end{equation*}
\noindent
We partition the event $W_{\widetilde{S}}^{(1)}=n$ into disjoint events given by $F_{n-k_1}=i,$ for $i=1, \ldots,\ \min(n-k_1,k_2-1).$ Adding the probabilities we have
\noindent
\begin{equation*}\label{eq:bn}
\begin{split}
\widetilde{P}_{q,S}^{(1)}(n)=\sum_{i=1}^{\min(n-k_1,k_2-1)}P_{q,\theta}\Big(L_{n-k_{1}}^{(1)}< k_{1}\ \wedge\ F_{n-k_1}=i&\ \wedge\ X_{n-k_{1}}=0\ \wedge\\
&X_{n-k_{1}+1}=\cdots =X_{n}=1\Big).\\
\end{split}
\end{equation*}
\noindent
Using $S_{n-k_1}=n-k_1-F_{n-k_1},$ we can rewrite as follows.
\begin{equation*}\label{eq:bn}
\begin{split}
\widetilde{P}_{q,S}^{(1)}(n)=\sum_{i=1}^{\min(n-k_1,k_2-1)}P_{q,\theta}\Big(L_{n-k_{1}}^{(1)}< k_{1}\ \wedge\ S_{n-k_1}=n-k_1-i&\ \wedge\ X_{n-k_{1}}=0\ \wedge\\
&X_{n-k_{1}+1}=\cdots =X_{n}=1\Big).\\
\end{split}
\end{equation*}
\noindent
If the number of $0$'s in the first $n-k_1$ trials is equal to $i,$ that is, $F_{n-k_1}=i,$ then the probability of successes in all of the $(n-k_1+1)$-th to $n$-th trials is
\noindent
\begin{equation*}\label{eq: kk}
\begin{split}
P_{q,\theta}(X_{n-k_1+1}=\cdots = X_n = 1 \,\mid \, F_{n-k_1}=i)=\left(\theta q^{i}\right)^{k_1}.
\end{split}
\end{equation*}
\noindent
We write $E_{n,i}^{(0)}$ for the event $\left\{L_n^{(1)}< k_1\ \wedge\ X_n=0\ \wedge\ S_n=n-i\right\}.$ We can now rewrite as follows.
\noindent
\begin{equation*}\label{eq:bn1}
\begin{split}
\widetilde{P}_{q,S}^{(1)}(n)=\sum_{i=1}^{\min(n-k_1,k_2-1)}P_{q,\theta}\Big(L_{n-k_{1}}^{(1)}< &k_{1}\ \wedge\ S_{n-k_1}=n-k_1-i\ \wedge\ X_{n-k_{1}}=0\Big)\\
&\times P_{q,\theta}\Big(X_{n-k_{1}+1}=\cdots =X_{n}=1\mid F_{n-k_{1}}=i\Big)\\
=\sum_{i=1}^{\min(n-k_1,k_2-1)}P_{q,\theta}\Big(E_{n-k_{1},\ i}^{(0)}&\Big)\left(\theta q^{i}\right)^{k_1}.\\
\end{split}
\end{equation*}
\noindent
We are going to focus on the event $E_{n-k_{1},\ i}^{(0)}$. A typical element of the event $\Big\{L_{n-k_{1}}^{(1)}< k_{1}\ \&\ S_{n-k_1}=n-k_1-i\ \wedge\ X_{n-k_{1}}=0\Big\}$ is an ordered sequence which consists of $n-k_1-i$ successes and $i$ failures such that the length of the longest success run is less than $k_1$. The number of these sequences can be derived as follows. First we will distribute the $i$ failures. Let $s$ $(1\leq s \leq i)$ be the number of runs of failures in the event $E_{n-k_{1},i}^{(0)}$. We divide into two cases: starting with a success run or starting with a failure run. Thus, we distinguish between two types of sequences in the event $$\left\{L_{n-k_{1}}^{(1)}< k_{1}\ \&
\ S_{n-k_1}=n-k_1-i\ \wedge\ X_{n-k_{1}}=0\right\},$$respectively named $(s-1,\ s)$-type and $(s,\ s)$-type, which are defined as follows.
\noindent
\begin{equation*}\label{eq:bn}
\begin{split}
(s-1,\ s)\text{-type}\ :&\quad \overbrace{0\ldots 0}^{y_{1}}\mid\overbrace{1\ldots 1}^{x_{1}}\mid\overbrace{0\ldots 0}^{y_{2}}\mid\overbrace{1\ldots 1}^{x_{2}}\mid \ldots \mid \overbrace{0\ldots 0}^{y_{s-1}}\mid\overbrace{1\ldots 1}^{x_{s-1}}\mid\overbrace{0\ldots 0}^{y_{s}},\\
\end{split}
\end{equation*}
\noindent
with $i$ $0$'s and $n-k_{1}-i$ $1$'s, where $x_{j}$ $(j=1,\ldots,s-1)$ represents a length of run of $1$'s and $y_{j}$ $(j=1,\ldots,s)$ represents the length of a run of $0$'s. And all integers $x_{1},\ldots,x_{s-1}$, and $y_{1},\ldots ,y_{s}$ satisfy the conditions
\noindent
\begin{equation*}\label{eq:bn}
\begin{split}
0 < x_j < k_1 \mbox{\ for\ } j=1,...,s-1, \mbox{\ and\ } x_1+\cdots +x_{s-1} = n-k_1-i,
\end{split}
\end{equation*}
\noindent
\begin{equation*}\label{eq:bn}
\begin{split}
(y_{1},\ldots,y_{s}) \in S_{i,s}^{0}.
\end{split}
\end{equation*}
\begin{equation*}\label{eq:bn}
\begin{split}
(s,\ s)\text{-type}\  :&\quad \overbrace{1\ldots 1}^{x_{1}}\mid\overbrace{0\ldots 0}^{y_{1}}\mid\overbrace{1\ldots 1}^{x_{2}}\mid\overbrace{0\ldots 0}^{y_{2}}\mid\overbrace{1\ldots 1}^{x_{3}}\mid \ldots \mid \overbrace{0\ldots 0}^{y_{s-1}}\mid\overbrace{1\ldots 1}^{x_{s}}\mid\overbrace{0\ldots 0}^{y_{s}},
\end{split}
\end{equation*}
\noindent
with $i$ $0$'s and $n-k_{1}-i$ $1$'s, where $x_{j}$ $(j=1,\ldots,s)$ represents a length of run of $1$'s and $y_{j}$ $(j=1,\ldots,s)$ represents the length of a run of $0$'s. Here all of $x_1,\ldots, x_{s}$, and $y_1,\ldots, y_s$ are integers, and they satisfy
\noindent
\begin{equation*}\label{eq:bn}
\begin{split}
0 < x_j < k_1 \mbox{\ for\ } j=1,...,s, \mbox{\ and\ } x_1+\cdots +x_s = n-k_1-i,
\end{split}
\end{equation*}
\noindent
\begin{equation*}\label{eq:bn}
\begin{split}
(y_{1},\ldots,y_{s}) \in S_{i,s}^{0}.
\end{split}
\end{equation*}
\noindent
Then the probability of the event $E_{n-k_1,i}^{(0)}$ is given by
\noindent
\begin{equation*}\label{eq: 1.1}
\begin{split}
P_{q,\theta}\Big(E_{n-k_{1},\ i}^{(0)}&\Big)=\\
\sum_{s=1}^{i}\Bigg[\bigg\{&\sum_{\substack{x_{1}+\cdots+x_{s-1}=n-k_{1}-i\\ x_{1},\ldots,x_{s-1} \in \{1,\ldots,k_{1}-1\}}}{\hspace{0.3cm}\sum_{(y_{1},\ldots,y_{s}) \in  S_{i,s}^{0}}} \big(1-\theta q^{0}\big)\cdots \big(1-\theta q^{y_{1}-1}\big)\times\\
&\Big(\theta q^{y_{1}}\Big)^{x_{1}}\big(1-\theta q^{y_{1}}\big)\cdots \big(1-\theta q^{y_{1}+y_{2}-1}\big)\times\\
&\Big(\theta q^{y_{1}+y_{2}}\Big)^{x_{2}}\big(1-\theta q^{y_{1}+y_2}\big)\cdots \big(1-\theta q^{y_{1}+y_{2}+y_3-1}\big)\times \\
&\quad \quad \quad \quad\quad \quad\quad \quad \quad \quad \quad \vdots\\
&\Big(\theta q^{y_{1}+\cdots+y_{s-1}}\Big)^{x_{s-1}}
\big(1-\theta q^{y_{1}+\cdots+y_{s-1}}\big)\cdots \big(1-\theta q^{y_{1}+\cdots+y_{s}-1}\big)\bigg\}+\\
\end{split}
\end{equation*}
\noindent
\begin{equation*}\label{eq: 1.1}
\begin{split}
\quad  \quad \quad \quad \quad \bigg\{&\sum_{\substack{x_{1}+\cdots+x_{s}=n-k_{1}-i\\ x_{1},\ldots,x_{s} \in \{1,\ldots,k_{1}-1\}}}{\hspace{0.3cm}\sum_{(y_{1},\ldots,y_{s}) \in  S_{i,s}^{0}}}\big(\theta q^{0}\big)^{x_{1}}\big(1-\theta q^{0}\big)\cdots (1-\theta q^{y_{1}-1})\times\\
&\Big(\theta q^{y_{1}}\Big)^{x_{2}}\big(1-\theta q^{y_{1}+y_2}\big)\cdots \big(1-\theta q^{y_{1}+y_{2}-1}\big)\times\\
&\Big(\theta q^{y_{1}+y_{2}}\Big)^{x_{3}}\big(1-\theta q^{y_{1}}\big)\cdots \big(1-\theta q^{y_{1}+y_{2}+y_3-1}\big)\times \\
&\quad \quad \quad \quad\quad \quad\quad \quad \quad \quad \quad \vdots\\
&\Big(\theta q^{y_{1}+\cdots+y_{s-1}}\Big)^{x_{s}}
\big(1-\theta q^{y_{1}+\cdots+y_{s-1}}\big)\cdots \big(1-\theta q^{y_{1}+\cdots+y_{s}-1}\big)\bigg\}\Bigg].\\
\end{split}
\end{equation*}
\noindent
Using simple exponentiation algebra arguments to simplify,
\noindent
\begin{equation*}\label{eq: 1.1}
\begin{split}
&P_{q,\theta}\Big(E_{n-k_{1},\ i}^{(0)}\Big)=\\
&\theta^{n-k_{1}-i}{\prod_{j=1}^{i}}\ (1-\theta q^{j-1})\times\\
&\sum_{s=1}^{i}\Bigg[\sum_{\substack{x_{1}+\cdots+x_{s-1}=n-k_{1}-i\\ x_{1},\ldots,x_{s-1} \in \{1,\ldots,k_{1}-1\}}}\
\sum_{(y_{1},\ldots,y_{s}) \in  S_{i,s}^{0}}q^{y_{1}x_{1}+(y_{1}+y_{2})x_{2}+\cdots+(y_{1}+\cdots+y_{s-1})x_{s-1}}+\\
&\quad \quad \sum_{\substack{x_{1}+\cdots+x_{s}=n-k_{1}-i\\ x_{1},\ldots,x_{s} \in \{1,\ldots,k_{1}-1\}}}\
\sum_{(y_{1},\ldots,y_{s}) \in  S_{i,s}^{0}}q^{y_{1}x_{2}+(y_{1}+y_{2})x_{3}+\cdots+(y_{1}+\cdots+y_{s-1})x_{s}}\Bigg].\\
\end{split}
\end{equation*}
\noindent
Using Lemma \ref{lemma:3.5} and Lemma \ref{lemma:3.7}, we can rewrite as follows.
\noindent
\begin{equation*}\label{eq: 1.1}
\begin{split}
P_{q,\theta}\Big(E_{n-k_{1},\ i}^{(0)}\Big)=\theta^{n-k_{1}-i}{\prod_{j=1}^{i}}(1-\theta q^{j-1})\sum_{s=1}^{i}\Bigg[\overline{E}&_{q,0,0}^{k_1,\infty}(n-k_{1}-i,\ i,\ s)+\overline{F}_{q,0,0}^{k_1,\infty}(n-k_{1}-i,\ i,\ s)\Bigg],\\
\end{split}
\end{equation*}\\
\noindent
where
\noindent
\begin{equation*}\label{eq: 1.1}
\begin{split}
\overline{E}&_{q,0,0}^{k_1,\infty}(n-k_{1}-i,\ i,\ s)\\
&=\sum_{\substack{x_{1}+\cdots+x_{s-1}=n-k_{1}-i\\ x_{1},\ldots,x_{s-1} \in \{1,\ldots,k_{1}-1\}}}\
\sum_{(y_{1},\ldots,y_{s}) \in S_{i,s}^{0}}q^{y_{1}x_{1}+(y_{1}+y_{2})x_{2}+\cdots+(y_{1}+\cdots+y_{s-1})x_{s-1}},
\end{split}
\end{equation*}
\noindent
and
\noindent
\begin{equation*}\label{eq: 1.1}
\begin{split}
\overline{F}&_{q,0,0}^{k_1,\infty}(n-k_{1}-i,\ i,\ s)\\
&=\sum_{\substack{x_{1}+\cdots+x_{s}=n-k_{1}-i\\ x_{1},\ldots,x_{s} \in \{1,\ldots,k_{1}-1\}}}\
\sum_{(y_{1},\ldots,y_{s}) \in  S_{i,s}^{0}}q^{y_{1}x_{2}+(y_{1}+y_{2})x_{3}+\cdots+(y_{1}+\cdots+y_{s-1})x_{s}}.
\end{split}
\end{equation*}
\noindent
Therefore we can compute the probability of the event $\widetilde{W}_{S}^{(1)}=n$ as follows
\noindent
\begin{equation*}\label{eq:bn1}
\begin{split}
\widetilde{P}_{q,S}^{(1)}(n)=&\sum_{i=1}^{\min(n-k_1,k_2-1)}P_{q,\theta}\Big(E_{n-k_{1},\ i}^{(0)}\Big)\Big(\theta q^{i}\Big)^{k_1}\\
=&\sum_{i=1}^{\min(n-k_1,k_2-1)}
\theta^{n-k_{1}-i}{\prod_{j=1}^{i}}(1-\theta q^{j-1})\sum_{s=1}^{i}\Bigg[\overline{E}_{q,0,0}^{k_1,\infty}(n-k_{1}-i,\ i,\ s)\\
&+\overline{F}_{q,0,0}^{k_1,\infty}(n-k_{1}-i,\ i,\ s)\Bigg]\times \Big(\theta q^{i}\Big)^{k_1}.\\
\end{split}
\end{equation*}
\noindent
Finally, applying typical factorization algebra arguments, $\widetilde{P}_{S}^{(1)}(n)$ can be rewritten as follows
\noindent
\begin{equation*}\label{eq:bn1}
\begin{split}
\widetilde{P}_{q,S}^{(1)}(n)=\sum_{i=1}^{\min(n-k_1,k_2-1)}\theta^{n-i} q^{ik_{1}}{\prod_{j=1}^{i}}(1-\theta q^{j-1})\sum_{s=1}^{i}\Bigg[\overline{E}_{q,0,0}^{k_1,\infty}&(n-k_{1}-i,\ i,\ s)\\
&+\overline{F}_{q,0,0}^{k_1,\infty}(n-k_{1}-i,\ i,\ s)\Bigg].\\
\end{split}
\end{equation*}
\noindent
We are now going to study $\widetilde{P}_{q,S}^{(0)}(n)$. It is easy to see that $\widetilde{P}_{q,S}^{(0)}(k_2)={\prod_{j=1}^{k_2}}(1-\theta q^{j-1})$. From now on we assume $n > k_2,$ and we can write $\widetilde{P}_{q,S}^{(0)}(n)$ as follows.
\noindent
\begin{equation*}\label{eq:bn}
\begin{split}
\widetilde{P}_{q,S}^{(0)}(n)&=P_{q,\theta}\left(L_{n}^{(1)}< k_{1}\ \wedge\ X_{n}=0\ \wedge\ S_{n}=n-k_2\right).\\
\end{split}
\end{equation*}
\noindent
We are going to focus on the event $\left\{L_{n}^{(1)}< k_{1}\ \wedge\ X_{n}=0\ \wedge\ S_{n}=n-k_2\right\}$. A typical element of the event $\left\{L_{n}^{(1)}< k_{1}\ \wedge\ X_{n}=0\ \wedge \ S_{n}=n-k_2\right\}$ is an ordered sequencewhich consists of $n-k_{2}$ successes and $k_2$ failures such that the length of the longest success run is less than $k_1$. The number of these sequences can be derived as follows. First we will distribute the $k_2$ failures. Let $s$ $(1\leq s \leq k_2)$ be the number of runs of failures in the event $\left\{L_{n}^{(1)}< k_{1}\ \&
\ X_{n}=0\ \wedge\ S_{n}=n-k_2\right\}$. We divide into two cases: starting with a success run or starting with a failure run. Thus, we distinguish between two types of sequences in the event  $$\left\{L_{n}^{(1)}< k_{1}\ \wedge\ X_{n}=0\ \wedge \ S_{n}=n-k_2\right\},$$ respectively named $(s-1,\ s)$-type and $(s,\ s)$-type, which are defined as follows.
\noindent
\begin{equation*}\label{eq:bn}
\begin{split}
(s-1,\ s)\text{-type}\ :&\quad \overbrace{0\ldots 0}^{y_{1}}\mid\overbrace{1\ldots 1}^{x_{1}}\mid\overbrace{0\ldots 0}^{y_{2}}\mid\overbrace{1\ldots 1}^{x_{2}}\mid \ldots \mid \overbrace{0\ldots 0}^{y_{s-1}}\mid\overbrace{1\ldots 1}^{x_{s-1}}\mid\overbrace{0\ldots 0}^{y_{s}},\\
\end{split}
\end{equation*}
\noindent
with $k_2$ $0$'s and $n-k_{2}$ $1$'s, where $x_{j}$ $(j=1,\ldots,s-1)$ represents a length of run of $1$'s and $y_{j}$ $(j=1,\ldots,s)$ represents the length of a run of $0$'s. And all integers $x_{1},\ldots,x_{s-1}$, and $y_{1},\ldots ,y_{s}$ satisfy the conditions
\noindent
\begin{equation*}\label{eq:bn}
\begin{split}
0 < x_j < k_1 \mbox{\ for\ } j=1,...,s-1, \mbox{\ and\ } x_1+\cdots +x_{s-1} = n-k_2,
\end{split}
\end{equation*}
\noindent
\begin{equation*}\label{eq:bn}
\begin{split}
(y_{1},\ldots,y_{s})\in S_{k_2,s}^{0},
\end{split}
\end{equation*}
\noindent
\begin{equation*}\label{eq:bn}
\begin{split}
(s,\ s)\text{-type}\  :&\quad \overbrace{1\ldots 1}^{x_{1}}\mid\overbrace{0\ldots 0}^{y_{1}}\mid\overbrace{1\ldots 1}^{x_{2}}\mid\overbrace{0\ldots 0}^{y_{2}}\mid\overbrace{1\ldots 1}^{x_{3}}\mid \ldots \mid \overbrace{0\ldots 0}^{y_{s-1}}\mid\overbrace{1\ldots 1}^{x_{s}}\mid\overbrace{0\ldots 0}^{y_{s}},
\end{split}
\end{equation*}
\noindent
with $k_2$ $0$'s and $n-k_{2}$ $1$'s, where $x_{j}$ $(j=1,\ldots,s)$ represents a length of run of $1$'s and $y_{j}$ $(j=1,\ldots,s)$ represents the length of a run of $0$'s. Here all of $x_1,\ldots, x_{s-1},$ and $y_1,\ldots, y_s$ are integers, and they satisfy
\noindent
\begin{equation*}\label{eq:bn}
\begin{split}
0 < x_j < k_1 \mbox{\ for\ } j=1,...,s, \mbox{\ and\ } x_1+\cdots +x_{s} = n-k_2,
\end{split}
\end{equation*}
\noindent
\begin{equation*}\label{eq:bn}
\begin{split}
(y_{1},\ldots,y_{s})\in S_{k_2,s}^{0},
\end{split}
\end{equation*}
\noindent
Then the probability of the event $\left\{L_{n}^{(1)}< k_{1}\ \wedge\ X_{n}=0\ \wedge\ S_{n}=n-k_2\right\}$ is given by
\noindent
\begin{equation*}\label{eq: 1.1}
\begin{split}
P_{q,\theta}\Big(L_{n}^{(1)}<& k_{1}\ \wedge\ X_{n}=0\ \wedge\ S_{n}=n-k_2\Big)\\
=\sum_{s=1}^{k_2}\Bigg[\bigg\{&\sum_{\substack{x_{1}+\cdots+x_{s-1}=n-k_{2}\\ x_{1},\ldots,x_{s-1} \in \{1,\ldots,k_{1}-1\}}}\
\sum_{(y_{1},\ldots,y_{s}) \in S_{k_{2},s}^{0}} \big(1-\theta q^{0}\big)\cdots \big(1-\theta q^{y_{1}-1}\big)\times\\
&\Big(\theta q^{y_{1}}\Big)^{x_{1}}\big(1-\theta q^{y_{1}}\big)\cdots \big(1-\theta q^{y_{1}+y_{2}-1}\big)\times\\
&\Big(\theta q^{y_{1}+y_{2}}\Big)^{x_{2}}\big(1-\theta q^{y_{1}+y_2}\big)\cdots \big(1-\theta q^{y_{1}+y_{2}+y_3-1}\big)\times \\
&\quad \quad \quad \quad\quad \quad\quad \quad \quad \quad \quad \vdots\\
&\Big(\theta q^{y_{1}+\cdots+y_{s-1}}\Big)^{x_{s-1}}
\big(1-\theta q^{y_{1}+\cdots+y_{s-1}}\big)\cdots \big(1-\theta q^{y_{1}+\cdots+y_{s}-1}\big)\bigg\}+\\
\end{split}
\end{equation*}
\noindent
\begin{equation*}\label{eq: 1.1}
\begin{split}
\quad \quad \quad \quad \quad \quad \quad \bigg\{&\sum_{\substack{x_{1}+\cdots+x_{s}=n-k_{2}\\ x_{1},\ldots,x_{s} \in \{1,\ldots,k_{1}-1\}}}\
\sum_{(y_{1},\ldots,y_{s}) \in S_{k_{2},s}^{0}}\big(\theta q^{0}\big)^{x_{1}}\big(1-\theta q^{0}\big)\cdots (1-\theta q^{y_{1}-1})\times\\
&\Big(\theta q^{y_{1}}\Big)^{x_{2}}\big(1-\theta q^{y_{1}}\big)\cdots \big(1-\theta q^{y_{1}+y_{2}-1}\big)\times\\
&\Big(\theta q^{y_{1}+y_{2}}\Big)^{x_{3}}\big(1-\theta q^{y_{1}+y_{2}}\big)\cdots \big(1-\theta q^{y_{1}+y_{2}+y_3-1}\big)\times \\
&\quad \quad \quad \quad\quad \quad\quad \quad \quad \quad \quad \vdots\\
&\Big(\theta q^{y_{1}+\cdots+y_{s-1}}\Big)^{x_{s}}
\big(1-\theta q^{y_{1}+\cdots+y_{s-1}}\big)\cdots \big(1-\theta q^{y_{1}+\cdots+y_{s}-1}\big)\bigg\}\Bigg].\\
\end{split}
\end{equation*}
\noindent
Using simple exponentiation algebra arguments to simplify,
\noindent
\begin{equation*}\label{eq: 1.1}
\begin{split}
&P_{q,\theta}\Big(L_{n}^{(1)}< k_{1}\ \wedge\ X_{n}=0\ \wedge
\ S_{n}=n-k_2\Big)=\\
&\theta^{n-k_{2}}{\prod_{j=1}^{k_2}}\ (1-\theta q^{j-1})\times\\
&\sum_{s=1}^{k_2}\Bigg[\sum_{\substack{x_{1}+\cdots+x_{s-1}=n-k_{2}\\ x_{1},\ldots,x_{s-1} \in \{1,\ldots,k_{1}-1\}}}\
\sum_{(y_{1},\ldots,y_{s}) \in S_{k_{2},s}^{0}}q^{y_{1}x_{1}+(y_{1}+y_{2})x_{2}+\cdots+(y_{1}+\cdots+y_{s-1})x_{s-1}}+\\
&\quad\quad\ \sum_{\substack{x_{1}+\cdots+x_{s}=n-k_{2}\\ x_{1},\ldots,x_{s} \in \{1,\ldots,k_{1}-1\}}}\
\sum_{(y_{1},\ldots,y_{s}) \in S_{k_{2},s}^{0}}q^{y_{1}x_{2}+(y_{1}+y_{2})x_{3}+\cdots+(y_{1}+\cdots+y_{s-1})x_{s}}\Bigg].\\
\end{split}
\end{equation*}
\noindent
Using Lemma \ref{lemma:3.5} and Lemma \ref{lemma:3.7}, we can rewrite as follows.
\noindent
\begin{equation*}\label{eq: 1.1}
\begin{split}
P_{q,\theta}\Big(L_{n}^{(1)}&< k_{1}\ \wedge\ X_{n}=0\ \wedge\ S_{n}=n-k_2\Big)\\
=\theta^{n-k_{2}}&{\prod_{j=1}^{k_2}}\ (1-\theta q^{j-1})\ \sum_{s=1}^{k_2}\bigg[\overline{E}_{q,0,0}^{k_1,\infty}(n-k_2,\ k_{2} ,\ s)+\overline{F}_{q,0,0}^{k_1,\infty}(n-k_2,\ k_{2} ,\ s)\bigg],\\
\end{split}
\end{equation*}\\
\noindent
where
\noindent
\begin{equation*}\label{eq: 1.1}
\begin{split}
\overline{E}&_{q,0,0}^{k_1,\infty}(n-k_{2},\ k_2,\ s)\\
&=\sum_{\substack{x_{1}+\cdots+x_{s-1}=n-k_{2}\\ x_{1},\ldots,x_{s-1} \in \{1,\ldots,k_{1}-1\}}}{\hspace{0.3cm}\sum_{(y_{1},\ldots,y_{s}) \in S_{k_2,s}^{0}}}q^{y_{1}x_{1}+(y_{1}+y_{2})x_{2}+\cdots+(y_{1}+\cdots+y_{s-1})x_{s-1}},
\end{split}
\end{equation*}
\noindent
and
\noindent
\begin{equation*}\label{eq: 1.1}
\begin{split}
\overline{F}&_{q,0,0}^{k_1,\infty}(n-k_{2},\ k_2,\ s)\\
&=\sum_{\substack{x_{1}+\cdots+x_{s}=n-k_{2}\\ x_{1},\ldots,x_{s} \in \{1,\ldots,k_{1}-1\}}}{\hspace{0.3cm}\sum_{y_{1},\ \ldots,\ y_{s} \in S_{k_2,s}^{0}}}q^{y_{1}x_{2}+(y_{1}+y_{2})x_{3}+\cdots+(y_{1}+\cdots+y_{s-1})x_{s}}.
\end{split}
\end{equation*}
\noindent
Therefore we can compute the probability of the event $\widetilde{W}_{S}^{(0)}=n$ as follows.
\noindent
\begin{equation*}\label{eq:bn1}
\begin{split}
\widetilde{P}_{q,S}^{(0)}(n)=\theta^{n-k_{2}}&{\prod_{j=1}^{k_2}}\ (1-\theta q^{j-1})\ \sum_{s=1}^{k_2}\bigg[\overline{E}_{q,0,0}^{k_1,\infty}(n-k_2,\ k_{2} ,\ s)+\overline{F}_{q,0,0}^{k_1,\infty}(n-k_2,\ k_{2} ,\ s)\bigg].\\
\end{split}
\end{equation*}
\noindent
Thus proof is completed.
\end{proof}
\noindent
It is worth mentioning here that the PMF $\widetilde{f}_{q,S}(n;\theta)$ approaches the probability function of the sooner waiting time distribution of order $(k_1,k_2)$
in the limit as $q$ tends to 1 when a run quota on runs of successes and a frequency quota on failures are imposed of IID model. The details are presented in the following remark.
\noindent
\begin{remark}
{\rm For $q=1$, the PMF $\widetilde{f}_{q,S}(n;\theta)$ reduces to the PMF $\widetilde{f}_{S}(n;\theta)=P_{\theta}(\widetilde{W}_{S}=n)$ for $n\geq min(k_1,k_2)$ is given by

\noindent
\begin{equation*}\label{eq:bn}
\begin{split}
P(\widetilde{W}_{S}=n)=\widetilde{P}_{S}^{(1)}(n)+\widetilde{P}_{S}^{(0)}(n),
\end{split}
\end{equation*}
\noindent
where $\widetilde{P}_S^{(1)}(k_1)=\theta^{k_{1}}$, $\widetilde{P}_S^{(0)}(k_2)=(1-\theta)^{k_2}$,
\noindent
\begin{equation*}\label{eq:bn1}
\begin{split}
\widetilde{P}_{S}^{(1)}(n)=\sum_{i=1}^{\min(n-k_1,k_2-1)}\theta^{n-i} (1-\theta)^{i}\sum_{s=1}^{i}M(s,\ n-k_1-i)\Big[S(s-1,\ k_2,\ i)+S(s,\ k_2,\ i)\Big],\ n>k_{1}\\
\end{split}
\end{equation*}
\noindent
and
\noindent
\begin{equation*}\label{eq:bn1}
\begin{split}
\widetilde{P}_{S}^{(0)}(n)=\theta^{n-k_{2}}(1-\theta)^{k_2}\sum_{s=1}^{k_2}\Big[S(s-1,\ k_1,\ n-k_2)+S(s,\ k_1,\ n-k_2)\Big]M(s,\ k_2),\ n>k_2,\\
\end{split}
\end{equation*}
\noindent
where
\noindent
\begin{equation*}\label{eq: 1.1}
\begin{split}
M(a,\ b)={b-1 \choose a-1},
\end{split}
\end{equation*}
\noindent
and
\noindent
\begin{equation*}\label{eq: 1.1}
\begin{split}
S(a,\ b,\ c)=\sum_{j=0}^{min(a,\left[\frac{c-a}{b-1}\right])}(-1)^{j}{a \choose j}{c-j(b-1)-1 \choose a-1}.
\end{split}
\end{equation*}
\noindent
See, e.g. \citet{charalambides2002enumerative}.
}

\end{remark}

\subsubsection{Later waiting time}

The problem of waiting time that will be discussed in this section is one of the 'later cases' and it emerges when a run quota on successes and a frequency quota on failures are imposed. More specifically, Binary (zero and one) trials with probability of ones varying according to a geometric rule, are performed sequentially until $k_1$ consecutive successes or $k_2$ failures in total are observed, whichever event gets observed later.
Let random variable $\widehat{W}_{L}$ denote the waiting time until both $k_{1}$ consecutive successes and $k_2$ failures in total have observed, whichever event gets observed later.  We now make some Lemma for the proofs of Theorem in the sequel.
\noindent
\begin{definition}
For $0<q\leq1$, we define
\begin{equation*}\label{eq: 1.1}
\begin{split}
\overline{K}_{q}(u,v,s)={\sum_{x_{1},\ldots,x_{s-1}}}\
\sum_{y_{1},\ldots,y_{s}} q^{y_{1}x_{1}+(y_{1}+y_{2})x_{2}+\cdots+(y_{1}+\cdots+y_{s-1})x_{s-1}},\
\end{split}
\end{equation*}
\noindent
where the summation is over all integers $x_1,\ldots,x_{s-1},$ and $y_1,\ldots,y_{s}$ satisfying
\noindent
\begin{equation*}\label{eq:1}
\begin{split}
(x_1,\ldots,x_{s-1})\in S_{u,s-1}^{0},\ \text{and}
\end{split}
\end{equation*}
\noindent
\begin{equation*}\label{eq:2}
\begin{split}
(y_{1},\ldots,y_{s}) \in S_{v,s}^{0}.
\end{split}
\end{equation*}
\end{definition}
\noindent
The following gives a recurrence relation useful for the computation of $B_{q}^{k_1,k_2}(m,r,s)$.
\noindent
\begin{lemma}
\label{lemma:4.5}
$\overline{K}_{q}(u,v,s)$ obeys the following recurrence relation.
\begin{equation*}\label{eq: 1.1}
\begin{split}
\overline{K}&_{q}(u,v,s)\\
&=\left\{
  \begin{array}{ll}
    \sum_{b=1}^{v-(s-1)}\sum_{a=1}^{u-(s-2)}q^{a(v-b)} \overline{K}_{q}(u-a,v-b,s-1), & \text{for}\ s>1,\ s-1\leq u\ \text{and}\ s\leq v \\
    1, & \text{for}\ s=1,\ u=0,\ \text{and}\ 1\leq v\\
    0, & \text{otherwise.}\\
  \end{array}
\right.
\end{split}
\end{equation*}
\end{lemma}
\begin{proof}
For $s > 1$, $s-1\leq u$ and $s\leq v$, we observe that $x_{s-1}$ may assume any value $1,\ldots,u-(s-2)$, then $\overline{K}_{q}(u,v,s)$ can be written as
\noindent
\begin{equation*}
\begin{split}
\overline{K}&_{q}(u,v,s)\\
=&\sum_{x_{s-1}=1}^{u-(s-2)}{\sum_{(x_1,\ldots,x_{s-1})\in S_{u-x_{s-1},s-1}^{0}}}{\hspace{0.3cm}\sum_{(y_{1},\ldots,y_{s}) \in S_{v,s}^{0}}}q^{x_{s-1}(v-y_{s})}\ q^{y_{1}x_{1}+(y_{1}+y_{2})x_{2}+\cdots+(y_{1}+\cdots+y_{s-2})x_{s-2}}
\end{split}
\end{equation*}
\noindent
Similarly, we observe that since $y_s$ can assume the values $1,\ldots,v-(s-1)$, then $\overline{K}_{q}(u,v,s)$ can be rewritten as
\noindent
\begin{equation*}
\begin{split}
\overline{K}&_{q}(u,v,s)\\
=&\sum_{y_{s}=1}^{v-(s-1)}\sum_{x_{s-1}=1}^{u-(s-2)}q^{x_{s-1}(v-y_{s})}{\sum_{(x_1,\ldots,x_{s-1})\in S_{u,s-1}^{0}}}{\hspace{0.3cm}\sum_{(y_{1},\ldots,y_{s}) \in S_{v,s}^{0}}}\ q^{y_{1}x_{1}+(y_{1}+y_{2})x_{2}+\cdots+(y_{1}+\cdots+y_{s-2})x_{s-2}}\\
=&\sum_{b=1}^{v-(s-1)}\sum_{a=1}^{u-(s-2)}q^{a(v-b)} \overline{K}_{q}(u-a,v-b,s-1).
\end{split}
\end{equation*}
\noindent
The other cases are obvious and thus the proof is completed.
\end{proof}
\begin{remark}
{\rm
We observe that $\overline{K}_{1}(m,r,s)$ is the number of integer solutions $(x_{1},\ldots,x_{s-1})$ and $(y_{1},\ldots,y_{s})$ of
\noindent
\begin{equation*}\label{eq:1}
\begin{split}
(x_1,\ldots,x_{s-1})\in S_{u,s-1}^{0},\ \text{and}
\end{split}
\end{equation*}
\noindent
\begin{equation*}\label{eq:2}
\begin{split}
(y_{1},\ldots,y_{s}) \in S_{v,s}^{0}
\end{split}
\end{equation*}
\noindent
given by
\noindent
\begin{equation*}\label{eq: 1.1}
\begin{split}
\overline{K}_{1}(m,r,s)=M(s-1,\ u)M(s,\ v ),
\end{split}
\end{equation*}
\noindent
where $M(a,\ b)$ denotes the total number of integer solution $y_{1}+x_{2}+\cdots+y_{a}=b$ such that $(y_{1},\ldots,y_{a}) \in S_{b,a}^{0}$. The number is given by
\noindent
\begin{equation*}\label{eq: 1.1}
\begin{split}
M(a,\ b)={b-1 \choose a-1}.
\end{split}
\end{equation*}
\noindent
See, e.g. \citet{charalambides2002enumerative}.
}
\end{remark}

\begin{definition}
For $0<q\leq1$, we define
\begin{equation*}\label{eq: 1.1}
\begin{split}
K_{q,k_1,0}(m,r,s)={\sum_{x_{1},\ldots,x_{s-1}}}\
\sum_{y_{1},\ldots,y_{s}} q^{y_{1}x_{1}+(y_{1}+y_{2})x_{2}+\cdots+(y_{1}+\cdots+y_{s-1})x_{s-1}},\
\end{split}
\end{equation*}
\noindent
where the summation is over all integers $x_1,\ldots,x_{s-1},$ and $y_1,\ldots,y_s$ satisfying
\noindent
\begin{equation*}\label{eq:1}
\begin{split}
(x_1,\ldots,x_{s-1})\in S_{m,s-1}^{k_1-1},\ \text{and}
\end{split}
\end{equation*}
\noindent
\begin{equation*}\label{eq:2}
\begin{split}
(y_1,\ldots,y_s)\in S_{r,s}^{0}.
\end{split}
\end{equation*}
\end{definition}
\noindent
The following gives a recurrence relation useful for the computation of $K_{q,k_1,0}(m,r,s)$.
\noindent
\begin{lemma}
\label{lemma:4.6}
$K_{q,k_1,0}(m,r,s)$ obeys the following recurrence relation.
\begin{equation*}\label{eq: 1.1}
\begin{split}
K&_{q,k_1,0}(m,r,s)\\
&=\left\{
  \begin{array}{ll}
    \sum_{a=1}^{k_1-1}\sum_{b=1}^{r-(s-1)}q^{a(r-b)} K_{q,k_1,0}(m-a,r-b,s-1)\\
    +\sum_{a=k_1}^{m-(s-2)}\sum_{b=1}^{r-(s-1)} q^{a(r-b)} \overline{K}_{q}(m-a,r-b,s-1), & \text{for}\ s>1,\ (s-2)+k_1\leq m\\
    &\text{and}\ s \leq r \\
    0, & \text{otherwise.}\\
  \end{array}
\right.
\end{split}
\end{equation*}
\end{lemma}
\begin{proof}
For $s > 1$, $(s-2)+k_1\leq m$ and $s \leq r$, we observe that $y_{s}$ may assume any value $1,\ldots,r-(s-1)$, then $K_{q,k_1,0}(m,r,s)$ can be written as
\begin{equation*}
\begin{split}
K&_{q,k_1,0}(m,r,s)\\
=&\sum_{y_{s}=1}^{r-(s-1)}{\sum_{(x_1,\ldots,x_{s-1})\in S_{m,s-1}^{k_1-1}}}{\hspace{0.3cm}\sum_{(y_{1},\ldots,y_{s}) \in S_{r-y_s,s}^{0}}}q^{x_{s-1}(r-y_{s})}\ q^{y_{1}x_{1}+(y_{1}+y_{2})x_{2}+\cdots+(y_{1}+\cdots+y_{s-2})x_{s-2}}
\end{split}
\end{equation*}
\noindent
Similarly, we observe that since $x_{s-1}$ can assume the values $1,\ldots,m-(s-2)$, then $K_{q,k_1,0}(m,r,s)$ can be rewritten as
\noindent
\begin{equation*}
\begin{split}
K&_{q,k_1,0}(m,r,s)\\
=&\sum_{x_{s-1}=1}^{k_1-1}\sum_{y_{s}=1}^{r-(s-1)}q^{x_{s-1}(r-y_{s})}{\sum_{(x_1,\ldots,x_{s-2})\in S_{m-x_{s-1},s-2}^{k_1-1}}}{\hspace{0.3cm}\sum_{(y_{1},\ldots,y_{s-1}) \in S_{r-y_s,s-1}^{0}}} q^{y_{1}x_{1}+(y_{1}+y_{2})x_{2}+\cdots+(y_{1}+\cdots+y_{s-2})x_{s-2}}\\
&+\sum_{x_{s-1}=k_1}^{m-(s-2)}\sum_{y_{s}=1}^{r-(s-1)}q^{x_{s-1}(r-y_{s})}{\sum_{(x_1,\ldots,x_{s-2})\in S_{m-x_{s-1},s-2}^{0}}}{\hspace{0.3cm}\sum_{(y_{1},\ldots,y_{s-1}) \in S_{r-y_s,s-1}^{0}}}q^{y_{1}x_{1}+(y_{1}+y_{2})x_{2}+\cdots+(y_{1}+\cdots+y_{s-2})x_{s-2}}\\
=&\sum_{a=1}^{k_1-1}\sum_{b=1}^{r-(s-1)}q^{a(r-b)} K_{q,k_1,0}(m-a,r-b,s-1)+\sum_{a=k_1}^{m-(s-2)}\sum_{b=1}^{r-(s-1)} q^{a(r-b)} \overline{K}_{q}(m-a,r-b,s-1).
\end{split}
\end{equation*}
The other cases are obvious and thus the proof is completed.
\end{proof}
\begin{remark}
{\rm
We observe that $K_{1,k_1,0}(m,r,s)$ is the number of integer solutions $(x_{1},\ldots,x_{s-1})$ and $(y_{1},\ldots,y_{s})$ of
\noindent
\begin{equation*}\label{eq:1}
\begin{split}
(x_1,\ldots,x_{s-1})\in S_{m,s-1}^{k_1-1},\ \text{and}
\end{split}
\end{equation*}
\noindent
\begin{equation*}\label{eq:2}
\begin{split}
(y_{1},\ldots,y_{s}) \in S_{r,s}^{0}
\end{split}
\end{equation*}
\noindent
given by
\noindent
\begin{equation*}\label{eq: 1.1}
\begin{split}
K_{1,k_1,0}(m,r,s)=R(s-1,\ k_{1},\ m)M(s,\ r),
\end{split}
\end{equation*}
\noindent
where $R(a,\ b,\ c)$ denotes the total number of integer solution $x_{1}+x_{2}+\cdots+x_{a}=c$ such that $(x_{1},\ldots,x_{a}) \in S_{r,a}^{k_2}$. The number is given by
\noindent
\begin{equation*}\label{eq: 1.1}
\begin{split}
R(a,\ b,\ c)=\sum_{j=1}^{min\left(a,\ \left[\frac{c-a}{b-1}\right]\right) }(-1)^{j+1}{a \choose j}{c-j(b-1)-1 \choose a-1}.
\end{split}
\end{equation*}
\noindent
where $M(a,\ b)$ denotes the total number of integer solution $y_{1}+x_{2}+\cdots+y_{a}=b$ such that $(y_{1},\ldots,y_{a}) \in S_{b,a}^{0}$. The number is given by
\noindent
\begin{equation*}\label{eq: 1.1}
\begin{split}
M(a,\ b)={b-1 \choose a-1}.
\end{split}
\end{equation*}
\noindent
See, e.g. \citet{charalambides2002enumerative}.
}
\end{remark}
\begin{definition}
For $0<q\leq1$, we define
\begin{equation*}\label{eq: 1.1}
\begin{split}
\overline{L}_{q}(u,v,s)={\sum_{x_{1},\ldots,x_{s}}}\
\sum_{y_{1},\ldots,y_{s}} q^{y_{1}x_{2}+(y_{1}+y_{2})x_{3}+\cdots+(y_{1}+\cdots+y_{s-1})x_{s}},\
\end{split}
\end{equation*}
\noindent
where the summation is over all integers $x_1,\ldots,x_{s},$ and $y_1,\ldots,y_s$ satisfying
\noindent
\begin{equation*}\label{eq:1}
\begin{split}
(x_1,\ldots,x_{s})\ \in S_{u,s}^{0},\ \text{and}
\end{split}
\end{equation*}
\noindent
\begin{equation*}\label{eq:2}
\begin{split}
(y_{1},\ldots,y_{s}) \in S_{v,s}^{0}
\end{split}
\end{equation*}
\noindent
\end{definition}
\noindent
The following gives a recurrence relation useful for the computation of $\overline{L}_{q}(u,v,s)$.
\noindent
\begin{lemma}
\label{lemma:4.7}
$\overline{L}_{q}(u,v,s)$ obeys the following recurrence relation.
\begin{equation*}\label{eq: 1.1}
\begin{split}
\overline{L}&_{q}(u,v,s)\\
&=\left\{
  \begin{array}{ll}
    \sum_{b=1}^{v-(s-1)}\sum_{a=1}^{u-(s-1)}q^{a(v-b)} \overline{L}_{q}(u-a,v-b,s-1), & \text{for}\ s>1,\ s\leq u\ \text{and}\ s\leq v \\
    1, & \text{for}\ s=1,\ 1\leq u\ \text{and}\ 1\leq v\\
    0, & \text{otherwise.}\\
  \end{array}
\right.
\end{split}
\end{equation*}
\end{lemma}
\begin{proof}
For $s > 1$, $s\leq u$ and $s\leq v$, we observe that $x_{s}$ may assume any value $1,\ldots,u-(s-1)$, then $\overline{L}_{q}(u,v,s)$ can be written as
\noindent
\begin{equation*}
\begin{split}
\overline{L}&_{q}(u,v,s)\\
=&\sum_{x_{s}=1}^{u-(s-1)}{\sum_{(x_1,\ldots,x_{s-1})\in S_{u-x_{s},s-1}^{0}}}{\hspace{0.3cm}\sum_{(y_{1},\ldots,y_{s}) \in S_{v,s}^{0}}}q^{x_{s}(v-y_{s})}\ q^{y_{1}x_{2}+(y_{1}+y_{2})x_{3}+\cdots+(y_{1}+\cdots+y_{s-2})x_{s-1}}
\end{split}
\end{equation*}
\noindent
Similarly, we observe that since $y_s$ can assume the values $1,\ldots,v-(s-1)$, then $\overline{L}_{q}(u,v,s)$ can be rewritten as
\noindent
\begin{equation*}
\begin{split}
\overline{L}&_{q}(u,v,s)\\
=&\sum_{y_{s}=1}^{v-(s-1)}\sum_{x_{s}=1}^{u-(s-1)}q^{x_{s}(v-y_{s})}{\sum_{(x_1,\ldots,x_{s-1})\in S_{u-x_{s},s-1}^{0}}}{\hspace{0.3cm}\sum_{(y_{1},\ldots,y_{s-1}) \in S_{v-y_{s},s-1}^{0}}}q^{y_{1}x_{2}+(y_{1}+y_{2})x_{3}+\cdots+(y_{1}+\cdots+y_{s-2})x_{s-1}}\\
=&\sum_{b=1}^{v-(s-1)}\sum_{a=1}^{u-(s-1)}q^{a(v-b)} \overline{L}_{q}(u-a,v-b,s-1).
\end{split}
\end{equation*}
\noindent
The other cases are obvious and thus the proof is completed.
\end{proof}
\begin{remark}
{\rm
We observe that $\overline{L}_{1}(m,r,s)$ is the number of integer solutions $(x_{1},\ldots,x_{s})$ and $(y_{1},\ldots,y_{s})$ of

\noindent
\begin{equation*}\label{eq:1}
\begin{split}
(x_1,\ldots,x_{s})\ \in S_{u,s}^{0},\ \text{and}
\end{split}
\end{equation*}
\noindent
\begin{equation*}\label{eq:2}
\begin{split}
(y_{1},\ldots,y_{s}) \in S_{v,s}^{0}
\end{split}
\end{equation*}
\noindent
given by
\noindent
\begin{equation*}\label{eq: 1.1}
\begin{split}
\overline{L}_{1}(u,v,s)=M(s,\ u)M(s,\ v ),
\end{split}
\end{equation*}
\noindent
where $M(a,\ b)$ denotes the total number of integer solution $y_{1}+x_{2}+\cdots+y_{a}=b$ such that $(y_{1},\ldots,y_{a}) \in S_{b,a}^{0}$. The number is given by
\noindent
\begin{equation*}\label{eq: 1.1}
\begin{split}
M(a,\ b)={b-1 \choose a-1}.
\end{split}
\end{equation*}
\noindent
See, e.g. \citet{charalambides2002enumerative}.
}
\end{remark}

\begin{definition}
For $0<q\leq1$, we define
\begin{equation*}\label{eq: 1.1}
\begin{split}
L_{q,k_1,0}(m,r,s)={\sum_{x_{1},\ldots,x_{s}}}\
\sum_{y_{1},\ldots,y_{s}}q^{y_{1}x_{2}+(y_{1}+y_{2})x_{3}+\cdots+(y_{1}+\cdots+y_{s-1})x_{s}},\
\end{split}
\end{equation*}
\noindent
where the summation is over all integers $x_1,\ldots,x_{s},$ and $y_1,\ldots,y_s$ satisfying
\noindent
\begin{equation*}\label{eq:1}
\begin{split}
(x_1,\ldots,x_{s})\ \in S_{m,s}^{k_1},\ \text{and}
\end{split}
\end{equation*}
\noindent
\begin{equation*}\label{eq:2}
\begin{split}
(y_{1},\ldots,y_{s}) \in S_{r,s}^{0}
\end{split}
\end{equation*}
\end{definition}
\noindent
The following gives a recurrence relation useful for the computation of $L_{q,k_1,0}(m,r,s)$.
\noindent
\begin{lemma}
\label{lemma:4.8}
$L_{q,k_1,0}(m,r,s)$ obeys the following recurrence relation,
\begin{equation*}\label{eq: 1.1}
\begin{split}
L&_{q,k_1,0}(m,r,s)\\
&=\left\{
  \begin{array}{ll}
    \sum_{a=1}^{k_{1}-1}\ \sum_{b=1}^{r-(s-1)}q^{a(r-b)}L_{q,k_1,0}(m-a,r-b,s-1)\\
    +\sum_{a=k_{1}}^{m-(s-1)}\ \sum_{b=1}^{r-(s-1)}q^{a(r-b)}\overline{L}_{q}(m-a,r-b,s-1), & \text{for}\quad s>1,\ (s-1)+k_1\leq m\\
    &\text{and}\quad s\leq r \\
    1, & \text{for}\quad s=1,\ k_1\leq m,\quad \text{and}\quad 1\leq r\\
    0, & \text{otherwise.}\\
  \end{array}
\right.
\end{split}
\end{equation*}
\end{lemma}
\begin{proof}
For $s > 1$, $k_1\leq m$, and $1\leq r$, we observe that $x_{s}$ may assume any value $1,\ \ldots,\ k_{1}-1$, then $L_{q,k_1,0}(m,r,s)$ can be written as

\begin{equation*}
\begin{split}
L&_{q,k_1,0}(m,r,s)\\
=&\sum_{y_{s}=1}^{r-(s-1)}{\sum_{(x_{1},\ldots,x_{s}) \in S_{m,s}^{k_1-1}}}\quad
{\sum_{(y_{1},\ldots,y_{s-1}) \in S_{r-y_s,s}^{0}}}q^{x_{s}(r-y_{s})}q^{y_{1}x_{2}+(y_{1}+y_{2})x_{3}+\cdots+(y_{1}+\cdots+y_{s-2})x_{s-1}}
\end{split}
\end{equation*}

\noindent
Similarly, we observe that since $y_s$ can assume the values $1,\ldots,r-(s-1)$, then $L_{q,k_1,0}(m,r,s)$ can be rewritten as
\noindent

\begin{equation*}
\begin{split}
L&_{q,k_1,0}(m,r,s)\\
=&\sum_{x_{s}=1}^{k_{1}-1}\sum_{y_{s}=1}^{r-(s-1)}q^{x_{s}(r-y_{s})}{\sum_{(x_{1},\ldots,x_{s-1}) \in S_{m-x_s,s}^{k_1-1}}}{\hspace{0.3cm}\sum_{(y_{1},\ldots,y_{s-1}) \in S_{r-y_{s},s-1}^{0}}} q^{y_{1}x_{2}+(y_{1}+y_{2})x_{3}+\cdots+(y_{1}+\cdots+y_{s-2})x_{s-1}}\\
&+\sum_{x_{s}=k_{1}}^{m-(s-1)}\sum_{y_{s}=1}^{r-(s-1)}q^{x_{s}(r-y_{s})}{\sum_{(x_{1},\ldots,x_{s-1}) \in S_{m-x_s,s}^{0}}}{\hspace{0.3cm}\sum_{(y_{1},\ldots,y_{s-1}) \in S_{r-y_{s},s-1}^{0}}} q^{y_{1}x_{2}+(y_{1}+y_{2})x_{3}+\cdots+(y_{1}+\cdots+y_{s-2})x_{s-1}}\\
=&\sum_{a=1}^{k_{1}-1}\ \sum_{b=1}^{r-(s-1)}q^{a(r-b)}L_{q,k_1,0}(m-a,r-b,s-1)+\sum_{a=k_{1}}^{m-(s-1)}\ \sum_{b=1}^{r-(s-1)}q^{a(r-b)}\overline{L}_{q}(m-a,r-b,s-1).
\end{split}
\end{equation*}
\noindent
The other cases are obvious and thus the proof is completed.
\end{proof}
\begin{remark}
{\rm
We observe that $L_{1,k_1,0}(m,r,s)$ is the number of integer solutions $(x_{1},\ldots,x_{s})$ and $(y_{1},\ldots,y_{s})$ of
\noindent
\noindent
\begin{equation*}\label{eq:1}
\begin{split}
(x_1,\ldots,x_{s})\ \in S_{m,s}^{k_1-1},\ \text{and}
\end{split}
\end{equation*}
\noindent
\begin{equation*}\label{eq:2}
\begin{split}
(y_{1},\ldots,y_{s}) \in S_{r,s}^{0}
\end{split}
\end{equation*}

\noindent
given by
\noindent
\begin{equation*}\label{eq: 1.1}
\begin{split}
L_{1,k_1,0}(m,r,s)=R(s,\ k_{1},\ m)M(s,\ r),
\end{split}
\end{equation*}

where $R(a,\ b,\ c)$ denotes the total number of integer solution $y_{1}+x_{2}+\cdots+y_{a}=c$ such that $(x_{1},\ldots,x_{a}) \in S_{r,a}^{k_2}$. The number is given by
\noindent
\begin{equation*}\label{eq: 1.1}
\begin{split}
R(a,\ b,\ c)=\sum_{j=1}^{min\left(a,\ \left[\frac{c-a}{b-1}\right]\right)}(-1)^{j+1}{a \choose j}{c-j(b-1)-1 \choose a-1}.
\end{split}
\end{equation*}
\noindent

where $M(a,\ b)$ denotes the total number of integer solution $y_{1}+y_{2}+\cdots+y_{a}=b$ such that $(y_{1},\ldots,y_{a}) \in S_{b,a}^{0}$. The number is given by
\noindent
\begin{equation*}\label{eq: 1.1}
\begin{split}
M(a,\ b)={b-1 \choose a-1}.
\end{split}
\end{equation*}
See, e.g. \citet{charalambides2002enumerative}.
}
\end{remark}
\noindent
The probability function of the $q$-later waiting time distribution impose a run quota on successes and a frequency quota on failures is obtained by the following theorem. It is evident that
\noindent
\begin{equation*}
P_{q,\theta}(\widehat{W}_{L}=n)=0\ \text{for}\ n< k_1+k_2
\end{equation*}
\noindent
and so we shall focus on determining the probability mass function for $n\geq k_{1}+k_{2}$.
\noindent
\begin{theorem}
\label{thm:4.4}
The PMF $P_{q,\theta}(\widetilde{W}_{L}=n)$ satisfies
\begin{equation*}\label{eq: 1.1}
\begin{split}
P(\widetilde{W}_{L}=n)=\widetilde{P}_{q,L}^{(1)}(n)+\widetilde{P}_{q,L}^{(0)}(n),\ n\geq k_1+k_2,
\end{split}
\end{equation*}
where
\begin{equation*}\label{eq:bn1}
\begin{split}
\widetilde{P}_{q,L}^{(1)}(n)=\sum_{i=k_2}^{n-k_{1}}\theta^{n-i} q^{ik_{1}}{\prod_{j=1}^{i}}(1-\theta q^{j-1})\sum_{s=1}^{i}\bigg[\overline{E}_{q,0,0}^{k_1,\infty}(n-k_{1}-i,\ i,\ s)+\overline{F}_{q,0,0}^{k_1,\infty}(n-k_{1}-i,\ i,\ s)\bigg],\\
\end{split}
\end{equation*}
and

\begin{equation*}\label{eq:bn1}
\begin{split}
\widetilde{P}_{q,L}^{(0)}(n)=\theta^{n-k_{2}}\ {\prod_{j=1}^{k_2}}\ (1-\theta q^{j-1})\ \sum_{s=1}^{k_2}\bigg[K_{q,k_1,0}(n-k_{2},\ k_2,\ s)+L_{q,k_1,0}(n-k_{2},\ k_2\ s)\bigg].
\end{split}
\end{equation*}
\end{theorem}
\begin{proof}
We start with the study of $\widetilde{P}_{q,L}^{(1)}(n)$. It is easy to see that $\widetilde{P}_{q,L}^{(1)}(k_1+k_2)={\prod_{j=1}^{k_2}}(1-\theta q^{j-1})(\theta q^{k_2})^{k_{1}}.$ From now on we assume $n > k_1+k_2,$ and we can write $\widetilde{P}_{q,L}^{(1)}(n)$ as follows
\begin{equation*}\label{eq:bn}
\begin{split}
\widetilde{P}_{q,L}^{(1)}(n)&=P_{q,\theta}\left(L_{n-k_{1}}^{(1)}< k_{1}\ \wedge\ F_{n-k_{1}}^{(0)}\geq k_{2}\ \wedge\ X_{n-k_{1}}=0\ \wedge
\ X_{n-k_{1}+1}=\cdots =X_{n}=1\right).\\
\end{split}
\end{equation*}
\noindent
We partition the event $W_{L}^{(1)}=n$ into disjoint events given by $F_{n-k_1}=i,$ for $i=k_2, \ldots,\ n-k_1.$ Adding the probabilities we have
\noindent
\begin{equation*}\label{eq:bn}
\begin{split}
\widetilde{P}_{q,L}^{(1)}(n)=\sum_{i=k_2}^{n-k_{1}}P_{q,\theta}\Big(L_{n-k_{1}}^{(1)}< k_{1}\ \wedge\ F_{n-k_{1}}=i&\ \wedge
\ X_{n-k_{1}}=0\ \wedge\\
&X_{n-k_{1}+1}=\cdots =X_{n}=1\Big).\\
\end{split}
\end{equation*}
\noindent
If the number of $0$'s in the first $n-k_1$ trials is equal to $i,$ that is, $F_{n-k_1}=i,$ then in the $(n-k_{1}+1)$-th to $n$-th trials with probability of success
\noindent
\begin{equation*}\label{eq: kk}
\begin{split}
p_{n-k_{1}+1}=\cdots=p_{n}=\theta q^{i}.
\end{split}
\end{equation*}
\noindent
We will write $E_{n,i}^{(0)}$ for the event $\left\{L_n^{(1)}< k_1\ \wedge\ X_n=0\ \wedge\ F_n=i\right\}.$ We can now rewrite as follows
\noindent
\begin{equation*}\label{eq:bn1}
\begin{split}
\widetilde{P}_{q,L}^{(1)}(n)=\sum_{i=k_2}^{n-k_{1}}&P_{q,\theta}\Big(L_{n-k_{1}}^{(1)}< k_{1}\ \wedge
\ F_{n-k_{1}}=i\ \wedge\ X_{n-k_{1}}=0\Big)\\
&\times P_{q,\theta}\Big(X_{n-k_{1}+1}=\cdots =X_{n}=1\mid F_{n-k_{1}}=i\Big)\\
=\sum_{i=k_2}^{n-k_{1}}&P_{q,\theta}\Big(E_{n-k_{1},\ i}^{(0)}\Big)\Big(\theta q^{i}\Big)^{k_1}.\\
\end{split}
\end{equation*}
\noindent
We are going to focus on the event $E_{n-k_{1},\ i}^{(0)}$. A typical element of the event $\Big\{L_{n-k_{1}}^{(1)}< k_{1}\ \wedge\ F_{n-k_{1}}=i\ \wedge\ X_{n-k_{1}}=0\Big\}$ is an ordered sequence which consists of $n-k_{1}-i$ successes and $i$ failures such that the length of the longest success run is less than $k_1$. The number of these sequences can be derived as follows. First we will distribute the $i$ failures. Let $s$ $(1\leq s \leq i)$ be the number of runs of failures in the event $E_{n-k_{1},\ i}^{(0)}$. We divide into two cases: starting with a failure run or starting with a success run. Thus, we distinguish between two types of sequences in the event $\left\{L_{n-k_{1}}^{(1)}< k_{1}\ \wedge\ F_{n-k_{1}}=i\ \wedge\ X_{n-k_{1}}=0\right\},$ respectively named $(s-1,\ s)$-type and $(s,\ s)$-type, which are defined as follows.
\noindent
\begin{equation*}\label{eq:bn}
\begin{split}
(s-1,\ s)\text{-type}\ :&\quad \overbrace{0\ldots 0}^{y_{1}}\mid\overbrace{1\ldots 1}^{x_{1}}\mid\overbrace{0\ldots 0}^{y_{2}}\mid\overbrace{1\ldots 1}^{x_{2}}\mid \ldots \mid \overbrace{0\ldots 0}^{y_{s-1}}\mid\overbrace{1\ldots 1}^{x_{s-1}}\mid\overbrace{0\ldots 0}^{y_{s}},\\
\end{split}
\end{equation*}
\noindent
with $i$ $0$'s and $n-k_{1}-i$ $1$'s, where $x_{j}$ $(j=1,\ldots,s-1)$ represents a length of run of $1$'s and $y_{j}$ $(j=1,\ldots,s)$ represents the length of a run of $0$'s. And all integers $x_{1},\ldots,x_{s-1}$, and $y_{1},\ldots ,y_{s}$ satisfy the conditions
\noindent
\begin{equation*}\label{eq:bn}
\begin{split}
0 < x_j < k_1 \mbox{\ for\ } j=1,...,s-1, \mbox{\ and\ } x_1+\cdots +x_{s-1} = n-k_1-i,
\end{split}
\end{equation*}
\noindent
\begin{equation*}\label{eq:bn}
\begin{split}
(y_{1},\ldots,y_{s})\in S_{i,s}^{0},
\end{split}
\end{equation*}
\noindent
\begin{equation*}\label{eq:bn}
\begin{split}
(s,\ s)\text{-type}\  :&\quad \overbrace{1\ldots 1}^{x_{1}}\mid\overbrace{0\ldots 0}^{y_{1}}\mid\overbrace{1\ldots 1}^{x_{2}}\mid\overbrace{0\ldots 0}^{y_{2}}\mid\overbrace{1\ldots 1}^{x_{3}}\mid \ldots \mid \overbrace{0\ldots 0}^{y_{s-1}}\mid\overbrace{1\ldots 1}^{x_{s}}\mid\overbrace{0\ldots 0}^{y_{s}},
\end{split}
\end{equation*}
\noindent
with $i$ $0$'s and $n-k_{1}-i$ $1$'s, where $x_{j}$ $(j=1,\ldots,s)$ represents a length of run of $1$'s and $y_{j}$ $(j=1,\ldots,s)$ represents the length of a run of $0$'s. Here all of $x_1,\ldots, x_{s-1},$ and $y_1,\ldots, y_s$ are integers, and they satisfy
\noindent
\begin{equation*}\label{eq:bn}
\begin{split}
0 < x_j < k_1 \mbox{\ for\ } j=1,...,s, \mbox{\ and\ } x_1+\cdots +x_{s} = n-k_1-i,
\end{split}
\end{equation*}
\noindent
\begin{equation*}\label{eq:bn}
\begin{split}
(y_{1},\ldots,y_{s})\in S_{i,s}^{0},
\end{split}
\end{equation*}
\noindent
Then the probability of the event $E_{n-k_1,i}^{(0)}$ is given by
\noindent
\begin{equation*}\label{eq: 1.1}
\begin{split}
P_{q,\theta}\Big(E_{n-k_{1},\ i}^{(0)}&\Big)=\\
\sum_{s=1}^{i}\Bigg[\bigg\{&\sum_{\substack{x_{1}+\cdots+x_{s-1}=n-k_{1}-i\\ x_{1},\ldots,x_{s-1} \in \{1,\ldots,k_{1}-1\}}}{\hspace{0.3cm}\sum_{\substack{(y_{1},\ldots,y_{s})\in S_{i,s}^{0}}}} \big(1-\theta q^{0}\big)\cdots \big(1-\theta q^{y_{1}-1}\big)\times\\
&\Big(\theta q^{y_{1}}\Big)^{x_{1}}\big(1-\theta q^{y_{1}}\big)\cdots \big(1-\theta q^{y_{1}+y_{2}-1}\big)\times\\
&\Big(\theta q^{y_{1}+y_{2}}\Big)^{x_{2}}\big(1-\theta q^{y_{1}}\big)\cdots \big(1-\theta q^{y_{1}+y_{2}+y_3-1}\big)\times \\
&\quad \quad \quad \quad\quad \quad\quad \quad \quad \quad \quad \vdots\\
&\Big(\theta q^{y_{1}+\cdots+y_{s-1}}\Big)^{x_{s-1}}
\big(1-\theta q^{y_{1}+\cdots+y_{s-1}}\big)\cdots \big(1-\theta q^{y_{1}+\cdots+y_{s}-1}\big)\bigg\}+\\
\end{split}
\end{equation*}

\noindent
\begin{equation*}\label{eq: 1.1}
\begin{split}
\quad \quad \quad \quad \quad \quad \bigg\{&\sum_{\substack{x_{1}+\cdots+x_{s}=n-k_{1}-i\\ x_{1},\ldots,x_{s} \in \{1,\ldots,k_{1}-1\}}}{\hspace{0.3cm}\sum_{\substack{(y_{1},\ldots,y_{s})\in S_{i,s}^{0}}}}\big(\theta q^{0}\big)^{x_{1}}\big(1-\theta q^{0}\big)\cdots (1-\theta q^{y_{1}-1})\times\\
&\Big(\theta q^{y_{1}}\Big)^{x_{2}}\big(1-\theta q^{y_{1}}\big)\cdots \big(1-\theta q^{y_{1}+y_{2}-1}\big)\times\\
&\Big(\theta q^{y_{1}+y_{2}}\Big)^{x_{3}}\big(1-\theta q^{y_{1}}\big)\cdots \big(1-\theta q^{y_{1}+y_{2}+y_3-1}\big)\times \\
&\quad \quad \quad \quad\quad \quad\quad \quad \quad \quad \quad \vdots\\
&\Big(\theta q^{y_{1}+\cdots+y_{s-1}}\Big)^{x_{s}}
\big(1-\theta q^{y_{1}+\cdots+y_{s-1}}\big)\cdots \big(1-\theta q^{y_{1}+\cdots+y_{s}-1}\big)\bigg\}\Bigg],\\
\end{split}
\end{equation*}
\noindent
and then using simple exponentiation algebra arguments to simplify,
\noindent
\begin{equation*}\label{eq: 1.1}
\begin{split}
&P_{q,\theta}\Big(E_{n-k_{1},\ i}^{(0)}\Big)=\\
&\theta^{n-k_{1}-i}{\prod_{j=1}^{i}}\ (1-\theta q^{j-1})\times\\
&\sum_{s=1}^{i}\Bigg[\sum_{\substack{x_{1}+\cdots+x_{s-1}=n-k_{1}-i\\ x_{1},\ldots,x_{s-1} \in \{1,\ldots,k_{1}-1\}}}{\hspace{0.3cm}\sum_{\substack{(y_{1},\ldots,y_{s})\in S_{i,s}^{0}}}}q^{y_{1}x_{1}+(y_{1}+y_{2})x_{2}+\cdots+(y_{1}+\cdots+y_{s-1})x_{s-1}}+\\
&\quad \quad \sum_{\substack{x_{1}+\cdots+x_{s}=n-k_{1}-i\\ x_{1},\ldots,x_{s} \in \{1,\ldots,k_{1}-1\}}}
{\hspace{0.3cm}\sum_{\substack{(y_{1},\ldots,y_{s})\in S_{i,s}^{0}}}}q^{y_{1}x_{2}+(y_{1}+y_{2})x_{3}+\cdots+(y_{1}+\cdots+y_{s-1})x_{s}}\Bigg]\\
\end{split}
\end{equation*}
\noindent
Using Lemma \ref{lemma:3.5} and Lemma \ref{lemma:3.7}, we can rewrite follow as.
\noindent
\begin{equation*}\label{eq: 1.1}
\begin{split}
P_{q,\theta}\Big(E_{n-k_{1},\ i}^{(0)}\Big)=\theta^{n-k_{1}-i}{\prod_{j=1}^{i}}(1-\theta q^{j-1})\sum_{s=1}^{i}\Bigg[\overline{E}&_{q,0,0}^{k_1,\infty}(n-k_{1}-i,\ i,\ s)+\overline{F}_{q,0,0}^{k_1,\infty}(n-k_{1}-i,\ i,\ s)\Bigg],\\
\end{split}
\end{equation*}\\
\noindent
where
\noindent
\begin{equation*}\label{eq: 1.1}
\begin{split}
\overline{E}_{q,0,0}^{k_1,\infty}(n&-k_{1}-i,\ i,\ s)\\
&=\sum_{\substack{x_{1}+\cdots+x_{s-1}=n-k_{1}-i\\ x_{1},\ldots,x_{s-1} \in \{1,\ldots,k_{1}-1\}}}{\hspace{0.3cm}\sum_{\substack{(y_{1},\ldots,y_{s})\in S_{i,s}^{0}}}}q^{y_{1}x_{1}+(y_{1}+y_{2})x_{2}+\cdots+(y_{1}+\cdots+y_{s-1})x_{s-1}},
\end{split}
\end{equation*}
\noindent
and
\noindent
\begin{equation*}\label{eq: 1.1}
\begin{split}
\overline{F}_{q,0,0}^{k_1,\infty}(n&-k_{1}-i,\ i,\ s)\\
&=\sum_{\substack{x_{1}+\cdots+x_{s}=n-k_{1}-i\\ x_{1},\ \ldots,\ x_{s} \in \{1,\ \ldots,\ k_{1}-1\}}}\
\sum_{\substack{(y_{1},\ldots,y_{s})\in S_{i,s}^{0}}}q^{y_{1}x_{2}+(y_{1}+y_{2})x_{3}+\cdots+(y_{1}+\cdots+y_{s-1})x_{s}}.
\end{split}
\end{equation*}
\noindent
Therefore we can compute the probability of the event $\widetilde{W}_{L}^{(1)}=n$ as follows.
\noindent
\begin{equation*}\label{eq:bn1}
\begin{split}
\widetilde{P}_{q,L}^{(1)}(n)=&\sum_{i=k_2}^{n-k_{1}}P_{q,\theta}\Big(E_{n-k_{1},\ i}^{(0)}\Big)\Big(\theta q^{i}\Big)^{k_{1}}\\
=&\sum_{i=k_2}^{n-k_{1}}\theta^{n-k_{1}-i}{\prod_{j=1}^{i}}(1-\theta q^{j-1})\sum_{s=1}^{i}\bigg[\overline{E}_{q,0,0}^{k_1,\infty}(n-k_{1}-i,\ i,\ s)\\
&\quad\quad\quad\quad\quad\quad\quad\quad\quad\quad\quad\quad\quad+\overline{F}_{q,0,0}^{k_1,\infty}(n-k_{1}-i,\ i,\ s)\bigg]\Big(\theta q^{i}\Big)^{k_{1}}.\\
\end{split}
\end{equation*}
\noindent
Finally, applying typical factorization algebra arguments, $P_{q,L}^{(1)}(n)$ can be rewritten as follows
\noindent
\begin{equation*}\label{eq:bn1}
\begin{split}
\widetilde{P}_{q,L}^{(1)}(n)=\sum_{i=k_2}^{n-k_{1}}\theta^{n-i} q^{ik_{1}}{\prod_{j=1}^{i}}(1-\theta q^{j-1})\sum_{s=1}^{i}\bigg[\overline{E}_{q,0,0}^{k_1,\infty}(n-k_{1}-i,\ i,\ s)+\overline{F}_{q,0,0}^{k_1,\infty}(n-k_{1}-i,\ i,\ s)\bigg].\\
\end{split}
\end{equation*}
\noindent
We start with the study of $\widetilde{P}_{q,L}^{(0)}(n)$. From now on we assume $n \geq k_1+k_2,$ and we can write $\widetilde{P}_{q,L}^{(0)}(n)$ as follows.
\noindent
\begin{equation*}\label{eq:bn}
\begin{split}
\widetilde{P}_{q,L}^{(0)}(n)&=P_{q,\theta}\left(L_{n}^{(1)}> k_{1}-1\ \wedge
\ X_{n}=0\ \wedge\ S_{n}=n-k_2\right).\\
\end{split}
\end{equation*}
\noindent
We are going to focus on the event $\left\{L_{n}^{(1)}\geq k_{1}\ \wedge\ X_{n}=0\ \wedge\ S_{n}=n-k_2\right\}$. A typical element of the event $\left\{L_{n}^{(1)}\geq k_{1}\ \wedge\ X_{n}=0\ \wedge\ S_{n}=n-k_2\right\}$ is an ordered sequence which consists of $n-k_{2}$ successes and $k_2$ failures such that the length of the longest success run is greater than or equal to $k_1$. The number of these sequences can be derived as follows. First we will distribute the $k_2$ failures. Let $s$ $(1\leq s \leq k_2)$ be the number of runs of failures in the event $\left\{L_{n}^{(1)}\geq k_{1}\ \wedge\ X_{n}=0\ \wedge
\ S_{n}=n-k_2\right\}$. We divide into two cases: starting with a failure run or starting with a success run. Thus, we distinguish between two types of sequences in the event  $$\left\{L_{n}^{(1)}\geq k_{1}\ \wedge\ X_{n}=0\ \wedge\ S_{n}=n-k_2\right\},$$ respectively named $(s-1,\ s)$-type and $(s,\ s)$-type, which are defined as follows.
\noindent
\begin{equation*}\label{eq:bn}
\begin{split}
(s-1,\ s)\text{-type}\ :&\quad \overbrace{0\ldots 0}^{y_{1}}\mid\overbrace{1\ldots 1}^{x_{1}}\mid\overbrace{0\ldots 0}^{y_{2}}\mid\overbrace{1\ldots 1}^{x_{2}}\mid \ldots \mid \overbrace{0\ldots 0}^{y_{s-1}}\mid\overbrace{1\ldots 1}^{x_{s-1}}\mid\overbrace{0\ldots 0}^{y_{s}},\\
\end{split}
\end{equation*}
\noindent
with $k_2$ $0$'s and $n-k_{2}$ $1$'s, where $x_{j}$ $(j=1,\ldots,s-1)$ represents a length of run of $1$'s and $y_{j}$ $(j=1,\ldots,s)$ represents the length of a run of $0$'s. And all integers $x_{1},\ldots,x_{s-1}$, and $y_{1},\ldots ,y_{s}$ satisfy the conditions
\noindent
\begin{equation*}\label{eq:bn}
\begin{split}
(x_{1},\ldots,x_{s-1})\in S_{n-k_2,s-1}^{k_1-1}, \mbox{\ and\ }
\end{split}
\end{equation*}
\noindent
\begin{equation*}\label{eq:bn}
\begin{split}
(y_{1},\ldots,y_{s})\in S_{k_2,s}^{0},
\end{split}
\end{equation*}
\noindent
\begin{equation*}\label{eq:bn}
\begin{split}
(s,\ s)\text{-type}\  :&\quad \overbrace{1\ldots 1}^{x_{1}}\mid\overbrace{0\ldots 0}^{y_{1}}\mid\overbrace{1\ldots 1}^{x_{2}}\mid\overbrace{0\ldots 0}^{y_{2}}\mid\overbrace{1\ldots 1}^{x_{3}}\mid \ldots \mid \overbrace{0\ldots 0}^{y_{s-1}}\mid\overbrace{1\ldots 1}^{x_{s}}\mid\overbrace{0\ldots 0}^{y_{s}},
\end{split}
\end{equation*}
\noindent
with $k_2$ $0$'s and $n-k_{2}$ $1$'s, where $x_{j}$ $(j=1,\ldots,s)$ represents a length of run of $1$'s and $y_{j}$ $(j=1,\ldots,s)$ represents the length of a run of $0$'s. Here all of $x_1,\ldots, x_{s},$ and $y_1,\ldots, y_s$ are integers, and they satisfy
\noindent
\begin{equation*}\label{eq:bn}
\begin{split}
(x_{1},\ldots,x_{s})\in S_{n-k_2,s}^{k_1-1}, \mbox{\ and\ }
\end{split}
\end{equation*}
\noindent
\begin{equation*}\label{eq:bn}
\begin{split}
(y_{1},\ldots,y_{s})\in S_{k_2,s}^{0},
\end{split}
\end{equation*}
\noindent
Then the probability of the event $\left\{L_{n}^{(1)}\geq k_{1}\ \wedge\ X_{n}=0\ \wedge\ S_{n}=n-k_2\right\}$ is given by
\noindent
\begin{equation*}\label{eq: 1.1}
\begin{split}
P_{q,\theta}\Big(L_{n}^{(1)}&\geq k_{1}\ \wedge
\ X_{n}=0\ \wedge\ S_{n}=n-k_2\Big)=\\
\sum_{s=1}^{k_1}\Bigg[\bigg\{&\sum_{(x_{1},\ldots,x_{s-1})\in S_{n-k_2,s-1}^{k_1-1}}\sum_{(y_{1},\ldots,y_{s})\in S_{k_2,s}^{0}} \big(1-\theta q^{0}\big)\cdots \big(1-\theta q^{y_{1}-1}\big)\times\\
&\Big(\theta q^{y_{1}}\Big)^{x_{1}}\big(1-\theta q^{y_{1}}\big)\cdots \big(1-\theta q^{y_{1}+y_{2}-1}\big)\times\\
&\Big(\theta q^{y_{1}+y_{2}}\Big)^{x_{2}}\big(1-\theta q^{y_{1}+y_2}\big)\cdots \big(1-\theta q^{y_{1}+y_{2}+y_3-1}\big)\times \\
&\quad \quad \quad \quad\quad \quad\quad \quad \quad \quad \quad \vdots\\
&\Big(\theta q^{y_{1}+\cdots+y_{s-1}}\Big)^{x_{s-1}}
\big(1-\theta q^{y_{1}+\cdots+y_{s-1}}\big)\cdots \big(1-\theta q^{y_{1}+\cdots+y_{s}-1}\big)\bigg\}+\\
\end{split}
\end{equation*}
\noindent
\begin{equation*}\label{eq: 1.1}
\begin{split}
\quad \quad \quad \quad \bigg\{&\sum_{(x_{1},\ldots,x_{s})\in S_{n-k_2,s}^{k_1-1}}\
\sum_{(y_{1},\ldots,y_{s})\in S_{k_2,s}^{0}}\big(\theta q^{0}\big)^{x_{1}}\big(1-\theta q^{0}\big)\cdots (1-\theta q^{y_{1}-1})\times\\
&\Big(\theta q^{y_{1}}\Big)^{x_{2}}\big(1-\theta q^{y_{1}}\big)\cdots \big(1-\theta q^{y_{1}+y_{2}-1}\big)\times\\
&\Big(\theta q^{y_{1}+y_{2}}\Big)^{x_{3}}\big(1-\theta q^{y_{1}+y_2}\big)\cdots \big(1-\theta q^{y_{1}+y_{2}+y_3-1}\big)\times \\
&\quad \quad \quad \quad\quad \quad\quad \quad \quad \quad \quad \vdots\\
&\Big(\theta q^{y_{1}+\cdots+y_{s-1}}\Big)^{x_{s}}
\big(1-\theta q^{y_{1}+\cdots+y_{s-1}}\big)\cdots \big(1-\theta q^{y_{1}+\cdots+y_{s}-1}\big)\bigg\}\Bigg].\\
\end{split}
\end{equation*}
\noindent
Using simple exponentiation algebra arguments to simplify,
\noindent
\begin{equation*}\label{eq: 1.1}
\begin{split}
&P_{q,\theta}\Big(L_{n}^{(1)}\geq k_{1}\ \wedge\ X_{n}=0\ \wedge\ S_{n}=n-k_2\Big)=\\
&\theta^{n-k_{2}}{\prod_{j=1}^{k_2}}\ (1-\theta q^{j-1})\times\\
&\sum_{s=1}^{k_2}\Bigg[\sum_{(x_{1},\ldots,x_{s-1})\in S_{n-k_2,s-1}^{k_1-1}}\
\sum_{(y_{1},\ldots,y_{s})\in S_{k_2,s}^{0}}q^{y_{1}x_{1}+(y_{1}+y_{2})x_{2}+\cdots+(y_{1}+\cdots+y_{s-1})x_{s-1}}+\\
&\quad \quad\sum_{(x_{1},\ldots,x_{s})\in S_{n-k_2,s}^{k_1-1}}\
\sum_{(y_{1},\ldots,y_{s})\in S_{k_2,s}^{0}}q^{y_{1}x_{2}+(y_{1}+y_{2})x_{3}+\cdots+(y_{1}+\cdots+y_{s-1})x_{s}}\Bigg].\\
\end{split}
\end{equation*}
\noindent
Using Lemma \ref{lemma:4.6} and Lemma \ref{lemma:4.8}, we can rewrite as follows.
\noindent
\begin{equation*}\label{eq: 1.1}
\begin{split}
P_{q,\theta}\Big(&L_{n}^{(1)}\geq k_{1} \wedge\ X_{n}=0\ \wedge
\ S_{n}=n-k_2\Big)\\
&=\theta^{n-k_{2}}\ {\prod_{j=1}^{k_2}}\ (1-\theta q^{j-1})\ \sum_{s=1}^{k_2}\bigg[K_{q,k_1,0}(n-k_{2},\ k_2,\ s)+L_{q,k_1,0}(n-k_{2},\ k_2\ s)\bigg],\\
\end{split}
\end{equation*}\\
\noindent
where
\noindent
\begin{equation*}\label{eq: 1.1}
\begin{split}
K_{q,k_1,0}(n-k_{2}&,\ k_2,\ s)\\
&=\sum_{(x_{1},\ldots,x_{s-1})\in S_{n-k_2,s-1}^{k_1-1}}{\hspace{0.3cm}\sum_{(y_{1},\ldots,y_{s})\in S_{k_2,s}^{0}}}q^{y_{1}x_{1}+(y_{1}+y_{2})x_{2}+\cdots+(y_{1}+\cdots+y_{s-1})x_{s-1}},
\end{split}
\end{equation*}
\noindent
and
\noindent
\begin{equation*}\label{eq: 1.1}
\begin{split}
L_{q,k_1,0}(n-k_{2}&,\ k_2,\ s)\\
&=\sum_{(x_{1},\ldots,x_{s})\in S_{n-k_2,s}^{k_1-1}}{\hspace{0.3cm}\sum_{(y_{1},\ldots,y_{s})\in S_{k_2,s}^{0}}}q^{y_{1}x_{2}+(y_{1}+y_{2})x_{3}+\cdots+(y_{1}+\cdots+y_{s-1})x_{s}}.
\end{split}
\end{equation*}
\noindent
Therefore we can compute the probability of the event $\widetilde{W}_{L}^{(0)}=n$ as follows.
\noindent
\begin{equation*}\label{eq:bn1}
\begin{split}
\widetilde{P}_{q,L}^{(0)}(n)=\theta^{n-k_{2}}\ {\prod_{j=1}^{k_2}}\ (1-\theta q^{j-1})\ \sum_{s=1}^{k_2}\bigg[K_{q,k_1,0}(n-k_{2},\ k_2,\ s)+L_{q,k_1,0}(n-k_{2},\ k_2\ s)\bigg].\\
\end{split}
\end{equation*}
\noindent
Thus proof is completed.
\end{proof}
\noindent
It is worth mentioning here that the PMF $f_{q,L}(n;\theta)$ approaches the probability function of the later waiting time distribution of order $(k_1,k_2)$
in the limit as $q$ tends to 1 when a run quota on runs of successes and a frequency quota on runs of failures are imposed of IID model. The details are presented in the following remark.
\noindent
\begin{remark}
{\rm
For $q=1$, the PMF $\widetilde{f}_{q,L}(n;\theta)$ reduces to the PMF $\widetilde{f}_{L}(n;\theta)=P_{\theta}(\widetilde{W}_{L}=n)$ for $n\geq k_1+k_2$ is given by

\begin{equation*}\label{eq: 1.1}
\begin{split}
P(\widetilde{W}_{L}=n)=\widetilde{P}_{L}^{(1)}(n)+\widetilde{P}_{L}^{(0)}(n),\ n\geq k_1+k_2,
\end{split}
\end{equation*}
where
\begin{equation*}\label{eq:bn1}
\begin{split}
\widetilde{P}_{L}^{(1)}(n)=\sum_{i=k_2}^{n-k_{1}}\theta^{n-i}(1-\theta)^{i}\sum_{s=1}^{i}\Big[S(s-1,\ k_1,\ n-k_{1}-i)+S(s,\ k_1,\ n-k_{1}-i)\Big]M(s,\ i),\\
\end{split}
\end{equation*}
and

\begin{equation*}\label{eq:bn1}
\begin{split}
\widetilde{P}_{L}^{(0)}(n)=\theta^{n-k_{2}}(1-\theta)^{k_2}\sum_{s=1}^{k_2}\bigg[R(s-1,\ k_1,\ n-k_{2})+R(s,\ k_1,\ n-k_{2})\bigg]M(s,\ k_2).
\end{split}
\end{equation*}

where

\noindent
\begin{equation*}\label{eq: 1.1}
\begin{split}
R(a,\ b,\ c)=\sum_{j=1}^{min\left(a,\ \left[\frac{c-a}{b-1}\right]\right)}(-1)^{j+1}{a \choose j}{c-j(b-1)-1 \choose a-1},
\end{split}
\end{equation*}
\noindent
\noindent
\begin{equation*}\label{eq: 1.1}
\begin{split}
S(a,\ b,\ c)=\sum_{j=0}^{min(a,\left[\frac{c-a}{b-1}\right])}(-1)^{j}{a \choose j}{c-j(b-1)-1 \choose a-1},
\end{split}
\end{equation*}
\noindent
and

\noindent
\begin{equation*}\label{eq: 1.1}
\begin{split}
M(a,\ b)={b-1 \choose a-1}.
\end{split}
\end{equation*}
\noindent
See, e.g. \citet{charalambides2002enumerative}.
}
\end{remark}
\section{Frequency quota for $q$-binary trials}

{\rm
In the present section we shall study of the $q$-sooner and later waiting times to be discussed will arise by setting a frequency quota of successes and failures. We consider sooner cases (or later cases) impose a frequency quotas on successes and failures.
}
\subsection{Sooner waiting time}

{\rm
The problem of waiting time that will be discussed in this section is one of the 'sooner cases' and it emerges when a frequency quota on successes and failures are imposed. More specifically, Binary (zero and one) trials with probability of ones varying according to a geometric rule, are performed sequentially until $k_1$ successes in total or $k_2$ failures in total are observed, whichever event occurs first. Let $\overline{W}_{S}$ be a random variable denoting that the waiting time until either $k_{1}$ successes in total or $k_{2}$ failures in total us occurred, whichever event observe sooner. In the following theorem we derive the PMF of the random variable $\overline{W}_{S}$.

\noindent
The probability function of the $q$-sooner waiting time distribution impose a frequency quotas on successes and failures is obtained by the following theorem. It is evident that
\noindent
\begin{equation*}
P_{q,\theta}(\widehat{W}_{S}=n)=0\ \text{for}\ 0\leq n<\text{min}(k_1,k_2)
\end{equation*}
\noindent
and so we shall focus on determining the probability mass function for $n=\text{min}(k_1,k_2),\ldots, k_1+k_2-1.$
\noindent
}

\begin{theorem}
\label{thm:5.1}
The PMF $P_{q,\theta}(\overline{W}_{S}=n)$ satisfies
\begin{equation*}\label{eq: 1.1}
\begin{split}
P_{q,\theta}(\overline{W}_{S}=n)=\overline{P}_{q,S}^{(1)}(n)+\overline{P}_{q,S}^{(0)}(n)\ \text{for}\ n=\text{min}(k_1,k_2),\ldots, k_1+k_2-1,
\end{split}
\end{equation*}
where

\begin{equation*}\label{eq:bn1}
\begin{split}
\overline{P}_{q,S}^{(1)}(n)=\theta^{k_{1}}\ {\prod_{j=1}^{n-k_1}}\ (1-\theta q^{j-1})\ \sum_{s=1}^{k_1}\bigg[\overline{I}_{q}(k_{1},\ n-k_1,\ s)+\overline{J}_{q}(k_{1},\ n-k_1,\ s)\bigg].\\
\end{split}
\end{equation*}
\noindent

and

\begin{equation*}\label{eq:bn1}
\begin{split}
\overline{P}_{q,S}^{(0)}(n)=\theta^{n-k_{2}}\ {\prod_{j=1}^{k_2}}\ (1-\theta q^{j-1})\ \sum_{s=1}^{k_2}\bigg[\overline{K}_{q}(n-k_{2},\ k_2,\ s)+\overline{L}_{q}(n-k_{2},\ k_2\ s)\bigg].\\
\end{split}
\end{equation*}
\end{theorem}

\begin{proof}
We start with the study of $\overline{P}_{q,S}^{(1)}(n)$. From now on we assume $n=\text{min}(k_1,k_2),\ldots, k_1+k_2-1,$ and we can write $\overline{P}_{q,S}^{(1)}(n)$ as follows.
\noindent
\begin{equation*}\label{eq:bn}
\begin{split}
\overline{P}_{q,S}^{(1)}(n)&=P_{q,\theta}\Big(S_{n}=k_{1}\ \wedge\ F_{n}=n-k_{1}\ \wedge\ X_{n}=1\Big).\\
\end{split}
\end{equation*}

\noindent
We are going to focus on the event $\left\{ S_{n}=k_{1}\ \wedge\ F_{n}=n-k_{1}\ \wedge\ X_{n}=1 \right\}$. A typical element of the event $\left\{ S_{n}=k_{1}\ \wedge\ F_{n}=n-k_{1}\ \wedge\ X_{n}=1 \right\}$ is an ordered sequence which consists of $k_{1}$ successes and $n-k_1$ failures. The number of these sequences can be derived as follows. First we will distribute the $k_1$ successes. Let $s$ $(1\leq s \leq k_1)$ be the number of runs of successes in the event $\left\{S_{n}=k_{1}\ \wedge\ F_{n}=n-k_{1}\ \wedge\ X_{n}=1\right\}$. We divide into two cases: starting with a success run or starting with a failure run. Thus, we distinguish between two types of sequences in the event  $$\Big\{S_{n}=k_{1}\ \wedge
\ F_{n}=n-k_{1}\ \wedge\ X_{n}=1\Big\},$$ respectively named $(s,\ s-1)$-type and $(s,\ s)$-type, which are defined as follows.
\noindent

\begin{equation*}\label{eq:bn}
\begin{split}
(s,\ s-1)\text{-type}\  :&\quad \overbrace{1\ldots 1}^{x_{1}}\mid\overbrace{0\ldots 0}^{y_{1}}\mid\overbrace{1\ldots 1}^{x_{2}}\mid\overbrace{0\ldots 0}^{y_{2}}\mid \ldots\mid\overbrace{1\ldots 1}^{x_{s-1}} \mid \overbrace{0\ldots 0}^{y_{s-1}}\mid\overbrace{1\ldots 1}^{x_{s}},
\end{split}
\end{equation*}
\noindent
with $n-k_1$ $0$'s and $k_{1}$ $1$'s, where $x_{j}$ $(j=1,\ldots,s)$ represents a length of run of $1$'s and $y_{j}$ $(j=1,\ldots,s-1)$ represents the length of a run of $0$'s. Here all of $x_1,\ldots, x_{s},$ and $y_1,\ldots, y_{s-1}$ are integers, and they satisfy
\noindent
\begin{equation*}\label{eq:bn}
\begin{split}
(x_{1},\ldots,x_{s})\in S_{k_1,s}^{0}, \text{and}
\end{split}
\end{equation*}
\noindent
\noindent
\begin{equation*}\label{eq:bn}
\begin{split}
(y_{1},\ldots,y_{s-1})\in S_{n-k_1,s-1}^{0},
\end{split}
\end{equation*}

\noindent

\begin{equation*}\label{eq:bn}
\begin{split}
(s,\ s)\text{-type}\ :&\quad \overbrace{0\ldots 0}^{y_{1}}\mid\overbrace{1\ldots 1}^{x_{1}}\mid\overbrace{0\ldots 0}^{y_{2}}\mid\overbrace{1\ldots 1}^{x_{2}}\mid \ldots \mid \overbrace{0\ldots 0}^{y_{s-1}}\mid\overbrace{1\ldots 1}^{x_{s-1}}\mid\overbrace{0\ldots 0}^{y_{s}}\mid\overbrace{1\ldots 1}^{x_{s}},\\
\end{split}
\end{equation*}
\noindent
with $n-k_1$ $0$'s and $k_{1}$ $1$'s, where $x_{j}$ $(j=1,\ldots,s)$ represents a length of run of $1$'s and $y_{j}$ $(j=1,\ldots,s)$ represents the length of a run of $0$'s. And all integers $x_{1},\ldots,x_{s}$, and $y_{1},\ldots ,y_{s}$ satisfy the conditions
\noindent
\begin{equation*}\label{eq:bn}
\begin{split}
(x_{1},\ldots,x_{s})\in S_{k_1,s}^{0}, \text{and}
\end{split}
\end{equation*}
\noindent
\begin{equation*}\label{eq:bn}
\begin{split}
(y_{1},\ldots,y_{s})\in S_{n-k_1,s}^{0},
\end{split}
\end{equation*}

\noindent
Then the probability of the event $\Big\{ S_{n}=k_{1}\ \wedge
\ F_{n}=n-k_{1}\ \wedge\ X_{n}=1 \Big\}$ is given by
\noindent
\begin{equation*}\label{eq: 1.1}
\begin{split}
P_{q,\theta}\Big(S_{n}=&k_{1}\ \wedge\ F_{n}=n-k_{1}\ \wedge\ X_{n}=1\Big)=\\
\sum_{s=1}^{k_1}\Bigg[\bigg\{&\sum_{(x_{1},\ldots,x_{s}) \in S_{k_{1},s}^{0}}{\hspace{0.3cm}\sum_{(y_{1},\ldots,y_{s-1})\in S_{n-k_1,s-1}^{0}}}\big(\theta q^{0}\big)^{x_{1}}\big(1-\theta q^{0}\big)\cdots (1-\theta q^{y_{1}-1})\times\\
&\Big(\theta q^{y_{1}}\Big)^{x_{2}}\big(1-\theta q^{y_{1}}\big)\cdots \big(1-\theta q^{y_{1}+y_{2}-1}\big)\times\\
&\quad \quad \quad \quad\quad \quad\quad \quad \quad \quad \quad \vdots\\
&\Big(\theta q^{y_{1}+\cdots+y_{s-2}}\Big)^{x_{s-1}}
\big(1-\theta q^{y_{1}+\cdots+y_{s-2}}\big)\cdots \big(1-\theta q^{y_{1}+\cdots+y_{s-1}-1}\big)\times\\
&\Big(\theta q^{y_{1}+\cdots+y_{s-1}}\Big)^{x_{s}}\bigg\}+\\
\end{split}
\end{equation*}

\noindent
\begin{equation*}\label{eq: 1.1}
\begin{split}
\quad \bigg\{&\sum_{\substack{(x_{1},\ldots,x_{s}) \in S_{k_{1},s}^{0}}}{\hspace{0.3cm}\sum_{(y_{1},\ldots,y_{s})\in S_{n-k_1,s}^{0}}} \big(1-\theta q^{0}\big)\cdots \big(1-\theta q^{y_{1}-1}\big)\Big(\theta q^{y_{1}}\Big)^{x_{1}}\times\\
&\big(1-\theta q^{y_{1}}\big)\cdots \big(1-\theta q^{y_{1}+y_{2}-1}\big)\Big(\theta q^{y_{1}+y_{2}}\Big)^{x_{2}}\times\\
&\quad \quad \quad \quad\quad \quad\quad \quad \quad \quad \quad \vdots\\
&\big(1-\theta q^{y_{1}+\cdots+y_{s-1}}\big)\cdots \big(1-\theta q^{y_{1}+\cdots+y_{s}-1}\big)\Big(\theta q^{y_{1}+\cdots+y_{s}}\Big)^{x_{s}}\bigg\}\Bigg].\\
\end{split}
\end{equation*}
\noindent
Using simple exponentiation algebra arguments to simplify,
\noindent
\begin{equation*}\label{eq: 1.1}
\begin{split}
&P_{q,\theta}\Big(S_{n}=k_{1}\ \wedge
\ F_{n}=n-k_{1}\ \wedge\ X_{n}=1\Big)=\\
&\theta^{k_{1}}{\prod_{j=1}^{n-k_1}}\ (1-\theta q^{j-1})\times\\
&\sum_{s=1}^{k_1}\Bigg[\sum_{\substack{(x_{1},\ldots,x_{s}) \in S_{k_{1},s}^{0}}}{\hspace{0.3cm}\sum_{(y_{1},\ldots,y_{s-1})\in S_{n-k_1,s-1}^{0}}}q^{y_{1}x_{2}+(y_{1}+y_{2})x_{3}+\cdots+(y_{1}+\cdots+y_{s-1})x_{s}}+\\
&\quad \quad \sum_{\substack{(x_{1},\ldots,x_{s}) \in S_{k_{1},s}^{0}}}{\hspace{0.3cm}\sum_{(y_{1},\ldots,y_{s})\in S_{n-k_1,s}^{0}}}q^{y_{1}x_{1}+(y_{1}+y_{2})x_{2}+\cdots+(y_{1}+\cdots+y_{s})x_{s}}\Bigg].\\
\end{split}
\end{equation*}
\noindent
Using Lemma \ref{lemma:4.1} and Lemma \ref{lemma:4.3}, we can rewrite as follows.
\noindent
\begin{equation*}\label{eq: 1.1}
\begin{split}
P_{q,\theta}\Big(&S_{n}=k_{1}\ \wedge\ F_{n}=n-k_{1}\ \wedge\ X_{n}=1\Big)\\
&=\theta^{k_{1}}\ {\prod_{j=1}^{n-k_1}}\ (1-\theta q^{j-1})\ \sum_{s=1}^{k_1}\bigg[\overline{I}_{q}(k_{1},\ n-k_1,\ s)+\overline{J}_{q}(k_{1},\ n-k_1,\ s)\bigg],\\
\end{split}
\end{equation*}\\
\noindent
where
\noindent
\begin{equation*}\label{eq: 1.1}
\begin{split}
\overline{I}_{q}(k_{1}&,\ n-k_1,\ s)\\
&=\sum_{\substack{ (x_{1},\ldots,x_{s}) \in S_{k_{1},s}^{0}}}{\hspace{0.3cm}\sum_{(y_{1},\ldots,y_{s-1})\in S_{n-k_1,s-1}^{0}}}q^{y_{1}x_{2}+(y_{1}+y_{2})x_{3}+\cdots+(y_{1}+\cdots+y_{s-1})x_{s}},
\end{split}
\end{equation*}
and
\begin{equation*}\label{eq: 1.1}
\begin{split}
\overline{J}_{q}(k_{1}&,\ n-k_1,\ s)\\
&=\sum_{\substack{x_{1},\ldots,x_{s} \in S_{k_{1},s}^{0}}}{\hspace{0.3cm}\sum_{(y_{1},\ldots,y_{s})\in S_{n-k_1,s}^{0}}}q^{y_{1}x_{1}+(y_{1}+y_{2})x_{2}+\cdots+(y_{1}+\cdots+y_{s})x_{s}}.
\end{split}
\end{equation*}
\noindent
Therefore we can compute the probability of the event $\overline{W}_{S}^{(1)}=n$ as follows.
\noindent
\begin{equation*}\label{eq:bn1}
\begin{split}
\overline{P}_{q,S}^{(1)}(n)=\theta^{k_{1}}\ {\prod_{j=1}^{n-k_1}}\ (1-\theta q^{j-1})\ \sum_{s=1}^{k_1}\bigg[\overline{I}_{q}(k_{1},\ n-k_1,\ s)+\overline{J}_{q}(k_{1},\ n-k_1,\ s)\bigg].\\
\end{split}
\end{equation*}
\noindent


We start with the study of $\overline{P}_{q,S}^{(0)}(n)$. From now on we assume $n=\text{min}(k_1,k_2),\ldots, k_1+k_2-1,$ and we can write $\overline{P}_{q,S}^{(0)}(n)$ as follows.
\noindent
\begin{equation*}\label{eq:bn}
\begin{split}
\overline{P}_{q,S}^{(0)}(n)&=P_{q,\theta}\Big(S_{n}=n-k_2\ \wedge\ F_{n}=k_2\ \wedge\ X_{n}=0\Big).\\
\end{split}
\end{equation*}

\noindent
We are going to focus on the event $\left\{S_{n}=n-k_2\ \wedge\ F_{n}=k_2\ \wedge\ X_{n}=0\right\}$. A typical element of the event $\left\{S_{n}=n-k_2\ \wedge\ F_{n}=k_2\ \wedge\ X_{n}=0\right\}$ is an ordered sequence which consists of $n-k_{2}$ successes and $k_2$ failures. The number of these sequences can be derived as follows. First we will distribute the $k_2$ failures. Let $s$ $(1\leq s \leq k_2)$ be the number of runs of failures in the event $\Big\{S_{n}=n-k_2\ \wedge\ F_{n}=k_2\ \wedge\ X_{n}=0\Big\}$. We divide into two cases: starting with a success run or starting with a failure run. Thus, we distinguish between two types of sequences in the event $\Big\{S_{n}=n-k_2\ \wedge\ F_{n}=k_2\ \wedge\ X_{n}=0\Big\},$ respectively named $(s-1,\ s)$-type and $(s,\ s)$-type, which are defined as follows.

\noindent
\begin{equation*}\label{eq:bn}
\begin{split}
(s-1,\ s)\text{-type}\ :&\quad \overbrace{0\ldots 0}^{y_{1}}\mid\overbrace{1\ldots 1}^{x_{1}}\mid\overbrace{0\ldots 0}^{y_{2}}\mid\overbrace{1\ldots 1}^{x_{2}}\mid \ldots \mid \overbrace{0\ldots 0}^{y_{s-1}}\mid\overbrace{1\ldots 1}^{x_{s-1}}\mid\overbrace{0\ldots 0}^{y_{s}},\\
\end{split}
\end{equation*}
\noindent
with $k_2$ $0$'s and $n-k_{2}$ $1$'s, where $x_{j}$ $(j=1,\ldots,s-1)$ represents a length of run of $1$'s and $y_{j}$ $(j=1,\ldots,s)$ represents the length of a run of $0$'s. And all integers $x_{1},\ldots,x_{s-1}$, and $y_{1},\ldots,y_{s}$ satisfy the conditions
\noindent
\begin{equation*}\label{eq:bn}
\begin{split}
(x_{1},\ldots,x_{s-1})\in S_{n-k_2,s-1}^{0}, \mbox{\ and\ }
\end{split}
\end{equation*}
\noindent
\begin{equation*}\label{eq:bn}
\begin{split}
(y_{1},\ldots,y_{s})\in S_{k_2,s}^{0},
\end{split}
\end{equation*}
\noindent
\begin{equation*}\label{eq:bn}
\begin{split}
(s,\ s)\text{-type}\  :&\quad \overbrace{1\ldots 1}^{x_{1}}\mid\overbrace{0\ldots 0}^{y_{1}}\mid\overbrace{1\ldots 1}^{x_{2}}\mid\overbrace{0\ldots 0}^{y_{2}}\mid\overbrace{1\ldots 1}^{x_{3}}\mid \ldots \mid \overbrace{0\ldots 0}^{y_{s-1}}\mid\overbrace{1\ldots 1}^{x_{s}}\mid\overbrace{0\ldots 0}^{y_{s}},
\end{split}
\end{equation*}
\noindent
with $k_2$ $0$'s and $n-k_{2}$ $1$'s, where $x_{j}$ $(j=1,\ldots,s)$ represents a length of run of $1$'s and $y_{j}$ $(j=1,\ldots,s)$ represents the length of a run of $0$'s. Here all of $x_1,\ldots, x_{s},$ and $y_1,\ldots, y_s$ are integers, and they satisfy
\noindent
\begin{equation*}\label{eq:bn}
\begin{split}
(x_{1},\ldots,x_{s})\in S_{n-k_2,s}^{0}, \mbox{\ and\ }
\end{split}
\end{equation*}
\noindent
\begin{equation*}\label{eq:bn}
\begin{split}
(y_{1},\ldots,y_{s})\in S_{k_2,s}^{0},
\end{split}
\end{equation*}

\noindent
Then the probability of the event $\Big\{S_{n}=n-k_2\ \wedge\ F_{n}=k_2\ \wedge\ X_{n}=0\Big\}$ is given by
\noindent
\begin{equation*}\label{eq: 1.1}
\begin{split}
P_{q,\theta}\Big(S_{n}=&n-k_2\ \wedge\ F_{n}=k_2\ \wedge\ X_{n}=0\Big)=\\
\sum_{s=1}^{k_1}\Bigg[\bigg\{&\sum_{(x_{1},\ldots,x_{s-1})\in S_{n-k_2,s-1}^{0}}{\hspace{0.3cm}\sum_{(y_{1},\ldots,y_{s})\in S_{k_2,s}^{0}}} \big(1-\theta q^{0}\big)\cdots \big(1-\theta q^{y_{1}-1}\big)\times\\
&\Big(\theta q^{y_{1}}\Big)^{x_{1}}\big(1-\theta q^{y_{1}}\big)\cdots \big(1-\theta q^{y_{1}+y_{2}-1}\big)\times\\
&\Big(\theta q^{y_{1}+y_{2}}\Big)^{x_{2}}\big(1-\theta q^{y_{1}+y_2}\big)\cdots \big(1-\theta q^{y_{1}+y_{2}+y_3-1}\big)\times \\
&\quad \quad \quad \quad\quad \quad\quad \quad \quad \quad \quad \vdots\\
&\Big(\theta q^{y_{1}+\cdots+y_{s-1}}\Big)^{x_{s-1}}
\big(1-\theta q^{y_{1}+\cdots+y_{s-1}}\big)\cdots \big(1-\theta q^{y_{1}+\cdots+y_{s}-1}\big)\bigg\}+\\
\end{split}
\end{equation*}

\noindent
\begin{equation*}\label{eq: 1.1}
\begin{split}
\quad \quad \quad \quad \bigg\{&\sum_{(x_{1},\ldots,x_{s})\in S_{n-k_2,s}^{0}}{\hspace{0.3cm}\sum_{(y_{1},\ldots,y_{s})\in S_{k_2,s}^{0}}}\big(\theta q^{0}\big)^{x_{1}}\big(1-\theta q^{0}\big)\cdots (1-\theta q^{y_{1}-1})\times\\
&\Big(\theta q^{y_{1}}\Big)^{x_{2}}\big(1-\theta q^{y_{1}}\big)\cdots \big(1-\theta q^{y_{1}+y_{2}-1}\big)\times\\
&\Big(\theta q^{y_{1}+y_{2}}\Big)^{x_{3}}\big(1-\theta q^{y_{1}+y_2}\big)\cdots \big(1-\theta q^{y_{1}+y_{2}+y_3-1}\big)\times \\
&\quad \quad \quad \quad\quad \quad\quad \quad \quad \quad \quad \vdots\\
&\Big(\theta q^{y_{1}+\cdots+y_{s-1}}\Big)^{x_{s}}
\big(1-\theta q^{y_{1}+\cdots+y_{s-1}}\big)\cdots \big(1-\theta q^{y_{1}+\cdots+y_{s}-1}\big)\bigg\}\Bigg].\\
\end{split}
\end{equation*}
\noindent
Using simple exponentiation algebra arguments to simplify,
\noindent
\begin{equation*}\label{eq: 1.1}
\begin{split}
&P_{q,\theta}\Big(S_{n}=n-k_2\ \wedge
\ F_{n}=k_2\ \wedge\ X_{n}=0\Big)=\\
&\theta^{n-k_{2}}{\prod_{j=1}^{k_2}}\ (1-\theta q^{j-1})\times\\
&\sum_{s=1}^{k_2}\Bigg[\sum_{(x_{1},\ldots,x_{s-1})\in S_{n-k_2,s-1}^{0}}{\hspace{0.3cm}\sum_{(y_{1},\ldots,y_{s})\in S_{k_2,s}^{0}}}q^{y_{1}x_{1}+(y_{1}+y_{2})x_{2}+\cdots+(y_{1}+\cdots+y_{s-1})x_{s-1}}+\\
&\quad \quad\sum_{(x_{1},\ldots,x_{s})\in S_{n-k_2,s}^{0}}{\hspace{0.3cm}\sum_{(y_{1},\ldots,y_{s})\in S_{k_2,s}^{0}}}q^{y_{1}x_{2}+(y_{1}+y_{2})x_{3}+\cdots+(y_{1}+\cdots+y_{s-1})x_{s}}\Bigg].\\
\end{split}
\end{equation*}
\noindent
Using Lemma \ref{lemma:4.5} and Lemma \ref{lemma:4.7}, we can rewrite as follows.
\noindent
\begin{equation*}\label{eq: 1.1}
\begin{split}
P_{q,\theta}\Big(&S_{n}=n-k_2\ \wedge\ F_{n}=k_2\ \wedge\ X_{n}=0\Big)\\
&=\theta^{n-k_{2}}\ {\prod_{j=1}^{k_2}}\ (1-\theta q^{j-1})\ \sum_{s=1}^{k_2}\bigg[\overline{K}_{q}(n-k_{2},\ k_2,\ s)+\overline{L}_{q}(n-k_{2},\ k_2\ s)\bigg],\\
\end{split}
\end{equation*}\\
\noindent

where
\noindent
\begin{equation*}\label{eq: 1.1}
\begin{split}
\overline{K}_{q}(n-k_{2}&,\ k_2,\ s)\\
&=\sum_{(x_{1},\ldots,x_{s-1})\in S_{n-k_2,s-1}^{0}}{\hspace{0.3cm}\sum_{(y_{1},\ldots,y_{s})\in S_{k_2,s}^{0}}}q^{y_{1}x_{1}+(y_{1}+y_{2})x_{2}+\cdots+(y_{1}+\cdots+y_{s-1})x_{s-1}},
\end{split}
\end{equation*}
\noindent
and
\noindent
\begin{equation*}\label{eq: 1.1}
\begin{split}
\overline{L}_{q}(n-k_{2}&,\ k_2,\ s)\\
&=\sum_{(x_{1},\ldots,x_{s})\in S_{n-k_2,s}^{0}}{\hspace{0.3cm}\sum_{(y_{1},\ldots,y_{s})\in S_{k_2,s}^{0}}}q^{y_{1}x_{2}+(y_{1}+y_{2})x_{3}+\cdots+(y_{1}+\cdots+y_{s-1})x_{s}}.
\end{split}
\end{equation*}
\noindent
Therefore we can compute the probability of the event $\Big\{\overline{W}_{L}^{(0)}=n\Big\}$ as follows.
\noindent
\begin{equation*}\label{eq:bn1}
\begin{split}
\overline{P}_{q,S}^{(0)}(n)=\theta^{n-k_{2}}\ {\prod_{j=1}^{k_2}}\ (1-\theta q^{j-1})\ \sum_{s=1}^{k_2}\bigg[\overline{K}_{q}(n-k_{2},\ k_2,\ s)+\overline{L}_{q}(n-k_{2},\ k_2\ s)\bigg].\\
\end{split}
\end{equation*}
\noindent
Thus proof is completed.
\end{proof}
\noindent
It is worth mentioning here that the PMF $\overline{f}_{q,S}(n;\theta)$ approaches the probability function of the sooner waiting time distribution of order $(k_1,k_2)$
in the limit as $q$ tends to 1 when a frequency quota is imposed on runs of successes and failures of IID model. The details are presented in the following remark.
\noindent
\begin{remark}
{\rm
For $q=1$, the PMF $\overline{f}_{q,S}(n;\theta)$ reduces to the PMF $\overline{f}_{S}(n;\theta)=P_{\theta}(\overline{W}_{S}=n)$ for $n=\text{min}(k_1,k_2),\ldots, k_1+k_2-1$ is given by

\begin{equation*}\label{eq: 1.1}
\begin{split}
P(\overline{W}_{S}=n)=\overline{P}_{S}^{(1)}(n)+\overline{P}_{S}^{(0)}(n),
\end{split}
\end{equation*}
where
\begin{equation*}\label{eq:bn1}
\begin{split}
\overline{P}_{S}^{(1)}(n)=\theta^{k_{1}}\ (1-\theta)^{n-k_1}\ \sum_{s=1}^{k_1}M(s,\ k_1)\bigg[M(s-1,\ n-k_1)+M(s,\ n-k_1)\bigg].\\
\end{split}
\end{equation*}
\noindent

and

\begin{equation*}\label{eq:bn1}
\begin{split}
\overline{P}_{S}^{(0)}(n)=\theta^{n-k_{2}}(1-\theta)^{k_2}\sum_{s=1}^{k_2}\bigg[M(s-1,\ n-k_{2})+M(s,\ n-k_{2})\bigg]M(s,\ k_2),
\end{split}
\end{equation*}
\noindent
where
\noindent
\begin{equation*}\label{eq: 1.1}
\begin{split}
M(a,\ b)={b-1 \choose a-1}.
\end{split}
\end{equation*}
See, e.g. \citet{charalambides2002enumerative}.
}
\end{remark}

\begin{theorem}
The PMF $P_{q,\theta}(\overline{W}_{S}=n)$ for $n=\text{min}(k_1,k_2),\ldots, k_1+k_2-1$ is given by
\begin{equation*}\label{}
\begin{split}
P_{q,\theta}(\overline{W}_{S}= n)&=\sum_{x=n-k_2}^{k_1-1}{n-1\brack x }_{q}
\theta^{x}\prod_{i=1}^{n-1-x}(1-\theta q^{i-1})-\sum_{x=n+1-k_2}^{k_1-1}{n\brack x }_{q}
\theta^{x}\prod_{i=1}^{n-x}(1-\theta q^{i-1}).
\end{split}
\end{equation*}
\end{theorem}
\begin{proof}
It follows from the definition of $\overline{W}_{S}$ that
\begin{equation*}\label{}
\begin{split}
P_{q,\theta}(\overline{W}_{S}\geq n)&=P_{q,\theta}(\overline{W}_{S}^{(1)}\geq n,\ \overline{W}_{S}^{(0)}\geq n)\\
&=P_{q,\theta}(F_{n-1}\leq k_2-1,\ S_{n-1}\leq k_1-1)\\
&=P_{q,\theta}(n-k_2\leq S_{n-1}\leq k_1-1)\\
&=\sum_{x=n-k_2}^{k_1-1}P_{q,\theta}(S_{n-1}=x).\\
\end{split}
\end{equation*}
Thus the proof is completed by noting that
\begin{equation*}\label{}
\begin{split}
P_{q,\theta}(\overline{W}_{S}= n)&=P_{q,\theta}(\overline{W}_{S}\geq n)-P_{q,\theta}(W_{S}\geq n+1)\\
&=\sum_{x=n-k_2}^{k_1-1}P_{q,\theta}(S_{n-1}=x)-\sum_{m=n+1-k_2}^{k_1-1}P_{q,\theta}(S_{n}=x)\\
&=\sum_{x=n-k_2}^{k_1-1} {n-1\brack x }_{q}
\theta^{x}\prod_{i=1}^{n-1-x}(1-\theta q^{i-1})-\sum_{x=n+1-k_2}^{k_1-1}{n\brack x }_{q}
\theta^{x}\prod_{i=1}^{n-x}(1-\theta q^{i-1}).
\end{split}
\end{equation*}
\end{proof}
\noindent
It is worth mentioning here that the PMF $\overline{f}_{q,S}(n;\theta)$ approaches the probability function of the sooner waiting time distribution of order $(k_1,k_2)$
in the limit as $q$ tends to 1 when a frequency quota is imposed on runs of successes and failures of IID model. The details are presented in the following remark.
\noindent
\begin{remark}
{\rm
For $q=1$, the PMF $\overline{f}_{q,S}(n;\theta)$ reduces to the PMF $\overline{f}_{S}(n;\theta)=P_{\theta}(\overline{W}_{S}=n)$ for $n=\text{min}(k_1,k_2),\ldots, k_1+k_2-1$ is given by
\begin{equation*}\label{}
\begin{split}
P_{q,\theta}(\overline{W}_{S}= n)&=\sum_{x=n-k_2}^{k_1-1}\binom{n-1}{x}\theta^{x}(1-\theta)^{n-1-x}-\sum_{x=n+1-k_2}^{k_1-1}\binom{n}{x}\theta^{x}(1-\theta)^{n-x}.
\end{split}
\end{equation*}
}

\end{remark}

\subsection{Later waiting time}

The problem of waiting time that will be discussed in this section is one of the 'later cases' and it emerges when a frequency quota on successes and failures are imposed. More specifically, Binary (zero and one) trials with probability of ones varying according to a geometric rule, are performed sequentially until $k_1$ successes in total or $k_2$ failures in total are observed, whichever event gets observed later.
Let random variable $\widehat{W}_{L}$ denote the waiting time until both $k_{1}$ successes in total and $k_{2}$ failures in total have observed, whichever event gets observed later. Let random variable $\widehat{W}_{L}$ denote the waiting time until both $k_{1}$ successes in total and $k_{2}$ failures in total have observed, whichever event gets observed later.
\noindent
The probability function of the $q$-later waiting time distribution impose a frequency quotas on both successes and failures is obtained by the following theorem. It is evident that
\noindent
\begin{equation*}
P_{q,\theta}(\widehat{W}_{L}=n)=0\ \text{for}\ n< k_1+k_2
\end{equation*}
\noindent
and so we shall focus on determining the probability mass function for $n\geq k_{1}+k_{2}$.
\noindent
\begin{theorem}
\label{thm:5.3}
The PMF $P_{q,\theta}(\overline{W}_{L}=n)$ satisfies
\begin{equation*}\label{eq: 1.1}
\begin{split}
P(\overline{W}_{L}=n)=\overline{P}_{q,L}^{(1)}(n)+\overline{P}_{q,L}^{(0)}(n),\ n\geq k_1+k_2,
\end{split}
\end{equation*}
where

\begin{equation*}\label{eq:bn1}
\begin{split}
\overline{P}_{q,L}^{(1)}(n)=\theta^{k_{1}}\ {\prod_{j=1}^{n-k_1}}\ (1-\theta q^{j-1})\ \sum_{s=1}^{k_1}\bigg[\overline{I}_{q}(k_{1},\ n-k_1,\ s)+\overline{J}_{q}(k_{1},\ n-k_1,\ s)\bigg]\\
\end{split}
\end{equation*}
\noindent

and

\begin{equation*}\label{eq:bn1}
\begin{split}
\overline{P}_{q,L}^{(0)}(n)=\theta^{n-k_{2}}\ {\prod_{j=1}^{k_2}}\ (1-\theta q^{j-1})\ \sum_{s=1}^{k_2}\bigg[\overline{K}_{q}(n-k_{2},\ k_2,\ s)+\overline{L}_{q}(n-k_{2},\ k_2\ s)\bigg].\\
\end{split}
\end{equation*}
\end{theorem}

\begin{proof}
We start with the study of $\overline{P}_{q,L}^{(1)}(n)$. From now on we assume $n \geq k_1+k_2,$ and we can write $\overline{P}_{q,L}^{(1)}(n)$ as follows.
\noindent
\begin{equation*}\label{eq:bn}
\begin{split}
\overline{P}_{q,L}^{(1)}(n)&=P_{q,\theta}\left(S_{n}=k_{1}\ \wedge\ F_{n}=n-k_{1}\ \wedge\ X_{n}=1\right).\\
\end{split}
\end{equation*}

\noindent
We are going to focus on the event $\left\{ S_{n}=k_{1}\ \wedge\ F_{n}=n-k_{1}\ \wedge\ X_{n}=1 \right\}$. A typical element of the event $\left\{ S_{n}=k_{1}\ \wedge\ F_{n}=n-k_{1}\ \wedge\ X_{n}=1 \right\}$ is an ordered sequence which consists of $k_{1}$ successes and $n-k_1$ failures. The number of these sequences can be derived as follows. First we will distribute the $k_1$ successes. Let $s$ $(1\leq s \leq k_1)$ be the number of runs of successes in the event $\left\{S_{n}=k_{1}\ \wedge\ F_{n}=n-k_{1}\ \wedge\ X_{n}=1\right\}$. We divide into two cases: starting with a success run or starting with a failure run. Thus, we distinguish between two types of sequences in the event  $$\Big\{S_{n}=k_{1}\ \wedge\ F_{n}=n-k_{1}\ \wedge\ X_{n}=1\Big\},$$ respectively named $(s,\ s-1)$-type and $(s,\ s)$-type, which are defined as follows.
\noindent

\begin{equation*}\label{eq:bn}
\begin{split}
(s,\ s-1)\text{-type}\  :&\quad \overbrace{1\ldots 1}^{x_{1}}\mid\overbrace{0\ldots 0}^{y_{1}}\mid\overbrace{1\ldots 1}^{x_{2}}\mid\overbrace{0\ldots 0}^{y_{2}}\mid \ldots\mid\overbrace{1\ldots 1}^{x_{s-1}} \mid \overbrace{0\ldots 0}^{y_{s-1}}\mid\overbrace{1\ldots 1}^{x_{s}},
\end{split}
\end{equation*}
\noindent
with $n-k_1$ $0$'s and $k_{1}$ $1$'s, where $x_{j}$ $(j=1,\ldots,s)$ represents a length of run of $1$'s and $y_{j}$ $(j=1,\ldots,s-1)$ represents the length of a run of $0$'s. Here all of $x_1,\ldots, x_{s},$ and $y_1,\ldots, y_{s-1}$ are integers, and they satisfy
\noindent
\begin{equation*}\label{eq:bn}
\begin{split}
(x_{1},\ldots,x_{s})\in S_{k_1,s}^{0}, \text{and}
\end{split}
\end{equation*}
\noindent
\noindent
\begin{equation*}\label{eq:bn}
\begin{split}
(y_{1},\ldots,y_{s-1})\in S_{n-k_1,s-1}^{0},
\end{split}
\end{equation*}

\noindent

\begin{equation*}\label{eq:bn}
\begin{split}
(s,\ s)\text{-type}\ :&\quad \overbrace{0\ldots 0}^{y_{1}}\mid\overbrace{1\ldots 1}^{x_{1}}\mid\overbrace{0\ldots 0}^{y_{2}}\mid\overbrace{1\ldots 1}^{x_{2}}\mid \ldots \mid \overbrace{0\ldots 0}^{y_{s-1}}\mid\overbrace{1\ldots 1}^{x_{s-1}}\mid\overbrace{0\ldots 0}^{y_{s}}\mid\overbrace{1\ldots 1}^{x_{s}},\\
\end{split}
\end{equation*}
\noindent
with $n-k_1$ $0$'s and $k_{1}$ $1$'s, where $x_{j}$ $(j=1,\ldots,s)$ represents a length of run of $1$'s and $y_{j}$ $(j=1,\ldots,s)$ represents the length of a run of $0$'s. And all integers $x_{1},\ldots,x_{s}$, and $y_{1},\ldots ,y_{s}$ satisfy the conditions
\noindent
\begin{equation*}\label{eq:bn}
\begin{split}
(x_{1},\ldots,x_{s})\in S_{k_1,s}^{0}, \text{and}
\end{split}
\end{equation*}
\noindent
\begin{equation*}\label{eq:bn}
\begin{split}
(y_{1},\ldots,y_{s})\in S_{n-k_1,s}^{0},
\end{split}
\end{equation*}

\noindent
Then the probability of the event $\Big\{ S_{n}=k_{1}\ \wedge\ F_{n}=n-k_{1}\ \wedge\ X_{n}=1 \Big\}$ is given by
\noindent
\begin{equation*}\label{eq: 1.1}
\begin{split}
P_{q,\theta}\Big(S_{n}=&k_{1}\ \wedge\ F_{n}=n-k_{1}\ \wedge\ X_{n}=1\Big)=\\
\sum_{s=1}^{k_1}\Bigg[\bigg\{&\sum_{(x_{1},\ldots,x_{s}) \in S_{k_{1},s}^{0}}{\hspace{0.3cm}\sum_{(y_{1},\ldots,y_{s-1})\in S_{n-k_1,s-1}^{0}}}\big(\theta q^{0}\big)^{x_{1}}\big(1-\theta q^{0}\big)\cdots (1-\theta q^{y_{1}-1})\times\\
&\Big(\theta q^{y_{1}}\Big)^{x_{2}}\big(1-\theta q^{y_{1}}\big)\cdots \big(1-\theta q^{y_{1}+y_{2}-1}\big)\times\\
&\quad \quad \quad \quad\quad \quad\quad \quad \quad \quad \quad \vdots\\
&\Big(\theta q^{y_{1}+\cdots+y_{s-2}}\Big)^{x_{s-1}}
\big(1-\theta q^{y_{1}+\cdots+y_{s-2}}\big)\cdots \big(1-\theta q^{y_{1}+\cdots+y_{s-1}-1}\big)\times\\
&\Big(\theta q^{y_{1}+\cdots+y_{s-1}}\Big)^{x_{s}}\bigg\}+\\
\end{split}
\end{equation*}

\noindent
\begin{equation*}\label{eq: 1.1}
\begin{split}
\quad \bigg\{&\sum_{\substack{x_{1},\ldots,x_{s} \in S_{k_{1},s}^{0}}}{\hspace{0.3cm}\sum_{(y_{1},\ldots,y_{s})\in S_{n-k_1,s}^{0}}} \big(1-\theta q^{0}\big)\cdots \big(1-\theta q^{y_{1}-1}\big)\Big(\theta q^{y_{1}}\Big)^{x_{1}}\times\\
&\big(1-\theta q^{y_{1}}\big)\cdots \big(1-\theta q^{y_{1}+y_{2}-1}\big)\Big(\theta q^{y_{1}+y_{2}}\Big)^{x_{2}}\times\\
&\quad \quad \quad \quad\quad \quad\quad \quad \quad \quad \quad \vdots\\
&\big(1-\theta q^{y_{1}+\cdots+y_{s-1}}\big)\cdots \big(1-\theta q^{y_{1}+\cdots+y_{s}-1}\big)\Big(\theta q^{y_{1}+\cdots+y_{s}}\Big)^{x_{s}}\bigg\}\Bigg].\\
\end{split}
\end{equation*}
\noindent
Using simple exponentiation algebra arguments to simplify,
\noindent
\begin{equation*}\label{eq: 1.1}
\begin{split}
&P_{q,\theta}\Big(S_{n}=k_{1}\ \wedge\ F_{n}=n-k_{1}\ \wedge\ X_{n}=1\Big)=\\
&\theta^{k_{1}}{\prod_{j=1}^{n-k_1}}\ (1-\theta q^{j-1})\times\\
&\sum_{s=1}^{k_1}\Bigg[\sum_{\substack{(x_{1},\ldots,x_{s}) \in S_{k_{1},s}^{0}}}{\hspace{0.3cm}\sum_{(y_{1},\ldots,y_{s-1})\in S_{n-k_1,s-1}^{0}}}q^{y_{1}x_{2}+(y_{1}+y_{2})x_{3}+\cdots+(y_{1}+\cdots+y_{s-1})x_{s}}+\\
&\quad \quad \sum_{\substack{(x_{1},\ldots,x_{s}) \in S_{k_{1},s}^{0}}}{\hspace{0.3cm}\sum_{(y_{1},\ \ldots,\ y_{s})\in S_{n-k_1,s}^{0}}}q^{y_{1}x_{1}+(y_{1}+y_{2})x_{2}+\cdots+(y_{1}+\cdots+y_{s})x_{s}}\Bigg].\\
\end{split}
\end{equation*}
\noindent
Using Lemma \ref{lemma:4.1} and Lemma \ref{lemma:4.3}, we can rewrite as follows.
\noindent
\begin{equation*}\label{eq: 1.1}
\begin{split}
P_{q,\theta}\Big(&S_{n}=k_{1}\ \wedge\ F_{n}=n-k_{1}\ \wedge\ X_{n}=1\Big)\\
&=\theta^{k_{1}}\ {\prod_{j=1}^{n-k_1}}\ (1-\theta q^{j-1})\ \sum_{s=1}^{k_1}\bigg[\overline{I}_{q}(k_{1},\ n-k_1,\ s)+\overline{J}_{q}(k_{1},\ n-k_1,\ s)\bigg],\\
\end{split}
\end{equation*}\\
\noindent
where
\noindent
\begin{equation*}\label{eq: 1.1}
\begin{split}
\overline{I}_{q}(k_{1}&,\ n-k_1,\ s)\\
&=\sum_{\substack{ (x_{1},\ldots,x_{s}) \in S_{k_{1},s}^{0}}}{\hspace{0.3cm}\sum_{(y_{1},\ldots,y_{s-1})\in S_{n-k_1,s-1}^{0}}}q^{y_{1}x_{2}+(y_{1}+y_{2})x_{3}+\cdots+(y_{1}+\cdots+y_{s-1})x_{s}},
\end{split}
\end{equation*}
and
\begin{equation*}\label{eq: 1.1}
\begin{split}
\overline{J}_{q}(k_{1}&,\ n-k_1,\ s)\\
&=\sum_{\substack{(x_{1},\ldots,x_{s}) \in S_{k_{1},s}^{0}}}{\hspace{0.3cm}\sum_{(y_{1},\ldots,y_{s})\in S_{n-k_1,s}^{0}}}q^{y_{1}x_{1}+(y_{1}+y_{2})x_{2}+\cdots+(y_{1}+\cdots+y_{s})x_{s}}.
\end{split}
\end{equation*}
\noindent
Therefore we can compute the probability of the event $\overline{W}_{L}^{(1)}=n$ as follows.
\noindent
\begin{equation*}\label{eq:bn1}
\begin{split}
\overline{P}_{q,L}^{(1)}(n)=\theta^{k_{1}}\ {\prod_{j=1}^{n-k_1}}\ (1-\theta q^{j-1})\ \sum_{s=1}^{k_1}\bigg[\overline{I}_{q}(k_{1},\ n-k_1,\ s)+\overline{J}_{q}(k_{1},\ n-k_1,\ s)\bigg].\\
\end{split}
\end{equation*}
\noindent


We start with the study of $\overline{P}_{q,L}^{(0)}(n)$. From now on we assume $n \geq k_1+k_2,$ and we can write $\overline{P}_{q,L}^{(0)}(n)$ as follows.
\noindent
\begin{equation*}\label{eq:bn}
\begin{split}
\overline{P}_{q,L}^{(0)}(n)&=P_{q,\theta}\Big(S_{n}=n-k_2\ \wedge\ F_{n}=k_2\ \wedge\ X_{n}=0\Big).\\
\end{split}
\end{equation*}

\noindent
We are going to focus on the event $\left\{S_{n}=n-k_2\ \wedge\ F_{n}=k_2\ \wedge\ X_{n}=0\right\}$. A typical element of the event $\left\{S_{n}=n-k_2\ \wedge\ F_{n}=k_2\ \wedge\ X_{n}=0\right\}$ is an ordered sequence which consists of $n-k_{2}$ successes and $k_2$ failures. The number of these sequences can be derived as follows. First we will distribute the $k_2$ failures. Let $s$ $(1\leq s \leq k_2)$ be the number of runs of failures in the event $\Big\{S_{n}=n-k_2\ \wedge\ F_{n}=k_2\ \wedge \ X_{n}=0\Big\}$. We divide into two cases: starting with a success run or starting with a failure run. Thus, we distinguish between two types of sequences in the event  $$\Big\{S_{n}=n-k_2\ \wedge\ F_{n}=k_2\ \wedge\ X_{n}=0\Big\},$$ respectively named $(s-1,\ s)$-type and $(s,\ s)$-type, which are defined as follows.

\noindent
\begin{equation*}\label{eq:bn}
\begin{split}
(s-1,\ s)\text{-type}\ :&\quad \overbrace{0\ldots 0}^{y_{1}}\mid\overbrace{1\ldots 1}^{x_{1}}\mid\overbrace{0\ldots 0}^{y_{2}}\mid\overbrace{1\ldots 1}^{x_{2}}\mid \ldots \mid \overbrace{0\ldots 0}^{y_{s-1}}\mid\overbrace{1\ldots 1}^{x_{s-1}}\mid\overbrace{0\ldots 0}^{y_{s}},\\
\end{split}
\end{equation*}
\noindent
with $k_2$ $0$'s and $n-k_{2}$ $1$'s, where $x_{j}$ $(j=1,\ldots,s-1)$ represents a length of run of $1$'s and $y_{j}$ $(j=1,\ldots,s)$ represents the length of a run of $0$'s. And all integers $x_{1},\ldots,x_{s-1}$, and $y_{1},\ldots,y_{s}$ satisfy the conditions
\noindent
\begin{equation*}\label{eq:bn}
\begin{split}
(x_{1},\ldots,x_{s-1})\in S_{n-k_2,s-1}^{0} \mbox{\ and\ }
\end{split}
\end{equation*}
\noindent
\begin{equation*}\label{eq:bn}
\begin{split}
(y_{1},\ldots,y_{s})\in S_{k_2,s}^{0},
\end{split}
\end{equation*}
\noindent
\begin{equation*}\label{eq:bn}
\begin{split}
(s,\ s)\text{-type}\  :&\quad \overbrace{1\ldots 1}^{x_{1}}\mid\overbrace{0\ldots 0}^{y_{1}}\mid\overbrace{1\ldots 1}^{x_{2}}\mid\overbrace{0\ldots 0}^{y_{2}}\mid\overbrace{1\ldots 1}^{x_{3}}\mid \ldots \mid \overbrace{0\ldots 0}^{y_{s-1}}\mid\overbrace{1\ldots 1}^{x_{s}}\mid\overbrace{0\ldots 0}^{y_{s}},
\end{split}
\end{equation*}
\noindent
with $k_2$ $0$'s and $n-k_{2}$ $1$'s, where $x_{j}$ $(j=1,\ldots,s)$ represents a length of run of $1$'s and $y_{j}$ $(j=1,\ldots,s)$ represents the length of a run of $0$'s. Here all of $x_1,\ldots, x_{s-1},$ and $y_1,\ldots, y_s$ are integers, and they satisfy
\noindent
\begin{equation*}\label{eq:bn}
\begin{split}
(x_{1},\ldots, x_{s})\in S_{n-k_2,s}^{0} \mbox{\ and\ }
\end{split}
\end{equation*}
\noindent
\begin{equation*}\label{eq:bn}
\begin{split}
(y_{1}, \ldots, y_{s})\in S_{k_2,s}^{0},
\end{split}
\end{equation*}

\noindent
Then the probability of the event $\Big\{S_{n}=n-k_2\ \wedge\ F_{n}=k_2\ \wedge\ X_{n}=0\Big\}$ is given by
\noindent
\begin{equation*}\label{eq: 1.1}
\begin{split}
P_{q,\theta}\Big(S_{n}=&n-k_2\ \wedge\ F_{n}=k_2\ \wedge\ X_{n}=0\Big)=\\
\sum_{s=1}^{k_1}\Bigg[\bigg\{&\sum_{(x_{1},\ldots,x_{s-1})\in S_{n-k_2,s-1}^{0}}{\hspace{0.3cm}\sum_{(y_{1},\ldots,y_{s})\in S_{k_2,s}^{0}}} \big(1-\theta q^{0}\big)\cdots \big(1-\theta q^{y_{1}-1}\big)\times\\
&\Big(\theta q^{y_{1}}\Big)^{x_{1}}\big(1-\theta q^{y_{1}}\big)\cdots \big(1-\theta q^{y_{1}+y_{2}-1}\big)\times\\
&\Big(\theta q^{y_{1}+y_{2}}\Big)^{x_{2}}\big(1-\theta q^{y_{1}+y_2}\big)\cdots \big(1-\theta q^{y_{1}+y_{2}+y_3-1}\big)\times \\
&\quad \quad \quad \quad\quad \quad\quad \quad \quad \quad \quad \vdots\\
&\Big(\theta q^{y_{1}+\cdots+y_{s-1}}\Big)^{x_{s-1}}
\big(1-\theta q^{y_{1}+\cdots+y_{s-1}}\big)\cdots \big(1-\theta q^{y_{1}+\cdots+y_{s}-1}\big)\bigg\}+\\
\end{split}
\end{equation*}

\noindent
\begin{equation*}\label{eq: 1.1}
\begin{split}
\quad \quad \quad \quad \bigg\{&\sum_{(x_{1},\ldots,x_{s})\in S_{n-k_2,s}^{0}}{\hspace{0.3cm}\sum_{(y_{1},\ldots,y_{s})\in S_{k_2,s}^{0}}}\big(\theta q^{0}\big)^{x_{1}}\big(1-\theta q^{0}\big)\cdots (1-\theta q^{y_{1}-1})\times\\
&\Big(\theta q^{y_{1}}\Big)^{x_{2}}\big(1-\theta q^{y_{1}}\big)\cdots \big(1-\theta q^{y_{1}+y_{2}-1}\big)\times\\
&\Big(\theta q^{y_{1}+y_{2}}\Big)^{x_{3}}\big(1-\theta q^{y_{1}+y_2}\big)\cdots \big(1-\theta q^{y_{1}+y_{2}+y_3-1}\big)\times \\
&\quad \quad \quad \quad\quad \quad\quad \quad \quad \quad \quad \vdots\\
&\Big(\theta q^{y_{1}+\cdots+y_{s-1}}\Big)^{x_{s}}
\big(1-\theta q^{y_{1}+\cdots+y_{s-1}}\big)\cdots \big(1-\theta q^{y_{1}+\cdots+y_{s}-1}\big)\bigg\}\Bigg].\\
\end{split}
\end{equation*}
\noindent
Using simple exponentiation algebra arguments to simplify,
\noindent
\begin{equation*}\label{eq: 1.1}
\begin{split}
&P_{q,\theta}\Big(S_{n}=n-k_2\ \wedge\ F_{n}=k_2\ \wedge\ X_{n}=0\Big)=\\
&\theta^{n-k_{2}}{\prod_{j=1}^{k_2}}\ (1-\theta q^{j-1})\times\\
&\sum_{s=1}^{k_2}\Bigg[\sum_{(x_{1},\ldots,x_{s-1})\in S_{n-k_2,s-1}^{0}}{\hspace{0.3cm}\sum_{(y_{1},\ldots,y_{s})\in S_{k_2,s}^{0}}}q^{y_{1}x_{1}+(y_{1}+y_{2})x_{2}+\cdots+(y_{1}+\cdots+y_{s-1})x_{s-1}}+\\
&\quad \quad\sum_{(x_{1},\ldots,x_{s})\in S_{n-k_2,s}^{0}}{\hspace{0.3cm}\sum_{(y_{1},\ldots,y_{s})\in S_{k_2,s}^{0}}}q^{y_{1}x_{2}+(y_{1}+y_{2})x_{3}+\cdots+(y_{1}+\cdots+y_{s-1})x_{s}}\Bigg].\\
\end{split}
\end{equation*}
\noindent
Using Lemma \ref{lemma:4.5} and Lemma \ref{lemma:4.7}, we can rewrite as follows.
\noindent
\begin{equation*}\label{eq: 1.1}
\begin{split}
P_{q,\theta}\Big(&S_{n}=n-k_2\ \wedge\ F_{n}=k_2\ \wedge\ X_{n}=0\Big)\\
&=\theta^{n-k_{2}}\ {\prod_{j=1}^{k_2}}\ (1-\theta q^{j-1})\ \sum_{s=1}^{k_2}\bigg[\overline{K}_{q}(n-k_{2},\ k_2,\ s)+\overline{L}_{q}(n-k_{2},\ k_2\ s)\bigg],\\
\end{split}
\end{equation*}\\
\noindent

where
\noindent
\begin{equation*}\label{eq: 1.1}
\begin{split}
\overline{K}_{q}(n-k_{2}&,\ k_2,\ s)\\
&=\sum_{(x_{1},\ldots,x_{s-1})\in S_{n-k_2,s-1}^{0}}{\hspace{0.3cm}\sum_{(y_{1},\ldots,y_{s})\in S_{k_2,s}^{0}}}q^{y_{1}x_{1}+(y_{1}+y_{2})x_{2}+\cdots+(y_{1}+\cdots+y_{s-1})x_{s-1}},
\end{split}
\end{equation*}
\noindent
and
\noindent
\begin{equation*}\label{eq: 1.1}
\begin{split}
\overline{L}_{q}(n-k_{2}&,\ k_2,\ s)\\
&=\sum_{(x_{1},\ldots,x_{s})\in S_{n-k_2,s}^{0}}{\hspace{0.3cm}\sum_{(y_{1},\ldots,y_{s})\in S_{k_2,s}^{0}}}q^{y_{1}x_{2}+(y_{1}+y_{2})x_{3}+\cdots+(y_{1}+\cdots+y_{s-1})x_{s}}.
\end{split}
\end{equation*}
\noindent
Therefore we can compute the probability of the event $\overline{W}_{L}^{(0)}=n$ as follows.
\noindent
\begin{equation*}\label{eq:bn1}
\begin{split}
\overline{P}_{q,L}^{(0)}(n)=\theta^{n-k_{2}}\ {\prod_{j=1}^{k_2}}\ (1-\theta q^{j-1})\ \sum_{s=1}^{k_2}\bigg[\overline{K}_{q}(n-k_{2},\ k_2,\ s)+\overline{L}_{q}(n-k_{2},\ k_2\ s)\bigg].\\
\end{split}
\end{equation*}
\noindent
Thus proof is completed.
\end{proof}
\noindent
It is worth mentioning here that the PMF $f_{q,L}(n;\theta)$ approaches the probability function of the later waiting time distribution of order $(k_1,k_2)$
in the limit as $q$ tends to 1 when a frequency quota is imposed on runs of successes and failures of IID model. The details are presented in the following remark.
\noindent
\begin{remark}
{\rm
For $q=1$, the PMF $\overline{f}_{q,L}(n;\theta)$ reduces to the PMF $\overline{f}_{L}(n;\theta)=P_{\theta}(\overline{W}_{L}=n)$ for $n\geq k_1+k_2$ is given by

\begin{equation*}\label{eq: 1.1}
\begin{split}
P(\overline{W}_{L}=n)=\overline{P}_{L}^{(1)}(n)+\overline{P}_{L}^{(0)}(n),\ n\geq k_1+k_2,
\end{split}
\end{equation*}
where
\begin{equation*}\label{eq:bn1}
\begin{split}
\overline{P}_{L}^{(1)}(n)=\theta^{k_{1}}\ (1-\theta)^{n-k_1}\ \sum_{s=1}^{k_1}M(s,\ k_1)\bigg[M(s-1,\ n-k_1)+M(s,\ n-k_1)\bigg].\\
\end{split}
\end{equation*}
\noindent
and
\begin{equation*}\label{eq:bn1}
\begin{split}
\overline{P}_{L}^{(0)}(n)=\theta^{n-k_{2}}(1-\theta)^{k_2}\sum_{s=1}^{k_2}\bigg[M(s-1,\ n-k_{2})+M(s,\ n-k_{2})\bigg]M(s,\ k_2),
\end{split}
\end{equation*}
where
\begin{equation*}
M(a,\ b)=\binom{b-1}{a-1}.
\end{equation*}
See, e.g. \citet{charalambides2002enumerative}.
}
\end{remark}

\section{Longest run for $q$-binary trials}
In this section, we shall study of the distribution of the lengths of the longest success run. More specifically binary (zero and one) trials with probability of ones varing according to a geometric rule. First, we obtain the exact PMF of the distribution of $\left(L_{n}^{(1)}= k\right)$. Next, we obtain the exact PMF of the distribution of $\left(L_{n}^{(1)}\leq k\right)$. The result are derived by means of enumerative combinatorics.

\subsection{Closed expression for the PMF of $\left(L_{n}^{(1)}= k\right)$}

\noindent
We now make some useful Definition and Lemma for the proofs of Theorem in the sequel.
\noindent
\begin{definition}
For $0<q\leq1$, we define

\begin{equation*}\label{eq: 1.1}
\begin{split}
U_{q}^{k}(r,s,t)={\sum_{x_{1},\ldots,x_{r}}} q^{x_{2}+2x_{3}+\cdots+(r-1)x_{r}},\
\end{split}
\end{equation*}
\noindent
where the summation is over all integers $x_1,\ldots,x_{r}$ satisfying
\noindent

\begin{equation*}\label{eq:1}
\begin{split}
0\leq x_{j}\leq k\ \text{for}\ j=1,\ldots,r,
\end{split}
\end{equation*}
\noindent
\begin{equation*}\label{eq:1}
\begin{split}
x_{1}+\cdots+x_{r}=s,\ \text{and}
\end{split}
\end{equation*}
\noindent
\begin{equation*}\label{eq:2}
\begin{split}
\delta_{y_{1},k}+\cdots+\delta_{y_{r},k}=t.
\end{split}
\end{equation*}

\end{definition}

\noindent
The following gives a recurrence relation useful for the computation of $U_{q}^{k}(r,s,t)$.\\
\noindent

\begin{lemma}
\label{lemma3.1-n}
For $0<q\leq1$, $U_{q}^{k}(r,s,t)$ obeys the following recurrence relation.
\noindent
\begin{equation*}\label{eq: 1.1}
\begin{split}
U_{q}&^{k}(r,s,t)\\
&=\left\{
  \begin{array}{ll}
    \sum_{a=0}^{k-1}q^{a(r-1)}U_{q}^{k}(r-1,s-a,t)+\\
    \quad \quad\ q^{k(r-1)}U_{q}^{k}(r-1,s-k,t-1) & \text{for}\ r>1,\ tk\leq s\leq rk\ \text{and}\ t\leq r<s, \\
    1 & \text{for}\ r=1,\ s=k,\ \text{and}\ t=1,\\
    & \text{or}\ r=1,\ 0\leq s<k,\ \text{and}\ t=0, \text{or}\\
    0 & \text{otherwise.}\\
  \end{array}
\right.
\end{split}
\end{equation*}
\end{lemma}
\begin{proof}
For $r > 1$, $tk\leq s\leq rk$ and $t\leq r<s$, we observe that since $x_{r}$ can assume the values $1,\ldots,k$, then $U_{q}^{k}(r,s,t)$ can be written as follows.
\noindent
\begin{equation*}
\begin{split}
U_{q}^{k}(r,s,t)=&\sum_{x_{r}=0}^{k-1}\ q^{x_{r}(r-1)} {\hspace{0.5cm}\sum_{\substack{x_{1}+\cdots+x_{s-1}=m-x_{s}\\ x_{1},\ldots,x_{r-1} \in \{0,1,\ldots,k\}\\\delta_{y_{1},k}+\cdots+\delta_{y_{r},k}=t}}}\quad
 q^{x_{2}+2x_{3}+\cdots+(r-2)x_{r-1}}\\
 &+q^{k(r-1)}{\hspace{0.5cm}\sum_{\substack{x_{1}+\cdots+x_{s-1}=m-x_{s}\\ x_{1},\ldots,x_{r-1} \in \{0,1,\ldots,k\}\\\delta_{y_{1},k}+\cdots+\delta_{y_{r},k}=t-1}}}\quad
\ q^{x_{2}+2x_{3}+\cdots+(r-2)x_{r-1}}\\
\end{split}
\end{equation*}
\noindent
Using simple algebraic arguments to simplify, then $U_{q}^{k}(r,s,t)$ can be rewritten as follows.
\noindent
\begin{equation*}
\begin{split}
U_{q}^{k}(r,s,t)=&\sum_{a=0}^{k-1}q^{a(r-1)}U_{q}^{k}(r-1,s-a,t)+ q^{k(r-1)}U_{q}^{k}(r-1,s-k,t-1).
\end{split}
\end{equation*}
\noindent
The other cases are obvious and thus the proof is completed.
\end{proof}

\begin{remark}
{\rm
We observe that $U_1^{k}(r,s,t)$ is the number of integer solutions $(x_{1},\ \ldots,\ x_{r})$ of
\noindent

\begin{equation*}\label{eq:1}
\begin{split}
0\leq x_{j}\leq k\ \text{for}\ j=1,\ldots,r,
\end{split}
\end{equation*}
\noindent
\begin{equation*}\label{eq:1}
\begin{split}
x_{1}+\cdots+x_{r}=s,\ \text{and}
\end{split}
\end{equation*}
\noindent
\begin{equation*}\label{eq:2}
\begin{split}
\delta_{y_{1},k}+\cdots+\delta_{y_{r},k}=t.
\end{split}
\end{equation*}
\noindent
which is
\noindent
\begin{equation*}\label{eq: 1.1}
\begin{split}
U_1^{k}(r,s,t)={r \choose t}C(s-tk,\ r-t,\ k-1),
\end{split}
\end{equation*}
\noindent
denotes the number of ways of distributing $s$ identical balls into $r$ different cells, so that each of $t$ of them receives $k$ balls is occupied,
remaining $s-tk$ balls placed in the remaining $r-t$ cells with no cell receiving more than $k-1$ balls.\\
where $C(a,\ b,\ c)$ denotes the total number of integer solutions $x_{1}+x_{2}+\cdots+x_{b}=a$ such that $0\leq x_{i}\leq c$ for $i=1,2,\ldots,b$. Alternatively, it is the number of ways of distributing $a$ identical balls into $b$ different cells , so that each cell is occupied by at most $c$ balls. The number can be expressed as
\noindent
\begin{equation*}\label{eq: 1.1}
\begin{split}
C(a,\ b,\ c)=\sum_{j=0}^{b}(-1)^{j}{b \choose j}{a-(c+1)j+b-1 \choose b-1}.
\end{split}
\end{equation*}
\noindent
if $a>0$: $1$ if $a=0$: and $0$ otherwise (See, e.g. \citet{charalambides2002enumerative}.
}
\end{remark}
\noindent
We shall study of the joint distribution of $\left(L_{n}^{(1)}= k\right)$. The probability function of the  $P_{q,\theta}\Big(L_{n}^{(1)}= k\Big)$ is obtained by the following theorem. It is evident that
\noindent
\begin{equation*}
P_{q,\theta}\left(L_{n}^{(1)}=0\right)={\prod_{j=1}^{n}}\left(1-\theta q^{j-1}\right)\ \text{for}\ k=0
\end{equation*}
\noindent
and so we shall focus on determining the probability mass function for $1\leq k \leq n$.
\noindent
\begin{theorem}
The PMF $P_{q,\theta}\Big(L_{n}^{(1)}= k\Big)$ for $k=1,\ldots,n,$
\begin{equation*}\label{}
\begin{split}
P_{q,\theta}\Big(L_{n}^{(1)}= k\Big)=&\sum_{y=[n/(k+1)]}^{n-k}\theta^{n-y}\ {\prod_{j=1}^{y}}\ \left(1-\theta q^{j-1}\right)\ \sum_{i=1}^{y+1} U_{q}^{k}(y+1,\ n-y,\ i).
\end{split}
\end{equation*}
\end{theorem}
\begin{proof}
We partition the event $\left\{L_{n}^{(1)}= k\right\}$ into disjoint events given by $F_{n}=y,$ for $y=[n/(k+1)], \ldots,n-k$. Adding the probabilities we have
\noindent
\begin{equation*}\label{eq:bn}
\begin{split}
P_{q,\theta}\Big(L_{n}^{(1)}= k\Big)=\sum_{y=[n/(k+1)]}^{n-k}P_{q,\theta}\Big(L_{n}^{(1)}= k\ \wedge\ F_{n}=y\Big).\\
\end{split}
\end{equation*}
\noindent
We are going to focus on the event $\left\{L_{n}^{(1)}= k\ \wedge\ F_{n}=y\right\}$. A typical element of the event $\Big\{L_{n}^{(1)}= k\ \wedge\ F_{n}=y\Big\}$ is an ordered sequence which consists of $n-y$ successes and $y$ failures such that the length of the longest success run is equal to $k$. The number of these sequences can be derived as follows. First we will distribute the $y$ failures which form $y+1$ cells. Next, we will distribute the $n-y$ successes in the $y+1$ distinguishable cells and they satisfy the conditions, $i$ ($i=1,\ldots,\text{min}\{y+1,[(n-y)/k]\})$ specified cells among the $y+1$ cells occupied exactly $k$ successes balls. And the remaining $n-y-ik$ successes can be distributed in the remaining $y+1-i$ cells with no cell receiving more than $k-1$ successes. Thus, the event $\left\{L_{n}^{(1)}= k\ \wedge\ F_{n}=y\right\}$ is defined as follows.
\noindent
\begin{equation*}\label{eq:bn}
\begin{split}
\underbrace{1\ldots 1}_{0\leq x_{1}\leq k}0\underbrace{1\ldots 1}_{0\leq x_{2}\leq k}0\underbrace{1\ldots 1}_{0\leq x_{3}\leq k}0 \ldots 0\underbrace{1\ldots 1}_{0\leq x_{y}\leq k}0\underbrace{1\ldots 1}_{0\leq x_{y+1}\leq k},\\
\end{split}
\end{equation*}
\noindent
with $y$ $0$'s and $n-y$ $1$'s, where $x_{j}$ $(j=1,\ldots,y+1$ represents a length of run of $1$'s . And all integers $x_{1},\ldots,x_{s}$ satisfy the conditions
\noindent
\begin{equation*}\label{eq:bn}
\begin{split}
0 \leq x_j \leq k \mbox{\ for\ } j=1,...,y+1,\  x_1+\cdots +x_{y+1} = n-y,\mbox{\ and}
\end{split}
\end{equation*}
\noindent
\begin{equation*}\label{eq:bn}
\begin{split}
\delta_{x_1,k}+\delta_{x_2,k}+\cdots+\delta_{x_{y+1},k}=i.
\end{split}
\end{equation*}
\noindent
Then the probability of the event $\left\{L_{n}^{(1)}= k\ \wedge\ F_{n}=y\right\}$ is given by
\noindent
\begin{equation*}\label{eq: 1.1}
\begin{split}
P_{q,\theta}\Big(L_{n}^{(1)}= k\ \wedge\ F_{n}=y\Big)&\\
=\sum_{i=1}^{y+1}\Bigg[\bigg\{\sum_{\substack{x_1+\cdots +x_{y+1} = n-y\\ x_{1},\ldots,x_{y+1} \in \{0.1,\ldots,k\}\\\delta_{x_1,k}+\cdots+\delta_{x_{y+1},k}=i}}&\Big(\theta q^{0}\Big)^{x_{1}}\Big(1-\theta q^{0}\Big)\Big(\theta q^{1}\Big)^{x_{2}}\Big(1-\theta q^{1}\Big)\times\\
&\quad \quad \quad \quad\quad \quad\quad \quad \quad  \vdots\\
& \Big(1-\theta q^{y-2}\Big)\Big(\theta q^{y-1}\Big)^{x_{y}} \Big(1-\theta q^{y-1}\Big)\Big(\theta q^{y}\Big)^{x_{y+1}}\bigg\}\Bigg].\\
\end{split}
\end{equation*}
\noindent
Using simple exponentiation algebra arguments to simplify,
\noindent
\begin{equation*}\label{eq: 1.1}
\begin{split}
P_{q,\theta}\Big(&L_{n}^{(1)}= k\ \wedge\ F_{n}=y\Big)=\\
&\theta^{n-y}\ {\prod_{j=1}^{y}}\ \left(1-\theta q^{j-1}\right)\ \sum_{i=1}^{y+1}\ \sum_{\substack{x_1+\cdots +x_{y+1} = n-y\\ x_{1},\ldots,x_{y+1} \in \{0,1,\ldots,k\}\\\delta_{x_1,k}+\cdots+\delta_{x_{y+1},k}=i}}q^{x_{2}+2y_{3}\cdots+(y-1)x_{y}+yx_{y+1}}.\\
\end{split}
\end{equation*}
\noindent
Using Lemma \ref{lemma3.1-n}, we can rewrite as follows.
\noindent
\begin{equation*}\label{eq: 1.1}
\begin{split}
P_{q,\theta}\Big(L_{n}^{(1)}= k\ \wedge\ F_{n}=y\Big)=\theta^{n-y}\ {\prod_{j=1}^{y}}\ \left(1-\theta q^{j-1}\right)\ \sum_{i=1}^{y+1} U_{q}^{k}(y+1,\ n-y,\ i),\\
\end{split}
\end{equation*}\\
\noindent
where
\noindent
\begin{equation*}\label{eq: 1.1}
\begin{split}
U_{q}^{k}(y+1,\ n-y,\ i)=\sum_{\substack{x_1+\cdots +x_{y+1} = n-y\\ x_{1},\ldots,x_{y+1} \in \{0.1,\ldots,k\}\\\delta_{x_1,k}+\cdots+\delta_{x_{y+1},k}=i}}q^{x_{2}+2y_{3}\cdots+(y-1)x_{y}+yx_{y+1}}.
\end{split}
\end{equation*}
\noindent
Therefore we can compute the probability of the event $\left\{L_{n}^{(1)}= k\right\}$ as follows.
\noindent
\begin{equation*}\label{eq:bn1}
\begin{split}
P_{q,\theta}\Big(L_{n}^{(1)}= k\Big)=&\sum_{y=[n/(k+1)]}^{n-k}\theta^{n-y}\ {\prod_{j=1}^{y}}\ \left(1-\theta q^{j-1}\right)\ \sum_{i=1}^{y+1} U_{q}^{k}(y+1,\ n-y,\ i).\\
\end{split}
\end{equation*}
\noindent
Thus proof is completed.
\end{proof}
\noindent
It is worth mentioning here that the PMF $P_{q,\theta}\Big(L_{n}^{(1)}= k\Big)$ approaches the probability function of $P_{\theta}\Big(L_{n}^{(1)}= k\Big)$ in the limit as $q$ tends to 1 of IID model. The details are presented in the following remark.
\noindent
\begin{remark}
{\rm For $q=1$, the PMF $P_{q,\theta}\Big(L_{n}^{(1)}= k\Big)$ reduces to the PMF $P_{\theta}\Big(L_{n}^{(1)}= k\Big)$. Then $P_{\theta}\Big(L_{n}^{(1)}= 0\Big)=(1-\theta)^{n}$ and for $1\leq k \leq n$ is given by
\noindent
\begin{equation*}\label{eq:bn1}
\begin{split}
P_{\theta}\Big(L_{n}^{(1)}= k\Big)=&\sum_{y=[n/(k+1)]}^{n-k}\theta^{n-y}(1-\theta)^{y}\ \sum_{i=1}^{y+1} U_{1}^{k}(y+1,\ n-y,\ i)\\
=&\sum_{y=[n/(k+1)]}^{n-k}\theta^{n-y}(1-\theta)^{y}\ \sum_{i=1}^{y+1} {y+1 \choose i}C(n-y+ik,\ y+1-i,\ k-1),
\end{split}
\end{equation*}
\noindent
where
\noindent
\begin{equation*}\label{eq: 1.1}
\begin{split}
C(a,\ b,\ c)=\sum_{j=0}^{b}(-1)^{j}{b \choose j}{a-(c+1)j+b-1 \choose b-1}.
\end{split}
\end{equation*}
\noindent
if $a>0$: $1$ if $a=0$: and $0$ otherwise (See, e.g. \citet{charalambides2002enumerative}).
}
\end{remark}

\subsection{Closed expression for the PMF of $\left(L_{n}^{(1)}\leq k\right)$}
\noindent
We now make some useful Definition and Lemma for the proofs of Theorem in the sequel.
\noindent
\begin{definition}
For $0<q\leq1$, we define
\noindent
\begin{equation*}\label{eq: 1.1}
\begin{split}
V_{q}(r,s,k)={\sum_{x_{1},\ldots,x_{r}}} q^{x_{2}+2x_{3}+\cdots+(r-1)x_{r}},\
\end{split}
\end{equation*}
\noindent
where the summation is over all integers $x_1,\ldots,x_{r}$ satisfying
\noindent
\begin{equation*}\label{eq:1}
\begin{split}
0\leq x_{j}\leq k\quad \text{for}\quad j=1,\ldots,r,\ \text{and}
\end{split}
\end{equation*}
\noindent
\begin{equation*}\label{eq:1}
\begin{split}
x_{1}+\cdots+x_{r}=s
\end{split}
\end{equation*}
\noindent
\end{definition}
\noindent
The following gives a recurrence relation useful for the computation of $V_{q}(r,s,k)$.\\
\noindent
\begin{lemma}
\label{lemma3.2-n}
For $0<q\leq1$, $V_{q}(r,s,k)$ obeys the following recurrence relation.
\noindent
\begin{equation*}\label{eq: 1.1}
\begin{split}
V_{q}&^{k}(r,s,k)=\left\{
  \begin{array}{ll}
    \sum_{a=0}^{k}q^{a(r-1)}V_{q}(r-1,s-a,k) & \text{for}\ r>1,\ \text{and}\ 0\leq s\leq kr\\
     1 & \text{for}\ r=1,\ \text{and}\ 0\leq s\leq k,\ \text{or}\\
    0 & \text{otherwise.}\\
  \end{array}
\right.
\end{split}
\end{equation*}
\end{lemma}
\begin{proof}
For $r > 1$ and $0\leq s\leq kr$, we observe that since $x_{r}$ can assume the values $1,\ldots,k$, then $V_{q}(r,s,k)$ can be written as follows.
\noindent
\begin{equation*}
\begin{split}
V_{q}(r,s,k)=&\sum_{x_{r}=0}^{k}\ q^{x_{r}(r-1)} {\hspace{0.5cm}\sum_{\substack{x_{1}+\cdots+x_{s-1}=m-x_{s}\\ x_{1},\ldots,x_{r-1} \in \{0,1,\ldots,k\}}}}\quad
 q^{x_{2}+2x_{3}+\cdots+(r-2)x_{r-1}}
\end{split}
\end{equation*}
\noindent
Using simple algebraic arguments to simplify, then $V_{q}(r,s,k)$ can be rewritten as follows.
\noindent
\begin{equation*}
\begin{split}
V_{q}&(r,s,k)=\sum_{a=0}^{k}q^{a(r-1)}V_{q}(r-1,s-a,k).
\end{split}
\end{equation*}
\noindent
The other cases are obvious and thus the proof is completed.
\end{proof}

\begin{remark}
{\rm
We observe that $V_1(r,s,k)$ is the number of integer solutions $(x_{1},\ \ldots,\ x_{r})$ of
\noindent
\begin{equation*}\label{eq:1}
\begin{split}
0\leq x_{j}\leq k\quad \text{for}\quad j=1,\ldots,r,\ \text{and}
\end{split}
\end{equation*}
\noindent
\begin{equation*}\label{eq:1}
\begin{split}
x_{1}+\cdots+x_{r}=s
\end{split}
\end{equation*}
\noindent
which is
\noindent
\begin{equation*}\label{eq: 1.1}
\begin{split}
V_1(r,s,k)=C(s,\ r,\ k),
\end{split}
\end{equation*}
\noindent
where $C(a,\ b,\ c)$ denotes the total number of integer solutions $x_{1}+x_{2}+\cdots+x_{b}=a$ such that $0\leq x_{i}\leq c$ for $i=1,2,\ldots,b$. Alternatively, it is the number of ways of distributing $a$ identical balls into $b$ different cells , so that each cell is occupied by at most $c$ balls. The number can be expressed as
\noindent
\begin{equation*}\label{eq: 1.1}
\begin{split}
C(a,\ b,\ c)=\sum_{j=0}^{b}(-1)^{j}{b \choose j}{a-(c+1)j+b-1 \choose b-1}.
\end{split}
\end{equation*}
\noindent
if $a>0$: $1$ if $a=0$: and $0$ otherwise (See, e.g. \citet{charalambides2002enumerative}.
}
\end{remark}
\noindent
We shall study of the cumulative distribution function $P_{q,\theta}\Big(L_{n}^{(1)}\leq k\Big)$ of the length $L_{n}^{(1)}$ of the longest success run. The probability function of the  $P_{q,\theta}\Big(L_{n}^{(1)}\leq k\Big)$ is obtained by the following theorem.
\noindent
\begin{theorem}
The PMF $P_{q,\theta}\Big(L_{n}^{(1)}\leq k\Big)$ is given by
\begin{equation*}\label{}
\begin{split}
P_{q,\theta}\Big(L_{n}^{(1)}\leq k\Big)=&\sum_{y=[n/(k+1)]}^{n-k}\theta^{n-y}\ {\prod_{j=1}^{y}}\ \left(1-\theta q^{j-1}\right)\  V_{q}(y+1,\ n-y,\ k).
\end{split}
\end{equation*}
\end{theorem}
\begin{proof}
We partition the event $\left\{L_{n}^{(1)}= k\right\}$ into disjoint events given by $F_{n}=y,$ for $y=[n/(k+1)], \ldots,n-k$. Adding the probabilities we have
\noindent
\begin{equation*}\label{eq:bn}
\begin{split}
P_{q,\theta}\Big(L_{n}^{(1)}= k\Big)=\sum_{y=[n/(k+1)]}^{n-k}P_{q,\theta}\Big(L_{n}^{(1)}= k\ \wedge\ F_{n}=y\Big).\\
\end{split}
\end{equation*}
\noindent
We are going to focus on the event $\left\{L_{n}^{(1)}\leq k\ \wedge\ F_{n}=y\right\}$. A typical element of the event $\Big\{L_{n}^{(1)}\leq k\ \wedge\ F_{n}=y\Big\}$ is an ordered sequence which consists of $n-y$ successes and $y$ failures such that the length of the longest success run is less and equal to $k$. The number of these sequences can be derived as follows. First we will distribute the $y$ failures which form $y+1$ cells. Next, we will distribute the $n-y$ successes in the $y+1$ distinguishable cells, so that each cell occupied by at most $k$ successes balls. Thus, the event $\left\{L_{n}^{(1)}\leq k\ \wedge\ F_{n}=y\right\}$ is defined as follows.
\noindent
\begin{equation*}\label{eq:bn}
\begin{split}
\underbrace{1\ldots 1}_{0\leq x_{1}\leq k}0\underbrace{1\ldots 1}_{0\leq x_{2}\leq k}0\underbrace{1\ldots 1}_{0\leq x_{3}\leq k}0 \ldots 0\underbrace{1\ldots 1}_{0\leq x_{y}\leq k}0\underbrace{1\ldots 1}_{0\leq x_{y+1}\leq k},\\
\end{split}
\end{equation*}
\noindent
with $y$ $0$'s and $n-y$ $1$'s, where $x_{j}$ $(j=1,\ldots,y+1$ represents a length of run of $1$'s . And all integers $x_{1},\ldots,x_{s}$ satisfy the conditions
\noindent
\begin{equation*}\label{eq:bn}
\begin{split}
0 \leq x_j \leq k \mbox{\ for\ } j=1,...,y+1,\  \mbox{\ and}\ x_1+\cdots +x_{y+1} = n-y,
\end{split}
\end{equation*}
\noindent
Then the probability of the event $\left\{L_{n}^{(1)}\leq k\ \wedge\ F_{n}=y\right\}$ is given by
\noindent
\begin{equation*}\label{eq: 1.1}
\begin{split}
P_{q,\theta}\Big(L_{n}^{(1)}\leq k\ \wedge\ F_{n}=y\Big)&\\
=\sum_{\substack{x_1+\cdots +x_{y+1} = n-y\\ x_{1},\ldots,x_{y+1} \in \{1,\ldots,k\}}}&\Big(\theta q^{0}\Big)^{x_{1}}\Big(1-\theta q^{0}\Big)\Big(\theta q^{1}\Big)^{x_{2}}\Big(1-\theta q^{1}\Big)\times\\
&\quad \quad \quad \quad\quad \quad\quad \quad \quad  \vdots\\
& \Big(1-\theta q^{y-2}\Big)\Big(\theta q^{y-1}\Big)^{x_{y}} \Big(1-\theta q^{y-1}\Big)\Big(\theta q^{y}\Big)^{x_{y+1}}.\\
\end{split}
\end{equation*}
\noindent
Using simple exponentiation algebra arguments to simplify,
\noindent
\begin{equation*}\label{eq: 1.1}
\begin{split}
P_{q,\theta}\Big(&L_{n}^{(1)}\leq k\ \wedge\ F_{n}=y\Big)=\\
&\theta^{n-y}\ {\prod_{j=1}^{y}}\ \left(1-\theta q^{j-1}\right)\ \sum_{\substack{x_1+\cdots +x_{y+1} = n-y\\ x_{1},\ldots,x_{y+1} \in \{0,1,\ldots,k\}}}q^{x_{2}+2y_{3}\cdots+(y-1)x_{y}+yx_{y+1}}.\\
\end{split}
\end{equation*}
\noindent
Using Lemma \ref{lemma3.2-n}, we can rewrite as follows.
\noindent
\begin{equation*}\label{eq: 1.1}
\begin{split}
P_{q,\theta}\Big(L_{n}^{(1)}\leq k\ \wedge\ F_{n}=y\Big)=\theta^{n-y}\ \left(1-\theta q^{j-1}\right)\  V_{q}(y+1,\ n-y,\ k),\\
\end{split}
\end{equation*}\\
\noindent
where
\noindent
\begin{equation*}\label{eq: 1.1}
\begin{split}
V_{q}(y+1,\ n-y,\ k)=\sum_{\substack{x_1+\cdots +x_{y+1} = n-y\\ x_{1},\ldots,x_{y+1} \in \{0,1,\ldots,k\}}}q^{x_{2}+2y_{3}\cdots+(y-1)x_{y}+yx_{y+1}}.
\end{split}
\end{equation*}
\noindent
Therefore we can compute the probability of the event $\left\{L_{n}^{(1)}\leq k\right\}$ as follows.
\noindent
\begin{equation*}\label{eq:bn1}
\begin{split}
P_{q,\theta}\Big(L_{n}^{(1)}\leq k\Big)=&\sum_{y=[n/(k+1)]}^{n}\theta^{n-y}\ {\prod_{j=1}^{y}}\ \left(1-\theta q^{j-1}\right)\ V_{q}(y+1,\ n-y,\ k).\\
\end{split}
\end{equation*}
\noindent
Thus proof is completed.
\end{proof}
\noindent
It is worth mentioning here that the PMF $P_{q,\theta}\Big(L_{n}^{(1)}\leq k\Big)$ approaches the probability function of $P_{\theta}\Big(L_{n}^{(1)}\leq k\Big)$ in the limit as $q$ tends to 1 of IID model. The details are presented in the following remark.
\noindent
\begin{remark}
{\rm For $q=1$, the PMF $P_{q,\theta}\Big(L_{n}^{(1)}\leq k\Big)$ reduces to the PMF $P_{\theta}\Big(L_{n}^{(1)}\leq k\Big)$ is given by
\noindent
\begin{equation*}\label{eq:bn1}
\begin{split}
P_{\theta}\Big(L_{n}^{(1)}\leq k\Big)=&\sum_{y=[n/(k+1)]}^{n}\theta^{n-y}(1-\theta)^{y}\ V_{1}(y+1,\ n-y,\ k)\\
=&\sum_{y=[n/(k+1)]}^{n}\theta^{n-y}(1-\theta)^{y}\ C(n-y,\ y+1,\ k),
\end{split}
\end{equation*}
\noindent
where
\noindent
\begin{equation*}\label{eq: 1.1}
\begin{split}
C(a,\ b,\ c)=\sum_{j=0}^{b}(-1)^{j}{b \choose j}{a-(c+1)j+b-1 \choose b-1}.
\end{split}
\end{equation*}
\noindent
if $a>0$: $1$ if $a=0$: and $0$ otherwise (See, e.g. \citet{charalambides2002enumerative}).
}
\end{remark}

\section{Joint distribution of the longest success and failure runs for $q$-binary trials}

%
%

In this section, we study the joint distribution of the lengths of the longest success and longest failure runs.
Let $\{X_{i}\}_{i\geq 1}$ be an arbitrary sequence of $n$ binary (zero and one) trials ordered on a line, with probability of ones varying according to a geometric rule. We consider that an ordered sequence of $n$ trials which consist of $n_1$ of the total number of successes and $n_2$ of the total number of failures, where $n_1+n_2=n$. We consider these $n$ trials how they are arranged to form a random sequence and we shall discuss various aspects of the representation of successes and failures trials which arise. We will distribute the $n_1$ successes and the $n_2$ failures. We first consider the failure runs. Let $s$$(1\leq s \leq n_2)$ be the number of runs of failures and then $n_1$ failures can be arranged to form $s$ runs. It is true that whenever there are $s$ runs of failures there are three alternatives for the successes :

\begin{itemize}
\item there are $s-1$ runs of successes and the series can begin only with a run of failures (B-type);
\item there are $s$ runs of failures and the series can begin with a run of successes or a run of failures (D and A-type, respectively) and
\item there are $s+1$ runs of successes and the series can begin only with a run of successes (C-type).
\end{itemize}
Thus sequence $X_{1}$, $\ldots$,$X_{n}$ has one of the following four forms:

%

\noindent

\begin{equation*}\label{eq:bn}
\begin{split}
\text{A-type}\ :&\quad \overbrace{0\ldots 0}^{y_{1}}\mid\overbrace{1\ldots 1}^{x_{1}}\mid\overbrace{0\ldots 0}^{y_{2}}\mid\overbrace{1\ldots 1}^{x_{2}}\mid \ldots \mid \overbrace{0\ldots 0}^{y_{s}}\mid\overbrace{1\ldots 1}^{x_{s}},\\
\text{B-type}\ :&\quad \overbrace{0\ldots 0}^{y_{1}}\mid\overbrace{1\ldots 1}^{x_{1}}\mid\overbrace{0\ldots 0}^{y_{2}}\mid\overbrace{1\ldots 1}^{x_{2}}\mid \ldots \mid \overbrace{0\ldots 0}^{y_{s-1}}\mid\overbrace{1\ldots 1}^{x_{s-1}}\mid\overbrace{0\ldots 0}^{y_{s}},\\
\text{C-type}\ :&\quad \overbrace{1\ldots 1}^{x_{1}}\mid\overbrace{0\ldots 0}^{y_{1}}\mid\overbrace{1\ldots 1}^{x_{2}}\mid\overbrace{0\ldots 0}^{y_{2}}\mid \ldots \mid\overbrace{0\ldots 0}^{y_{s}}\mid\overbrace{1\ldots 1}^{x_{s+1}},\\
\text{D-type}\ :&\quad \overbrace{1\ldots 1}^{x_{1}}\mid\overbrace{0\ldots 0}^{y_{1}}\mid\overbrace{1\ldots 1}^{x_{2}}\mid\overbrace{0\ldots 0}^{y_{2}}\mid \ldots \mid \overbrace{0\ldots 0}^{y_{s-1}}\mid\overbrace{1\ldots 1}^{x_{s}}\mid\overbrace{0\ldots 0}^{y_{s}},\\
\end{split}
\end{equation*}
\noindent
It is noteworthy that the representations of A to D-type are based on the arguments given in the proof of Theorem 3.2.1 of \citet{gibbons2020nonparametric}. It is of interest to note that, for the first and fourth forms total number of success and failure runs are both equal and we have $(s-1)$ and $(s+1)$ success runs for the second and third forms, respectively. Here, we derive the joint distribution of the lengths of the longest success and longest failure runs. More specifically, we derive the joint distributions for $P_{q,\theta}\Big(L_{n}^{(1)}\leq k_{1}\ \wedge\ L_{n}^{(0)}\leq k_{2}\Big)$, $P_{q,\theta}\Big(L_{n}^{(1)}\leq k_{1}\ \wedge\ L_{n}^{(0)}\geq k_{2}\Big)$, $P_{q,\theta}\Big(L_{n}^{(1)}\geq k_{1}\ \wedge\ L_{n}^{(0)}\leq k_{2}\Big)$ and $P_{q,\theta}\Big(L_{n}^{(1)}\geq k_{1}\ \wedge\ L_{n}^{(0)}\geq k_{2}\Big)$.

\subsection{Closed expression for the PMF of $\left(L_{n}^{(1)}\leq k_{1}\ \wedge\ L_{n}^{(0)}\leq k_{2}\right)$}

\noindent
We shall study of the joint distribution of $\left(L_{n}^{(1)}\leq k_{1}\ \wedge\ L_{n}^{(0)}\leq k_{2}\right)$. The probability function of the  $P_{q,\theta}\Big(L_{n}^{(1)}\leq k_{1}\ \wedge\ L_{n}^{(0)}\leq k_{2}\Big)$ is obtained by the following theorem.
\noindent
\begin{theorem}
The PMF $P_{q,\theta}\Big(L_{n}^{(1)}\leq k_{1}\ \wedge\ L_{n}^{(0)}\leq k_{2}\Big)$ for $n\geq1$
\begin{equation*}\label{eq:bn1}
\begin{split}
P_{q,\theta}\Big(L_{n}^{(1)}\leq k_{1}\ \wedge\ L_{n}^{(0)}\leq k_{2}\Big)=&\sum_{i=1}^{n}\theta^{n-i}\ {\prod_{j=1}^{i}}\ (1-\theta q^{j-1})\sum_{s=1}^{i}\Big[D_{q}^{k_1+1,k_2+1}(n-i,\ i,\ s)\\
&+A_{q}^{k_1+1,k_2+1}(n-i,\ i,\ s)+C_{q}^{k_1+1,k_2+1}(n-i,\ i,\ s+1)\\
&+B_{q}^{k_1+1,k_2+1}(n-i,\ i,\ s)\Big]+\eta(n,k_1)\theta^n,
\end{split}
\end{equation*}
where $\eta(n,k_1)$ is the function by $\eta(n,k_1)=1$ if $n\leq k_1$ and $0$ otherwise.
\end{theorem}
\begin{proof}
We partition the event $\left\{L_{n}^{(1)}\leq k_{1}\ \wedge\ L_{n}^{(0)}\leq k_{2}\right\}$ into disjoint events given by $F_{n}=i,$ for $i=1, \ldots,n$. Adding the probabilities we have
\noindent
\begin{equation*}\label{eq:bn}
\begin{split}
P_{q,\theta}\Big(L_{n}^{(1)}\leq k_{1}\ \wedge\ L_{n}^{(0)}\leq k_{2}\Big)=\sum_{i=1}^{n}P_{q,\theta}\Big(L_{n}^{(1)}\leq k_{1}\ \wedge\ L_{n}^{(0)}\leq k_{2}\ \wedge\ F_{n}=i\Big).\\
\end{split}
\end{equation*}
\noindent
We will write $E_{n,i}=\left\{L_{n}^{(1)}\leq k_{1}\ \wedge\ L_{n}^{(0)}\leq k_{2}\ \wedge\ F_{n}=i\right\}$.
\noindent
We can now rewrite as follows
\noindent
\begin{equation*}\label{eq:bn1}
\begin{split}
P_{q,\theta}\Big(L_{n}^{(1)}\leq k_{1}\ \wedge\ L_{n}^{(0)}\leq k_{2}\Big)=\sum_{i=1}^{n}P_{q,\theta}\left(E_{n,i}\right).
\end{split}
\end{equation*}

We are going to focus on the event $E_{n,\ i}$. A typical element of the event $\left\{L_{n}^{(1)}\leq k_{1}\ \wedge\ L_{n}^{(0)}\leq k_{2}\right\}$ is an ordered sequence which consists of $n-i$ successes and $i$ failures such that the length of the longest success run is less than or equal to $k_1$. The number of these sequences can be derived as follows. First we will distribute the $i$ failures and the $n-i$ successes.  Let $s$ $(1\leq s \leq i)$ be the number of runs of failures in the event $E_{n,\ i}$. We divide into two cases: starting with a success run or starting with a failure run. Thus, we distinguish between four types of sequences in the event $\left\{L_{n}^{(1)}\leq k_{1}\ \wedge\ L_{n}^{(0)}\leq k_{2}\right\}$, respectively named A, B, C and D-type, which are defined as follows.

\begin{equation*}\label{eq:bn}
\begin{split}
\text{A-type}\ :&\quad \overbrace{0\ldots 0}^{y_{1}}\mid\overbrace{1\ldots 1}^{x_{1}}\mid\overbrace{0\ldots 0}^{y_{2}}\mid\overbrace{1\ldots 1}^{x_{2}}\mid \ldots \mid \overbrace{0\ldots 0}^{y_{s}}\mid\overbrace{1\ldots 1}^{x_{s}},\\
\end{split}
\end{equation*}
\noindent
with $i$ $0$'s and $n-i$ $1$'s, where $x_{j}$ $(j=1,\ldots,s)$ represents a length of run of $1$'s and $y_{j}$ $(j=1,\ldots,s)$ represents the length of a run of $0$'s. And all integers $x_{1},\ldots,x_{s}$, and $y_{1},\ldots,y_{s}$ satisfy the conditions
\noindent
\begin{equation*}\label{eq:bn}
\begin{split}
0 < x_j < k_1+1 \mbox{\ for\ } j=1,...,s, \mbox{\ and\ } x_1+\cdots +x_{s} = n-i,
\end{split}
\end{equation*}
\noindent
\begin{equation*}\label{eq:bn}
\begin{split}
0 < y_j < k_2+1 \mbox{\ for\ } j=1,...,s, \mbox{\ and\ } y_1+\cdots +y_s=i.
\end{split}
\end{equation*}

\begin{equation*}\label{eq:bn}
\begin{split}
\text{B-type}\ :&\quad \overbrace{0\ldots 0}^{y_{1}}\mid\overbrace{1\ldots 1}^{x_{1}}\mid\overbrace{0\ldots 0}^{y_{2}}\mid\overbrace{1\ldots 1}^{x_{2}}\mid \ldots \mid \overbrace{0\ldots 0}^{y_{s-1}}\mid\overbrace{1\ldots 1}^{x_{s-1}}\mid\overbrace{0\ldots 0}^{y_{s}},\\
\end{split}
\end{equation*}
\noindent
with $i$ $0$'s and $n-i$ $1$'s, where $x_{j}$ $(j=1,\ldots,s-1)$ represents a length of run of $1$'s and $y_{j}$ $(j=1,\ldots,s)$ represents the length of a run of $0$'s. And all integers $x_{1},\ldots,x_{s-1}$, and $y_{1},\ldots,y_{s}$ satisfy the conditions
\noindent
\begin{equation*}\label{eq:bn}
\begin{split}
0 < x_j < k_1+1 \mbox{\ for\ } j=1,...,s-1, \mbox{\ and\ } x_1+\cdots +x_{s-1} = n-i,
\end{split}
\end{equation*}
\noindent
\begin{equation*}\label{eq:bn}
\begin{split}
0 < y_j < k_2+1 \mbox{\ for\ } j=1,...,s, \mbox{\ and\ } y_1+\cdots +y_s=i.
\end{split}
\end{equation*}

\begin{equation*}\label{eq:bn}
\begin{split}
\text{C-type}\ :&\quad \overbrace{1\ldots 1}^{x_{1}}\mid\overbrace{0\ldots 0}^{y_{1}}\mid\overbrace{1\ldots 1}^{x_{2}}\mid\overbrace{0\ldots 0}^{y_{2}}\mid \ldots \mid\overbrace{0\ldots 0}^{y_{s}}\mid\overbrace{1\ldots 1}^{x_{s+1}},\\
\end{split}
\end{equation*}
\noindent
with $i$ $0$'s and $n-i$ $1$'s, where $x_{j}$ $(j=1,\ldots,s+1)$ represents a length of run of $1$'s and $y_{j}$ $(j=1,\ldots,s)$ represents the length of a run of $0$'s. And all integers $x_{1},\ldots,x_{s+1}$, and $y_{1},\ldots,y_{s}$ satisfy the conditions
\noindent
\begin{equation*}\label{eq:bn}
\begin{split}
0 < x_j < k_1+1 \mbox{\ for\ } j=1,...,s+1, \mbox{\ and\ } x_1+\cdots +x_{s+1} = n-i,
\end{split}
\end{equation*}
\noindent
\begin{equation*}\label{eq:bn}
\begin{split}
0 < y_j < k_2+1 \mbox{\ for\ } j=1,...,s, \mbox{\ and\ } y_1+\cdots +y_{s}=i.
\end{split}
\end{equation*}

\begin{equation*}\label{eq:bn}
\begin{split}
\text{D-type}\ :&\quad \overbrace{1\ldots 1}^{x_{1}}\mid\overbrace{0\ldots 0}^{y_{1}}\mid\overbrace{1\ldots 1}^{x_{2}}\mid\overbrace{0\ldots 0}^{y_{2}}\mid\overbrace{1\ldots 1}^{x_{3}}\mid \ldots \mid \overbrace{0\ldots 0}^{y_{s-1}}\mid\overbrace{1\ldots 1}^{x_{s}}\mid\overbrace{0\ldots 0}^{y_{s}},\\
\end{split}
\end{equation*}
\noindent
with $i$ $0$'s and $n-i$ $1$'s, where $x_{j}$ $(j=1,\ldots,s)$ represents a length of run of $1$'s and $y_{j}$ $(j=1,\ldots,s)$ represents the length of a run of $0$'s. And all integers $x_{1},\ldots,x_{s}$, and $y_{1},\ldots,y_{s}$ satisfy the conditions
\noindent
\begin{equation*}\label{eq:bn}
\begin{split}
0 < x_j < k_1+1 \mbox{\ for\ } j=1,...,s, \mbox{\ and\ } x_1+\cdots +x_{s} = n-i,
\end{split}
\end{equation*}
\noindent
\begin{equation*}\label{eq:bn}
\begin{split}
0 < y_j < k_2+1 \mbox{\ for\ } j=1,...,s, \mbox{\ and\ } y_1+\cdots +y_{s}=i.
\end{split}
\end{equation*}

Then the probability of the event $E_{n,\ i}$ is given by

\noindent
\begin{equation*}\label{eq: 1.1}
\begin{split}
P_{q,\theta}\Big(E_{n,\ i}\Big)&\\
=\sum_{s=1}^{i}\Bigg[\bigg\{&\sum_{\substack{x_{1}+\cdots+x_{s}=n-i\\ x_{1},\ldots,x_{s} \in \{1,\ldots,k_{1}\}}}{\hspace{0.3cm}\sum_{\substack{y_{1}+\cdots+y_{s}=i\\ y_{1},\ldots,y_{s} \in \{1,\ldots,k_{2}\}}}}\big(1-\theta q^{0}\big)\cdots \big(1-\theta q^{y_{1}-1}\big)\Big(\theta q^{y_{1}}\Big)^{x_{1}}\times\\
&\big(1-\theta q^{y_{1}}\big)\cdots \big(1-\theta q^{y_{1}+y_{2}-1}\big)\Big(\theta q^{y_{1}+y_{2}}\Big)^{x_{2}}\times\\
&\big(1-\theta q^{y_{1}+y_2}\big)\cdots \big(1-\theta q^{y_{1}+y_{2}+y_3-1}\big)\Big(\theta q^{y_{1}+y_{2}+y_3}\Big)^{x_{3}}\times \\
&\quad \quad \quad \quad\quad \quad\quad \quad \quad \quad \quad \vdots\\
&\big(1-\theta q^{y_{1}+\cdots+y_{s-1}}\big)\cdots \big(1-\theta q^{y_{1}+\cdots+y_{s}-1}\big)\Big(\theta q^{y_{1}+\cdots+y_{s}}\Big)^{x_{s}}\bigg\}+\\
\end{split}
\end{equation*}

\noindent
\begin{equation*}\label{eq: 1.1}
\begin{split}
\quad \bigg\{&\sum_{\substack{x_{1}+\cdots+x_{s-1}=n-i\\ x_{1},\ldots,x_{s-1} \in \{1,\ldots,k_{1}\}}}{\hspace{0.3cm}\sum_{\substack{y_{1}+\cdots+y_{s}=i\\ y_{1},\ldots,y_{s} \in \{1,\ldots,k_{2}\}}}} \big(1-\theta q^{0}\big)\cdots \big(1-\theta q^{y_{1}-1}\big)\times\\
&\Big(\theta q^{y_{1}}\Big)^{x_{1}}\big(1-\theta q^{y_{1}}\big)\cdots \big(1-\theta q^{y_{1}+y_{2}-1}\big)\times\\
&\Big(\theta q^{y_{1}+y_{2}}\Big)^{x_{2}}\big(1-\theta q^{y_{1}+y_2}\big)\cdots \big(1-\theta q^{y_{1}+y_{2}+y_3-1}\big)\times \\
&\quad \quad \quad \quad\quad \quad\quad \quad \quad \quad \quad \vdots\\
&\Big(\theta q^{y_{1}+\cdots+y_{s-1}}\Big)^{x_{s-1}}
\big(1-\theta q^{y_{1}+\cdots+y_{s-1}}\big)\cdots \big(1-\theta q^{y_{1}+\cdots+y_{s}-1}\big)\bigg\}+\\
\end{split}
\end{equation*}

\noindent
\begin{equation*}\label{eq: 1.1}
\begin{split}
 \quad\quad \quad \quad\bigg\{&\sum_{\substack{x_{1}+\cdots+x_{s+1}=n-i\\ x_{1},\ldots,x_{s+1} \in \{1,\ldots,k_{1}\}}}
{\hspace{0.3cm}\sum_{\substack{y_{1}+\cdots+y_{s}=i\\ y_{1},\ldots,y_{s} \in \{1,\ldots,k_{2}\}}}} \big(\theta q^{0}\big)^{x_{1}}\big(1-\theta q^{0}\big)\cdots (1-\theta q^{y_{1}-1})\times\\
&\Big(\theta q^{y_{1}}\Big)^{x_{2}}\big(1-\theta q^{y_{1}}\big)\cdots \big(1-\theta q^{y_{1}+y_{2}-1}\big)\times\\
&\Big(\theta q^{y_{1}+y_{2}}\Big)^{x_{3}}\big(1-\theta q^{y_{1}+y_2}\big)\cdots \big(1-\theta q^{y_{1}+y_{2}+y_3-1}\big)\times \\
&\quad \quad \quad \quad \vdots\\
&\Big(\theta q^{y_{1}+\cdots+y_{s}}\Big)^{x_{s+1}}\bigg\}+\\
\end{split}
\end{equation*}

\noindent
\begin{equation*}\label{eq: 1.1}
\begin{split}
\quad \quad \quad \quad\bigg\{&\sum_{\substack{x_{1}+\cdots+x_{s}=n-i\\ x_{1},\ldots,x_{s} \in \{1,\ldots,k_{1}\}}}{\hspace{0.3cm}\sum_{\substack{y_{1}+\cdots+y_{s}=i\\ y_{1},\ldots,y_{s} \in \{1,\ldots,k_{2}\}}}}\big(\theta q^{0}\big)^{x_{1}}\big(1-\theta q^{0}\big)\cdots (1-\theta q^{y_{1}-1})\times\\
&\Big(\theta q^{y_{1}}\Big)^{x_{2}}\big(1-\theta q^{y_{1}+y_2}\big)\cdots \big(1-\theta q^{y_{1}+y_{2}-1}\big)\times\\
&\Big(\theta q^{y_{1}+y_{2}}\Big)^{x_{3}}\big(1-\theta q^{y_{1}}\big)\cdots \big(1-\theta q^{y_{1}+y_{2}+y_3-1}\big)\times \\
&\quad \quad \quad \quad\quad \quad\quad \quad \quad \quad \quad \vdots\\
&\Big(\theta q^{y_{1}+\cdots+y_{s-1}}\Big)^{x_{s}}
\big(1-\theta q^{y_{1}+\cdots+y_{s-1}}\big)\cdots \big(1-\theta q^{y_{1}+\cdots+y_{s}-1}\big)\bigg\}\Bigg].\\
\end{split}
\end{equation*}
\noindent
Using simple exponentiation algebra arguments to simplify,
\noindent
\begin{equation*}\label{eq: 1.1}
\begin{split}
&P_{q,\theta}\Big(E_{n,i}\Big)=\\
&\theta^{n-i}{\prod_{j=1}^{i}}\ \left(1-\theta q^{j-1}\right)\times\\
&\sum_{s=1}^{i}\Bigg[\sum_{\substack{x_{1}+\cdots+x_{s}=n-i\\ x_{1},\ldots,x_{s} \in \{1,\ldots,k_{1}\}}}{\hspace{0.3cm}\sum_{\substack{y_{1}+\cdots+y_{s}=i\\ y_{1},\ldots,y_{s} \in \{1,\ldots,k_{2}\}}}}q^{y_{1}x_{1}+(y_{1}+y_{2})x_{2}+\cdots+(y_{1}+\cdots+y_{s})x_{s}}+\\
\end{split}
\end{equation*}

\noindent
\begin{equation*}\label{eq: 1.1}
\begin{split}
&\quad \quad \quad \sum_{\substack{x_{1}+\cdots+x_{s-1}=n-i\\ x_{1},\ldots,x_{s-1} \in \{1,\ldots,k_{1}\}}}{\hspace{0.3cm}\sum_{\substack{y_{1}+\cdots+y_{s}=i\\ y_{1},\ldots,y_{s} \in \{1,\ldots,k_{2}\}}}}q^{y_{1}x_{1}+(y_{1}+y_{2})x_{2}+\cdots+(y_{1}+\cdots+y_{s-1})x_{s-1}}+\\
&\quad \quad \sum_{\substack{x_{1}+\cdots+x_{s+1}=n-i\\ x_{1},\ldots,x_{s+1} \in \{1,\ldots,k_{1}\}}}{\hspace{0.3cm}\sum_{\substack{y_{1}+\cdots+y_{s}=i\\ y_{1},\ldots,y_{s} \in \{1,\ldots,k_{2}\}}}}q^{y_{1}x_{2}+(y_{1}+y_{2})x_{3}+\cdots+(y_{1}+\cdots+y_{s})x_{s+1}}+\\
&\quad \quad \sum_{\substack{x_{1}+\cdots+x_{s}=n-i\\ x_{1},\ldots,x_{s} \in \{1,\ldots,k_{1}\}}}
\hspace{0.3cm}\sum_{\substack{y_{1}+\cdots+y_{s}=i\\ y_{1},\ldots,y_{s} \in \{1,\ldots,k_{2}\}}}q^{y_{1}x_{2}+(y_{1}+y_{2})x_{3}+\cdots+(y_{1}+\cdots+y_{s-1})x_{s}}\Bigg].\\
\end{split}
\end{equation*}
\noindent
Using Lemma \ref{lemma:3.1}, Lemma \ref{lemma:3.2}, Lemma \ref{lemma:3.3} and Lemma \ref{lemma:3.4}, we can rewrite as follows.
\noindent
\begin{equation*}\label{eq: 1.1}
\begin{split}
P_{q,\theta}\Big(E_{n,i}\Big)=\theta^{n-i}\ {\prod_{j=1}^{i}}\ (1-\theta q^{j-1})&\sum_{s=1}^{i}\Big[D_{q}^{k_1+1,k_2+1}(n-i,\ i,\ s)+A_{q}^{k_1+1,k_2+1}(n-i,\ i,\ s)\\
&+C_{q}^{k_1+1,k_2+1}(n-i,\ i,\ s+1)+B_{q}^{k_1+1,k_2+1}(n-i,\ i,\ s)\Big],\\
\end{split}
\end{equation*}\\
\noindent
where
\noindent
\begin{equation*}\label{eq: 1.1}
\begin{split}
D_{q}^{k_1+1,k_2+1}(n&-i,\ i,\ s)\\
&=\sum_{\substack{x_{1}+\cdots+x_{s}=n-i\\ x_{1},\ldots,x_{s} \in \{1,\ldots,k_{1}\}}}\
\hspace{0.3cm}\sum_{\substack{y_{1}+\cdots+y_{s}=i\\ y_{1},\ldots,y_{s} \in \{1,\ldots,k_{2}\}}}q^{y_{1}x_{1}+(y_{1}+y_{2})x_{2}+\cdots+(y_{1}+\cdots+y_{s})x_{s}},
\end{split}
\end{equation*}
\noindent
\begin{equation*}\label{eq: 1.1}
\begin{split}
A_{q}^{k_1+1,k_2+1}(n&-i,\ i,\ s)\\
&=\sum_{\substack{x_{1}+\cdots+x_{s-1}=n-i\\ x_{1},\ldots,x_{s-1} \in \{1,\ldots,k_{1}\}}}
\hspace{0.3cm}\sum_{\substack{y_{1}+\cdots+y_{s}=i\\ y_{1},\ldots,y_{s} \in \{1,\ldots,k_{2}\}}} q^{y_{1}x_{1}+(y_{1}+y_{2})x_{2}+\cdots+(y_{1}+\cdots+y_{s-1})x_{s-1}},
\end{split}
\end{equation*}
\noindent
\begin{equation*}\label{eq: 1.1}
\begin{split}
C_{q}^{k_1+1,k_2+1}(n&-i,\ i,\ s+1)\\
&=\sum_{\substack{x_{1}+\cdots+x_{s+1}=n-i\\ x_{1},\ldots,x_{s+1} \in \{1,\ldots,k_{1}\}}}\
\hspace{0.3cm}\sum_{\substack{y_{1}+\cdots+y_{s}=i\\ y_{1},\ldots,y_{s} \in \{1,\ldots,k_{2}\}}}q^{y_{1}x_{2}+(y_{1}+y_{2})x_{3}+\cdots+(y_{1}+\cdots+y_{s})x_{s+1}},
\end{split}
\end{equation*}
\noindent
and
\noindent
\begin{equation*}\label{eq: 1.1}
\begin{split}
B_{q}^{k_1+1,k_2+1}(n&-i,\ i,\ s)\\
&=\sum_{\substack{x_{1}+\cdots+x_{s}=n-i\\ x_{1},\ldots,x_{s} \in \{1,\ldots,k_{1}\}}}\
\hspace{0.3cm}\sum_{\substack{y_{1}+\cdots+y_{s}=i\\ y_{1},\ldots,y_{s} \in \{1,\ldots,k_{2}\}}} q^{y_{1}x_{2}+(y_{1}+y_{2})x_{3}+\cdots+(y_{1}+\cdots+y_{s-1})x_{s}}.
\end{split}
\end{equation*}
\noindent
Therefore we can compute the probability of the event $\left\{L_{n}^{(1)}\leq k_{1}\ \wedge\ L_{n}^{(0)}\leq k_{2}\right\}$ as follows.
\noindent
\begin{equation*}\label{eq:bn1}
\begin{split}
P_{q,\theta}\Big(L_{n}^{(1)}\leq k_{1}\ \wedge\ L_{n}^{(0)}\leq k_{2}\Big)=&\sum_{i=1}^{n}\theta^{n-i}\ {\prod_{j=1}^{i}}\ (1-\theta q^{j-1})\sum_{s=1}^{i}\Big[D_{q}^{k_1+1,k_2+1}(n-i,\ i,\ s)\\
&+A_{q}^{k_1+1,k_2+1}(n-i,\ i,\ s)+C_{q}^{k_1+1,k_2+1}(n-i,\ i,\ s+1)\\
&+B_{q}^{k_1+1,k_2+1}(n-i,\ i,\ s)\Big].\\
\end{split}
\end{equation*}
\noindent
Thus proof is completed.
\end{proof}
\noindent
It is worth mentioning here that the PMF $P_{q,\theta}\Big(L_{n}^{(1)}\leq k_{1}\ \wedge\ L_{n}^{(0)}\leq k_{2}\Big)$ approaches the probability function of $P_{\theta}\Big(L_{n}^{(1)}\leq k_{1}\ \wedge\ L_{n}^{(0)}\leq k_{2}\Big)$ in the limit as $q$ tends to 1 of IID model. The details are presented in the following remark.
\noindent
\begin{remark}
{\rm For $q=1$, the PMF $P_{q,\theta}\Big(L_{n}^{(1)}\leq k_{1}\ \wedge\ L_{n}^{(0)}\leq k_{2}\Big)$ reduces to the PMF $P_{\theta}\Big(L_{n}^{(1)}\leq k_{1}\ \wedge\ L_{n}^{(0)}\leq k_{2}\Big)$ for $n\geq1$ is given by

\begin{equation*}\label{eq:bn1}
\begin{split}
P_{\theta}\Big(L_{n}^{(1)}\leq k_{1}\ \wedge\ L_{n}^{(0)}\leq k_{2}\Big)=&\sum_{i=1}^{n}\theta^{n-i}(1-\theta)^{i}\sum_{s=1}^{i}\Big[S(s,\ k_1+1,\ n-i)\\
&+S(s,\ k_1+1,\ n-i)+S(s+1,\ k_1+1,\ n-i)\\
&+S(s,\ k_1+1,\ n-i)\Big]S(s,\ k_2+1,\ i)+\eta(n,k_1)\theta^n,\\
\end{split}
\end{equation*}

\noindent
where $\ S(a,\ b,\ c)=\sum_{j=0}^{min(a,\left[\frac{c-a}{b-1}\right])}(-1)^{j}{a \choose j}{c-j(b-1)-1 \choose a-1}$, and $\eta(n,k_1)$ is the function by $\eta(n,k_1)=1$ if $n\leq k_1$ and $0$ otherwise.
}
\end{remark}

\subsection{Closed expression for the PMF of $\left(L_{n}^{(1)}\leq k_{1}\ \wedge\ L_{n}^{(0)}\geq k_{2}\right)$}
We shall study of the joint distribution of $\left(L_{n}^{(1)}\leq k_{1}\ \wedge\ L_{n}^{(0)}\geq k_{2}\right)$. We now make some useful Definition and Lemma for the proofs of Theorem in the sequel.
\noindent
\begin{definition}
For $0<q\leq1$, we define

\begin{equation*}\label{eq: 1.1}
\begin{split}
\overline{M}_{q,0,0}^{k_1,\infty}(u,v,s)={\sum_{x_{1},\ldots,x_{s}}}\
\sum_{y_{1},\ldots,y_{s}} q^{y_{1}x_{1}+(y_{1}+y_{2})x_{2}+\cdots+(y_{1}+\cdots+y_{s})x_{s}},\
\end{split}
\end{equation*}
\noindent
where the summation is over all integers $x_1,\ldots,x_{s},$ and $y_1,\ldots,y_{s}$ satisfying

\noindent
\begin{equation*}\label{eq:1}
\begin{split}
0 < x_j < k_1 \mbox{\ for\ } j=1,...,s,
\end{split}
\end{equation*}
\noindent
\begin{equation*}\label{eq:1}
\begin{split}
(x_1,\ldots,x_{s})\in S_{u,a}^{0},\ \text{and}
\end{split}
\end{equation*}
\noindent
\begin{equation*}\label{eq:2}
\begin{split}
(y_{1},\ldots,y_{s}) \in S_{v,s}^{0}.
\end{split}
\end{equation*}
\end{definition}
\noindent
The following gives a recurrence relation useful for the computation of $\overline{M}_{q,0,0}^{k_1,\infty}(u,v,s)$.
\noindent
\begin{lemma}
\label{lemma:6.1}
$\overline{M}_{q,0,0}^{k_1,\infty}(u,v,s)$ obeys the following recurrence relation.
\begin{equation*}\label{eq: 1.1}
\begin{split}
&\overline{M}_{q,0,0}^{k_1,\infty}(u,v,s)\\
&=\left\{
  \begin{array}{ll}
    \sum_{b=1}^{v-(s-1)}\sum_{a=1}^{k_1-1}q^{va}\overline{M}_{q,0,0}^{k_1,\infty}(u-a,v-b,s-1), & \text{for}\ s>1,\ s\leq u\leq s(k_1-1)\\
    &\text{and}\ s\leq v \\
    1, & \text{for}\ s=1,\ 1\leq u\leq k_1-1,\ \text{and}\  1\leq v\\
    0, & \text{otherwise.}\\
  \end{array}
\right.
\end{split}
\end{equation*}
\end{lemma}
\begin{proof}
For $s > 1$, $s\leq u\leq s(k_1-1)$ and $s\leq v$, we observe that $x_{s}$ may assume any value $1,\ldots,k_1-1$, then $\overline{M}_{q,0,0}^{k_1,\infty}(u,v,s)$ can be written as

\noindent
\begin{equation*}
\begin{split}
\overline{M}&_{q,0,0}^{k_1,\infty}(u,v,s)\\
=&\sum_{x_{s}=1}^{k_1-1}{\sum_{\substack{x_{1},\ldots,x_{s-1} \in \{1,\ldots,k_{1}-1\}\\(x_1,\ldots,x_{s-1})\in S_{u-x_{s},s-1}^{0}}}}\
{\sum_{\substack{(y_{1},\ldots,y_{s}) \in S_{v,s}^{0}}}}q^{vx_{s}}q^{y_{1}x_{1}+(y_{1}+y_{2})x_{2}+\cdots+(y_{1}+\cdots+y_{s-1})x_{s-1}}.
\end{split}
\end{equation*}
\noindent
Similarly, we observe that since $y_s$ can assume the values $1,\ldots,v-(s-1)$, then $\overline{M}_{q,0,0}^{k_1,\infty}(u,v,s)$ can be rewritten as
\noindent

\begin{equation*}
\begin{split}
\overline{M}&_{q,0,0}^{k_1,\infty}(u,v,s)\\
=&\sum_{y_{s}=1}^{v-(s-1)}\sum_{x_{s}=1}^{k_1-1}{\hspace{0.3cm}\sum_{\substack{x_{1},\ldots,x_{s-1} \in \{1,\ldots,k_{1}-1\}\\(x_1,\ldots,x_{s-1})\in S_{u-x_{s},s-1}^{0}}}}{\hspace{0.3cm}\sum_{\substack{(y_{1},\ldots,y_{s}) \in S_{v,s}^{0}}}}q^{vx_{s}}q^{y_{1}x_{1}+(y_{1}+y_{2})x_{2}+\cdots+(y_{1}+\cdots+y_{s-1})x_{s-1}}\\
=&\sum_{b=1}^{v-(s-1)}\sum_{a=1}^{k_1-1}q^{va}\overline{M}_{q,0,0}^{k_1,\infty}(u-a,v-b,s-1)
\end{split}
\end{equation*}
\noindent
The other cases are obvious and thus the proof is completed.
\end{proof}
\begin{remark}
{\rm
We observe that $\overline{M}_{1,0,0}^{k_1,\infty}(u,v,s)$ is the number of integer solutions $(x_{1},\ldots,x_{s})$ and $(y_{1},\ldots,y_{s})$ of
\noindent
\begin{equation*}\label{eq:1}
\begin{split}
0 < x_j < k_1 \mbox{\ for\ } j=1,...,s,
\end{split}
\end{equation*}
\noindent
\begin{equation*}\label{eq:1}
\begin{split}
(x_1,\ldots,x_{s})\in S_{u,s}^{0},\ \text{and}
\end{split}
\end{equation*}
\noindent
\begin{equation*}\label{eq:2}
\begin{split}
(y_{1},\ldots,y_{s}) \in S_{v,s}^{0}.
\end{split}
\end{equation*}
which is
\noindent
\begin{equation*}\label{eq: 1.1}
\begin{split}
\overline{M}_{1,0,0}^{k_1,\infty}(u,v,s)=S(s,\ k_{1},\ u)M(s,\ v ),
\end{split}
\end{equation*}
\noindent
where $S(a,\ b,\ c)$ denotes the total number of integer solutions $x_{1}+x_{2}+\cdots+x_{a}=c$ such that $0<x_{i}<b$ for $i=1,2,\ldots,a$. The number can be expressed as
\noindent
\begin{equation*}\label{eq: 1.1}
\begin{split}
S(a,\ b,\ c)=\sum_{j=0}^{min(a,\left[\frac{c-a}{b-1}\right])}(-1)^{j}{a \choose j}{c-j(b-1)-1 \choose a-1}.
\end{split}
\end{equation*}
\noindent

\noindent
where $M(a,\ b)$ denotes the total number of integer solution $y_{1}+y_{2}+\cdots+y_{a}=b$ such that $(y_{1},\ldots,y_{a}) \in S_{b,a}^{0}$. The number is given by
\noindent
\begin{equation*}\label{eq: 1.1}
\begin{split}
M(a,\ b)={b-1 \choose a-1}.
\end{split}
\end{equation*}
\noindent
See, e.g. \citet{charalambides2002enumerative}.
}
\end{remark}
\begin{definition}
For $0<q\leq1$, we define

\begin{equation*}\label{eq: 1.1}
\begin{split}
M_{q,0,k_2}^{k_1,\infty}(m,r,s)={\sum_{x_{1},\ldots,x_{s}}}\
\sum_{y_{1},\ldots,y_{s}} q^{y_{1}x_{1}+(y_{1}+y_{2})x_{2}+\cdots+(y_{1}+\cdots+y_{s})x_{s}},\
\end{split}
\end{equation*}
\noindent
where the summation is over all integers $x_1,\ldots,x_{s},$ and $y_1,\ldots,y_{s}$ satisfying

\noindent
\begin{equation*}\label{eq:1}
\begin{split}
0 < x_j < k_1 \mbox{\ for\ } j=1,...,s,
\end{split}
\end{equation*}
\noindent
\begin{equation*}\label{eq:1}
\begin{split}
(x_1,\ldots,x_{s})\in S_{u,s}^{0},\ \text{and}
\end{split}
\end{equation*}
\noindent
\begin{equation*}\label{eq:2}
\begin{split}
(y_{1},\ldots,y_{s}) \in S_{v,s}^{k_2-1}.
\end{split}
\end{equation*}
\end{definition}
\noindent
The following gives a recurrence relation useful for the computation of $M_{q,0,k_2}^{k_1,\infty}(m,r,s)$.
\noindent
\begin{lemma}
\label{lemma:6.2}
$M_{q,0,k_2}^{k_1,\infty}(m,r,s)$ obeys the following recurrence relation.
\begin{equation*}\label{eq: 1.1}
\begin{split}
M&_{q,0,k_2}^{k_1,\infty}(m,r,s)\\
&=\left\{
  \begin{array}{ll}
    \sum_{b=1}^{k_2-1}\sum_{a=1}^{k_{1}-1}q^{ra}M_{q,0,k_2}^{k_1,\infty}(m-a,r-b,s-1)\\
    +\sum_{b=k_2}^{r-(s-1)}\sum_{a=1}^{k_1-1}q^{ra}\overline{M}_{q}(m-a,r-b,s-1), & \text{for}\ s>1,\ s\leq m\leq s(k_1-1)\\
    &\text{and}\ (s-1)+k_2\leq r \\
    1, & \text{for}\ s=1,\ 1\leq m\leq k_1-1,\ \text{and}\ k_2\leq r \\
    0, & \text{otherwise.}\\
  \end{array}
\right.
\end{split}
\end{equation*}
\end{lemma}
\begin{proof}
For $s > 1$, $s\leq m\leq s(k_1-1)$ and $(s-1)+k_2\leq r$, we observe that $x_{s}$ may assume any value $1,\ldots,k_1-1$, then $M_{q,0,k_2}^{k_1,\infty}(m,r,s)$ can be written as

\noindent
\begin{equation*}
\begin{split}
M&_{q,0,k_2}^{k_1,\infty}(m,r,s)\\
=&\sum_{x_{s}=1}^{k_{1}-1}{\hspace{0.3cm}\sum_{\substack{x_{1},\ldots,x_{s-1} \in \{1,\ldots,k_{1}-1\}\\(x_1,\ldots,x_{s-1})\in S_{m-x_{s},s-1}^{0}}}}{\hspace{0.3cm}\sum_{\substack{(y_{1},\ldots,y_{s}) \in S_{r,s}^{k_2-1}}}}q^{rx_{s}}q^{y_{1}x_{1}+(y_{1}+y_{2})x_{2}+\cdots+(y_{1}+\cdots+y_{s-1})x_{s-1}}.
\end{split}
\end{equation*}
\noindent
Similarly, we observe that since $y_s$ can assume the values $1,\ldots,r-(s-1)$, then $M_{q,0,k_2}^{k_1,\infty}(m,r,s)$ can be rewritten as
\noindent

\begin{equation*}
\begin{split}
M&_{q,0,k_2}^{k_1,\infty}(m,r,s)\\
=&\sum_{y_{s}=1}^{k_2-1}\sum_{x_{s}=1}^{k_{1}-1}{\hspace{0.3cm}\sum_{\substack{x_{1},\ldots,x_{s-1} \in \{1,\ldots,k_{1}-1\}\\(x_1,\ldots,x_{s-1})\in S_{m-x_{s},s-1}^{0}}}}{\hspace{0.3cm}\sum_{\substack{(y_{1},\ldots,y_{s-1}) \in S_{r-y_{s},s-1}^{k_2-1}}}}q^{rx_{s}}q^{y_{1}x_{1}+(y_{1}+y_{2})x_{2}+\cdots+(y_{1}+\cdots+y_{s-1})x_{s-1}}\\
&+\sum_{y_{s}=k_2}^{r-(s-1)}\sum_{x_{s}=1}^{k_{1}-1}{\hspace{0.3cm}\sum_{\substack{x_{1},\ldots,x_{s-1} \in \{1,\ldots,k_{1}-1\}\\(x_1,\ldots,x_{s-1})\in S_{m-x_{s},s-1}^{0}}}}{\hspace{0.3cm}\sum_{\substack{(y_{1},\ldots,y_{s-1}) \in S_{r-y_{s},s-1}^{0}}}}\hspace{-0.5cm} q^{rx_{s}}q^{y_{1}x_{1}+(y_{1}+y_{2})x_{2}+\cdots+(y_{1}+\cdots+y_{s-1})x_{s-1}}\\
=&\sum_{b=1}^{k_2-1}\sum_{a=1}^{k_{1}-1}q^{ra}M_{q,0,k_2}^{k_1,\infty}(m-a,r-b,s-1)+\sum_{b=k_2}^{r-(s-1)}\sum_{a=1}^{k_1-1}q^{ra}\overline{M}_{q,0,0}^{k_1,\infty}(m-a,r-b,s-1).
\end{split}
\end{equation*}
\noindent
The other cases are obvious and thus the proof is completed.
\end{proof}

\begin{remark}
{\rm
We observe that $M_{1,0,k_2}^{k_1,\infty}(m,r,s)$ is the number of integer solutions $(x_{1},\ldots,x_{s})$ and $(y_{1},\ldots,y_{s})$ of

\noindent
\begin{equation*}\label{eq:1}
\begin{split}
0 < x_j < k_1 \mbox{\ for\ } j=1,...,s,
\end{split}
\end{equation*}
\noindent
\begin{equation*}\label{eq:1}
\begin{split}
(x_1,\ldots,x_{s})\in S_{m,s}^{0},\ \text{and}
\end{split}
\end{equation*}
\noindent
\begin{equation*}\label{eq:2}
\begin{split}
(y_{1},\ldots,y_{s}) \in S_{r,s}^{k_2-1}.
\end{split}
\end{equation*}
\noindent
given by
\noindent
\begin{equation*}\label{eq: 1.1}
\begin{split}
M_{1,0,k_2}^{k_1,\infty}(m,r,s)=S(s,\ k_{1},\ m)R(s,\ k_{2},\ r),
\end{split}
\end{equation*}
\noindent
where $S(a,\ b,\ c)$ denotes the total number of integer solution $x_{1}+x_{2}+\cdots+x_{a}=c$ such that $0<x_{i}<b$ for $i=1,2,\ldots,a$. The number is given by
\noindent
\begin{equation*}\label{eq: 1.1}
\begin{split}
S(a,\ b,\ c)=\sum_{j=0}^{min\left(a,\ \left[\frac{c-a}{b-1}\right]\right) }(-1)^{j}{a \choose j}{c-j(b-1)-1 \choose a-1}.
\end{split}
\end{equation*}
\noindent
where $R(a,\ b,\ c)$ denotes the total number of integer solution $y_{1}+x_{2}+\cdots+y_{a}=c$ such that $(y_{1},\ldots,y_{a})\in S_{r,a}^{k_2}$. The number is given by
\noindent
\begin{equation*}\label{eq: 1.1}
\begin{split}
R(a,\ b,\ c)=\sum_{j=1}^{min\left(a,\ \left[\frac{c-a}{b-1}\right]\right) }(-1)^{j+1}{a \choose j}{c-j(b-1)-1 \choose a-1}.
\end{split}
\end{equation*}
\noindent
See, e.g. \citet{charalambides2002enumerative}.
}
\end{remark}

\begin{definition}
For $0<q\leq1$, we define

\begin{equation*}\label{eq: 1.1}
\begin{split}
\overline{N}_{q,0,0}^{k_1,\infty}(u,v,s)={\sum_{x_{1},\ldots,x_{s}}}\
\sum_{y_{1},\ldots,y_{s-1}} q^{y_{1}x_{2}+(y_{1}+y_{2})x_{3}+\cdots+(y_{1}+\cdots+y_{s-1})x_{s}},\
\end{split}
\end{equation*}
\noindent
where the summation is over all integers $x_1,\ldots,x_{s},$ and $y_1,\ldots,y_{s-1}$ satisfying
\noindent
\begin{equation*}\label{eq:1}
\begin{split}
0 < x_j < k_1 \mbox{\ for\ } j=1,...,s,
\end{split}
\end{equation*}
\noindent
\begin{equation*}\label{eq:1}
\begin{split}
(x_1,\ldots,x_{s})\in S_{u,s}^{0},\ \text{and}
\end{split}
\end{equation*}
\noindent
\begin{equation*}\label{eq:2}
\begin{split}
(y_{1},\ldots,y_{s-1}) \in S_{v,s-1}^{0}.
\end{split}
\end{equation*}
\end{definition}
\noindent
The following gives a recurrence relation useful for the computation of $\overline{N}_{q,0,0}^{k_1,\infty}(u,v,s)$.
\noindent
\begin{lemma}
\label{lemma:6.3}
$\overline{N}_{q,0,0}^{k_1,\infty}(u,v,s)$ obeys the following recurrence relation.
\begin{equation*}\label{eq: 1.1}
\begin{split}
\overline{N}&_{q,0,0}^{k_1,\infty}(u,v,s)\\
&=\left\{
  \begin{array}{ll}
    \sum_{b=1}^{v-(s-2)}\sum_{a=1}^{k_1-1}q^{va}\overline{N}_{q,0,0}^{k_1,\infty}(u-a,v-b,s-1), & \text{for}\ s>1,\ s\leq u\leq s(k_1-1)\\
    &\text{and}\ s-1 \leq v \\
    1, & \text{for}\ s=1,\ 1\leq u\leq k_1-1,\ \text{and}\  v=0\\
    0, & \text{otherwise.}\\
  \end{array}
\right.
\end{split}
\end{equation*}
\end{lemma}
\begin{proof}
For $s > 1$, $s\leq u\leq s(k_1-1)$ and $s-1 \leq v$, we observe that $x_{s}$ may assume any value $1,\ldots,k_1-1$, then $\overline{N}_{q,0,0}^{k_1,\infty}(u,v,s)$ can be written as

\noindent
\begin{equation*}
\begin{split}
\overline{N}&_{q,0,0}^{k_1,\infty}(u,v,s)\\
=&\sum_{x_{s}=1}^{k_1-1}{\hspace{0.3cm}\sum_{\substack{x_{1},\ldots,x_{s-1} \in \{1,\ldots,k_{1}-1\}\\(x_1,\ldots,x_{s-1})\in S_{u-x_{s},s-1}^{0}}}}{\hspace{0.3cm}\sum_{\substack{(y_{1},\ldots,y_{s-1}) \in S_{v,s-1}^{0}}}}q^{vx_{s}}q^{y_{1}x_{2}+(y_{1}+y_{2})x_{3}+\cdots+(y_{1}+\cdots+y_{s-2})x_{s-1}}.
\end{split}
\end{equation*}
\noindent
Similarly, we observe that since $y_{s-1}$ can assume the values $1,\ldots,v-(s-2)$, then $\overline{N}_{q,0,0}^{k_1,\infty}(u,v,s)$ can be rewritten as
\noindent

\begin{equation*}
\begin{split}
\overline{N}&_{q,0,0}^{k_1,\infty}(u,v,s)\\
=&\sum_{y_{s-1}=1}^{v-(s-2)}\sum_{x_{s}=1}^{k_1-1}{\hspace{0.3cm}\sum_{\substack{x_{1},\ldots,x_{s-1} \in \{1,\ldots,k_{1}-1\}\\(x_1,\ldots,x_{s-1})\in S_{u-x_{s},s-1}^{0}}}}
{\hspace{0.3cm}\sum_{\substack{(y_{1},\ldots,y_{s-2}) \in S_{v-y_{s-1},s-2}^{0}}}}q^{vx_{s}} q^{y_{1}x_{2}+(y_{1}+y_{2})x_{3}+\cdots+(y_{1}+\cdots+y_{s-2})x_{s-1}}\\
=&\sum_{b=1}^{v-(s-2)}\sum_{a=1}^{k_1-1}q^{va}\overline{N}_{q,0,0}^{k_1,\infty}(u-a,v-b,s-1)
\end{split}
\end{equation*}
\noindent
The other cases are obvious and thus the proof is completed.
\end{proof}
\begin{remark}
{\rm
We observe that $\overline{N}_{1,0,0}^{k_1,\infty}(u,v,s)$ corresponds to the number of integer solutions $(x_{1},\ldots,x_{s})$ and $(y_{1},\ldots,y_{s-1})$ satisfying
\noindent
\begin{equation*}\label{eq:1}
\begin{split}
0 < x_j < k_1 \mbox{\ for\ } j=1,...,s,
\end{split}
\end{equation*}
\noindent
\begin{equation*}\label{eq:1}
\begin{split}
(x_1,\ldots,x_{s})\in S_{u,s}^{0},\ \text{and}
\end{split}
\end{equation*}
\noindent
\begin{equation*}\label{eq:2}
\begin{split}
(y_{1},\ldots,y_{s-1}) \in S_{v,s-1}^{0}
\end{split}
\end{equation*}
\noindent
given by
\noindent
\begin{equation*}\label{eq: 1.1}
\begin{split}
\overline{N}_{1,0,0}^{k_1,\infty}(u,v,s)=S(s,\ k_{1},\ u)M(s-1,\ v)
\end{split}
\end{equation*}

\noindent
where $S(a,\ b,\ c)$ denotes the total number of integer solutions $x_{1}+x_{2}+\cdots+x_{a}=c$ such that $0<x_{i}<b$ for $i=1,2,\ldots,a$. The number can be expressed as
\noindent
\begin{equation*}\label{eq: 1.1}
\begin{split}
S(a,\ b,\ c)=\sum_{j=0}^{min(a,\left[\frac{c-a}{b-1}\right])}(-1)^{j}{a \choose j}{c-j(b-1)-1 \choose a-1}.
\end{split}
\end{equation*}

\noindent
where $M(a,\ b)$ denotes the total number of integer solution $y_{1}+y_{2}+\cdots+y_{a}=b$ such that $(y_{1},\ldots,y_{a}) \in S_{b,a}^{0}$. The number is given by
\noindent
\begin{equation*}\label{eq: 1.1}
\begin{split}
M(a,\ b)={b-1 \choose a-1}.
\end{split}
\end{equation*}
\noindent
See, e.g. \citet{charalambides2002enumerative}.
}
\end{remark}
\begin{definition}
For $0<q\leq1$, we define

\begin{equation*}\label{eq: 1.1}
\begin{split}
N_{q,0,k_2}^{k_1,\infty}(m,r,s)={\sum_{x_{1},\ldots,x_{s}}}\
\sum_{y_{1},\ldots,y_{s-1}} q^{y_{1}x_{2}+(y_{1}+y_{2})x_{3}+\cdots+(y_{1}+\cdots+y_{s-1})x_{s}},\
\end{split}
\end{equation*}
\noindent
where the summation is over all integers $x_1,\ldots,x_{s},$ and $y_1,\ldots,y_{s-1}$ satisfying

\noindent
\begin{equation*}\label{eq:1}
\begin{split}
0 < x_j < k_1 \mbox{\ for\ } j=1,...,s,
\end{split}
\end{equation*}
\noindent
\begin{equation*}\label{eq:1}
\begin{split}
(x_1,\ldots,x_{s})\in S_{m,s}^{0},\ \text{and}
\end{split}
\end{equation*}
\noindent
\begin{equation*}\label{eq:2}
\begin{split}
(y_{1},\ldots,y_{s-1}) \in S_{r,s-1}^{k_2-1}
\end{split}
\end{equation*}
\noindent
\end{definition}
\noindent
The following gives a recurrence relation useful for the computation of $N_{q,0,k_2}^{k_1,\infty}(m,r,s)$.
\noindent
\begin{lemma}
\label{lemma:6.4}
$N_{q,0,k_2}^{k_1,\infty}(m,r,s)$ obeys the following recurrence relation.
\begin{equation*}\label{eq: 1.1}
\begin{split}
N&_{q,0,k_2}^{k_1,\infty}(m,r,s)\\
&=\left\{
  \begin{array}{ll}
    \sum_{b=1}^{k_2-1}\sum_{a=1}^{k_{1}-1}q^{ra}N_{q,0,k_2}^{k_1,\infty}(m-a,r-b,s-1)\\
    +\sum_{b=k_2}^{r-(s-2)}\sum_{a=1}^{k_1-1}q^{ra}\overline{N}_{q,0,0}^{k_1,\infty}(m-a,r-b,s-1), & \text{for}\ s>1,\ s\leq m\leq s(k_1-1)\\
    &\text{and}\ (s-2)+k_2\leq r \\
    0, & \text{otherwise.}\\
  \end{array}
\right.
\end{split}
\end{equation*}
\end{lemma}
\begin{proof}
For $s > 1$, $s\leq m\leq s(k_1-1)$ and $(s-2)+k_2\leq r$, we observe that $x_{s}$ may assume any value $1,\ldots,k_1-1$, then $N_{q,0,k_2}^{k_1,\infty}(m,r,s)$ can be written as

\noindent
\begin{equation*}
\begin{split}
N&_{q,0,k_2}^{k_1,\infty}(m,r,s)\\
=&\sum_{x_{s}=1}^{k_{1}-1}{\hspace{0.3cm}\sum_{\substack{x_{1},\ldots,x_{s-1} \in \{1,\ldots,k_{1}-1\}\\(x_1,\ldots,x_{s-1})\in S_{m-x_{s},s-1}^{0}}}}{\hspace{0.3cm}\sum_{\substack{(y_{1},\ldots,y_{s-1}) \in S_{r,s-1}^{k_2-1}}}}q^{rx_{s}}q^{y_{1}x_{2}+(y_{1}+y_{2})x_{3}+\cdots+(y_{1}+\cdots+y_{s-2})x_{s-1}}.
\end{split}
\end{equation*}
\noindent
Similarly, we observe that since $y_{s-1}$ can assume the values $1,\ldots,r-(s-2)$, then $N_{q,0,k_2}^{k_1,\infty}(m,r,s)$ can be rewritten as
\noindent

\begin{equation*}
\begin{split}
N&_{q,0,k_2}^{k_1,\infty}(m,r,s)\\
=&\sum_{y_{s-1}=1}^{k_2-1}\sum_{x_{s}=1}^{k_{1}-1}{\hspace{0.3cm}\sum_{\substack{x_{1},\ldots,x_{s-1} \in \{1,\ldots,k_{1}-1\}\\(x_1,\ldots,x_{s-1})\in S_{m-x_{s},s-1}^{k_2-1}}}}{\hspace{0.3cm}\sum_{\substack{(y_{1},\ldots,y_{s-2}) \in S_{r-y_{s-1},s-2}^{0}}}}q^{rx_{s}}q^{y_{1}x_{2}+(y_{1}+y_{2})x_{3}+\cdots+(y_{1}+\cdots+y_{s-2})x_{s-1}}\\
&+\sum_{y_{s-1}=k_2}^{r-(s-2)}\sum_{x_{s}=k_{1}}^{k_1-1}{\hspace{0.3cm}\sum_{\substack{x_{1},\ldots,x_{s-1} \in \{1,\ldots,k_{1}-1\}\\(x_1,\ldots,x_{s-1})\in S_{m-x_{s},s-1}^{0}}}}{\hspace{0.3cm}\sum_{\substack{(y_{1},\ \ldots,\ y_{s-2}) \in S_{r-y_{s-1},s-2}^{0}}}}q^{rx_{s}}q^{y_{1}x_{2}+(y_{1}+y_{2})x_{3}+\cdots+(y_{1}+\cdots+y_{s-2})x_{s-1}}\\
=&\sum_{b=1}^{k_2-1}\sum_{a=1}^{k_{1}-1}q^{ra}N_{q,0,k_2}^{k_1,\infty}(m-a,r-b,s-1)+\sum_{b=k_2}^{r-(s-2)}\sum_{a=1}^{k_1-1}q^{ra}\overline{N}_{q,0,0}^{k_1,\infty}(m-a,r-b,s-1,k_1,k_2).
\end{split}
\end{equation*}
\noindent
The other cases are obvious and thus the proof is completed.
\end{proof}
\begin{remark}
{\rm
We observe that $N_{1,0,k_2}^{k_1,\infty}(m,r,s)$ is the number of integer solutions $(x_{1},\ldots,x_{s})$ and $(y_{1},\ldots,y_{s-1})$ of
\noindent
\begin{equation*}\label{eq:1}
\begin{split}
0 < x_j < k_1 \mbox{\ for\ } j=1,...,s,
\end{split}
\end{equation*}
\noindent
\begin{equation*}\label{eq:1}
\begin{split}
(x_1,\ldots,x_{s})\in S_{m,s}^{0},\ \text{and}
\end{split}
\end{equation*}
\noindent
\begin{equation*}\label{eq:2}
\begin{split}
(y_{1},\ldots,y_{s-1}) \in S_{r,s-1}^{k_2-1}
\end{split}
\end{equation*}
\noindent
given by
\noindent
\begin{equation*}\label{eq: 1.1}
\begin{split}
N_{1,0,k_2}^{k_1,\infty}(m,r,s)=S(s,\ k_{1},\ m)R(s-1,\ k_2,\ r)
\end{split}
\end{equation*}
\noindent
where $S(a,\ b,\ c)$ denotes the total number of integer solution $x_{1}+x_{2}+\cdots+x_{a}=c$ such that $0<x_{i}<b$ for $i=1,2,\ldots,a$. The number is given by
\noindent
\begin{equation*}\label{eq: 1.1}
\begin{split}
S(a,\ b,\ c)=\sum_{j=0}^{min\left(a,\ \left[\frac{c-a}{b-1}\right]\right) }(-1)^{j}{a \choose j}{c-j(b-1)-1 \choose a-1}.
\end{split}
\end{equation*}
\noindent
where $R(a,\ b,\ c)$ denotes the total number of integer solution $y_{1}+x_{2}+\cdots+y_{a}=c$ such that $(y_{1},\ldots,y_{a}) \in S_{r,a}^{k_2}$. The number is given by
\noindent
\begin{equation*}\label{eq: 1.1}
\begin{split}
R(a,\ b,\ c)=\sum_{j=1}^{min\left(a,\ \left[\frac{c-a}{b-1}\right]\right) }(-1)^{j+1}{a \choose j}{c-j(b-1)-1 \choose a-1}.
\end{split}
\end{equation*}
}
\end{remark}
\noindent
The probability function of the $P_{q,\theta}\Big(L_{n}^{(1)}\leq k_{1}\ \wedge\ L_{n}^{(0)}\geq k_{2}\Big)$ is obtained by the following theorem.
\noindent
\begin{theorem}
The PMF $P_{q,\theta}\Big(L_{n}^{(1)}\leq k_{1}\ \wedge\ L_{n}^{(0)}\geq k_{2}\Big)$ for $n\geq k_2$
\begin{equation*}\label{eq:bn1}
\begin{split}
P_{q,\theta}\Big(L_{n}^{(1)}\leq k_{1}\ \wedge\ L_{n}^{(0)}\geq k_{2}\Big)=&\sum_{i=k_2}^{n}\theta^{n-i}\ {\prod_{j=1}^{i}}\ (1-\theta q^{j-1})\sum_{s=1}^{i-k_2+1}\Big[M_{q,0,k_2}^{k_1+1,\infty}(n-i,\ i,\ s)\\
&+E_{q,0,k_2}^{k_1+1,\infty}(n-i,\ i,\ s)+N_{q,0,k_2}^{k_1+1,\infty}(n-i,\ i,\ s+1)\\
&+F_{q,0,k_2}^{k_1+1,\infty}(n-i,\ i,\ s)\Big].\\
\end{split}
\end{equation*}
\end{theorem}
\begin{proof}
We partition the event $\left\{L_{n}^{(1)}\leq k_{1}\ \wedge\ L_{n}^{(0)}\geq k_{2}\right\}$ into disjoint events given by $F_{n}=i,$ for $i=k_2,\ldots,n$. Adding the probabilities we have
\noindent
\begin{equation*}\label{eq:bn}
\begin{split}
P_{q,\theta}\Big(L_{n}^{(1)}\leq k_{1}\ \wedge\ L_{n}^{(0)}\geq k_{2}\Big)=\sum_{i=k_2}^{n}P\Big(L_{n}^{(1)}\leq k_{1}\ \wedge\ L_{n}^{(0)}\geq k_{2}\ \wedge\ F_{n}=i\Big).\\
\end{split}
\end{equation*}
\noindent
We will write $E_{n,i}=\left\{L_{n}^{(1)}\leq k_{1}\ \wedge\ L_{n}^{(0)}\geq k_{2}\ \wedge\ F_{n}=i\right\}$.
\noindent
We can now rewrite as follows
\noindent
\begin{equation*}\label{eq:bn1}
\begin{split}
P_{q,\theta}\Big(L_{n}^{(1)}\leq k_{1}\ \wedge\ L_{n}^{(0)}\geq k_{2}\Big)=\sum_{i=k_2}^{n}P_{q,\theta}\left(E_{n,i}\right).
\end{split}
\end{equation*}
\noindent
We are going to focus on the event $E_{n,\ i}$ which consists of $n-i$ successes and $i$ failures. First we will distribute the $i$ failures and the $n-i$ successes.  Let $s$ $(1\leq s \leq i-k_2+1)$ be the number of runs of failures in the event $E_{n,\ i}$. We distinguish between four types of sequences in the event $\left\{L_{n}^{(1)}\leq k_{1}\ \wedge\ L_{n}^{(0)}\geq k_{2}\ \wedge\ F_{n}=i\right\}$, respectively named A, B, C and D-type, which are defined as follows.
\noindent
\begin{equation*}\label{eq:bn}
\begin{split}
\text{A-type}\ :&\quad \overbrace{0\ldots 0}^{y_{1}}\mid\overbrace{1\ldots 1}^{x_{1}}\mid\overbrace{0\ldots 0}^{y_{2}}\mid\overbrace{1\ldots 1}^{x_{2}}\mid \ldots \mid \overbrace{0\ldots 0}^{y_{s}}\mid\overbrace{1\ldots 1}^{x_{s}},\\
\end{split}
\end{equation*}
\noindent
with $i$ $0$'s and $n-i$ $1$'s, where $x_{j}$ $(j=1,\ldots,s)$ represents a length of run of $1$'s and $y_{j}$ $(j=1,\ldots,s)$ represents the length of a run of $0$'s. And all integers $x_{1},\ldots,x_{s}$, and $y_{1},\ldots,y_{s}$ satisfy the conditions
\noindent
\begin{equation*}\label{eq:bn}
\begin{split}
0 < x_j < k_1+1 \mbox{\ for\ } j=1,...,s, \mbox{\ and\ } x_1+\cdots +x_{s} = n-i,
\end{split}
\end{equation*}
\noindent
\begin{equation*}\label{eq:bn}
\begin{split}
(y_{1},\ldots,y_{s}) \in S_{i,s}^{k_2-1}.
\end{split}
\end{equation*}
\begin{equation*}\label{eq:bn}
\begin{split}
\text{B-type}\ :&\quad \overbrace{0\ldots 0}^{y_{1}}\mid\overbrace{1\ldots 1}^{x_{1}}\mid\overbrace{0\ldots 0}^{y_{2}}\mid\overbrace{1\ldots 1}^{x_{2}}\mid \ldots \mid \overbrace{0\ldots 0}^{y_{s-1}}\mid\overbrace{1\ldots 1}^{x_{s-1}}\mid\overbrace{0\ldots 0}^{y_{s}},\\
\end{split}
\end{equation*}
\noindent
with $i$ $0$'s and $n-i$ $1$'s, where $x_{j}$ $(j=1,\ldots,s-1)$ represents a length of run of $1$'s and $y_{j}$ $(j=1,\ldots,s)$ represents the length of a run of $0$'s. And all integers $x_{1},\ldots,x_{s-1}$, and $y_{1},\ldots,y_{s}$ satisfy the conditions
\noindent
\begin{equation*}\label{eq:bn}
\begin{split}
0 < x_j < k_1+1 \mbox{\ for\ } j=1,...,s-1, \mbox{\ and\ } x_1+\cdots +x_{s-1} = n-i,
\end{split}
\end{equation*}
\noindent
\begin{equation*}\label{eq:bn}
\begin{split}
(y_{1},\ldots,y_{s}) \in S_{i,s}^{k_2-1}.
\end{split}
\end{equation*}

\begin{equation*}\label{eq:bn}
\begin{split}
\text{C-type}\ :&\quad \overbrace{1\ldots 1}^{x_{1}}\mid\overbrace{0\ldots 0}^{y_{1}}\mid\overbrace{1\ldots 1}^{x_{2}}\mid\overbrace{0\ldots 0}^{y_{2}}\mid \ldots \mid\overbrace{0\ldots 0}^{y_{s}}\mid\overbrace{1\ldots 1}^{x_{s+1}},\\
\end{split}
\end{equation*}
\noindent
with $i$ $0$'s and $n-i$ $1$'s, where $x_{j}$ $(j=1,\ldots,s+1)$ represents a length of run of $1$'s and $y_{j}$ $(j=1,\ldots,s)$ represents the length of a run of $0$'s. And all integers $x_{1},\ldots,x_{s+1}$, and $y_{1},\ldots,y_{s}$ satisfy the conditions
\noindent
\begin{equation*}\label{eq:bn}
\begin{split}
0 < x_j < k_1+1 \mbox{\ for\ } j=1,...,s+1, \mbox{\ and\ } x_1+\cdots +x_{s+1} = n-i,
\end{split}
\end{equation*}
\noindent
\begin{equation*}\label{eq:bn}
\begin{split}
(y_{1},\ldots,y_{s}) \in S_{i,s}^{k_2-1}.
\end{split}
\end{equation*}

\begin{equation*}\label{eq:bn}
\begin{split}
\text{D-type}\ :&\quad \overbrace{1\ldots 1}^{x_{1}}\mid\overbrace{0\ldots 0}^{y_{1}}\mid\overbrace{1\ldots 1}^{x_{2}}\mid\overbrace{0\ldots 0}^{y_{2}}\mid\overbrace{1\ldots 1}^{x_{3}}\mid \ldots \mid \overbrace{0\ldots 0}^{y_{s-1}}\mid\overbrace{1\ldots 1}^{x_{s}}\mid\overbrace{0\ldots 0}^{y_{s}},\\
\end{split}
\end{equation*}
\noindent
with $i$ $0$'s and $n-i$ $1$'s, where $x_{j}$ $(j=1,\ldots,s)$ represents a length of run of $1$'s and $y_{j}$ $(j=1,\ldots,s)$ represents the length of a run of $0$'s. And all integers $x_{1},\ldots,x_{s}$, and $y_{1},\ldots ,y_{s}$ satisfy the conditions
\noindent
\begin{equation*}\label{eq:bn}
\begin{split}
0 < x_j < k_1+1 \mbox{\ for\ } j=1,...,s, \mbox{\ and\ } x_1+\cdots +x_{s} = n-i, \text{and}
\end{split}
\end{equation*}
\noindent
\begin{equation*}\label{eq:bn}
\begin{split}
(y_{1},\ldots,y_{s}) \in S_{i,s}^{k_2-1}.
\end{split}
\end{equation*}
\noindent
Then the probability of the event $E_{n,\ i}$ is given by
\noindent
\begin{equation*}\label{eq: 1.1}
\begin{split}
P_{q,\theta}\Big(E_{n,\ i}\Big)\quad\ &\\
=\sum_{s=1}^{i-k_2+1}\Bigg[\bigg\{&\sum_{\substack{x_{1}+\cdots+x_{s}=n-i\\ x_{1},\ldots,x_{s} \in \{1,\ldots,k_{1}\}}}\
\sum_{\substack{(y_{1},\ldots,y_{s})\in S_{i,s}^{k_2-1}}}\big(1-\theta q^{0}\big)\cdots \big(1-\theta q^{y_{1}-1}\big)\Big(\theta q^{y_{1}}\Big)^{x_{1}}\times\\
&\big(1-\theta q^{y_{1}}\big)\cdots \big(1-\theta q^{y_{1}+y_{2}-1}\big)\Big(\theta q^{y_{1}+y_{2}}\Big)^{x_{2}}\times\\
&\big(1-\theta q^{y_{1}+y_2}\big)\cdots \big(1-\theta q^{y_{1}+y_{2}+y_3-1}\big)\Big(\theta q^{y_{1}+y_{2}+y_3}\Big)^{x_{3}}\times \\
&\quad \quad \quad \quad\quad \quad\quad \quad \quad \quad \quad \vdots\\
&\big(1-\theta q^{y_{1}+\cdots+y_{s-1}}\big)\cdots \big(1-\theta q^{y_{1}+\cdots+y_{s}-1}\big)\Big(\theta q^{y_{1}+\cdots+y_{s}}\Big)^{x_{s}}\bigg\}+\\
\end{split}
\end{equation*}

\noindent
\begin{equation*}\label{eq: 1.1}
\begin{split}
\quad \quad \quad \quad \quad \quad\bigg\{&\sum_{\substack{x_{1}+\cdots+x_{s-1}=n-i\\ x_{1},\ldots,x_{s-1} \in \{1,\ldots,k_{1}\}}}\
\sum_{\substack{(y_{1},\ldots,y_{s})\in S_{i,s}^{k_2-1}}} \big(1-\theta q^{0}\big)\cdots \big(1-\theta q^{y_{1}-1}\big)\times\\
&\Big(\theta q^{y_{1}}\Big)^{x_{1}}\big(1-\theta q^{y_{1}}\big)\cdots \big(1-\theta q^{y_{1}+y_{2}-1}\big)\times\\
&\Big(\theta q^{y_{1}+y_{2}}\Big)^{x_{2}}\big(1-\theta q^{y_{1}+y_2}\big)\cdots \big(1-\theta q^{y_{1}+y_{2}+y_3-1}\big)\times \\
&\quad \quad \quad \quad\quad \quad\quad \quad \quad \quad \quad \vdots\\
&\Big(\theta q^{y_{1}+\cdots+y_{s-1}}\Big)^{x_{s-1}}
\big(1-\theta q^{y_{1}+\cdots+y_{s-1}}\big)\cdots \big(1-\theta q^{y_{1}+\cdots+y_{s}-1}\big)\bigg\}+\\
\end{split}
\end{equation*}

\noindent
\begin{equation*}\label{eq: 1.1}
\begin{split}
\quad \quad \quad \quad \quad \quad&\bigg\{\sum_{\substack{x_{1}+\cdots+x_{s+1}=n-i\\ x_{1},\ldots,x_{s+1} \in \{1,\ldots,k_{1}\}}}\
\sum_{\substack{(y_{1},\ldots,y_{s}) \in S_{i,s}^{k_2-1}}} \big(\theta q^{0}\big)^{x_{1}}\big(1-\theta q^{0}\big)\cdots (1-\theta q^{y_{1}-1})\times\\
&\Big(\theta q^{y_{1}}\Big)^{x_{2}}\big(1-\theta q^{y_{1}}\big)\cdots \big(1-\theta q^{y_{1}+y_{2}-1}\big)\times\\
&\Big(\theta q^{y_{1}+y_{2}}\Big)^{x_{3}}\big(1-\theta q^{y_{1}+y_2}\big)\cdots \big(1-\theta q^{y_{1}+y_{2}+y_3-1}\big)\times \\
&\quad \quad \quad \quad \vdots\\
&\Big(\theta q^{y_{1}+\cdots+y_{s}}\Big)^{x_{s+1}}\bigg\}+\\
\end{split}
\end{equation*}

\noindent
\begin{equation*}\label{eq: 1.1}
\begin{split}
\quad \quad \quad \quad \quad \quad&\bigg\{\sum_{\substack{x_{1}+\cdots+x_{s}=n-i\\ x_{1},\ldots,x_{s} \in \{1,\ldots,k_{1}\}}}\
\sum_{\substack{(y_{1},\ldots,y_{s}) \in S_{i,s}^{k_2-1}}}\big(\theta q^{0}\big)^{x_{1}}\big(1-\theta q^{0}\big)\cdots (1-\theta q^{y_{1}-1})\times\\
&\Big(\theta q^{y_{1}}\Big)^{x_{2}}\big(1-\theta q^{y_{1}+y_2}\big)\cdots \big(1-\theta q^{y_{1}+y_{2}-1}\big)\times\\
&\Big(\theta q^{y_{1}+y_{2}}\Big)^{x_{3}}\big(1-\theta q^{y_{1}}\big)\cdots \big(1-\theta q^{y_{1}+y_{2}+y_3-1}\big)\times \\
&\quad \quad \quad \quad\quad \quad\quad \quad \quad \quad \quad \vdots\\
&\Big(\theta q^{y_{1}+\cdots+y_{s-1}}\Big)^{x_{s}}
\big(1-\theta q^{y_{1}+\cdots+y_{s-1}}\big)\cdots \big(1-\theta q^{y_{1}+\cdots+y_{s}-1}\big)\bigg\}\Bigg].\\
\end{split}
\end{equation*}

Using simple exponentiation algebra arguments to simplify,
\noindent
\begin{equation*}\label{eq: 1.1}
\begin{split}
&P_{q,\theta}\Big(E_{n,i}\Big)=\\
&\theta^{n-i}{\prod_{j=1}^{i}}\ (1-\theta q^{j-1})\times\\
&\sum_{s=1}^{i-k_2+1}\Bigg[\sum_{\substack{x_{1}+\cdots+x_{s}=n-i\\ x_{1},\ldots,x_{s} \in \{1,\ldots,k_{1}\}}}\
\sum_{\substack{(y_{1},\ldots,y_{s})\in S_{i,s}^{k_2-1}}}q^{y_{1}x_{1}+(y_{1}+y_{2})x_{2}+\cdots+(y_{1}+\cdots+y_{s})x_{s}}+\\
&\quad \quad\quad\sum_{\substack{x_{1}+\cdots+x_{s-1}=n-i\\ x_{1},\ldots,x_{s-1} \in \{1,\ldots,k_{1}\}}}\
\sum_{\substack{(y_{1},\ldots,y_{s})\in S_{i,s}^{k_2-1}}}q^{y_{1}x_{1}+(y_{1}+y_{2})x_{2}+\cdots+(y_{1}+\cdots+y_{s-1})x_{s-1}}+\\
&\quad \quad \quad\sum_{\substack{x_{1}+\cdots+x_{s+1}=n-i\\ x_{1},\ \ldots,\ x_{s+1} \in \{1,\ldots,k_{1}\}}}\
\sum_{\substack{(y_{1},\ldots,y_{s}) \in S_{i,s}^{k_2-1}}}q^{y_{1}x_{2}+(y_{1}+y_{2})x_{3}+\cdots+(y_{1}+\cdots+y_{s})x_{s+1}}+\\
&\quad \quad \quad \sum_{\substack{x_{1}+\cdots+x_{s}=n-i\\ x_{1},\ldots,x_{s} \in \{1,\ldots,k_{1}\}}}\
\sum_{\substack{(y_{1},\ldots,y_{s}) \in S_{i,s}^{k_2-1}}}q^{y_{1}x_{2}+(y_{1}+y_{2})x_{3}+\cdots+(y_{1}+\cdots+y_{s-1})x_{s}}\Bigg].\\
\end{split}
\end{equation*}
\noindent
Using Lemma \ref{lemma:3.6}, Lemma \ref{lemma:3.8}, Lemma \ref{lemma:6.2} and Lemma \ref{lemma:6.4}, we can rewrite as follows.
\noindent
\begin{equation*}\label{eq: 1.1}
\begin{split}
P_{q,\theta}\Big(E_{n,i}\Big)=&\theta^{n-i}\ {\prod_{j=1}^{i}}\ (1-\theta q^{j-1})\sum_{s=1}^{i-k_2+1}\Big[M_{q,0,k_2}^{k_1+1,\infty}(n-i,\ i,\ s)+E_{q,0,k_2}^{k_1+1,\infty}(n-i,\ i,\ s)\\
&+N_{q,0,k_2}^{k_1+1,\infty}(n-i,\ i,\ s+1)+F_{q,0,k_2}^{k_1+1,\infty}(n-i,\ i,\ s)\Big],\\
\end{split}
\end{equation*}\\
\noindent
where
\noindent
\begin{equation*}\label{eq: 1.1}
\begin{split}
M_{q,0,k_2}^{k_1+1,\infty}(n&-i,\ i,\ s)\\
&=\sum_{\substack{x_{1}+\cdots+x_{s}=n-i\\ x_{1},\ldots,x_{s} \in \{1,\ldots,k_{1}\}}}\
\sum_{\substack{(y_{1},\ldots,y_{s}) \in S_{i,s}^{k_2-1}}}q^{y_{1}x_{1}+(y_{1}+y_{2})x_{2}+\cdots+(y_{1}+\cdots+y_{s})x_{s}},
\end{split}
\end{equation*}
\noindent
\begin{equation*}\label{eq: 1.1}
\begin{split}
E_{q,0,k_2}^{k_1+1,\infty}(n&-i,\ i,\ s)\\
&=\sum_{\substack{x_{1}+\cdots+x_{s-1}=n-i\\ x_{1},\ldots,x_{s-1} \in \{1,\ldots,k_{1}\}}}\
\sum_{\substack{(y_{1},\ldots,y_{s})\in S_{i,s}^{k_2}}}q^{y_{1}x_{1}+(y_{1}+y_{2})x_{2}+\cdots+(y_{1}+\cdots+y_{s-1})x_{s-1}},
\end{split}
\end{equation*}
\noindent
\begin{equation*}\label{eq: 1.1}
\begin{split}
N_{q,0,k_2}^{k_1+1,\infty}(n&-i,\ i,\ s+1)\\
&=\sum_{\substack{x_{1}+\cdots+x_{s+1}=n-i\\ x_{1},\ldots,x_{s+1} \in \{1,\ldots,k_{1}\}}}\
\sum_{\substack{(y_{1},\ldots,y_{s}) \in S_{i,s}^{k_2-1}}}q^{y_{1}x_{2}+(y_{1}+y_{2})x_{3}+\cdots+(y_{1}+\cdots+y_{s})x_{s+1}},
\end{split}
\end{equation*}
\noindent
and
\noindent
\begin{equation*}\label{eq: 1.1}
\begin{split}
F_{q,0,k_2}^{k_1+1,\infty}(n&-i,\ i,\ s)\\
&=\sum_{\substack{x_{1}+\cdots+x_{s}=n-i\\ x_{1},\ldots,x_{s} \in \{1,\ldots,k_{1}\}}}\
\sum_{\substack{(y_{1},\ldots,y_{s})\in S_{i,s}^{k_2-1}}}q^{y_{1}x_{2}+(y_{1}+y_{2})x_{3}+\cdots+(y_{1}+\cdots+y_{s-1})x_{s}}.
\end{split}
\end{equation*}
\noindent
Therefore we can compute the probability of the event $\left\{L_{n}^{(1)}\leq k_{1}\ \wedge\ L_{n}^{(0)}\geq k_{2}\right\}$ as follows
\noindent
\begin{equation*}\label{eq:bn1}
\begin{split}
P_{q,\theta}\Big(L_{n}^{(1)}\leq k_{1}\ \wedge\ L_{n}^{(0)}\geq k_{2}\Big)=&\sum_{i=k_2}^{n}\theta^{n-i}\ {\prod_{j=1}^{i}}\ (1-\theta q^{j-1})\sum_{s=1}^{i-k_2+1}\Big[M_{q,0,k_2}^{k_1+1,\infty}(n-i,\ i,\ s)\\
&+E_{q,0,k_2}^{k_1+1,\infty}(n-i,\ i,\ s)+N_{q,0,k_2}^{k_1+1,\infty}(n-i,\ i,\ s+1)\\
&+F_{q,0,k_2}^{k_1+1,\infty}(n-i,\ i,\ s)\Big].\\
\end{split}
\end{equation*}
Thus proof is completed.
\noindent
\end{proof}
\noindent
It is worth mentioning here that the PMF $P_{q,\theta}\Big(L_{n}^{(1)}\leq k_{1}\ \wedge\ L_{n}^{(0)}\geq k_{2}\Big)$ approaches the probability function of $P_{\theta}\Big(L_{n}^{(1)}\leq k_{1}\ \wedge\ L_{n}^{(0)}\geq k_{2}\Big)$ for $n\geq1$ in the limit as $q$ tends to 1 of IID model. The details are presented in the following remark.
\noindent
\begin{remark}
{\rm For $q=1$, the PMF $P_{q,\theta}\Big(L_{n}^{(1)}\leq k_{1}\ \wedge\ L_{n}^{(0)}\geq k_{2}\Big)$ for $n\geq k_2$ reduces to the PMF $P_{\theta}\Big(L_{n}^{(1)}\leq k_{1}\ \wedge\ L_{n}^{(0)}\geq k_{2}\Big)$ for $n\geq1$ is given by
\noindent
\begin{equation*}\label{eq:bn1}
\begin{split}
P_{q,\theta}\Big(L_{n}^{(1)}\leq k_{1}\ \wedge\ L_{n}^{(0)}\geq k_{2}\Big)=&\sum_{i=k_2}^{n}\theta^{n-i}(1-\theta)^{i}\ \sum_{s=1}^{i-k_2+1}\Big[S(s,\ k_1+1,\ n-i)\\
&+S(s-1,\ k_1+1,\ n-i)+S(s+1,\ k_1+1,\ n-i)\\
&+S(s,\ k_1+1,\ n-i)\Big]R(s,\ k_2,\ i),\\
\end{split}
\end{equation*}

where

\noindent
\begin{equation*}\label{eq: 1.1}
\begin{split}
S(a,\ b,\ c)=\sum_{j=0}^{min(a,\left[\frac{c-a}{b-1}\right])}(-1)^{j}{a \choose j}{c-j(b-1)-1 \choose a-1}
\end{split}
\end{equation*}
\noindent
and
\noindent
\begin{equation*}\label{eq: 1.1}
\begin{split}
R(a,\ b,\ c)=\sum_{j=1}^{min\left(a,\ \left[\frac{c-a}{b-1}\right]\right)}(-1)^{j+1}{a \choose j}{c-j(b-1)-1 \choose a-1}.
\end{split}
\end{equation*}
\noindent
See, e.g. \citet{charalambides2002enumerative}.
}
\end{remark}

\subsection{Closed expression for the PMF of $\left(L_{n}^{(1)}\geq k_{1}\ \wedge\ L_{n}^{(0)}\leq k_{2}\right)$}
We shall study of the joint distribution of $\left(L_{n}^{(1)}\geq k_{1}\ \wedge\ L_{n}^{(0)}\leq k_{2}\right)$. We now make some useful Definition and Lemma for the proofs of Theorem in the sequel.

\noindent
\begin{definition}
For $0<q\leq1$, we define

\begin{equation*}\label{eq: 1.1}
\begin{split}
\overline{O}_{q,0,0}^{0,k_2}(u,v,s)={\sum_{x_{1},\ldots,x_{s-1}}}\
\sum_{y_{1},\ldots,y_{s}} q^{y_{1}x_{1}+(y_{1}+y_{2})x_{2}+\cdots+(y_{1}+\cdots+y_{s-1})x_{s-1}},\
\end{split}
\end{equation*}
\noindent
where the summation is over all integers $x_1,\ldots,x_{s-1},$ and $y_1,\ldots,y_{s}$ satisfying
\noindent
\begin{equation*}\label{eq:1}
\begin{split}
(x_1,\ldots,x_{s-1})\in S_{u,s-1}^{0},\ \text{and}
\end{split}
\end{equation*}
\noindent

\begin{equation*}\label{eq:1}
\begin{split}
0<y_{j}<k_2\ \text{for}\ j=1,\ldots,s,
\end{split}
\end{equation*}

\noindent
\begin{equation*}\label{eq:2}
\begin{split}
(y_{1},\ldots,y_{s}) \in S_{v,s}^{0}.
\end{split}
\end{equation*}
\end{definition}
\noindent
The following gives a recurrence relation useful for the computation of $\overline{O}_{q,0,0}^{\infty,k_2}(u,v,s)$.\\
\begin{lemma}
\label{lemma:6.5}
$\overline{O}_{q,0,0}^{\infty,k_2}(u,v,s)$ obeys the following recurrence relation.
\begin{equation*}\label{eq: 1.1}
\begin{split}
\overline{O}&_{q,0,0}^{\infty,k_2}(u,v,s)\\
&=\left\{
  \begin{array}{ll}
    \sum_{a=1}^{u-(s-2)}\sum_{b=1}^{k_2-1}q^{a(v-b)} \overline{O}_{q,0,0}^{\infty,k_2}(u-a,v-b,s-1), & \text{for}\ s>1,\ s-1\leq u\\
    &\text{and}\ s\leq v\leq s(k_2-1) \\
    1, & \text{for}\ s=1,\ u=0,\ \text{and}\ 1\leq v\leq k_2-1\\
    0, & \text{otherwise.}\\
  \end{array}
\right.
\end{split}
\end{equation*}
\end{lemma}
\begin{proof}
For $s > 1$, $s-1\leq u$ and $s\leq v\leq s(k_2-1)$, we observe that $y_{s}$ may assume any value $1,\ldots,k_{2}-1$, then $\overline{O}_{q,0,0}^{\infty,k_2}(u,v,s)$ can be written as

\begin{equation*}
\begin{split}
\overline{O}&_{q,0,0}^{\infty,k_2}(u,v,s)\\
=&\sum_{y_{s}=1}^{k_2-1}{\hspace{0.3cm}\sum_{\substack{(x_{1},\ldots,x_{s-1}) \in S_{u,s-1}^{0}}}}{\hspace{0.3cm}\sum_{\substack{0<y_1,\ldots,y_{s-1}<k_2\\(y_1,\ldots,y_{s-1})\in S_{v-y_{s},s-1}^{0}}}}q^{x_{s-1}(v-y_{s})} q^{y_{1}x_{1}+(y_{1}+y_{2})x_{2}+\cdots+(y_{1}+\cdots+y_{s-2})x_{s-2}}
\end{split}
\end{equation*}
\noindent
Similarly, we observe that since $x_{s-1}$ can assume the values $1,\ldots,r-(s-2)$, then $\overline{O}_{q,0,0}^{\infty,k_2}(u,v,s)$ can be rewritten as
\noindent

\begin{equation*}
\begin{split}
\overline{O}&_{q,0,0}^{\infty,k_2}(u,v,s)\\
=&\sum_{x_{s-1}=1}^{u-(s-2)}\sum_{y_{s}=1}^{k_2-1}q^{x_{s-1}(v-y_{s})}{\hspace{0.3cm}\sum_{\substack{(x_{1},\ldots,x_{s-2}) \in S_{u-x_{s-1},s-2}^{0}}}}{\hspace{0.3cm}\sum_{\substack{0<y_1,\ldots,y_{s-1}<k_2\\(y_1,\ldots,y_{s-1})\in S_{v-y_{s},s-1}{0}}}}q^{y_{1}x_{1}+(y_{1}+y_{2})x_{2}+\cdots+(y_{1}+\cdots+y_{s-2})x_{s-2}}\\
=&\sum_{a=1}^{u-(s-2)}\sum_{b=1}^{k_2-1}q^{a(v-b)} \overline{O}_{q,0,0}^{\infty,k_2}(u-a,v-b,s-1).
\end{split}
\end{equation*}
\noindent
The other cases are obvious and thus the proof is completed.
\end{proof}
\begin{remark}
{\rm
We observe that $\overline{O}_{1,0,0}^{\infty,k_2}(m,r,s)$ is the number of integer solutions $(x_{1},\ldots,x_{s-1})$ and $(y_{1},\ldots,y_{s})$ of

\noindent
\begin{equation*}\label{eq:1}
\begin{split}
(x_1,\ldots,x_{s-1})\in S_{u,s-1}^{0},\ \text{and}
\end{split}
\end{equation*}
\noindent

\begin{equation*}\label{eq:1}
\begin{split}
0<y_{j}<k_2\ \text{for}\ j=1,\ldots,s,
\end{split}
\end{equation*}

\noindent
\begin{equation*}\label{eq:2}
\begin{split}
(y_{1},\ldots,y_{s}) \in S_{v,s}^{0}.
\end{split}
\end{equation*}
\noindent
given by
\noindent
\begin{equation*}\label{eq: 1.1}
\begin{split}
\overline{O}_{1,0,0}^{\infty,k_2}(m,r,s)=M(s-1,\ u )S(s,\ k_{2},\ v),
\end{split}
\end{equation*}
\noindent
where $M(a,\ b)$ denotes the total number of integer solution $y_{1}+x_{2}+\cdots+y_{a}=b$ such that $(y_{1},\ldots,y_{a}) \in S_{b,a}^{0}$. The number is given by
\noindent
\begin{equation*}\label{eq: 1.1}
\begin{split}
M(a,\ b)={b-1 \choose a-1}.
\end{split}
\end{equation*}
\noindent
where $S(a,\ b,\ c)$ denotes the total number of integer solution $x_{1}+x_{2}+\cdots+x_{a}=c$ such that $0<x_{i}<b$ for $i=1,2,\ldots,a$. The number is given by
\noindent
\begin{equation*}\label{eq: 1.1}
\begin{split}
S(a,\ b,\ c)=\sum_{j=0}^{min\left(a,\ \left[\frac{c-a}{b-1}\right]\right) }(-1)^{j}{a \choose j}{c-j(b-1)-1 \choose a-1}.
\end{split}
\end{equation*}
\noindent
See, e.g. \citet{charalambides2002enumerative}.
}
\end{remark}
\begin{definition}
For $0<q\leq1$, we define
\noindent
\begin{equation*}\label{eq: 1.1}
\begin{split}
O_{q,k_1,0}^{\infty,k_2}(m,r,s)={\sum_{x_{1},\ldots,x_{s-1}}}\
\sum_{y_{1},\ldots,y_{s}} q^{y_{1}x_{1}+(y_{1}+y_{2})x_{2}+\cdots+(y_{1}+\cdots+y_{s-1})x_{s-1}},\
\end{split}
\end{equation*}
\noindent
where the summation is over all integers $x_1,\ldots,x_{s-1},$ and $y_1,\ldots,y_s$ satisfying
\noindent
\begin{equation*}\label{eq:1}
\begin{split}
(x_1,\ldots,x_{s-1})\in S_{u,s-1}^{k_1-1},\ \text{and}
\end{split}
\end{equation*}
\noindent

\begin{equation*}\label{eq:1}
\begin{split}
0<y_{j}<k_2\ \text{for}\ j=1,\ldots,s,
\end{split}
\end{equation*}

\noindent
\begin{equation*}\label{eq:2}
\begin{split}
(y_{1},\ldots,y_{s}) \in S_{v,s}^{0}.
\end{split}
\end{equation*}
\end{definition}
\noindent
The following gives a recurrence relation useful for the computation of $O_{q,k_1,0}^{\infty,k_2}(m,r,s)$.
\noindent\\
\begin{lemma}
\label{lemma:6.6}
$O_{q,k_1,0}^{\infty,k_2}(m,r,s)$ obeys the following recurrence relation.
\begin{equation*}\label{eq: 1.1}
\begin{split}
O&_{q,k_1,0}^{\infty,k_2}(m,r,s)\\
&=\left\{
  \begin{array}{ll}
    \sum_{a=1}^{k_1-1}\sum_{b=1}^{k_2-1}q^{a(r-b)} O_{q,k_1,0}^{\infty,k_2}(m-a,r-b,s-1)\\
    +\sum_{a=k_1}^{m-(s-2)}\sum_{b=1}^{k_2-1} q^{a(r-b)} \overline{O}_{q,0,0}^{\infty,k_2}(m-a,r-b,s-1), & \text{for}\ s>1,\ (s-2)+k_1\leq m)\\
    &\text{and}\ s\leq r \leq s(k_2-1)\\
    1, & \text{for}\ s=1,\ m=0,\ \text{and}\ 1\leq r \leq k_2-1\\
    0, & \text{otherwise.}\\
  \end{array}
\right.
\end{split}
\end{equation*}
\end{lemma}
\begin{proof}
For $s > 1$, $(s-2)+k_1\leq m)$ and $s\leq r \leq s(k_2-1)$, we observe that $y_{s}$ may assume any value $1,\ldots,k_{2}-1$, then $O_{q,k_1,0}^{\infty,k_2}(m,r,s)$ can be written as
\noindent
\begin{equation*}
\begin{split}
O&_{q,k_1,0}^{\infty,k_2}(m,r,s)\\
=&\sum_{y_{s}=1}^{k_2-1}{\hspace{0.3cm}\sum_{\substack{(x_1,\ldots,x_{s-1})\in S_{m,s-1}^{k_1-1}}}}{\hspace{0.3cm}\sum_{\substack{0<y_1,\ldots,y_{s-1}<k_2\\(y_{1},\ldots,y_{s-1}) \in S_{r-y_s,s-1}^{0}}}}q^{x_{s-1}(r-y_{s})}\ q^{y_{1}x_{1}+(y_{1}+y_{2})x_{2}+\cdots+(y_{1}+\cdots+y_{s-2})x_{s-2}}
\end{split}
\end{equation*}
\noindent
Similarly, we observe that since $x_{s-1}$ can assume the values $1,\ldots,m-(s-2)$, then $O_{q,k_1,0}^{\infty,k_2}(m,r,s)$ can be rewritten as
\noindent
\begin{equation*}
\begin{split}
O&_{q,k_1,0}^{\infty,k_2}(m,r,s)\\
=&\sum_{x_{s-1}=1}^{k_1-1}\sum_{y_{s}=1}^{k_2-1}q^{x_{s-1}(r-y_{s})}{\hspace{-0.5cm}\sum_{\substack{(x_1,\ldots,x_{s-2})\in S_{m-x_{s-1},s-2}^{k_1-1}}}}\quad {\sum_{\substack{0<y_1,\ldots,y_{s-1}<k_2\\(y_{1},\ldots,y_{s-1}) \in S_{r-y_s,s-1}^{0}}}}\ q^{y_{1}x_{1}+(y_{1}+y_{2})x_{2}+\cdots+(y_{1}+\cdots+y_{s-2})x_{s-2}}\\
&+\sum_{x_{s-1}=k_1}^{m-(s-2)}\sum_{y_{s}=1}^{k_2-1}q^{x_{s-1}(r-y_{s})}{\hspace{-0.5cm}\sum_{\substack{(x_1,\ldots,x_{s-2})\in S_{m-x_{s-1},s-2}^{0}}}}\quad{\sum_{\substack{0<y_1,\ldots,y_{s-1}<k_2\\(y_{1},\ldots,y_{s-1}) \in S_{r-y_s,s-1}^{0}}}}q^{y_{1}x_{1}+(y_{1}+y_{2})x_{2}+\cdots+(y_{1}+\cdots+y_{s-2})x_{s-2}}\\
=&\sum_{a=1}^{k_1-1}\ \sum_{b=1}^{k_2-1}\ q^{a(r-b)}\ O_{q,k_1,0}^{\infty,k_2}(m-a,r-b,s-1)+\sum_{a=k_1}^{m-(s-2)}\ \sum_{b=1}^{k_2-1}\ q^{a(r-b)}\ \overline{O}_{q,0,0}^{\infty,k_2}(m-a,r-b,s-1).
\end{split}
\end{equation*}
\noindent
The other cases are obvious and thus the proof is completed.
\end{proof}

\begin{remark}
{\rm
We observe that $O_{1,k_1,0}^{\infty,k_2}(m,r,s)$ is the number of integer solutions $(x_{1},\ldots,x_{s-1})$ and $(y_{1},\ldots,y_{s})$ of

\noindent
\begin{equation*}\label{eq:1}
\begin{split}
(x_1,\ldots,x_{s-1})\in S_{m,s-1}^{k_1-1},\ \text{and}
\end{split}
\end{equation*}
\noindent

\begin{equation*}\label{eq:1}
\begin{split}
0<y_{j}<k_2\ \text{for}\ j=1,\ldots,s,
\end{split}
\end{equation*}

\noindent
\begin{equation*}\label{eq:2}
\begin{split}
(y_{1},\ldots,y_{s}) \in S_{r,s}^{0},
\end{split}
\end{equation*}
\noindent
given by
\noindent
\begin{equation*}\label{eq: 1.1}
\begin{split}
O_{1,k_1,0}^{\infty,k_2}(m,r,s)=R(s-1,\ k_{1},\ m)S(s,\ k_{2},\ r),
\end{split}
\end{equation*}

\noindent
where $R(a,\ b,\ c)$ denotes the total number of integer solution $y_{1}+x_{2}+\cdots+y_{a}=c$ such that $(y_{1},\ldots,y_{a}) \in S_{r,a}^{k_2}$. The number is given by
\noindent
\begin{equation*}\label{eq: 1.1}
\begin{split}
R(a,\ b,\ c)=\sum_{j=1}^{min\left(a,\ \left[\frac{c-a}{b-1}\right]\right) }(-1)^{j+1}{a \choose j}{c-j(b-1)-1 \choose a-1}.
\end{split}
\end{equation*}

\noindent
where $S(a,\ b,\ c)$ denotes the total number of integer solution $x_{1}+x_{2}+\cdots+x_{a}=c$ such that $0<x_{i}<b$ for $i=1,2,\ldots,a$. The number is given by
\noindent
\begin{equation*}\label{eq: 1.1}
\begin{split}
S(a,\ b,\ c)=\sum_{j=0}^{min\left(a,\ \left[\frac{c-a}{b-1}\right]\right) }(-1)^{j}{a \choose j}{c-j(b-1)-1 \choose a-1}.
\end{split}
\end{equation*}
\noindent
See, e.g. \citet{charalambides2002enumerative}.
}

\end{remark}

\begin{definition}
For $0<q\leq1$, we define

\begin{equation*}\label{eq: 1.1}
\begin{split}
\overline{P}_{q,0,0}^{\infty,k_2}(u,v,s)={\sum_{x_{1},\ldots,x_{s}}}\
\sum_{y_{1},\ldots,y_{s}} q^{y_{1}x_{2}+(y_{1}+y_{2})x_{3}+\cdots+(y_{1}+\cdots+y_{s-1})x_{s}},\
\end{split}
\end{equation*}
\noindent
where the summation is over all integers $x_1,\ldots,x_{s},$ and $y_1,\ldots,y_s$ satisfying
\noindent
\begin{equation*}\label{eq:1}
\begin{split}
(x_1,\ldots,x_{s})\in S_{u,s}^{0},\ \text{and}
\end{split}
\end{equation*}
\noindent

\begin{equation*}\label{eq:1}
\begin{split}
0<y_{j}<k_2\ \text{for}\ j=1,\ldots,s,
\end{split}
\end{equation*}

\noindent
\begin{equation*}\label{eq:2}
\begin{split}
(y_{1},\ldots,y_{s}) \in S_{v,s}^{0}.
\end{split}
\end{equation*}
\end{definition}
\noindent
The following gives a recurrence relation useful for the computation of $\overline{P}_{q,0,0}^{\infty,k_2}(u,v,s)$.
\noindent\\
\begin{lemma}
\label{lemma:6.7}
$\overline{P}_{q,0,0}^{\infty,k_2}(u,v,s)$ obeys the following recurrence relation.
\begin{equation*}\label{eq: 1.1}
\begin{split}
\overline{P}&_{q,0,0}^{\infty,k_2}(u,v,s)\\
&=\left\{
  \begin{array}{ll}
    \sum_{a=1}^{u-(s-1)}\sum_{b=1}^{k_2-1}q^{a(v-b)} \overline{P}_{q,0,0}^{\infty,k_2}(u-a,v-b,s-1), & \text{for}\ s>1,\ s\leq u\\
    &\text{and}\ s\leq v \leq s(k_{2}-1)\\
    1, & \text{for}\ s=1,\ 1\leq u,\ \text{and}\ 1\leq v\leq k_{2}-1\\
    0, & \text{otherwise.}\\
  \end{array}
\right.
\end{split}
\end{equation*}
\end{lemma}
\begin{proof}
For $s > 1$, $s\leq u$ and $s\leq v \leq s(k_{2}-1)$, we observe that $y_{s}$ may assume any value $1,\ldots,k_{2}-1$, then $\overline{P}_{q,0,0}^{\infty,k_2}(u,v,s)$ can be written as
\noindent
\begin{equation*}
\begin{split}
\overline{P}&_{q,0,0}^{\infty,k_2}(u,v,s)\\
=&\sum_{y_{s}=1}^{k_2-1}{\sum_{\substack{(x_1,\ldots,x_{s})\in S_{u,s}^{0}}}}
{\hspace{0.3cm}\sum_{\substack{0<y_1,\ldots,y_{s-1}<k_2\\(y_{1},\ldots,y_{s-1}) \in S_{v-y_s,s-1}^{0}}}}q^{x_{s}(v-y_{s})}\ q^{y_{1}x_{2}+(y_{1}+y_{2})x_{3}+\cdots+(y_{1}+\cdots+y_{s-2})x_{s-1}}
\end{split}
\end{equation*}
\noindent
Similarly, we observe that since $x_s$ can assume the values $1,\ldots,u-(s-1)$, then $\overline{P}_{q,0,0}^{\infty,k_2}(u,v,s)$ can be rewritten as
\noindent
\begin{equation*}
\begin{split}
\overline{P}&_{q,0,0}^{\infty,k_2}(u,v,s)\\
=&\sum_{x_{s}=1}^{u-(s-1)}\sum_{y_{s}=1}^{k_2-1}q^{x_{s}(v-y_{s})}{\hspace{-0.3cm}\sum_{\substack{(x_1,\ldots,x_{s-1})\in S_{u-x_s,s-1}^{0}}}}{\hspace{0.3cm}\sum_{\substack{0<y_1,\ldots,y_{s-1}<k_2\\(y_{1},\ldots,y_{s-1}) \in S_{v-y_s,s-1}^{0}}}}\ q^{y_{1}x_{2}+(y_{1}+y_{2})x_{3}+\cdots+(y_{1}+\cdots+y_{s-2})x_{s-1}}\\
=&\sum_{a=1}^{u-(s-1)}\sum_{b=1}^{k_2-1}q^{a(v-b)} \overline{P}_{q,0,0}^{\infty,k_2}(u-a,v-b,s-1).
\end{split}
\end{equation*}
\noindent
The other cases are obvious and thus the proof is completed.
\end{proof}
\begin{remark}
{\rm
We observe that $\overline{P}_{1,0,0}^{\infty,k_2}(u,v,s)$ is the number of integer solutions $(x_{1},\ldots,x_{s})$ and $(y_{1},\ldots,y_{s})$ of
\noindent
\begin{equation*}\label{eq:1}
\begin{split}
(x_1,\ldots,x_{s})\in S_{u,s}^{0},\ \text{and}
\end{split}
\end{equation*}
\noindent

\begin{equation*}\label{eq:1}
\begin{split}
0<y_{j}<k_2\ \text{for}\ j=1,\ldots,s,
\end{split}
\end{equation*}

\noindent
\begin{equation*}\label{eq:2}
\begin{split}
(y_{1},\ldots,y_{s}) \in S_{v,s}^{0},
\end{split}
\end{equation*}

\noindent
given by
\noindent
\begin{equation*}\label{eq: 1.1}
\begin{split}
\overline{P}_{1,0,0}^{\infty,k_2}(m,r,s)=M(s,\ u )S(s,\ k_{2},\ v),
\end{split}
\end{equation*}

\noindent
where $M(a,\ b)$ denotes the total number of integer solution $y_{1}+x_{2}+\cdots+y_{a}=b$ such that $y_{1},\ldots,y_{a} \in S_{b,a}^{0}$. The number is given by
\noindent
\begin{equation*}\label{eq: 1.1}
\begin{split}
M(a,\ b)={b-1 \choose a-1}.
\end{split}
\end{equation*}

\noindent
where $S(a,\ b,\ c)$ denotes the total number of integer solution $x_{1}+x_{2}+\cdots+x_{a}=c$ such that $0<x_{i}<b$ for $i=1,2,\ldots,a$. The number is given by
\noindent
\begin{equation*}\label{eq: 1.1}
\begin{split}
S(a,\ b,\ c)=\sum_{j=0}^{min\left(a,\ \left[\frac{c-a}{b-1}\right]\right) }(-1)^{j}{a \choose j}{c-j(b-1)-1 \choose a-1}.
\end{split}
\end{equation*}
\noindent
See, e.g. \citet{charalambides2002enumerative}.
}
\end{remark}

\begin{definition}
For $0<q\leq1$, we define
\begin{equation*}\label{eq: 1.1}
\begin{split}
P_{q,k_1,0}^{\infty,k_2}(m,r,s)={\sum_{x_{1},\ldots,x_{s}}}\
\sum_{y_{1},\ldots,y_{s}}q^{y_{1}x_{2}+(y_{1}+y_{2})x_{3}+\cdots+(y_{1}+\cdots+y_{s-1})x_{s}},\
\end{split}
\end{equation*}
\noindent
where the summation is over all integers $x_1,\ldots,x_{s},$ and $y_1,\ldots,y_s$ satisfying
\noindent
\begin{equation*}\label{eq:1}
\begin{split}
(x_1,\ldots,x_{s})\in S_{m,s}^{k_1-1},\ \text{and}
\end{split}
\end{equation*}
\noindent

\begin{equation*}\label{eq:1}
\begin{split}
0<y_{j}<k_2\ \text{for}\ j=1,\ldots,s,
\end{split}
\end{equation*}

\noindent
\begin{equation*}\label{eq:2}
\begin{split}
(y_{1},\ldots,y_{s}) \in S_{r,s}^{0}.
\end{split}
\end{equation*}
\end{definition}
\noindent
The following gives a recurrence relation useful for the computation of $P_{q,k_1,0}^{\infty,k_2}(m,r,s)$.
\noindent\\
\begin{lemma}
\label{lemma:6.8}
$P_{q,k_1,0}^{\infty,k_2}(m,r,s)$ obeys the following recurrence relation,
\begin{equation*}\label{eq: 1.1}
\begin{split}
P&_{q,k_1,0}^{\infty,k_2}(m,r,s)\\
&=\left\{
  \begin{array}{ll}
    \sum_{b=1}^{k_{2}-1}\ \sum_{a=1}^{k_{1}-1}q^{a(r-b)}P_{q,k_1,0}^{\infty,k_2}(m-a,r-b,s-1)\\
    +\sum_{b=k_{2}}^{r-(s-1)}\ \sum_{a=1}^{k_{1}-1}q^{a(r-b)}\overline{P}_{q,0,0}^{\infty,k_2}(m-a,r-b,s-1), & \text{for}\quad s>1,\ (s-1)+k_{1}\leq m\\
    &\text{and}\quad s\leq r \leq s(k_{2}-1)\\
    1, & \text{for}\quad s=1,\ k_1\leq m\ \text{and}\\
    & 1\leq r \leq k_{2}-1\\
    0, & \text{otherwise.}\\
  \end{array}
\right.
\end{split}
\end{equation*}
\end{lemma}
\begin{proof}
For $s > 1$, $(s-1)+k_{1}\leq m$ and $s\leq r \leq s(k_{2}-1)$, we observe that $y_{s}$ may assume any value $1,\ldots,k_{2}-1$, then $P_{q,k_1,0}^{\infty,k_2}(m,r,s)$ can be written as
\noindent
\begin{equation*}
\begin{split}
P&_{q,k_1,0}^{\infty,k_2}(m,r,s)\\
=&\sum_{y_{s}=1}^{k_2-1}{\hspace{0.3cm}\sum_{\substack{(x_1,\ldots,x_{s})\in S_{m,s}^{k_1-1}}}}{\hspace{0.3cm}\sum_{\substack{0<y_1,\ldots,y_{s-1}<k_2\\(y_{1},\ldots,y_{s-1}) \in S_{r-y_s,s-1}^{0}}}}q^{x_{s}(r-y_{s})}q^{y_{1}x_{2}+(y_{1}+y_{2})x_{3}+\cdots+(y_{1}+\cdots+y_{s-2})x_{s-1}}
\end{split}
\end{equation*}
\noindent
Similarly, we observe that since $x_s$ can assume the values $1,\ldots,m-(s-1)$, then $L_{q,k_1,0}(m,r,s)$ can be rewritten as
\noindent
\begin{equation*}
\begin{split}
P&_{q,k_1,0}^{\infty,k_2}(m,r,s)\\
=&\sum_{x_{s}=1}^{k_1-1}\sum_{y_{s}=1}^{k_{2}-1}q^{x_{s}(r-y_{s})}{\hspace{-0.3cm}\sum_{\substack{(x_1,\ldots,x_{s-1})\in S_{m-x_s,s-1}^{k_1-1}}}}{\hspace{0.3cm}\sum_{\substack{(y_{1},\ldots,y_{s-1}) \in S_{r-y_s,s-1}^{0}}}} q^{y_{1}x_{2}+(y_{1}+y_{2})x_{3}+\cdots+(y_{1}+\cdots+y_{s-2})x_{s-1}}\\
&+\sum_{x_{s}=k_1}^{m-(s-1)}\sum_{y_{s}=1}^{k_2-1}q^{x_{s}(r-y_{s})}{\hspace{-0.3cm}\sum_{\substack{(x_1,\ldots,x_{s-1})\in S_{m-x_s,s-1}^{0}}}}{\hspace{0.3cm}\sum_{\substack{(y_{1},\ldots,y_{s-1}) \in S_{r-y_{s},s-1}^{0}}}} q^{y_{1}x_{2}+(y_{1}+y_{2})x_{3}+\cdots+(y_{1}+\cdots+y_{s-2})x_{s-1}}\\
=&\sum_{a=1}^{k_{1}-1}\ \sum_{b=1}^{k_{2}-1}q^{a(r-b)}P_{q,k_1,0}^{\infty,k_2}(m-a,r-b,s-1)+\sum_{a=k_{1}}^{m-(s-1)}\ \sum_{b=1}^{k_{2}-1}q^{a(r-b)}\overline{P}_{q,0,0}^{\infty,k_2}(m-a,r-b,s-1).
\end{split}
\end{equation*}
\noindent
The other cases are obvious and thus the proof is completed.
\end{proof}
\begin{remark}
{\rm
We observe that $P_{1,k_1,0}^{\infty,k_2}(m,r,s)$ is the number of integer solutions $(x_{1},\ldots,x_{s})$ and $(y_{1},\ldots,y_{s})$ of

\noindent
\begin{equation*}\label{eq:1}
\begin{split}
(x_1,\ldots,x_{s})\in S_{m,s}^{k_1-1},\ \text{and}
\end{split}
\end{equation*}
\noindent

\begin{equation*}\label{eq:1}
\begin{split}
0<y_{j}<k_2\ \text{for}\ j=1,\ldots,s,
\end{split}
\end{equation*}

\noindent
\begin{equation*}\label{eq:2}
\begin{split}
(y_{1},\ldots,y_{s}) \in S_{r,s}^{0},
\end{split}
\end{equation*}
\noindent
given by
\noindent
\begin{equation*}\label{eq: 1.1}
\begin{split}
P_{1,k_1,0}^{\infty,k_2}(m,r,s)=R(s,\ k_{1},\ m)S(s,\ k_{2},\ r),
\end{split}
\end{equation*}

\noindent
where $R(a,\ b,\ c)$ denotes the total number of integer solution $y_{1}+x_{2}+\cdots+y_{a}=c$ such that $(y_{1},\ldots,y_{a}) \in S_{r,a}^{k_2}$. The number is given by
\noindent
\begin{equation*}\label{eq: 1.1}
\begin{split}
R(a,\ b,\ c)=\sum_{j=1}^{min\left(a,\ \left[\frac{c-a}{b-1}\right]\right)}(-1)^{j+1}{a \choose j}{c-j(b-1)-1 \choose a-1}.
\end{split}
\end{equation*}

\noindent
where $S(a,\ b,\ c)$ denotes the total number of integer solution $x_{1}+x_{2}+\cdots+x_{a}=c$ such that $0<x_{i}<b$ for $i=1,2,\ldots,a$. The number is given by
\noindent
\begin{equation*}\label{eq: 1.1}
\begin{split}
S(a,\ b,\ c)=\sum_{j=0}^{min\left(a,\ \left[\frac{c-a}{b-1}\right]\right)}(-1)^{j}{a \choose j}{c-j(b-1)-1 \choose a-1}.
\end{split}
\end{equation*}
\noindent
See, e.g. \citet{charalambides2002enumerative}.
}
\end{remark}
\noindent
The probability function of $P_{q,\theta}\Big(L_{n}^{(1)}\geq k_{1}\ \wedge\ L_{n}^{(0)}\leq k_{2}\Big)$ for $n\geq k_1$ is obtained by the following theorem.
\noindent\\
\begin{theorem}
The PMF $P_{q,\theta}\Big(L_{n}^{(1)}\geq k_{1}\ \wedge\ L_{n}^{(0)}\leq k_{2}\Big)$ for $n\geq k_1$
\noindent
\begin{equation*}\label{eq:bn1}
\begin{split}
P_{q,\theta}\Big(L_{n}^{(1)}\geq k_{1}\ \wedge\ L_{n}^{(0)}\leq k_{2}\Big)=&\sum_{i=1}^{n-k_1}\theta^{n-i}\ {\prod_{j=1}^{i}}\ (1-\theta q^{j-1})\sum_{s=1}^{i}\Big[H_{q,k_1,0}^{\infty,k_2+1}(n-i,\ i,\ s)\\
&+O_{q,k_1,0}^{\infty,k_2+1}(n-i,\ i,\ s)+G_{q,k_1,0}^{\infty,k_2+1}(n-i,\ i,\ s+1)\\
&+P_{q,k_1,0}^{\infty,k_2+1}(n-i,\ i,\ s)\Big]J(n,k_1+1)+J(n,k_1)\theta^n,\\
\end{split}
\end{equation*}
\noindent
where $J(n,k_1)$ is the function defined by $J(n,k_1)=1$ if $n\geq k_1$ and $0$ otherwise.
\end{theorem}
\begin{proof}
We partition the event $\left\{L_{n}^{(1)}\geq k_{1}\ \wedge\ L_{n}^{(0)}\leq k_{2}\right\}$ into disjoint events given by $F_{n}=i,$ for $i=k_2, \ldots,\ n$. Adding the probabilities we have
\noindent
\begin{equation*}\label{eq:bn}
\begin{split}
P_{q,\theta}\Big(L_{n}^{(1)}\geq k_{1}\ \wedge\ L_{n}^{(0)}\leq k_{2}\Big)=\sum_{i=1}^{n-k_1}P_{q,\theta}\Big(L_{n}^{(1)}\geq k_{1}\ \wedge\ L_{n}^{(0)}\leq k_{2}\ \wedge\ F_{n}=i\Big).\\
\end{split}
\end{equation*}
\noindent
We will write $E_{n,i}=\left\{L_{n}^{(1)}\geq k_{1}\ \wedge\ L_{n}^{(0)}\leq k_{2}\ \wedge\ F_{n}=i\right\}$.
\noindent
We can now rewrite as follows
\noindent
\begin{equation*}\label{eq:bn1}
\begin{split}
P_{q,\theta}\Big(L_{n}^{(1)}\geq k_{1}\ \wedge\ L_{n}^{(0)}\leq k_{2}\Big)=\sum_{i=1}^{n-k_1}P_{q,\theta}\left(E_{n,i}\right).
\end{split}
\end{equation*}
\noindent
We are going to focus on the event $E_{n,\ i}$. A typical element of the event $\Big\{L_{n}^{(1)}\geq k_{1}\ \wedge\ L_{n}^{(0)}\leq k_{2}\ \wedge\ F_{n}=i\Big\}$ is an ordered sequence which consists of $n-i$ successes and $i$ failures such that the length of the longest success run is greater than or equal to $k_1$ and the length of the longest failure run is less than or equal to $k_2$. The number of these sequences can be derived as follows. First we will distribute the $i$ failures and the $n-i$ successes.  Let $s$ $(1\leq s \leq i)$ be the number of runs of failures in the event $E_{n,\ i}$. We divide into two cases: starting with a failure run or starting with a success run. Thus, we distinguish between four types of sequences in the event $\Big\{L_{n}^{(1)}\geq k_{1}\ \wedge\ L_{n}^{(0)}\leq k_{2}\ \wedge\ F_{n}=i\Big\}$, respectively named A, B, C and D-type, which are defined as follows.
\begin{equation*}\label{eq:bn}
\begin{split}
\text{A-type}\ :&\quad \overbrace{0\ldots 0}^{y_{1}}\mid\overbrace{1\ldots 1}^{x_{1}}\mid\overbrace{0\ldots 0}^{y_{2}}\mid\overbrace{1\ldots 1}^{x_{2}}\mid \ldots \mid \overbrace{0\ldots 0}^{y_{s}}\mid\overbrace{1\ldots 1}^{x_{s}},\\
\end{split}
\end{equation*}
\noindent
with $i$ $0$'s and $n-i$ $1$'s, where $x_{j}$ $(j=1,\ldots,s)$ represents a length of run of $1$'s and $y_{j}$ $(j=1,\ldots,s)$ represents the length of a run of $0$'s. And all integers $x_{1},\ldots,x_{s}$, and $y_{1},\ldots ,y_{s}$ satisfy the conditions
\noindent
\begin{equation*}\label{eq:bn}
\begin{split}
(x_1,\ldots,x_{s})\in S_{n-i,s}^{k_1-1},\ \text{and}
\end{split}
\end{equation*}
\noindent
\begin{equation*}\label{eq:bn}
\begin{split}
0 < y_j < k_2+1 \mbox{\ for\ } j=1,...,s, \mbox{\ and\ } y_1+\cdots +y_s=i.
\end{split}
\end{equation*}

\begin{equation*}\label{eq:bn}
\begin{split}
\text{B-type}\ :&\quad \overbrace{0\ldots 0}^{y_{1}}\mid\overbrace{1\ldots 1}^{x_{1}}\mid\overbrace{0\ldots 0}^{y_{2}}\mid\overbrace{1\ldots 1}^{x_{2}}\mid \ldots \mid \overbrace{0\ldots 0}^{y_{s-1}}\mid\overbrace{1\ldots 1}^{x_{s-1}}\mid\overbrace{0\ldots 0}^{y_{s}},\\
\end{split}
\end{equation*}
\noindent
with $i$ $0$'s and $n-i$ $1$'s, where $x_{j}$ $(j=1,\ldots,s-1)$ represents a length of run of $1$'s and $y_{j}$ $(j=1,\ldots,s)$ represents the length of a run of $0$'s. And all integers $x_{1},\ldots,x_{s-1}$, and $y_{1},\ldots ,y_{s}$ satisfy the conditions
\noindent
\begin{equation*}\label{eq:bn}
\begin{split}
(x_1,\ldots,x_{s-1})\in S_{n-i,s-1}^{k_1-1},\ \text{and}\end{split}
\end{equation*}
\noindent
\begin{equation*}\label{eq:bn}
\begin{split}
0 < y_j < k_2+1 \mbox{\ for\ } j=1,...,s, \mbox{\ and\ } y_1+\cdots +y_s=i.
\end{split}
\end{equation*}

\begin{equation*}\label{eq:bn}
\begin{split}
\text{C-type}\ :&\quad \overbrace{1\ldots 1}^{x_{1}}\mid\overbrace{0\ldots 0}^{y_{1}}\mid\overbrace{1\ldots 1}^{x_{2}}\mid\overbrace{0\ldots 0}^{y_{2}}\mid \ldots \mid\overbrace{0\ldots 0}^{y_{s}}\mid\overbrace{1\ldots 1}^{x_{s+1}},\\
\end{split}
\end{equation*}
\noindent
with $i$ $0$'s and $n-i$ $1$'s, where $x_{j}$ $(j=1,\ldots,s+1)$ represents a length of run of $1$'s and $y_{j}$ $(j=1,\ldots,s)$ represents the length of a run of $0$'s. And all integers $x_{1},\ldots,x_{s+1}$, and $y_{1},\ldots ,y_{s}$ satisfy the conditions
\noindent
\begin{equation*}\label{eq:bn}
\begin{split}
(x_1,\ldots,x_{s+1})\in S_{n-i,s+1}^{k_1-1},\ \text{and}
\end{split}
\end{equation*}
\noindent
\begin{equation*}\label{eq:bn}
\begin{split}
0 < y_j < k_2+1 \mbox{\ for\ } j=1,...,s, \mbox{\ and\ } y_1+\cdots +y_{s}=i.
\end{split}
\end{equation*}

\begin{equation*}\label{eq:bn}
\begin{split}
\text{D-type}\ :&\quad \overbrace{1\ldots 1}^{x_{1}}\mid\overbrace{0\ldots 0}^{y_{1}}\mid\overbrace{1\ldots 1}^{x_{2}}\mid\overbrace{0\ldots 0}^{y_{2}}\mid\overbrace{1\ldots 1}^{x_{3}}\mid \ldots \mid \overbrace{0\ldots 0}^{y_{s-1}}\mid\overbrace{1\ldots 1}^{x_{s}}\mid\overbrace{0\ldots 0}^{y_{s}},\\
\end{split}
\end{equation*}
\noindent
with $i$ $0$'s and $n-i$ $1$'s, where $x_{j}$ $(j=1,\ldots,s)$ represents a length of run of $1$'s and $y_{j}$ $(j=1,\ldots,s)$ represents the length of a run of $0$'s. And all integers $x_{1},\ldots,x_{s}$, and $y_{1},\ldots ,y_{s}$ satisfy the conditions
\noindent
\begin{equation*}\label{eq:bn}
\begin{split}
(x_1,\ldots,x_{s})\in S_{n-i,s}^{k_1-1},\ \text{and}
\end{split}
\end{equation*}
\noindent
\begin{equation*}\label{eq:bn}
\begin{split}
0 < y_j < k_2+1 \mbox{\ for\ } j=1,...,s, \mbox{\ and\ } y_1+\cdots +y_{s}=i.
\end{split}
\end{equation*}

Then the probability of the event $E_{n,\ i}$ is given by

\noindent
\begin{equation*}\label{eq: 1.1}
\begin{split}
P_{q,\theta}\Big(E_{n,\ i}\Big)&\\
=\sum_{s=1}^{i}\Bigg[\bigg\{&\sum_{\substack{(x_1,\ldots,x_{s})\in S_{n-i,s}^{k_1-1}}}\hspace{0.3cm}\sum_{\substack{y_{1}+\cdots+y_{s}=i\\ y_{1}, \ldots, y_{s} \in \{1,\ldots,k_{2}\}}}\big(1-\theta q^{0}\big)\cdots \big(1-\theta q^{y_{1}-1}\big)\Big(\theta q^{y_{1}}\Big)^{x_{1}}\times\\
&\big(1-\theta q^{y_{1}}\big)\cdots \big(1-\theta q^{y_{1}+y_{2}-1}\big)\Big(\theta q^{y_{1}+y_{2}}\Big)^{x_{2}}\times\\
&\big(1-\theta q^{y_{1}+y_2}\big)\cdots \big(1-\theta q^{y_{1}+y_{2}+y_3-1}\big)\Big(\theta q^{y_{1}+y_{2}+y_3}\Big)^{x_{3}}\times \\
&\quad \quad \quad \quad\quad \quad\quad \quad \quad \quad \quad \vdots\\
&\big(1-\theta q^{y_{1}+\cdots+y_{s-1}}\big)\cdots \big(1-\theta q^{y_{1}+\cdots+y_{s}-1}\big)\Big(\theta q^{y_{1}+\cdots+y_{s}}\Big)^{x_{s}}\bigg\}+\\
\end{split}
\end{equation*}

\noindent
\begin{equation*}\label{eq: 1.1}
\begin{split}
\quad \  &\bigg\{\sum_{\substack{(x_1,\ldots,x_{s-1})\in S_{n-i,s-1}^{k_1-1}}}\hspace{0.3cm}\sum_{\substack{y_{1}+\cdots+y_{s}=i\\ y_{1},\ldots,y_{s} \in \{1,\ldots,k_{2}\}}} \big(1-\theta q^{0}\big)\cdots \big(1-\theta q^{y_{1}-1}\big)\times\\
&\Big(\theta q^{y_{1}}\Big)^{x_{1}}\big(1-\theta q^{y_{1}}\big)\cdots \big(1-\theta q^{y_{1}+y_{2}-1}\big)\times\\
&\Big(\theta q^{y_{1}+y_{2}}\Big)^{x_{2}}\big(1-\theta q^{y_{1}+y_2}\big)\cdots \big(1-\theta q^{y_{1}+y_{2}+y_3-1}\big)\times \\
&\quad \quad \quad \quad\quad \quad\quad \quad \quad \quad \quad \vdots\\
&\Big(\theta q^{y_{1}+\cdots+y_{s-1}}\Big)^{x_{s-1}}
\big(1-\theta q^{y_{1}+\cdots+y_{s-1}}\big)\cdots \big(1-\theta q^{y_{1}+\cdots+y_{s}-1}\big)\bigg\}+\\
\end{split}
\end{equation*}

\noindent
\begin{equation*}\label{eq: 1.1}
\begin{split}
\quad \quad \quad \quad &\bigg\{\sum_{\substack{(x_1,\ldots,x_{s+1})\in S_{n-i,s+1}^{k_1-1}}}
\hspace{0.3cm}\sum_{\substack{y_{1}+\cdots+y_{s}=i\\ y_{1},\ldots,y_{s} \in \{1,\ldots,k_{2}\}}} \big(\theta q^{0}\big)^{x_{1}}\big(1-\theta q^{0}\big)\cdots (1-\theta q^{y_{1}-1})\times\\
&\Big(\theta q^{y_{1}}\Big)^{x_{2}}\big(1-\theta q^{y_{1}}\big)\cdots \big(1-\theta q^{y_{1}+y_{2}-1}\big)\times\\
&\Big(\theta q^{y_{1}+y_{2}}\Big)^{x_{3}}\big(1-\theta q^{y_{1}+y_2}\big)\cdots \big(1-\theta q^{y_{1}+y_{2}+y_3-1}\big)\times \\
&\quad \quad \quad \quad \vdots\\
&\Big(\theta q^{y_{1}+\cdots+y_{s}}\Big)^{x_{s+1}}\bigg\}+\\
\end{split}
\end{equation*}

\noindent
\begin{equation*}\label{eq: 1.1}
\begin{split}
\quad \quad \quad &\bigg\{\sum_{\substack{(x_1,\ldots,x_{s})\in S_{n-i,s}^{k_1-1}}}\hspace{0.3cm}
\sum_{\substack{y_{1}+\cdots+y_{s}=i\\ y_{1},\ldots,y_{s} \in \{1,\ldots,k_{2}\}}}\big(\theta q^{0}\big)^{x_{1}}\big(1-\theta q^{0}\big)\cdots (1-\theta q^{y_{1}-1})\times\\
&\Big(\theta q^{y_{1}}\Big)^{x_{2}}\big(1-\theta q^{y_{1}+y_2}\big)\cdots \big(1-\theta q^{y_{1}+y_{2}-1}\big)\times\\
&\Big(\theta q^{y_{1}+y_{2}}\Big)^{x_{3}}\big(1-\theta q^{y_{1}}\big)\cdots \big(1-\theta q^{y_{1}+y_{2}+y_3-1}\big)\times \\
&\quad \quad \quad \quad\quad \quad\quad \quad \quad \quad \quad \vdots\\
&\Big(\theta q^{y_{1}+\cdots+y_{s-1}}\Big)^{x_{s}}
\big(1-\theta q^{y_{1}+\cdots+y_{s-1}}\big)\cdots \big(1-\theta q^{y_{1}+\cdots+y_{s}-1}\big)\bigg\}\Bigg].\\
\end{split}
\end{equation*}
\noindent
Using simple exponentiation algebra arguments to simplify,
\noindent
\begin{equation*}\label{eq: 1.1}
\begin{split}
&P_{q,\theta}\Big(E_{n,i}\Big)=\\
&\theta^{n-i}{\prod_{j=1}^{i}}\ (1-\theta q^{j-1})\times\\
&\sum_{s=1}^{i}\Bigg[\sum_{\substack{(x_1,\ldots,x_{s})\in S_{n-i,s}^{k_1-1}}}\quad
\sum_{\substack{y_{1}+\cdots+y_{s}=i\\ y_{1},\ldots,y_{s} \in \{1,\ldots,k_{2}\}}}q^{y_{1}x_{1}+(y_{1}+y_{2})x_{2}+\cdots+(y_{1}+\cdots+y_{s})x_{s}}+\\
&\quad \sum_{\substack{(x_1,\ldots,x_{s-1})\in S_{n-i,s-1}^{k_1-1}}}\quad
\sum_{\substack{y_{1}+\cdots+y_{s}=i\\ y_{1},\ldots,y_{s} \in \{1,\ldots,k_{2}\}}}q^{y_{1}x_{1}+(y_{1}+y_{2})x_{2}+\cdots+(y_{1}+\cdots+y_{s-1})x_{s-1}}+\\
&\quad\sum_{\substack{(x_1,\ldots,x_{s+1})\in S_{n-i,s+1}^{k_1-1}}}\quad\sum_{\substack{y_{1}+\cdots+y_{s}=i\\ y_{1},\ldots,y_{s} \in \{1,\ldots,k_{2}\}}}q^{y_{1}x_{2}+(y_{1}+y_{2})x_{3}+\cdots+(y_{1}+\cdots+y_{s})x_{s+1}}+\\
&\quad\sum_{\substack{(x_1,\ldots,x_{s})\in S_{n-i,s}^{k_1-1}}}\quad
\sum_{\substack{y_{1}+\cdots+y_{s}=i\\ y_{1},\ldots,y_{s} \in \{1,\ldots,k_{2}\}}}q^{y_{1}x_{2}+(y_{1}+y_{2})x_{3}+\cdots+(y_{1}+\cdots+y_{s-1})x_{s}}\Bigg].\\
\end{split}
\end{equation*}
\noindent
Using Lemma \ref{lemma:3.10}, Lemma \ref{lemma:3.12}, Lemma \ref{lemma:6.6} and Lemma \ref{lemma:6.8}, we can rewrite as follows.
\noindent
\begin{equation*}\label{eq: 1.1}
\begin{split}
P_{q,\theta}\Big(E_{n,i}\Big)=\theta^{n-i}\ {\prod_{j=1}^{i}}\ (1-\theta q^{j-1})\sum_{s=1}^{i}&\Big[H_{q,k_1,0}^{\infty,k_2+1}(n-i,\ i,\ s)+O_{q,k_1,0}^{\infty,k_2+1}(n-i,\ i,\ s)\\
&+G_{q,k_1,0}^{\infty,k_2+1}(n-i,\ i,\ s+1)+P_{q,k_1,0}^{\infty,k_2+1}(n-i,\ i,\ s)\Big],\\
\end{split}
\end{equation*}\\
\noindent
where
\noindent
\begin{equation*}\label{eq: 1.1}
\begin{split}
H_{q,k_1,0}^{\infty,k_2+1}(n&-i,\ i,\ s)\\
&=\sum_{\substack{(x_1,\ldots,x_{s})\in S_{n-i,s}^{k_1-1}}}\quad
\sum_{\substack{y_{1}+\cdots+y_{s}=i\\ y_{1},\ldots,y_{s} \in \{1,\ldots,k_{2}\}}}q^{y_{1}x_{1}+(y_{1}+y_{2})x_{2}+\cdots+(y_{1}+\cdots+y_{s})x_{s}},
\end{split}
\end{equation*}
\noindent
\begin{equation*}\label{eq: 1.1}
\begin{split}
O_{q,k_1,0}^{\infty,k_2+1}(n&-i,\ i,\ s)\\
&=\sum_{\substack{(x_1,\ldots,x_{s-1})\in S_{n-i,s-1}^{k_1-1}}}\quad
\sum_{\substack{y_{1}+\cdots+y_{s}=i\\ y_{1},\ldots,y_{s} \in \{1,\ldots,k_{2}\}}} q^{y_{1}x_{1}+(y_{1}+y_{2})x_{2}+\cdots+(y_{1}+\cdots+y_{s-1})x_{s-1}},
\end{split}
\end{equation*}
\noindent
\begin{equation*}\label{eq: 1.1}
\begin{split}
G_{q,k_1,0}^{\infty,k_2+1}(n&-i,\ i,\ s+1)\\
&=\sum_{\substack{(x_1,\ldots,x_{s+1})\in S_{n-i,s+1}^{k_1-1}}}\quad
\sum_{\substack{y_{1}+\cdots+y_{s}=i\\ y_{1},\ldots,y_{s} \in \{1,\ldots,k_{2}\}}}q^{y_{1}x_{2}+(y_{1}+y_{2})x_{3}+\cdots+(y_{1}+\cdots+y_{s})x_{s+1}},
\end{split}
\end{equation*}
\noindent
and
\noindent
\begin{equation*}\label{eq: 1.1}
\begin{split}
P_{q,k_1,0}^{\infty,k_2+1}(n&-i,\ i,\ s)\\
&=\sum_{\substack{(x_1,\ldots,x_{s})\in S_{n-i,s}^{k_1-1}}}
\quad\sum_{\substack{y_{1}+\cdots+y_{s}=i\\ y_{1},\ldots,y_{s} \in \{1,\ldots,k_{2}\}}} q^{y_{1}x_{2}+(y_{1}+y_{2})x_{3}+\cdots+(y_{1}+\cdots+y_{s-1})x_{s}}.
\end{split}
\end{equation*}
\noindent
Therefore we can compute the probability of the event $\left\{L_{n}^{(1)}\geq k_{1}\ \wedge\ L_{n}^{(0)}\leq k_{2}\right\}$ as follows
\noindent
\begin{equation*}\label{eq:bn1}
\begin{split}
P_{q,\theta}\Big(L_{n}^{(1)}\geq k_{1}\ \wedge\ L_{n}^{(0)}\leq k_{2}\Big)=&\sum_{i=1}^{n-k_1}\theta^{n-i}\ {\prod_{j=1}^{i}}\ (1-\theta q^{j-1})\sum_{s=1}^{i}\Big[H_{q,k_1,0}^{\infty,k_2+1}(n-i,\ i,\ s)\\
&+O_{q,k_1,0}^{\infty,k_2+1}(n-i,\ i,\ s)+G_{q,k_1,0}^{\infty,k_2+1}(n-i,\ i,\ s+1)\\
&+P_{q,k_1,0}^{\infty,k_2+1}(n-i,\ i,\ s)\Big].\\
\end{split}
\end{equation*}
\noindent
Thus proof is completed.
\end{proof}
\noindent
It is worth mentioning here that the PMF $P_{q,\theta}\Big(L_{n}^{(1)}\geq k_{1}\ \wedge\ L_{n}^{(0)}\leq k_{2}\Big)$ approaches the probability function of $P_{\theta}\Big(L_{n}^{(1)}\geq k_{1}\ \wedge\ L_{n}^{(0)}\leq k_{2}\Big)$ of IID model. The details are presented in the following remark.
\noindent\\
\begin{remark}
{\rm For $q=1$, the PMF $P_{q,\theta}\Big(L_{n}^{(1)}\geq k_{1}\ \wedge\ L_{n}^{(0)}\leq k_{2}\Big)$ for $n\geq k_1$ reduces to the PMF $P_{\theta}\Big(L_{n}^{(1)}\geq k_{1}\ \wedge\ L_{n}^{(0)}\leq k_{2}\Big)$ for $n\geq k_1$ is given by
\noindent
\begin{equation*}\label{eq:bn1}
\begin{split}
P\Big(L_{n}^{(1)}\geq k_{1}\ \wedge\ L_{n}^{(0)}\leq k_{2}\Big)=&\sum_{i=1}^{n-k_1}\theta^{n-i}\ (1-\theta)^{i}\sum_{s=1}^{i}S(s,\ k_2+1,\ i)\Big[R(s,\ k_1,\ n-i)\\
&+R(s-1,\ k_1,\ n-i)+S(s+1,\ k_1,\ n-i)\\
&+R(s,\ k_1,\ n-i)\Big]J(n,k_1+1)+J(n,k_1)\theta^n,\\
\end{split}
\end{equation*}
\noindent
where
\noindent
\begin{equation*}\label{eq: 1.1}
\begin{split}
S(a,\ b,\ c)=\sum_{j=0}^{min(a,\left[\frac{c-a}{b-1}\right])}(-1)^{j}{a \choose j}{c-j(b-1)-1 \choose a-1},
\end{split}
\end{equation*}

\noindent
\begin{equation*}\label{eq: 1.1}
\begin{split}
R(a,\ b,\ c)=\sum_{j=1}^{min\left(a,\ \left[\frac{c-a}{b-1}\right]\right)}(-1)^{j+1}{a \choose j}{c-j(b-1)-1 \choose a-1}.
\end{split}
\end{equation*}
\noindent

and $J(n,k_1)$ is the function defined by $J(n,k_1)=1$ if $n\geq k_1$ and $0$ otherwise.
}
\end{remark}

\subsection{Closed expression for the PMF of $\left(L_{n}^{(1)}\geq k_{1}\ \wedge\ L_{n}^{(0)}\geq k_{2}\right)$}
We shall study of the joint distribution of $\left(L_{n}^{(1)}\geq k_{1}\ \wedge\ L_{n}^{(0)}\geq k_{2}\right)$. We now make some useful Definition and Lemma for the proofs of Theorem in the sequel.
\begin{definition}
For $0<q\leq1$, we define

\begin{equation*}\label{eq: 1.1}
\begin{split}
\overline{Q}_{q,k_1,0}(u,v,s)={\sum_{x_{1},\ldots,x_{s}}}\
\sum_{y_{1},\ldots,y_{s}} q^{y_{1}x_{1}+(y_{1}+y_{2})x_{2}+\cdots+(y_{1}+\cdots+y_{s})x_{s}},\
\end{split}
\end{equation*}
\noindent
where the summation is over all integers $x_1,\ldots,x_{s},$ and $y_1,\ldots,y_{s}$ satisfying

\noindent
\begin{equation*}\label{eq:1}
\begin{split}
(x_1,\ldots,x_{s})\in S_{u,s}^{k_1-1},\ \text{and}
\end{split}
\end{equation*}
\noindent
\begin{equation*}\label{eq:2}
\begin{split}
(y_{1},\ldots,y_{s}) \in S_{v,s}^{0}.
\end{split}
\end{equation*}
\end{definition}
\noindent
The following gives a recurrence relation useful for the computation of $\overline{Q}_{q,k_1,0}(u,v,s)$.
\noindent
\begin{lemma}
\label{lemma:6.9}
$\overline{Q}_{q,k_1,0}(u,v,s)$ obeys the following recurrence relation.
\begin{equation*}\label{eq: 1.1}
\begin{split}
\overline{Q}&_{q,k_1,0}(u,v,s)\\
&=\left\{
  \begin{array}{ll}
    \sum_{a=1}^{k_1-1}\sum_{b=1}^{v-(s-1)}q^{va}\overline{Q}_{q,k_1,0}(u-a,v-b,s-1)\\
    +\sum_{a=k_1}^{u-(s-1)}\sum_{b=1}^{v-(s-1)}q^{va}\overline{J}_{q}(u-a,v-b,s-1), & \text{for}\ s>1,\ (s-1)+k_1\leq u\\
    &\text{and}\ s\leq v \\
    1, & \text{for}\ s=1,\ k_1\leq u,\ \text{and}\  1\leq v\\
    0, & \text{otherwise.}\\
  \end{array}
\right.
\end{split}
\end{equation*}
\end{lemma}
\begin{proof}
For $s > 1$, $(s-1)+k_1\leq u$ and $s\leq v$, we observe that $y_{s}$ may assume any value $1,\ldots,v-(s-1)$, then $\overline{Q}_{q}(u,v,s)$ can be written as
\noindent
\begin{equation*}
\begin{split}
\overline{Q}&_{q,k_1,0}(u,v,s)\\
=&\sum_{y_{s}=1}^{v-(s-1)}{\sum_{\substack{(x_1,\ldots,x_{s})\in S_{u,s}^{k_1-1}}}}{\hspace{0.3cm}\sum_{\substack{(y_{1},\ldots,y_{s-1}) \in S_{v-y_s,s-1}^{0}}}}q^{vx_{s}}q^{y_{1}x_{1}+(y_{1}+y_{2})x_{2}+\cdots+(y_{1}+\cdots+y_{s-1})x_{s-1}}.
\end{split}
\end{equation*}
\noindent
Similarly, we observe that since $x_s$ can assume the values $1,\ldots,u-(s-1)$, then $\overline{Q}_{q}(u,v,s)$ can be rewritten as
\noindent
\begin{equation*}
\begin{split}
\overline{Q}&_{q,k_1,0}(u,v,s)\\
=&\sum_{x_{s}=1}^{k_1-1}\sum_{y_{s}=1}^{v-(s-1)}{\sum_{\substack{(x_1,\ldots,x_{s-1})\in S_{u-x_{s},s-1}^{k_1-1}}}}{\hspace{0.3cm}\sum_{\substack{(y_{1},\ldots,y_{s-1}) \in S_{v-y_{s},s-1}^{0}}}}q^{vx_{s}}q^{y_{1}x_{1}+(y_{1}+y_{2})x_{2}+\cdots+(y_{1}+\cdots+y_{s-1})x_{s-1}}\\
&+\sum_{x_{s}=k_1}^{u-(s-1)}\sum_{y_{s}=1}^{v-(s-1)}{\sum_{\substack{(x_1,\ldots,x_{s-1})\in S_{u-x_{s},s-1}^{0}}}}{\hspace{0.3cm}\sum_{\substack{(y_{1},\ldots,y_{s-1}) \in S_{v-y_{s},s-1}^{0}}}}q^{vx_{s}}q^{y_{1}x_{1}+(y_{1}+y_{2})x_{2}+\cdots+(y_{1}+\cdots+y_{s-1})x_{s-1}}\\
=&\sum_{a=1}^{k_1-1}\sum_{b=1}^{v-(s-1)}q^{va}\overline{Q}_{q,k_1,0}(u-a,v-b,s-1)+\sum_{a=k_1}^{u-(s-1)}\sum_{b=1}^{v-(s-1)}q^{va}\overline{J}_{q}(u-a,v-b,s-1).
\end{split}
\end{equation*}
\noindent
The other cases are obvious and thus the proof is completed.
\end{proof}

\begin{remark}
{\rm
We observe that $\overline{Q}_{1,k_1,0}(u,v,s)$ is the number of integer solutions $(x_{1},\ldots,x_{s})$ and $(y_{1},\ldots,y_{s})$ of

\noindent
\begin{equation*}\label{eq:1}
\begin{split}
(x_1,\ldots,x_{s})\in S_{u,s}^{k_1-1},\ \text{and}
\end{split}
\end{equation*}
\noindent
\begin{equation*}\label{eq:2}
\begin{split}
(y_{1},\ldots,y_{s}) \in S_{v,s}^{0}.
\end{split}
\end{equation*}
\noindent
given by
\noindent
\begin{equation*}\label{eq: 1.1}
\begin{split}
\overline{Q}_{1,k_1,0}(u,v,s)=R(s,\ k_{1},\ u)M(s,\ v ),
\end{split}
\end{equation*}
\noindent
where $R(a,\ b,\ c)$ denotes the total number of integer solution $y_{1}+x_{2}+\cdots+y_{a}=c$ such that $(y_{1},\ldots,y_{a}) \in S_{r,a}^{k_2}$. The number is given by
\noindent
\begin{equation*}\label{eq: 1.1}
\begin{split}
R(a,\ b,\ c)=\sum_{j=1}^{min\left(a,\ \left[\frac{c-a}{b-1}\right]\right) }(-1)^{j+1}{a \choose j}{c-j(b-1)-1 \choose a-1}.
\end{split}
\end{equation*}
\noindent
where $M(a,\ b)$ denotes the total number of integer solution $x_{1}+x_{2}+\cdots+x_{a}=b$ such that $(x_{1},\ldots,x_{a}) \in S_{b,a}^{0}$. The number is given by
\noindent
\begin{equation*}\label{eq: 1.1}
\begin{split}
M(a,\ b)={b-1 \choose a-1}.
\end{split}
\end{equation*}
\noindent
See, e.g. \citet{charalambides2002enumerative}.
}
\end{remark}

\begin{definition}
For $0<q\leq1$, we define
\begin{equation*}\label{eq: 1.1}
\begin{split}
Q_{q,k_1,k_2}(m,r,s)={\sum_{x_{1},\ldots,x_{s}}}\
\sum_{y_{1},\ldots,y_{s}} q^{y_{1}x_{1}+(y_{1}+y_{2})x_{2}+\cdots+(y_{1}+\cdots+y_{s})x_{s}},\
\end{split}
\end{equation*}
\noindent
where the summation is over all integers $x_1,\ldots,x_{s},$ and $y_1,\ldots,y_{s}$ satisfying
\begin{equation*}\label{eq:1}
\begin{split}
(x_1,\ldots,x_{s})\in S_{m,s}^{k_1-1},\ \text{and}
\end{split}
\end{equation*}
\noindent
\begin{equation*}\label{eq:2}
\begin{split}
(y_{1},\ldots,y_{s}) \in S_{r,s}^{k_2-1}.
\end{split}
\end{equation*}
\end{definition}
\noindent
The following gives a recurrence relation useful for the computation of $Q_{q,k_1,k_2}(m,r,s)$.
\noindent\\
\begin{lemma}
\label{lemma:6.10}
$Q_{q,k_1,k_2}(m,r,s)$ obeys the following recurrence relation.
\begin{equation*}\label{eq: 1.1}
\begin{split}
Q&_{q,k_1,k_2}(m,r,s)\\
&=\left\{
  \begin{array}{ll}
    \sum_{b=1}^{k_2-1}\sum_{a=1}^{k_{1}-1}q^{ra}Q_{q,k_1,k_2}(m-a,r-b,s-1)\\
    +\sum_{b=k_2}^{r-(s-1)}\sum_{a=1}^{k_1-1}q^{ra}\overline{Q}_{q,k_1,0}(m-a,r-b,s-1)\\
+\sum_{b=1}^{k_2-1}\sum_{a=k_1}^{m-(s-1)}q^{ra}J_{q,0,k_2}(m-a,r-b,s-1)\\
+\sum_{b=k_2}^{r-(s-1)}\sum_{a=k_1}^{m-(s-1)}q^{ra}\overline{J}_{q}(m-a,r-b,s-1), & \text{for}\ s>1,\ (s-1)+k_{1}\leq m\\
    &\text{and}\ (s-1)+k_2\leq r \\
    1, & \text{for}\ s=1,\ k_1\leq m,\ \text{and}\ k_2\leq r \\
    0, & \text{otherwise.}\\
  \end{array}
\right.
\end{split}
\end{equation*}
\end{lemma}
\begin{proof}
For $s > 1$, $(s-1)+k_{1}\leq m$ and $(s-1)+k_2\leq r$, we observe that $x_{s}$ may assume any value $1,\ldots,m-(s-1)$, then $Q_{q,k_1,k_2}(m,r,s)$ can be written as
\noindent
\begin{equation*}
\begin{split}
Q&_{q,k_1,k_2}(m,r,s)\\
=&\sum_{x_{s}=1}^{k_{1}-1}{\sum_{\substack{(x_1,\ldots,x_{s-1})\in S_{m-x_{s},s-1}^{k_{1}-1}}}}{\hspace{0.3cm}\sum_{\substack{(y_{1},\ldots,y_{s}) \in S_{r,s}^{k_2-1}}}}q^{rx_{s}}q^{y_{1}x_{1}+(y_{1}+y_{2})x_{2}+\cdots+(y_{1}+\cdots+y_{s-1})x_{s-1}}\\
&+\sum_{x_{s}=k_{1}}^{m-(s-1)}{\sum_{\substack{(x_1,\ldots,x_{s-1})\in S_{m-x_{s},s-1}^{0}}}}{\hspace{0.3cm}\sum_{\substack{(y_{1},\ldots,y_{s}) \in S_{r,s}^{k_2-1}}}}q^{rx_{s}}q^{y_{1}x_{1}+(y_{1}+y_{2})x_{2}+\cdots+(y_{1}+\cdots+y_{s-1})x_{s-1}}.
\end{split}
\end{equation*}
\noindent
Similarly, we observe that since $y_s$ can assume the values $1,\ldots,r-(s-1)$, then $Q_{q,k_1,k_2}(m,r,s)$ can be rewritten as
\noindent

\begin{equation*}
\begin{split}
Q&_{q}(m,r,s)\\
=&\sum_{y_{s}=1}^{k_2-1}\sum_{x_{s}=1}^{k_{1}-1}{\sum_{\substack{(x_1,\ldots,x_{s-1})\in S_{m-x_{s},s-1}^{k_{1}-1}}}}{\hspace{0.3cm}\sum_{\substack{(y_{1},\ldots,y_{s-1}) \in S_{r-y_{s},s-1}^{k_2-1}}}}q^{rx_{s}}q^{y_{1}x_{1}+(y_{1}+y_{2})x_{2}+\cdots+(y_{1}+\cdots+y_{s-1})x_{s-1}}\\
&+\sum_{y_{s}=k_2}^{r-(s-1)}\sum_{x_{s}=1}^{k_{1}-1}{\sum_{\substack{(x_1,\ldots,x_{s-1})\in S_{m-x_{s},s-1}^{k_1-1}}}}{\hspace{0.3cm}\sum_{\substack{(y_{1},\ldots,y_{s-1}) \in S_{r-y_{s},s-1}^{0}}}}q^{rx_{s}}q^{y_{1}x_{1}+(y_{1}+y_{2})x_{2}+\cdots+(y_{1}+\cdots+y_{s-1})x_{s-1}}\\
&+\sum_{y_{s}=1}^{k_2-1}\sum_{x_{s}=k_{1}}^{m-(s-1)}{\sum_{\substack{(x_1,\ldots,x_{s-1})\in S_{m-x_{s},s-1}^{0}}}}{\hspace{0.3cm}\sum_{\substack{(y_{1},\ldots,y_{s-1}) \in S_{r-y_{s},s-1}^{k_2-1}}}}q^{rx_{s}}q^{y_{1}x_{1}+(y_{1}+y_{2})x_{2}+\cdots+(y_{1}+\cdots+y_{s-1})x_{s-1}}\\
&+\sum_{y_{s}=k_2}^{r-(s-1)}\sum_{x_{s}=k_{1}}^{m-(s-1)}{\sum_{\substack{(x_1,\ldots,x_{s-1})\in S_{m-x_{s},s-1}^{0}}}}{\hspace{0.3cm}\sum_{\substack{(y_{1},\ldots,y_{s-1}) \in S_{r-y_{s},s-1}^{0}}}}q^{rx_{s}}q^{y_{1}x_{1}+(y_{1}+y_{2})x_{2}+\cdots+(y_{1}+\cdots+y_{s-1})x_{s-1}}\\
=&\sum_{b=1}^{k_2-1}\sum_{a=1}^{k_{1}-1}q^{ra}Q_{q,k_1,k_2}(m-a,r-b,s-1)+\sum_{b=k_2}^{r-(s-1)}\sum_{a=1}^{k_1-1}q^{ra}\overline{Q}_{q,k_1,0}(m-a,r-b,s-1)\\
&+\sum_{b=1}^{k_2-1}\sum_{a=k_1}^{m-(s-1)}q^{ra}J_{q,0,k_2}(m-a,r-b,s-1)+\sum_{b=k_2}^{r-(s-1)}\sum_{a=k_1}^{m-(s-1)}q^{ra}\overline{J}_{q}(m-a,r-b,s-1).
\end{split}
\end{equation*}
\noindent
The other cases are obvious and thus the proof is completed.
\end{proof}

\begin{remark}
{\rm
We observe that $Q_{1,k_1,k_2}(m,r,s)$ is the number of integer solutions $(x_{1},\ldots,x_{s})$ and $(y_{1},\ldots,y_{s})$ of

\noindent
\begin{equation*}\label{eq:1}
\begin{split}
(x_1,\ldots,x_{s})\in S_{m,s}^{k_{1}-1},\ \text{and}
\end{split}
\end{equation*}

\noindent
\begin{equation*}\label{eq:2}
\begin{split}
(y_{1},\ldots,y_{s}) \in S_{r,s}^{k_2-1}
\end{split}
\end{equation*}
\noindent
given by
\noindent
\begin{equation*}\label{eq: 1.1}
\begin{split}
Q_{1,k_1,k_2}(m,r,s)=R(s,\ k_{1},\ m)R(s,\ k_{2},\ r),
\end{split}
\end{equation*}
\noindent
where $R(a,\ b,\ c)$ denotes the total number of integer solution $y_{1}+x_{2}+\cdots+y_{a}=c$ such that $(y_{1},\ldots,y_{a}) \in S_{r,a}^{k_2}$. The number is given by
\noindent
\begin{equation*}\label{eq: 1.1}
\begin{split}
R(a,\ b,\ c)=\sum_{j=1}^{min\left(a,\ \left[\frac{c-a}{b-1}\right]\right) }(-1)^{j+1}{a \choose j}{c-j(b-1)-1 \choose a-1}.
\end{split}
\end{equation*}
}

\end{remark}

\begin{definition}
For $0<q\leq1$, we define
\begin{equation*}\label{eq: 1.1}
\begin{split}
\overline{R}_{q,0,k_2}(u,v,s)={\sum_{x_{1},\ldots,x_{s-1}}}\
\sum_{y_{1},\ldots,y_{s}} q^{y_{1}x_{1}+(y_{1}+y_{2})x_{2}+\cdots+(y_{1}+\cdots+y_{s-1})x_{s-1}},\
\end{split}
\end{equation*}
\noindent
where the summation is over all integers $x_1,\ldots,x_{s-1},$ and $y_1,\ldots,y_{s}$ satisfying

\noindent
\begin{equation*}\label{eq:1}
\begin{split}
(x_1,\ldots,x_{s-1})\in S_{u,s-1}^{0},\ \text{and}
\end{split}
\end{equation*}

\noindent
\begin{equation*}\label{eq:2}
\begin{split}
(y_{1},\ldots,y_{s}) \in S_{v,s}^{k_2-1}.
\end{split}
\end{equation*}
\end{definition}
\noindent
The following gives a recurrence relation useful for the computation of $\overline{R}_{q,0,k_2}(u,v,s)$.
\noindent\\
\begin{lemma}
\label{lemma:6.11}
$\overline{R}_{q,0,k_2}(u,v,s)$ obeys the following recurrence relation.
\begin{equation*}\label{eq: 1.1}
\begin{split}
\overline{R}&_{q,0,k_2}(u,v,s)\\
&=\left\{
  \begin{array}{ll}
    \sum_{b=1}^{k_2-1}\sum_{a=1}^{u-(s-2)}q^{a(v-b)} \overline{R}_{q,0,k_2}(u-a,v-b,s-1)\\
    +\sum_{b=k_2}^{v-(s-1)}\sum_{a=1}^{u-(s-2)}q^{a(v-b)} \overline{K}_{q}(u-a,v-b,s-1), & \text{for}\ s>1,\ s-1\leq u\ \text{and}\\ &(s-1)+k_2\leq v \\
    1, & \text{for}\ s=1,\ u=0,\ \text{and}\ k_2\leq v\\
    0, & \text{otherwise.}\\
  \end{array}
\right.
\end{split}
\end{equation*}
\end{lemma}
\begin{proof}
For $s > 1$, $s-1\leq u$ and $(s-1)+k_2\leq v$, we observe that $x_{s-1}$ may assume any value $1,\ldots,u-(s-2)$, then $\overline{R}_{q,0,k_2}(u,v,s)$ can be written as

\begin{equation*}
\begin{split}
\overline{R}&_{q,0,k_2}(u,v,s)\\
=&\sum_{x_{s-1}=1}^{u-(s-2)}{\sum_{(x_1,\ldots,x_{s-1})\in S_{u,s-1}^{0}}}{\hspace{0.3cm}\sum_{(y_{1},\ldots,y_{s}) \in S_{v,s}^{k_2-1}}}q^{x_{s-1}(v-y_{s})}q^{y_{1}x_{1}+(y_{1}+y_{2})x_{2}+\cdots+(y_{1}+\cdots+y_{s-2})x_{s-2}}
\end{split}
\end{equation*}
\noindent
Similarly, we observe that since $y_s$ can assume the values $1,\ldots,v-(s-1)$, then $\overline{R}_{q,0,k_2}(u,v,s)$ can be rewritten as
\noindent
\begin{equation*}
\begin{split}
\overline{R}&_{q,0,k_2}(u,v,s)\\
=&\sum_{y_{s}=1}^{k_2-1}\sum_{x_{s-1}=1}^{u-(s-2)}q^{x_{s-1}(v-y_{s})}{\sum_{(x_1,\ldots,x_{s-1})\in S_{u-x_{s-1},s-1}^{0}}}{\hspace{0.3cm}\sum_{(y_{1},\ldots,y_{s-1}) \in S_{v-y_{s},s-1}^{k_2-1}}}q^{y_{1}x_{1}+(y_{1}+y_{2})x_{2}+\cdots+(y_{1}+\cdots+y_{s-2})x_{s-2}}\\
&+\sum_{y_{s}=k_2}^{v-(s-1)}\sum_{x_{s-1}=1}^{u-(s-2)}q^{x_{s-1}(v-y_{s})}{\hspace{-0.9cm}\sum_{(x_1,\ldots,x_{s-1})\in S_{u-x_{s-1},s-1}^{0}}}{\hspace{0.3cm}\sum_{(y_{1},\ldots,y_{s-1}) \in S_{v-y_{s},s-1}^{0}}}q^{y_{1}x_{1}+(y_{1}+y_{2})x_{2}+\cdots+(y_{1}+\cdots+y_{s-2})x_{s-2}}\\
=&\sum_{b=1}^{k_2-1}\sum_{a=1}^{u-(s-2)}q^{a(v-b)} \overline{R}_{q,0,k_2}(u-a,v-b,s-1)+\sum_{b=k_2}^{v-(s-1)}\sum_{a=1}^{u-(s-2)}q^{a(v-b)} \overline{K}_{q}(u-a,v-b,s-1).
\end{split}
\end{equation*}
\noindent
The other cases are obvious and thus the proof is completed.
\end{proof}
\begin{remark}
{\rm
We observe that $\overline{R}_{1,0,k_2}(m,r,s)$ is the number of integer solutions $(x_{1},\ldots,x_{s-1})$ and $(y_{1},\ldots,y_{s})$ of

\noindent
\begin{equation*}\label{eq:1}
\begin{split}
(x_1,\ldots,x_{s-1})\in S_{u,s-1}^{0},\ \text{and}
\end{split}
\end{equation*}

\noindent
\begin{equation*}\label{eq:2}
\begin{split}
(y_{1},\ldots,y_{s}) \in S_{v,s}^{k_2-1}
\end{split}
\end{equation*}
\noindent
given by
\noindent
\begin{equation*}\label{eq: 1.1}
\begin{split}
\overline{R}_{1,0,k_2}(m,r,s)=M(s-1,\ u )R(s,\ k_2,\ v),
\end{split}
\end{equation*}
\noindent
where $M(a,\ b)$ denotes the total number of integer solution $y_{1}+x_{2}+\cdots+y_{a}=b$ such that $(y_{1},\ldots,y_{a}) \in S_{b,a}^{0}$. The number is given by
\noindent
\begin{equation*}\label{eq: 1.1}
\begin{split}
M(a,\ b)={b-1 \choose a-1}.
\end{split}
\end{equation*}

\noindent
where $R(a,\ b,\ c)$ denotes the total number of integer solution $y_{1}+x_{2}+\cdots+y_{a}=c$ such that $(y_{1},\ldots,y_{a}) \in S_{r,a}^{k_2}$. The number is given by
\noindent
\begin{equation*}\label{eq: 1.1}
\begin{split}
R(a,\ b,\ c)=\sum_{j=1}^{min\left(a,\ \left[\frac{c-a}{b-1}\right]\right) }(-1)^{j+1}{a \choose j}{c-j(b-1)-1 \choose a-1}.
\end{split}
\end{equation*}
\noindent
See, e.g. \citet{charalambides2002enumerative}.
}
\end{remark}

\begin{definition}
For $0<q\leq1$, we define

\begin{equation*}\label{eq: 1.1}
\begin{split}
R_{q,k_1,k_2}(m,r,s)={\sum_{x_{1},\ldots,x_{s-1}}}\sum_{y_{1},\ldots,y_{s}} q^{y_{1}x_{1}+(y_{1}+y_{2})x_{2}+\cdots+(y_{1}+\cdots+y_{s-1})x_{s-1}},\
\end{split}
\end{equation*}
\noindent
where the summation is over all integers $x_1,\ldots,x_{s-1},$ and $y_1,\ldots,y_s$ satisfying
\noindent
\begin{equation*}\label{eq:1}
\begin{split}
(x_1,\ldots,x_{s-1})\in S_{m,s-1}^{k_1-1},\ \text{and}
\end{split}
\end{equation*}
\noindent
\begin{equation*}\label{eq:2}
\begin{split}
(y_1,\ldots,y_s)\in S_{r,s}^{k_2-1}.
\end{split}
\end{equation*}
\end{definition}
\noindent
The following gives a recurrence relation useful for the computation of $R_{q,k_1,k_2}(m,r,s)$.
\noindent\\
\begin{lemma}
\label{lemma:6.12}
$R_{q,k_1,k_2}(m,r,s)$ obeys the following recurrence relation.
\begin{equation*}\label{eq: 1.1}
\begin{split}
R&_{q,k_1,k_2}(m,r,s)\\
&=\left\{
  \begin{array}{ll}
    \sum_{b=1}^{k_2-1}\sum_{a=1}^{k_1-1}q^{a(r-b)} R_{q,k_1,k_2}(m-a,r-b,s-1)\\
    +\sum_{b=k_2}^{r-(s-1)}\sum_{a=1}^{k_1-1}q^{a(r-b)} K_{q,k_1,0}(m-a,r-b,s-1)\\
+\sum_{b=1}^{k_2-1} \sum_{a=k_1}^{m-(s-2)}q^{a(r-b)} \overline{R}_{q,0,k_2}(m-a,r-b,s-1)\\
+\sum_{b=k_2}^{r-(s-1)}\sum_{a=k_1}^{m-(s-2)} q^{a(r-b)}\overline{K}_{q}(m-a,r-b,s-1) , & \text{for}\ s>1,\ (s-2)+k_1\leq m\\
    &\text{and}\ (s-1)+k_2 \leq r \\
    0, & \text{otherwise.}\\
  \end{array}
\right.
\end{split}
\end{equation*}
\end{lemma}
\begin{proof}
For $s > 1$, $(s-2)+k_1\leq m$ and $(s-1)+k_2 \leq r$, we observe that $x_{s-1}$ may assume any value $1,\ldots,m-(s-2)$, then $R_{q,k_1,k_2}(m,r,s)$ can be written as
\noindent
\begin{equation*}
\begin{split}
R&_{q,k_1,k_2}(m,r,s)\\
=&\sum_{x_{s-1}=1}^{k_1-1}{\hspace{0.3cm}\sum_{(x_1,\ldots,x_{s-2})\in S_{m-x_{s-1},s-2}^{k_1-1}}}{\hspace{0.3cm}\sum_{(y_{1},\ldots,y_{s}) \in S_{r,s}^{k_2-1}}}q^{x_{s-1}(r-y_{s})}q^{y_{1}x_{1}+(y_{1}+y_{2})x_{2}+\cdots+(y_{1}+\cdots+y_{s-2})x_{s-2}}\\
&+\sum_{x_{s-1}=k_1}^{m-(s-2)}{\hspace{0.3cm}\sum_{(x_1,\ldots,x_{s-2})\in S_{m-x_{s-1},s-2}^{0}}}{\hspace{0.3cm}\sum_{(y_{1},\ldots,y_{s}) \in S_{r,s}^{k_2-1}}}q^{x_{s-1}(r-y_{s})}q^{y_{1}x_{1}+(y_{1}+y_{2})x_{2}+\cdots+(y_{1}+\cdots+y_{s-2})x_{s-2}}.
\end{split}
\end{equation*}
\noindent
Similarly, we observe that since $y_s$ can assume the values $1,\ldots,r-(s-1)$, then $R_{q,k_1,k_2}(m,r,s)$ can be rewritten as
\noindent
\begin{equation*}
\begin{split}
R&_{q,k_1,k_2}(m,r,s)\\
=&\sum_{y_{s}=1}^{k_2-1}\sum_{x_{s-1}=1}^{k_1-1}q^{x_{s-1}(r-y_{s})}{\hspace{-0.5cm}\sum_{(x_1,\ldots,x_{s-2})\in S_{m-x_{s-1},s-2}^{k_1-1}}}{\hspace{0.3cm}\sum_{(y_{1},\ldots,y_{s-1}) \in S_{r-y_s,s-1}^{k_2-1}}}q^{y_{1}x_{1}+(y_{1}+y_{2})x_{2}+\cdots+(y_{1}+\cdots+y_{s-2})x_{s-2}}\\
&+\sum_{y_{s}=k_2}^{r-(s-1)}\sum_{x_{s-1}=1}^{k_1-1}q^{x_{s-1}(r-y_{s})}{\hspace{-0.5cm}\sum_{(x_1,\ldots,x_{s-2})\in S_{m-x_{s-1},s-2}^{k_1-1}}}{\hspace{0.3cm}\sum_{(y_{1},\ldots,y_{s-1}) \in S_{r-y_s,s-1}^{0}}}q^{y_{1}x_{1}+(y_{1}+y_{2})x_{2}+\cdots+(y_{1}+\cdots+y_{s-2})x_{s-2}}\\
&+\sum_{y_{s}=1}^{k_2-1}\sum_{x_{s-1}=k_1}^{m-(s-2)}q^{x_{s-1}(r-y_{s})}{\hspace{-0.5cm}\sum_{(x_1,\ldots,x_{s-2})\in S_{m-x_{s-1},s-2}^{0}}}{\hspace{0.3cm}\sum_{(y_{1},\ \ldots,\ y_{s-1}) \in S_{r-y_s,s-1}^{k_2-1}}}q^{y_{1}x_{1}+(y_{1}+y_{2})x_{2}+\cdots+(y_{1}+\cdots+y_{s-2})x_{s-2}}\\
&+\sum_{y_{s}=k_2}^{r-(s-1)}\sum_{x_{s-1}=k_1}^{m-(s-2)}q^{x_{s-1}(r-y_{s})}{\hspace{-0.5cm}\sum_{(x_1,\ldots,x_{s-2})\in S_{m-x_{s-1},s-2}^{0}}}{\hspace{0.3cm}\sum_{(y_{1},\ldots,y_{s-1}) \in S_{r-y_s,s-1}^{0}}}q^{y_{1}x_{1}+(y_{1}+y_{2})x_{2}+\cdots+(y_{1}+\cdots+y_{s-2})x_{s-2}}\\
=&\sum_{b=1}^{k_2-1}\sum_{a=1}^{k_1-1}q^{a(r-b)} R_{q,k_1,k_2}(m-a,r-b,s-1)+\sum_{b=k_2}^{r-(s-1)}\sum_{a=1}^{k_1-1}q^{a(r-b)} K_{q,k_1,0}(m-a,r-b,s-1)\\
&+\sum_{b=1}^{k_2-1} \sum_{a=k_1}^{m-(s-2)}q^{a(r-b)} \overline{R}_{q,0,k_2}(m-a,r-b,s-1)+\sum_{b=k_2}^{r-(s-1)}\sum_{a=k_1}^{m-(s-2)} q^{a(r-b)} \overline{K}_{q}(m-a,r-b,s-1).
\end{split}
\end{equation*}
\noindent
The other cases are obvious and thus the proof is completed.
\end{proof}
\begin{remark}
{\rm
We observe that $R_{1,k_1,k_2}(m,r,s)$ is the number of integer solutions $(x_{1},\ldots,x_{s-1})$ and $(y_{1},\ldots,y_{s})$ of

\noindent
\begin{equation*}\label{eq:1}
\begin{split}
(x_1,\ldots,x_{s-1})\in S_{m,s-1}^{k_1-1},\ \text{and}
\end{split}
\end{equation*}
\noindent
\begin{equation*}\label{eq:2}
\begin{split}
(y_{1},\ldots,y_{s}) \in S_{r,s}^{k_2-1}
\end{split}
\end{equation*}
\noindent
given by
\noindent
\begin{equation*}\label{eq: 1.1}
\begin{split}
R_{1,k_1,k_2}(m,r,s)=R(s-1,\ k_{1},\ m)R(s,\ k_{2},\ r),
\end{split}
\end{equation*}
\noindent
where $R(a,\ b,\ c)$ denotes the total number of integer solution $x_{1}+x_{2}+\cdots+x_{a}=c$ such that $(x_{1},\ldots,x_{a}) \in S_{r,a}^{k_2}$. The number is given by
\noindent
\begin{equation*}\label{eq: 1.1}
\begin{split}
R(a,\ b,\ c)=\sum_{j=1}^{min\left(a,\ \left[\frac{c-a}{b-1}\right]\right) }(-1)^{j+1}{a \choose j}{c-j(b-1)-1 \choose a-1}.
\end{split}
\end{equation*}
\noindent
See, e.g. \citet{charalambides2002enumerative}.
}
\end{remark}

\begin{definition}
For $0<q\leq1$, we define

\begin{equation*}\label{eq: 1.1}
\begin{split}
\overline{S}_{q,k_1,0}(u,v,s)={\sum_{x_{1},\ldots,x_{s}}}\
\sum_{y_{1},\ldots,y_{s-1}} q^{y_{1}x_{2}+(y_{1}+y_{2})x_{3}+\cdots+(y_{1}+\cdots+y_{s-1})x_{s}},\
\end{split}
\end{equation*}
\noindent
where the summation is over all integers $x_1,\ldots,x_{s},$ and $y_1,\ldots,y_{s-1}$ satisfying

\begin{equation*}\label{eq:1}
\begin{split}
(x_1,\ldots,x_{s})\in S_{u,s}^{k_1-1},\ \text{and}
\end{split}
\end{equation*}
\noindent
\begin{equation*}\label{eq:2}
\begin{split}
(y_{1},\ldots,y_{s-1}) \in S_{v,s-1}^{0}.
\end{split}
\end{equation*}
\end{definition}
\noindent
The following gives a recurrence relation useful for the computation of $\overline{S}_{q,k_1,0}(u,v,s)$.
\noindent\\
\begin{lemma}
\label{lemma:6.13}
$\overline{S}_{q,k_1,0}(u,v,s)$ obeys the following recurrence relation.
\begin{equation*}\label{eq: 1.1}
\begin{split}
\overline{S}&_{q,k_1,0}(u,v,s)\\
&=\left\{
  \begin{array}{ll}
   \sum_{a=1}^{k_1-1}\sum_{b=1}^{v-(s-2)}q^{va}\overline{S}_{q,k_1,0}(u-a,v-b,s-1)\\
   +\sum_{a=k_1}^{u-(s-1)}\sum_{b=1}^{v-(s-2)}q^{va}\overline{I}_{q}(u-a,v-b,s-1), & \text{for}\ s>1,\ (s-1)+k_1\leq u\ \text{and}\ (s-1)\leq v \\
    1, & \text{for}\ s=1,\ k_1\leq u,\ \text{and}\  v=0\\
    0, & \text{otherwise.}\\
  \end{array}
\right.
\end{split}
\end{equation*}
\end{lemma}
\begin{proof}
For $s > 1$, $(s-1)+k_1\leq u$ and $(s-1)\leq v$, we observe that $y_{s-1}$ may assume any value $1,\ldots,v-(s-2)$, then $\overline{S}_{q,k_1,0}(u,v,s)$ can be written as,
\noindent
\begin{equation*}
\begin{split}
\overline{S}&_{q,k_1,0}(u,v,s)\\
=&\sum_{y_{s-1}=1}^{v-(s-2)}{\hspace{0.3cm}\sum_{\substack{(x_1,\ldots,x_{s})\in S_{u,s}^{k_1-1}}}}{\hspace{0.5cm}\sum_{(y_{1},\ldots,y_{s-2}) \in S_{v-y_{s-1},s-2}^{0}}}q^{vx_{s}}q^{y_{1}x_{2}+(y_{1}+y_{2})x_{3}+\cdots+(y_{1}+\cdots+y_{s-2})x_{s-1}}.
\end{split}
\end{equation*}
\noindent
Similarly, we observe that since $x_s$ can assume the values $1,\ldots,u-(s-1)$, then $\overline{S}_{q,k_1,0}(u,v,s)$ can be rewritten as
\noindent
\begin{equation*}
\begin{split}
\overline{R}&_{q,k_1,0}(u,v,s)\\
=&\sum_{x_{s}=1}^{k_1-1}\sum_{y_{s-1}=1}^{v-(s-2)}{\hspace{0.2cm}\sum_{\substack{(x_1,\ldots,x_{s-1})\in S_{u-x_{s},s-1}^{k_1-1}}}}{\hspace{0.3cm}\sum_{(y_{1},\ldots,y_{s-2}) \in S_{v-y_{s-1},s-2}^{0}}}q^{vx_{s}}q^{y_{1}x_{2}+(y_{1}+y_{2})x_{3}+\cdots+(y_{1}+\cdots+y_{s-2})x_{s-1}}\\
&+\sum_{x_{s}=k_1}^{u-(s-1)}\sum_{y_{s-1}=1}^{v-(s-2)}{\hspace{0.2cm}\sum_{\substack{(x_1,\ldots,x_{s-1})\in S_{u-x_{s},s-1}^{0}}}}{\hspace{0.3cm}\sum_{(y_{1},\ldots,y_{s-2}) \in S_{v-y_{s-1},s-2}^{0}}}q^{vx_{s}}q^{y_{1}x_{2}+(y_{1}+y_{2})x_{3}+\cdots+(y_{1}+\cdots+y_{s-2})x_{s-1}}\\
=&\sum_{a=1}^{k_1-1}\sum_{b=1}^{v-(s-2)}q^{va}\overline{S}_{q,k_1,0}(u-a,v-b,s-1)+\sum_{a=k_1}^{u-(s-1)}\sum_{b=1}^{v-(s-2)}q^{va}\overline{I}_{q}(u-a,v-b,s-1).
\end{split}
\end{equation*}
\noindent
The other cases are obvious and thus the proof is completed.
\end{proof}

\begin{remark}
{\rm
We observe that $\overline{S}_{1,k_1,0}(u,v,s)$ is the number of integer solutions $(x_{1},\ldots,x_{s})$ and $(y_{1},\ldots,y_{s-1})$ of

\noindent
\begin{equation*}\label{eq:1}
\begin{split}
(x_1,\ldots,x_{s})\in S_{u,s}^{k_1-1},\ \text{and}
\end{split}
\end{equation*}

\noindent
\begin{equation*}\label{eq:2}
\begin{split}
(y_{1},\ldots,y_{s-1}) \in S_{v,s-1}^{0}
\end{split}
\end{equation*}
\noindent
given by
\noindent
\begin{equation*}\label{eq: 1.1}
\begin{split}
\overline{S}_{1,k_1,0}(u,v,s)=R(s,\ k_{1},\ u)M(s-1,\ v),
\end{split}
\end{equation*}

\noindent
where $R(a,\ b,\ c)$ denotes the total number of integer solution $y_{1}+x_{2}+\cdots+y_{a}=c$ such that $(y_{1},\ldots,y_{a}) \in S_{r,a}^{k_2}$. The number is given by
\noindent
\begin{equation*}\label{eq: 1.1}
\begin{split}
R(a,\ b,\ c)=\sum_{j=1}^{min\left(a,\ \left[\frac{c-a}{b-1}\right]\right) }(-1)^{j+1}{a \choose j}{c-j(b-1)-1 \choose a-1}.
\end{split}
\end{equation*}

\noindent
where $M(a,\ b)$ denotes the total number of integer solution $x_{1}+x_{2}+\cdots+x_{a}=b$ such that $(x_{1},\ldots,x_{a}) \in S_{b,a}^{0}$. The number is given by
\noindent
\begin{equation*}\label{eq: 1.1}
\begin{split}
M(a,\ b)={b-1 \choose a-1}.
\end{split}
\end{equation*}
See, e.g. \citet{charalambides2002enumerative}.
}
\end{remark}

\begin{definition}
For $0<q\leq1$, we define

\begin{equation*}\label{eq: 1.1}
\begin{split}
S_{q,k_1,k_2}(m,r,s)={\sum_{x_{1},\ldots,x_{s}}}\
\sum_{y_{1},\ldots,y_{s-1}} q^{y_{1}x_{2}+(y_{1}+y_{2})x_{3}+\cdots+(y_{1}+\cdots+y_{s-1})x_{s}},\
\end{split}
\end{equation*}
\noindent
where the summation is over all integers $x_1,\ldots,x_{s},$ and $y_1,\ldots,y_{s-1}$ satisfying

\begin{equation*}\label{eq:1}
\begin{split}
(x_1,\ldots,x_{s})\in S_{m,s}^{k_1-1},\ \text{and}
\end{split}
\end{equation*}
\noindent
\begin{equation*}\label{eq:2}
\begin{split}
(y_{1},\ldots,y_{s-1}) \in S_{r,s-1}^{k_2-1}.
\end{split}
\end{equation*}
\end{definition}
\noindent
The following gives a recurrence relation useful for the computation of $S_{q,k_1,k_2}(m,r,s)$.
\noindent\\
\begin{lemma}
\label{lemma:6.14}
$S_{q,k_1,k_2}(m,r,s)$ obeys the following recurrence relation.
\begin{equation*}\label{eq: 1.1}
\begin{split}
S&_{q,k_1,k_2}(m,r,s)\\
&=\left\{
  \begin{array}{ll}
    \sum_{b=1}^{k_2-1}\sum_{a=1}^{k_1-1}q^{ra}S_{q,k_1,k_2}(u-a,v-b,s-1)\\
    +\sum_{b=k_2}^{r-(s-2)}\sum_{a=1}^{k_1-1}q^{ra}\overline{S}_{q,k_1,0}(u-a,v-b,s-1)\\
+\sum_{b=1}^{k_2-1}\sum_{a=k_1}^{m-(s-1)}q^{ra}I_{q,0,k_2}(u-a,v-b,s-1)\\
+\sum_{b=k_2}^{r-(s-2)}\sum_{a=k_1}^{m-(s-1)}q^{ra}\overline{I}_{q}(u-a,v-b,s-1), & \text{for}\ s>1,\ (s-1)+k_1\leq m\\
    &\text{and}\ (s-2)+k_{2}\leq r\\
    1, & \text{for}\ s=1,\ k_1\leq m,\ \text{and}\ r=0\\
    0, & \text{otherwise.}\\
  \end{array}
\right.
\end{split}
\end{equation*}
\end{lemma}
\begin{proof}
For $s > 1$, $(s-1)+k_1\leq m$ and $(s-2)+k_{2}\leq r$, we observe that $x_{s}$ may assume any value $1,\ldots,m-(s-1)$, then $S_{q,k_1,k_2}(m,r,s)$ can be written as
\noindent
\begin{equation*}
\begin{split}
S&_{q,k_1,k_2}(m,r,s)\\
=&\sum_{x_{s}=1}^{k_1-1}{\hspace{0.3cm}\sum_{\substack{(x_1,\ldots,x_{s-1})\in S_{m-x_{s},s-1}^{k_1-1}}}}{\hspace{0.3cm}\sum_{(y_{1},\ldots,y_{s-1}) \in S_{r,s-1}^{k_2-1}}}q^{rx_{s}}q^{y_{1}x_{2}+(y_{1}+y_{2})x_{3}+\cdots+(y_{1}+\cdots+y_{s-2})x_{s-1}}\\
&+\sum_{x_{s}=k_1}^{m-(s-1)}{\sum_{\substack{(x_1,\ldots,x_{s-1})\in S_{m-x_{s},s-1}^{0}}}}{\hspace{0.3cm}\sum_{(y_{1},\ldots,y_{s-1}) \in S_{r,s-1}^{k_2-1}}}q^{rx_{s}}q^{y_{1}x_{2}+(y_{1}+y_{2})x_{3}+\cdots+(y_{1}+\cdots+y_{s-2})x_{s-1}}.
\end{split}
\end{equation*}
\noindent
Similarly, we observe that since $y_s$ can assume the values $1,\ldots,r-(s-2)$, then $S_{q,k_1,k_2}(m,r,s)$ can be rewritten as
\noindent

\begin{equation*}
\begin{split}
S&_{q,k_1,k_2}(m,r,s)\\
=&\sum_{y_{s-1}=1}^{k_2-1}\sum_{x_{s}=1}^{m-(s-1)}{\hspace{0.3cm}\sum_{\substack{(x_1,\ldots,x_{s-1})\in S_{m-x_{s},s-1}^{k_{1}-1}}}}{\hspace{0.3cm}\sum_{(y_{1},\ldots,y_{s-2}) \in S_{r-y_{s-1},s-2}^{k_2-1}}}q^{rx_{s}}q^{y_{1}x_{2}+(y_{1}+y_{2})x_{3}+\cdots+(y_{1}+\cdots+y_{s-2})x_{s-1}}\\
&+\sum_{y_{s-1}=k_2}^{r-(s-2)}\sum_{x_{s}=1}^{k_1-1}\hspace{0.3cm}{\sum_{\substack{(x_1,\ldots,x_{s-1})\in S_{m-x_{s},s-1}^{k_{1}-1}}}}\hspace{0.3cm}{\sum_{(y_{1},\ldots,y_{s-2}) \in S_{r-y_{s-1},s-2}^{0}}}q^{rx_{s}}q^{y_{1}x_{2}+(y_{1}+y_{2})x_{3}+\cdots+(y_{1}+\cdots+y_{s-2})x_{s-1}}\\
&+\sum_{y_{s-1}=1}^{k_2-1}\sum_{x_{s}=k_1}^{m-(s-1)}{\hspace{0.3cm}\sum_{\substack{(x_1,\ldots,x_{s-1})\in S_{m-x_{s},s-1}^{0}}}}{\hspace{0.3cm}\sum_{(y_{1},\ldots,y_{s-2}) \in S_{r-y_{s-1},s-2}^{k_2-1}}}q^{rx_{s}}q^{y_{1}x_{2}+(y_{1}+y_{2})x_{3}+\cdots+(y_{1}+\cdots+y_{s-2})x_{s-1}}\\
&+\sum_{y_{s-1}=k_2}^{r-(s-2)}\sum_{x_{s}=k_1}^{m-(s-1)}{\hspace{0.3cm}\sum_{\substack{(x_1,\ldots,x_{s-1})\in S_{m-x_{s},s-1}^{0}}}}{\hspace{0.3cm}\sum_{(y_{1},\ldots,y_{s-2}) \in S_{r-y_{s-1},s-2}^{0}}}q^{rx_{s}}q^{y_{1}x_{2}+(y_{1}+y_{2})x_{3}+\cdots+(y_{1}+\cdots+y_{s-2})x_{s-1}}\\
\end{split}
\end{equation*}

\noindent

\begin{equation*}
\begin{split}
=&\sum_{b=1}^{k_2-1}\sum_{a=1}^{k_1-1}q^{ra}S_{q,k_1,k_2}(u-a,v-b,s-1)+\sum_{b=k_2}^{r-(s-2)}\sum_{a=1}^{k_1-1}q^{ra}\overline{S}_{q,k_1,0}(u-a,v-b,s-1)\\
&+\sum_{b=1}^{k_2-1}\sum_{a=1}^{m-(s-1)}q^{ra}I_{q,0,k_2}(u-a,v-b,s-1)+\sum_{b=k_2}^{r-(s-2)}\sum_{a=1}^{m-(s-1)}q^{ra}\overline{I}_{q}(u-a,v-b,s-1).
\end{split}
\end{equation*}
\noindent
The other cases are obvious and thus the proof is completed.

\end{proof}

\begin{remark}
{\rm
We observe that $S_{1,k_1,k_2}(m,r,s)$ is the number of integer solutions $(x_{1},\ldots,x_{s})$ and $(y_{1},\ldots,y_{s-1})$ of

\noindent
\begin{equation*}\label{eq:1}
\begin{split}
(x_1,\ldots,x_{s})\in S_{m,s}^{k_1-1},\ \text{and}
\end{split}
\end{equation*}

\noindent
\begin{equation*}\label{eq:2}
\begin{split}
(y_{1},\ldots,y_{s-1}) \in S_{r,s-1}^{k_2-1}
\end{split}
\end{equation*}
\noindent
given by
\noindent
\begin{equation*}\label{eq: 1.1}
\begin{split}
S_{1,k_1,k_2}(m,r,s)=R(s,\ k_{1},\ u)R(s-1,\ k_{2},\ r),
\end{split}
\end{equation*}

\noindent

where $R(a,\ b,\ c)$ denotes the total number of integer solution $y_{1}+x_{2}+\cdots+y_{a}=c$ such that $(y_{1},\ldots,y_{a}) \in S_{r,a}^{k_2}$. The number is given by
\noindent
\begin{equation*}\label{eq: 1.1}
\begin{split}
R(a,\ b,\ c)=\sum_{j=1}^{min\left(a,\ \left[\frac{c-a}{b-1}\right]\right) }(-1)^{j+1}{a \choose j}{c-j(b-1)-1 \choose a-1}.
\end{split}
\end{equation*}
\noindent
See, e.g. \citet{charalambides2002enumerative}.
}
\end{remark}

\begin{definition}
For $0<q\leq1$, we define

\begin{equation*}\label{eq: 1.1}
\begin{split}
\overline{T}_{q}(u,v,s)={\sum_{x_{1},\ldots,x_{s-1}}}\
\sum_{y_{1},\ldots,y_{s}} q^{y_{1}x_{2}+(y_{1}+y_{2})x_{3}+\cdots+(y_{1}+\cdots+y_{s-1})x_{s}},\
\end{split}
\end{equation*}
\noindent
where the summation is over all integers $x_1,\ldots,x_{s},$ and $y_1,\ldots,y_s$ satisfying

\noindent
\begin{equation*}\label{eq:1}
\begin{split}
(x_1,\ldots,x_{s})\ \in S_{u,s}^{0},\ \text{and}
\end{split}
\end{equation*}
\noindent
\begin{equation*}\label{eq:2}
\begin{split}
(y_{1},\ldots,y_{s}) \in S_{v,s}^{k_2-1}.
\end{split}
\end{equation*}
\end{definition}
\noindent
The following gives a recurrence relation useful for the computation of $\overline{T}_{q,0,k_2}(u,v,s)$.
\noindent
\begin{lemma}
\label{lemma:6.15}
$\overline{T}_{q,0,k_2}(u,v,s)$ obeys the following recurrence relation.
\begin{equation*}\label{eq: 1.1}
\begin{split}
\overline{T}&_{q,0,k_2}(u,v,s)\\
&=\left\{
  \begin{array}{ll}
    \sum_{b=1}^{k_2-1}\sum_{a=1}^{u-(s-1)}q^{a(v-b)} \overline{T}_{q,0,k_2}(u-a,v-b,s-1)\\
    +\sum_{b=k_2}^{v-(s-1)}\sum_{a=1}^{u-(s-1)}q^{a(v-b)} \overline{L}_{q}(u-a,v-b,s-1), & \text{for}\ s>1,\ s\leq u\ \text{and}\ (s-1)+k_2\leq v \\
    1, & \text{for}\ s=1,\ 1\leq u\ \text{and}\ k_2\leq v\\
    0, & \text{otherwise.}\\
  \end{array}
\right.
\end{split}
\end{equation*}
\end{lemma}
\begin{proof}
For $s > 1$, $s\leq u$ and $(s-1)+k_2\leq v$, we observe that $x_{s}$ may assume any value $1,\ldots,u-(s-1)$, then $\overline{T}_{q,0,k_2}(u,v,s)$ can be written as
\noindent
\begin{equation*}
\begin{split}
\overline{T}&_{q,0,k_2}(u,v,s)\\
=&\sum_{x_{s}=1}^{u-(s-1)}{\sum_{(x_1,\ldots,x_{s-1})\in S_{u-x_{s},s-1}^{0}}}{\hspace{0.3cm}\sum_{(y_{1},\ldots,y_{s}) \in S_{v,s}^{k_2-1}}}q^{x_{s}(v-y_{s})}q^{y_{1}x_{2}+(y_{1}+y_{2})x_{3}+\cdots+(y_{1}+\cdots+y_{s-2})x_{s-1}}
\end{split}
\end{equation*}
\noindent
Similarly, we observe that since $y_s$ can assume the values $1,\ldots,v-(s-1)$, then $\overline{T}_{q,0,k_2}(u,v,s)$ can be rewritten as
\noindent
\begin{equation*}
\begin{split}
\overline{T}&_{q,0,k_2}(u,v,s)\\
=&\sum_{y_{s}=1}^{k_2-1}\sum_{x_{s}=1}^{u-(s-1)}q^{x_{s}(v-y_{s})}{\hspace{-0.5cm}\sum_{(x_1,\ldots,x_{s-1})\in S_{u-x_{s},s-1}^{0}}}{\hspace{0.3cm}\sum_{(y_{1},\ldots,y_{s-1}) \in S_{v-y_{s},s-1}^{k_2-1}}}q^{y_{1}x_{2}+(y_{1}+y_{2})x_{3}+\cdots+(y_{1}+\cdots+y_{s-2})x_{s-1}}\\
&+\sum_{y_{s}=k_2}^{v-(s-1)}\sum_{x_{s}=1}^{u-(s-1)}q^{x_{s}(v-y_{s})}{\hspace{-0.5cm}\sum_{(x_1,\ldots,x_{s-1})\in S_{u-x_{s},s-1}^{0}}}{\hspace{0.3cm}\sum_{(y_{1},\ldots,y_{s-1}) \in S_{v-y_{s},s-1}^{0}}}q^{y_{1}x_{2}+(y_{1}+y_{2})x_{3}+\cdots+(y_{1}+\cdots+y_{s-2})x_{s-1}}\\
=&\sum_{b=1}^{k_2-1}\sum_{a=1}^{u-(s-1)}q^{a(v-b)} \overline{T}_{q,0,k_2}(u-a,v-b,s-1)+\sum_{b=k_2}^{v-(s-1)}\sum_{a=1}^{u-(s-1)}q^{a(v-b)} \overline{L}_{q}(u-a,v-b,s-1).
\end{split}
\end{equation*}
\noindent
The other cases are obvious and thus the proof is completed.
\end{proof}
\begin{remark}
{\rm
We observe that $\overline{T}_{1,0,k_2}(m,r,s)$ is the number of integer solutions $(x_{1},\ldots,x_{s})$ and $(y_{1},\ldots,y_{s})$ of

\noindent
\begin{equation*}\label{eq:1}
\begin{split}
(x_1,\ldots,x_{s})\ \in S_{u,s}^{0},\ \text{and}
\end{split}
\end{equation*}
\noindent
\begin{equation*}\label{eq:2}
\begin{split}
(y_{1},\ldots,y_{s}) \in S_{v,s}^{k_2-1}
\end{split}
\end{equation*}
\noindent

given by
\noindent
\begin{equation*}\label{eq: 1.1}
\begin{split}
\overline{T}_{1,0,k_2}(u,v,s)=M(s,\ u)R(s,\ k_{2},\ v),
\end{split}
\end{equation*}
\noindent
where $M(a,\ b)$ denotes the total number of integer solution $y_{1}+x_{2}+\cdots+y_{a}=b$ such that $(y_{1},\ldots,y_{a}) \in S_{b,a}^{0}$. The number is given by
\noindent
\begin{equation*}\label{eq: 1.1}
\begin{split}
M(a,\ b)={b-1 \choose a-1}.
\end{split}
\end{equation*}
\noindent

where $R(a,\ b,\ c)$ denotes the total number of integer solution $y_{1}+x_{2}+\cdots+y_{a}=c$ such that $(y_{1},\ldots,y_{a}) \in S_{r,a}^{k_2}$. The number is given by
\noindent
\begin{equation*}\label{eq: 1.1}
\begin{split}
R(a,\ b,\ c)=\sum_{j=1}^{min\left(a,\ \left[\frac{c-a}{b-1}\right]\right) }(-1)^{j+1}{a \choose j}{c-j(b-1)-1 \choose a-1}.
\end{split}
\end{equation*}
\noindent
See, e.g. \citet{charalambides2002enumerative}.
}
\end{remark}

\begin{definition}
For $0<q\leq1$, we define
\begin{equation*}\label{eq: 1.1}
\begin{split}
T_{q,k_1,k_2}(m,r,s)={\sum_{x_{1},\ldots,x_{s}}}\
\sum_{y_{1},\ldots,y_{s}}q^{y_{1}x_{2}+(y_{1}+y_{2})x_{3}+\cdots+(y_{1}+\cdots+y_{s-1})x_{s}},\
\end{split}
\end{equation*}
\noindent
where the summation is over all integers $x_1,\ldots,x_{s},$ and $y_1,\ldots,y_s$ satisfying
\noindent
\begin{equation*}\label{eq:1}
\begin{split}
(x_1,\ldots,x_{s})\ \in S_{m,s}^{k_1-1},\ \text{and}
\end{split}
\end{equation*}
\noindent
\begin{equation*}\label{eq:2}
\begin{split}
(y_{1},\ldots,y_{s}) \in S_{r,s}^{k_2-1}
\end{split}
\end{equation*}
\end{definition}
\noindent
The following gives a recurrence relation useful for the computation of $T_{q,k_1,k_2}(m,r,s)$.
\noindent\\
\begin{lemma}
\label{lemma:6.16}
$T_{q,k_1,k_2}(m,r,s)$ obeys the following recurrence relation,
\begin{equation*}\label{eq: 1.1}
\begin{split}
T&_{q,k_1,k_2}(m,r,s)\\
&=\left\{
  \begin{array}{ll}
    \sum_{b=1}^{k_2-1}\sum_{a=1}^{k_{1}-1}\ q^{a(r-b)}T_{q,k_1,k_2}(m-a,r-b,s-1)\\
    +\sum_{b=k_2}^{r-(s-1)}\sum_{a=1}^{k_{1}-1}\ q^{a(r-b)}L_{q,k_1,0}(m-a,r-b,s-1)\\
+\sum_{b=1}^{k_2-1}\sum_{a=k_{1}}^{m-(s-1)}\ q^{a(r-b)}\overline{T}_{q,0,k_2}(m-a,r-b,s-1)\\
+\sum_{b=k_2}^{r-(s-1)}\sum_{a=k_{1}}^{m-(s-1)}\ q^{a(r-b)}\overline{L}_{q}(m-a,r-b,s-1), & \text{for}\quad s>1,\ (s-1)+k_1\leq m\\
    &\text{and}\ (s-1)+k_2\leq r \\
    1, & \text{for}\ s=1,\ k_1\leq m,\ \text{and}\ k_2\leq r\\
    0, & \text{otherwise.}\\
  \end{array}
\right.
\end{split}
\end{equation*}
\end{lemma}
\begin{proof}
For $s > 1$, $(s-1)+k_1\leq m$ and $(s-1)+k_2\leq r$, we observe that $x_{s}$ may assume any value $1,\ldots,m-(s-1)$, then $T_{q,k_1,k_2}(m,r,s)$ can be written as
\noindent
\begin{equation*}
\begin{split}
T&_{q,k_1,k_2}(m,r,s)\\
=&\sum_{x_{s}=1}^{k_1-1}{\hspace{0.3cm}\sum_{(x_{1},\ldots,x_{s-1}) \in S_{m-x_s,s}^{k_1-1}}}{\hspace{0.3cm}\sum_{(y_{1},\ldots,y_{s}) \in S_{r,s}^{k_2-1}}}q^{x_{s}(r-y_{s})}q^{y_{1}x_{2}+(y_{1}+y_{2})x_{3}+\cdots+(y_{1}+\cdots+y_{s-2})x_{s-1}}\\
&+\sum_{x_{s}=k_1}^{m-(s-1)}{\sum_{(x_{1},\ldots,x_{s}) \in S_{m,s}^{0}}}{\hspace{0.3cm}\sum_{(y_{1},\ldots,y_{s}) \in S_{r,s}^{k_2-1}}}q^{x_{s}(r-y_{s})}q^{y_{1}x_{2}+(y_{1}+y_{2})x_{3}+\cdots+(y_{1}+\cdots+y_{s-2})x_{s-1}}
\end{split}
\end{equation*}
\noindent
Similarly, we observe that since $y_s$ can assume the values $1,\ldots,r-(s-1)$, then $T_{q,k_1,k_2}(m,r,s)$ can be rewritten as
\noindent
\begin{equation*}
\begin{split}
T&_{q,k_1,k_2}(m,r,s)\\
=&\sum_{y_{s}=1}^{k_2-1}\sum_{x_{s}=1}^{k_{1}-1}q^{x_{s}(r-y_{s})}{\hspace{-0.3cm}\sum_{(x_{1},\ldots,x_{s-1}) \in S_{m-x_s,s}^{k_1-1}}}{\hspace{0.3cm}\sum_{(y_{1},\ldots,y_{s-1}) \in S_{r-y_{s},s-1}^{k_2-1}}} q^{y_{1}x_{2}+(y_{1}+y_{2})x_{3}+\cdots+(y_{1}+\cdots+y_{s-2})x_{s-1}}\\
&+\sum_{y_{s}=k_2}^{r-(s-1)}\sum_{x_{s}=1}^{k_{1}-1}q^{x_{s}(r-y_{s})}{\hspace{-0.3cm}\sum_{(x_{1},\ldots,x_{s-1}) \in S_{m-x_s,s}^{k_1-1}}}
{\hspace{0.3cm}\sum_{(y_{1},\ldots,y_{s-1}) \in S_{r-y_{s},s-1}^{0}}} q^{y_{1}x_{2}+(y_{1}+y_{2})x_{3}+\cdots+(y_{1}+\cdots+y_{s-2})x_{s-1}}\\
&+\sum_{y_{s}=1}^{k_2-1}\sum_{x_{s}=k_{1}}^{m-(s-1)}q^{x_{s}(r-y_{s})}{\hspace{-0.3cm}\sum_{(x_{1},\ldots,x_{s-1}) \in S_{m-x_s,s}^{0}}}{\hspace{0.3cm}\sum_{(y_{1},\ldots,y_{s-1}) \in S_{r-y_{s},s-1}^{k_2-1}}} q^{y_{1}x_{2}+(y_{1}+y_{2})x_{3}+\cdots+(y_{1}+\cdots+y_{s-2})x_{s-1}}\\
&+\sum_{y_{s}=k_2}^{r-(s-1)}\sum_{x_{s}=k_{1}}^{m-(s-1)}q^{x_{s}(r-y_{s})}{\hspace{-0.3cm}\sum_{(x_{1},\ldots,x_{s-1}) \in S_{m-x_s,s}^{0}}}{\hspace{0.3cm}\sum_{(y_{1},\ldots,y_{s-1}) \in S_{r-y_{s},s-1}^{0}}} q^{y_{1}x_{2}+(y_{1}+y_{2})x_{3}+\cdots+(y_{1}+\cdots+y_{s-2})x_{s-1}}\\
=&\sum_{b=1}^{k_2-1}\sum_{a=1}^{k_{1}-1}\ q^{a(r-b)}T_{q,0,k_2}(m-a,r-b,s-1)+\sum_{b=k_2}^{r-(s-1)}\sum_{a=1}^{k_{1}-1}\ q^{a(r-b)}L_{q,k_1,0}(m-a,r-b,s-1)\\
&+\sum_{b=1}^{k_2-1}\sum_{a=k_{1}}^{m-(s-1)}q^{a(r-b)}\overline{T}_{q,0,k_2}(m-a,r-b,s-1)+\sum_{b=k_2}^{r-(s-1)}\sum_{a=k_{1}}^{m-(s-1)}\ q^{a(r-b)} \overline{L}_{q}(m-a,r-b,s-1).
\end{split}
\end{equation*}
\noindent
The other cases are obvious and thus the proof is completed.
\end{proof}
\begin{remark}
{\rm
We observe that $T_{1,k_1,k_2}(m,r,s)$ is the number of integer solutions $(x_{1},\ldots,x_{s})$ and $(y_{1},\ldots,y_{s})$ of
\noindent
\noindent
\begin{equation*}\label{eq:1}
\begin{split}
(x_1,\ldots,x_{s})\ \in S_{m,s}^{k_1-1},\ \text{and}
\end{split}
\end{equation*}
\noindent
\begin{equation*}\label{eq:2}
\begin{split}
(y_{1},\ldots,y_{s}) \in S_{r,s}^{k_2-1}
\end{split}
\end{equation*}

\noindent
given by
\noindent
\begin{equation*}\label{eq: 1.1}
\begin{split}
T_{1,k_1,k_2}(m,r,s)=R(s,\ k_{1},\ m)R(s,\ k_{2},\ r),
\end{split}
\end{equation*}

where $R(a,\ b,\ c)$ denotes the total number of integer solution $y_{1}+x_{2}+\cdots+y_{a}=c$ such that $(y_{1},\ldots,y_{a}) \in S_{r,a}^{k_2}$. The number is given by
\noindent
\begin{equation*}\label{eq: 1.1}
\begin{split}
R(a,\ b,\ c)=\sum_{j=1}^{min\left(a,\ \left[\frac{c-a}{b-1}\right]\right)}(-1)^{j+1}{a \choose j}{c-j(b-1)-1 \choose a-1}.
\end{split}
\end{equation*}
\noindent
See, e.g. \citet{charalambides2002enumerative}.
}
\end{remark}
\noindent
The probability function of $P_{q,\theta}\Big(L_{n}^{(1)}\geq k_{1}\ \wedge\ L_{n}^{(0)}\geq k_{2}\Big)$ is obtained by the following theorem.
\noindent
\begin{theorem}
The PMF $P_{q,\theta}\Big(L_{n}^{(1)}\geq k_{1}\ \wedge\ L_{n}^{(0)}\geq k_{2}\Big)$ for $n\geq k_1+k_2$
\noindent
\begin{equation*}\label{eq:bn1}
\begin{split}
P_{q,\theta}\Big(L_{n}^{(1)}\geq k_{1}\ \wedge\ L_{n}^{(0)}\geq k_{2}\Big)=&\sum_{i=k_2}^{n-k_1}\theta^{n-i}\ {\prod_{j=1}^{i}}\ (1-\theta q^{j-1})\sum_{s=1}^{i}\Big[Q_{q,k_1,k_2}(n-i,\ i,\ s)\\
&+R_{q,k_1,k_2}(n-i,\ i,\ s)+S_{q,k_1,k_2}(n-i,\ i,\ s+1)\\
&+T_{q,k_1,k_2}(n-i,\ i,\ s)\Big].\\
\end{split}
\end{equation*}
\end{theorem}
\begin{proof}
We partition the event $\left\{L_{n}^{(1)}\geq k_{1}\ \wedge\ L_{n}^{(0)}\geq k_{2}\right\}$ into disjoint events given by $F_{n}=i,$ for $i=k_2, \ldots,\ n-k_1$. Adding the probabilities we have
\noindent
\begin{equation*}\label{eq:bn}
\begin{split}
P_{q,\theta}\Big(L_{n}^{(1)}\geq k_{1}\ \wedge\ L_{n}^{(0)}\geq k_{2}\Big)=\sum_{i=k_2}^{n-k_1}P_{q,\theta}\Big(L_{n}^{(1)}\geq k_{1}\ \wedge\ L_{n}^{(0)}\geq k_{2}\ \wedge\ F_{n}=i\Big).\\
\end{split}
\end{equation*}
\noindent
We will write $E_{n,i}=\left\{L_{n}^{(1)}\geq k_{1}\ \wedge\ L_{n}^{(0)}\geq k_{2}\ \wedge\ F_{n}=i\right\}$.
\noindent
We can now rewrite as follows
\noindent
\begin{equation*}\label{eq:bn1}
\begin{split}
P_{q,\theta}\Big(L_{n}^{(1)}\geq k_{1}\ \wedge\ L_{n}^{(0)}\geq k_{2}\Big)=\sum_{i=k_2}^{n-k_1}P_{q,\theta}\left(E_{n,i}\right).
\end{split}
\end{equation*}

We are going to focus on the event $E_{n,\ i}$. A typical element of the event $\Big\{L_{n}^{(1)}\geq k_{1}\ \wedge\ L_{n}^{(0)}\geq k_{2}\ \wedge\ F_{n}=i\Big\}$ is an ordered sequence which consists of $n-i$ successes and $i$ failures such that the length of the longest success run is greater than or equal to $k_1$ and the length of the longest failure run is greater than or equal to $k_2$. The number of these sequences can be derived as follows. First we will distribute the $i$ failures and the $n-i$ successes.  Let $s$ $(1\leq s \leq i)$ be the number of runs of failures in the event $E_{n,\ i}$. We divide into two cases: starting with a failure run or starting with a success run. Thus, we distinguish between four types of sequences in the event $\Big\{L_{n}^{(1)}\geq k_{1}\ \wedge\ L_{n}^{(0)}\geq k_{2}\ \wedge\ F_{n}=i\Big\}$, respectively named A, B, C and D-type, which are defined as follows.

\begin{equation*}\label{eq:bn}
\begin{split}
\text{A-type}\ :&\quad \overbrace{0\ldots 0}^{y_{1}}\mid\overbrace{1\ldots 1}^{x_{1}}\mid\overbrace{0\ldots 0}^{y_{2}}\mid\overbrace{1\ldots 1}^{x_{2}}\mid \ldots \mid \overbrace{0\ldots 0}^{y_{s}}\mid\overbrace{1\ldots 1}^{x_{s}},\\
\end{split}
\end{equation*}
\noindent
with $i$ $0$'s and $n-i$ $1$'s, where $x_{j}$ $(j=1,\ldots,s)$ represents a length of run of $1$'s and $y_{j}$ $(j=1,\ldots,s)$ represents the length of a run of $0$'s. And all integers $x_{1},\ldots,x_{s}$, and $y_{1},\ldots,y_{s}$ satisfy the conditions
\noindent
\begin{equation*}\label{eq:bn}
\begin{split}
(x_1,\ldots,x_{s})\in S_{n-i,s}^{k_1-1},\ \text{and}
\end{split}
\end{equation*}
\noindent
\begin{equation*}\label{eq:bn}
\begin{split}
(y_{1},\ldots,y_{s}) \in S_{i,s}^{k_2-1}.
\end{split}
\end{equation*}

\begin{equation*}\label{eq:bn}
\begin{split}
\text{B-type}\ :&\quad \overbrace{0\ldots 0}^{y_{1}}\mid\overbrace{1\ldots 1}^{x_{1}}\mid\overbrace{0\ldots 0}^{y_{2}}\mid\overbrace{1\ldots 1}^{x_{2}}\mid \ldots \mid \overbrace{0\ldots 0}^{y_{s-1}}\mid\overbrace{1\ldots 1}^{x_{s-1}}\mid\overbrace{0\ldots 0}^{y_{s}},\\
\end{split}
\end{equation*}
\noindent
with $i$ $0$'s and $n-i$ $1$'s, where $x_{j}$ $(j=1,\ldots,s-1)$ represents a length of run of $1$'s and $y_{j}$ $(j=1,\ldots,s)$ represents the length of a run of $0$'s. And all integers $x_{1},\ldots,x_{s-1}$, and $y_{1},\ldots,y_{s}$ satisfy the conditions
\noindent
\begin{equation*}\label{eq:bn}
\begin{split}
(x_1,\ldots,x_{s-1})\in S_{n-i,s-1}^{k_1-1},\ \text{and}\end{split}
\end{equation*}
\noindent
\begin{equation*}\label{eq:bn}
\begin{split}
(y_{1},\ldots,y_{s}) \in S_{i,s}^{k_2-1}.
\end{split}
\end{equation*}

\begin{equation*}\label{eq:bn}
\begin{split}
\text{C-type}\ :&\quad \overbrace{1\ldots 1}^{x_{1}}\mid\overbrace{0\ldots 0}^{y_{1}}\mid\overbrace{1\ldots 1}^{x_{2}}\mid\overbrace{0\ldots 0}^{y_{2}}\mid \ldots \mid\overbrace{0\ldots 0}^{y_{s}}\mid\overbrace{1\ldots 1}^{x_{s+1}},\\
\end{split}
\end{equation*}
\noindent
with $i$ $0$'s and $n-i$ $1$'s, where $x_{j}$ $(j=1,\ldots,s+1)$ represents a length of run of $1$'s and $y_{j}$ $(j=1,\ldots,s)$ represents the length of a run of $0$'s. And all integers $x_{1},\ldots,x_{s+1}$, and $y_{1},\ldots,y_{s}$ satisfy the conditions
\noindent
\begin{equation*}\label{eq:bn}
\begin{split}
(x_1,\ldots,x_{s+1})\in S_{n-i,s+1}^{k_1-1},\ \text{and}
\end{split}
\end{equation*}
\noindent
\begin{equation*}\label{eq:bn}
\begin{split}
(y_{1},\ldots,y_{s}) \in S_{i,s}^{k_2-1}.
\end{split}
\end{equation*}

\begin{equation*}\label{eq:bn}
\begin{split}
\text{D-type}\ :&\quad \overbrace{1\ldots 1}^{x_{1}}\mid\overbrace{0\ldots 0}^{y_{1}}\mid\overbrace{1\ldots 1}^{x_{2}}\mid\overbrace{0\ldots 0}^{y_{2}}\mid\overbrace{1\ldots 1}^{x_{3}}\mid \ldots \mid \overbrace{0\ldots 0}^{y_{s-1}}\mid\overbrace{1\ldots 1}^{x_{s}}\mid\overbrace{0\ldots 0}^{y_{s}},\\
\end{split}
\end{equation*}
\noindent
with $i$ $0$'s and $n-i$ $1$'s, where $x_{j}$ $(j=1,\ldots,s)$ represents a length of run of $1$'s and $y_{j}$ $(j=1,\ldots,s)$ represents the length of a run of $0$'s. And all integers $x_{1},\ldots,x_{s}$, and $y_{1},\ldots,y_{s}$ satisfy the conditions
\noindent
\begin{equation*}\label{eq:bn}
\begin{split}
(x_1,\ldots,x_{s})\in S_{n-i,s}^{k_1-1},\ \text{and}
\end{split}
\end{equation*}
\noindent
\begin{equation*}\label{eq:bn}
\begin{split}
(y_{1},\ldots,y_{s}) \in S_{i,s}^{k_2-1}.
\end{split}
\end{equation*}

Then the probability of the event $E_{n,\ i}$ is given by

\noindent
\begin{equation*}\label{eq: 1.1}
\begin{split}
P_{q,\theta}\Big(E_{n,\ i}\Big)\ &\\
=\sum_{s=1}^{i}\Bigg[\bigg\{&\sum_{\substack{(x_1,\ldots,x_{s})\in S_{n-i,s}^{k_1-1}}}\quad\sum_{\substack{(y_{1},\ldots,y_{s}) \in S_{i,s}^{k_2-1}}}\big(1-\theta q^{0}\big)\cdots \big(1-\theta q^{y_{1}-1}\big)\Big(\theta q^{y_{1}}\Big)^{x_{1}}\times\\
&\big(1-\theta q^{y_{1}}\big)\cdots \big(1-\theta q^{y_{1}+y_{2}-1}\big)\Big(\theta q^{y_{1}+y_{2}}\Big)^{x_{2}}\times\\
&\big(1-\theta q^{y_{1}+y_2}\big)\cdots \big(1-\theta q^{y_{1}+y_{2}+y_3-1}\big)\Big(\theta q^{y_{1}+y_{2}+y_3}\Big)^{x_{3}}\times \\
&\quad \quad \quad \quad\quad \quad\quad \quad \quad \quad \quad \vdots\\
&\big(1-\theta q^{y_{1}+\cdots+y_{s-1}}\big)\cdots \big(1-\theta q^{y_{1}+\cdots+y_{s}-1}\big)\Big(\theta q^{y_{1}+\cdots+y_{s}}\Big)^{x_{s}}\bigg\}+\\
\end{split}
\end{equation*}

\noindent
\begin{equation*}\label{eq: 1.1}
\begin{split}
\quad \  &\bigg\{\sum_{\substack{(x_1,\ldots,x_{s-1})\in S_{n-i,s-1}^{k_1-1}}}\quad
\sum_{\substack{(y_{1},\ldots,y_{s}) \in S_{i,s}^{k_2-1}}} \big(1-\theta q^{0}\big)\cdots \big(1-\theta q^{y_{1}-1}\big)\times\\
&\Big(\theta q^{y_{1}}\Big)^{x_{1}}\big(1-\theta q^{y_{1}}\big)\cdots \big(1-\theta q^{y_{1}+y_{2}-1}\big)\times\\
&\Big(\theta q^{y_{1}+y_{2}}\Big)^{x_{2}}\big(1-\theta q^{y_{1}+y_2}\big)\cdots \big(1-\theta q^{y_{1}+y_{2}+y_3-1}\big)\times \\
&\quad \quad \quad \quad\quad \quad\quad \quad \quad \quad \quad \vdots\\
&\Big(\theta q^{y_{1}+\cdots+y_{s-1}}\Big)^{x_{s-1}}
\big(1-\theta q^{y_{1}+\cdots+y_{s-1}}\big)\cdots \big(1-\theta q^{y_{1}+\cdots+y_{s}-1}\big)\bigg\}+\\
\end{split}
\end{equation*}

\noindent
\begin{equation*}\label{eq: 1.1}
\begin{split}
\quad \quad \quad \quad &\bigg\{\sum_{\substack{(x_1,\ldots,x_{s+1})\in S_{n-i,s+1}^{k_1-1}}}\quad\sum_{\substack{(y_{1},\ldots,y_{s}) \in S_{i,s}^{k_2-1}}} \big(\theta q^{0}\big)^{x_{1}}\big(1-\theta q^{0}\big)\cdots (1-\theta q^{y_{1}-1})\times\\
&\Big(\theta q^{y_{1}}\Big)^{x_{2}}\big(1-\theta q^{y_{1}}\big)\cdots \big(1-\theta q^{y_{1}+y_{2}-1}\big)\times\\
&\Big(\theta q^{y_{1}+y_{2}}\Big)^{x_{3}}\big(1-\theta q^{y_{1}+y_2}\big)\cdots \big(1-\theta q^{y_{1}+y_{2}+y_3-1}\big)\times \\
&\quad \quad \quad \vdots\\
&\Big(\theta q^{y_{1}+\cdots+y_{s}}\Big)^{x_{s+1}}\bigg\}+\\
\end{split}
\end{equation*}

\noindent
\begin{equation*}\label{eq: 1.1}
\begin{split}
\quad \quad &\bigg\{\sum_{\substack{(x_1,\ldots,x_{s})\in S_{n-i,s}^{k_1-1}}}\quad\sum_{\substack{(y_{1},\ldots,y_{s}) \in S_{i,s}^{k_2-1}}}\big(\theta q^{0}\big)^{x_{1}}\big(1-\theta q^{0}\big)\cdots (1-\theta q^{y_{1}-1})\times\\
&\Big(\theta q^{y_{1}}\Big)^{x_{2}}\big(1-\theta q^{y_{1}+y_2}\big)\cdots \big(1-\theta q^{y_{1}+y_{2}-1}\big)\times\\
&\Big(\theta q^{y_{1}+y_{2}}\Big)^{x_{3}}\big(1-\theta q^{y_{1}}\big)\cdots \big(1-\theta q^{y_{1}+y_{2}+y_3-1}\big)\times \\
&\quad \quad \quad \quad\quad \quad\quad \quad \quad \quad \quad \vdots\\
&\Big(\theta q^{y_{1}+\cdots+y_{s-1}}\Big)^{x_{s}}
\big(1-\theta q^{y_{1}+\cdots+y_{s-1}}\big)\cdots \big(1-\theta q^{y_{1}+\cdots+y_{s}-1}\big)\bigg\}\Bigg].\\
\end{split}
\end{equation*}

Using simple exponentiation algebra arguments to simplify,
\noindent
\begin{equation*}\label{eq: 1.1}
\begin{split}
&P_{q,\theta}\Big(E_{n,i}\Big)=\\
&\theta^{n-i}{\prod_{j=1}^{i}}\ (1-\theta q^{j-1})\times\\
&\sum_{s=1}^{i}\Bigg[\sum_{\substack{(x_1,\ldots,x_{s})\in S_{n-i,s}^{k_1-1}}}\quad\sum_{\substack{(y_{1},\ldots,y_{s}) \in S_{i,s}^{k_2-1}}}q^{y_{1}x_{1}+(y_{1}+y_{2})x_{2}+\cdots+(y_{1}+\cdots+y_{s})x_{s}}+\\
&\quad \ \sum_{\substack{(x_1,\ldots,x_{s-1})\in S_{n-i,s-1}^{k_1-1}}}\quad\sum_{\substack{(y_{1},\ldots,y_{s}) \in S_{i,s}^{k_2-1}}}q^{y_{1}x_{1}+(y_{1}+y_{2})x_{2}+\cdots+(y_{1}+\cdots+y_{s-1})x_{s-1}}+\\
&\quad \ \sum_{\substack{(x_1,\ldots,x_{s+1})\in S_{n-i,s+1}^{k_1-1}}}\quad\sum_{\substack{(y_{1},\ldots,y_{s}) \in S_{i,s}^{k_2-1}}}q^{y_{1}x_{2}+(y_{1}+y_{2})x_{3}+\cdots+(y_{1}+\cdots+y_{s})x_{s+1}}+\\
&\quad \quad \sum_{\substack{(x_1,\ldots,x_{s})\in S_{n-i,s}^{k_1-1}}}\quad\sum_{\substack{(y_{1},\ldots,y_{s}) \in S_{i,s}^{k_2-1}}}q^{y_{1}x_{2}+(y_{1}+y_{2})x_{3}+\cdots+(y_{1}+\cdots+y_{s-1})x_{s}}\Bigg].\\
\end{split}
\end{equation*}
\noindent
Using Lemma \ref{lemma:6.10}, Lemma \ref{lemma:6.12}, Lemma \ref{lemma:6.14} and Lemma \ref{lemma:6.16}, we can rewrite as follows.
\noindent
\begin{equation*}\label{eq: 1.1}
\begin{split}
P_{q,\theta}\Big(E_{n,i}\Big)=\theta^{n-i}\ {\prod_{j=1}^{i}}\ (1-\theta q^{j-1})\sum_{s=1}^{i}&\Big[Q_{q,k_1,k_2}(n-i,\ i,\ s)+R_{q,k_1,k_2}(n-i,\ i,\ s)\\
&+S_{q,k_1,k_2}(n-i,\ i,\ s+1)+T_{q,k_1,k_2}(n-i,\ i,\ s)\Big],\\
\end{split}
\end{equation*}\\
\noindent
where
\noindent
\begin{equation*}\label{eq: 1.1}
\begin{split}
Q_{q,k_1,k_2}(n&-i,\ i,\ s)\\
&=\sum_{\substack{(x_1,\ldots,x_{s})\in S_{n-i,s}^{k_1-1}}}\quad
\sum_{\substack{(y_{1},\ldots,y_{s}) \in S_{i,s}^{k_2-1}}}q^{y_{1}x_{1}+(y_{1}+y_{2})x_{2}+\cdots+(y_{1}+\cdots+y_{s})x_{s}},
\end{split}
\end{equation*}
\noindent
\begin{equation*}\label{eq: 1.1}
\begin{split}
R_{q,k_1,k_2}(n&-i,\ i,\ s)\\
&=\sum_{\substack{(x_1,\ldots,x_{s-1})\in S_{n-i,s-1}^{k_1-1}}}\quad
\sum_{\substack{(y_{1},\ldots,y_{s}) \in S_{i,s}^{k_2-1}}} q^{y_{1}x_{1}+(y_{1}+y_{2})x_{2}+\cdots+(y_{1}+\cdots+y_{s-1})x_{s-1}},
\end{split}
\end{equation*}
\noindent
\begin{equation*}\label{eq: 1.1}
\begin{split}
S_{q,k_1,k_2}(n&-i,\ i,\ s+1)\\
&=\sum_{\substack{(x_1,\ldots,x_{s+1})\in S_{n-i,s+1}^{k_1-1}}}\quad\sum_{\substack{(y_{1},\ldots,y_{s}) \in S_{i,s}^{k_2-1}}}q^{y_{1}x_{2}+(y_{1}+y_{2})x_{3}+\cdots+(y_{1}+\cdots+y_{s})x_{s+1}},
\end{split}
\end{equation*}
\noindent
and
\noindent
\begin{equation*}\label{eq: 1.1}
\begin{split}
T_{q,k_1,k_2}(n&-i,\ i,\ s)\\
&=\sum_{\substack{(x_1,\ldots,x_{s})\in S_{n-i,s}^{k_1-1}}}\quad
\sum_{\substack{(y_{1},\ldots,y_{s}) \in S_{i,s}^{k_2-1}}} q^{y_{1}x_{2}+(y_{1}+y_{2})x_{3}+\cdots+(y_{1}+\cdots+y_{s-1})x_{s}}.
\end{split}
\end{equation*}
\noindent
Therefore we can compute the probability of the event $\left\{L_{n}^{(1)}\geq k_{1}\ \wedge\ L_{n}^{(0)}\geq k_{2}\right\}$ as follows
\noindent
\begin{equation*}\label{eq:bn1}
\begin{split}
P_{q,\theta}\Big(L_{n}^{(1)}\geq k_{1}\ \wedge\ L_{n}^{(0)}\geq k_{2}\Big)=&\sum_{i=k_2}^{n-k_1}\theta^{n-i}\ {\prod_{j=1}^{i}}\ (1-\theta q^{j-1})\sum_{s=1}^{i}\Big[Q_{q,k_1,k_2}(n-i,\ i,\ s)\\
&+R_{q,k_1,k_2}(n-i,\ i,\ s)+S_{q,k_1,k_2}(n-i,\ i,\ s+1)\\
&+T_{q,k_1,k_2}(n-i,\ i,\ s)\Big].\\
\end{split}
\end{equation*}
\noindent
Thus proof is completed.
\end{proof}
\noindent
It is worth mentioning here that the PMF $P_{q,\theta}\Big(L_{n}^{(1)}\geq k_{1}\ \wedge\ L_{n}^{(0)}\geq k_{2}\Big)$ for $n\geq k_1+k_2$ approaches the probability function of $P_{\theta}\Big(L_{n}^{(1)}\geq k_{1}\ \wedge\ L_{n}^{(0)}\geq k_{2}\Big)$ in the limit as $q$ tends to 1 of IID model. The details are presented in the following remark.
\noindent
\begin{remark}
{\rm For $q=1$, the PMF $P_{q,\theta}\Big(L_{n}^{(1)}\geq k_{1}\ \wedge\ L_{n}^{(0)}\geq k_{2}\Big)$ for $n\geq k_1+k_2$ reduces to the PMF $P_{\theta}\Big(L_{n}^{(1)}\geq k_{1}\ \wedge\ L_{n}^{(0)}\geq k_{2}\Big)$ for $n\geq k_1+k_2$ is given by
\noindent
\begin{equation*}\label{eq:bn1}
\begin{split}
P\Big(L_{n}^{(1)}\geq k_{1}\ \wedge\ L_{n}^{(0)}\geq k_{2}\Big)=&\sum_{i=k_2}^{n-k_1}\theta^{n-i}\ {\prod_{j=1}^{i}}\ (1-\theta q^{j-1})\sum_{s=1}^{i}\Big[R(s,n-i,\ k_1,\ n-i)\\
&+R(s-1,\ k_1,\ n-i)+R(s+1,\ k_1,\ n-i)\\
&+R(s,\ k_1,\ n-i)\Big]R(s,\ k_2,\ i).\\
\end{split}
\end{equation*}
where
\noindent
\begin{equation*}\label{eq: 1.1}
\begin{split}
R(a,\ b,\ c)=\sum_{j=1}^{min\left(a,\ \left[\frac{c-a}{b-1}\right]\right)}(-1)^{j+1}{a \choose j}{c-j(b-1)-1 \choose a-1}.
\end{split}
\end{equation*}
\noindent
See, e.g. \citet{charalambides2002enumerative}.
}
\end{remark}

\bibliography{biblio}
\addcontentsline{toc}{section}{References}
\bibliographystyle{apalike}

\end{document}